\documentclass[11pt]{amsart}
\usepackage{marginnote}

\usepackage{geometry}
\usepackage{mathrsfs}
\usepackage{mathrsfs} 
\PassOptionsToPackage{linktocpage}{hyperref}

\usepackage[utf8]{inputenc}
\usepackage[T1]{fontenc}
\usepackage[all]{xy}
\usepackage{calc}

\usepackage{hyperref}
\hypersetup{
           breaklinks=true,   
        }

\usepackage{cleveref}

\usepackage{amsmath,amstext,amsthm,amssymb,amscd,version, mathrsfs}
\usepackage{xcolor}
\usepackage{comment}
\usepackage{tikz}
\usepackage{tikz-cd}

\tikzset{
    labl/.style={anchor=south, rotate=135, inner sep=.5mm}
}
\tikzcdset{scale cd/.style={every label/.append style={scale=#1},
    cells={nodes={scale=#1}}}}

\newtheorem{thm}{Theorem}[section]
\newtheorem*{thm*}{Theorem}
\newtheorem*{claim}{Claim}
\newtheorem*{subclaim}{Sub-Claim}

\newtheorem{lem}[thm]{Lemma}
\newtheorem{prop}[thm]{Proposition}
\newtheorem{cor}[thm]{Corollary}

\theoremstyle{definition}
\newtheorem{exa}[thm]{Example}
\newtheorem{defn}[thm]{Definition}
\newtheorem{setup}[thm]{Setup}
\newtheorem{rem}[thm]{Remark}
\newtheorem{notn}[thm]{Notation}

\def\bA{\mathbb{A}}
\def\bB{\mathbb{B}}
\def\bD{\mathbb{D}}
\def\bH{\mathbb{H}}
\def\bN{\mathbb{N}}
\def\bZ{\mathbb{Z}}
\def\bQ{\mathbb{Q}}
\def\bR{\mathbb{R}}
\def\bC{\mathbb{C}}
\def\bP{\mathbb{P}}
\def\bL{\mathbb{L}}

\def\bG{\mathbf{G}}
\def\G{\mathbf{G}}
\def\H{\mathbf{H}}
\def\bT{\mathbf{T}}
\def\T{\mathbf{T}}
\def\bfV{\mathbf{V}}
\def\cA{\mathcal{A}}

\def\cC{\mathcal{C}}
\def\cD{\mathcal{D}}
\def\cE{\mathcal{E}}
\def\cF{\mathcal{F}}

\def\cL{\mathcal{L}}
\def\cM{\mathcal{M}}
\def\cN{\mathcal{N}}
\def\cO{\mathcal{O}}

\def\cS{\mathcal{S}}
\def\cT{\mathcal{T}}
\def\cU{\mathcal{U}}
\def\cV{\mathcal{V}}
\def\cW{\mathcal{W}}
\def\cX{\mathcal{X}}
\def\cY{\mathcal{Y}}
\def\cZ{\mathcal{Z}}

\def\m{\mathfrak{m}}
\def\ft{\mathfrak{t}}
\def\fm{\mathfrak{m}}
\def\fa{\mathfrak{a}}
\def\fo{\mathfrak{o}}

\def\scrX{\mathscr{X}}

\def\scrM{\mathscr{M}}
\def\scrT{\mathscr{T}}
\def\scrL{\mathscr{L}}
\def\scrE{\mathscr{E}}
\def\twR{\mathscr{R}}
\def\twF{\mathscr{F}}

\def\twZ{\mathscr{Z}}

\def\scrV{\mathscr{V}}

\def\scrA{\mathscr{A}}
\def\scrD{\mathscr{D}}

\def\Zar{\mathrm{Zar}} 
\def\Sing{\mathrm{Sing}} 
\def\Reg{\mathrm{Reg}} 
\def\CVHS{\bC\mhyphen\mathrm{VHS}}
\def\df{\mathrm{def}}
\def\bGL{\mathbf{GL}}
\def\an{\mathrm{an}}
\def\cHom{\mathcal{H}\hspace{-.2em}\operatorname{om}}
\def\cEnd{\mathcal{E}\hspace{-.2em}\operatorname{nd}}
\def\sing{\mathrm{sing}}
\def\nilp{\mathrm{nilp}}
\def\unip{\mathrm{unip}}
\def\Hod{\mathrm{Hod}}
\def\img{\operatorname{img}}
\def\qu{\mathrm{qu}}
\def\ss{\mathrm{ss}}
\def\Sh{Sh}

\def\gr{\operatorname{gr}}

\def\tate{\bT(0)}
\def\id{\operatorname{id}}
\def\zero{\mathbf{0}}

\mathchardef\mhyphen="2D
\def\CMHS{\bC\mhyphen\mathrm{MHS}}
\def\MHS{\mathrm{MHS}}
\def\MS{\mathrm{MS}}
\def\CMS{\bC\mhyphen\mathrm{MS}}
    \def\MTS{\mathrm{MTS}}
\def\MHM{\operatorname{MHM}}
\def\DR{\operatorname{DR}}
\def\HM{\operatorname{HM}}
\def\MF{\operatorname{MF}}

\def\MFW{\operatorname{MFW}}
\def\LS{\operatorname{LS}}
\def\MM{\operatorname{MM}}

\def\tpi{\mathrm{tpi}}
\def\MTM{\operatorname{MTM}^\tpi}
\def\Hol{\operatorname{Hol}}
\def\sp{\mathrm{sp}}
\def\pt{\mathrm{pt}}
\def\Art{\mathrm{Art}}
\def\Aff{\mathrm{Aff}}
\def\Der{\operatorname{Der}}
\def\op{\mathrm{op}}

\def\obs{\mathrm{obs}}

\def\MSArt{\mathrm{MS}\mhyphen\Art}
\def\Res{\operatorname{Res}}
\def\op{\mathrm{op}}

\def\triv{\mathrm{triv}}
\def\Exp{\operatorname{Exp}}

\def\reg{\mathrm{reg}}
\def\good{\mathrm{good}}
\def\barX{\bar X}
\def\lnabla{\nabla}
\def\loc{\mathrm{loc}}
\def\qun{\qu|n}
\def\uSpec{\underline{\Spec}\,}

\def\Sect{\operatorname{Sect}}
\def\ev{\operatorname{ev}}
\def\PrefSect{\operatorname{PrefSect}}
\def\bsV{V}
\def\bsphi{\phi}
\def\bss{s}    
\def\bsVp{V'}

\def\bsphip{\phi'} 
\def\sat{\mathrm{sat}}

\def\form{E}
\def\codim{\operatorname{codim}}
\def\sing{\mathrm{sing}}

\def\cHom{\mathcal{H}\hspace{-.2em}\operatorname{om}}
\def\cEnd{\mathcal{E}\hspace{-.2em}\operatorname{nd}}
\def\sing{\mathrm{sing}}

\def\CVHS{\bC-\mathrm{VHS}}

\DeclareMathOperator{\Aut}{Aut}

\DeclareMathOperator{\End}{End}
\DeclareMathOperator{\Hom}{Hom}
\DeclareMathOperator{\Ext}{Ext}
\DeclareMathOperator{\rk}{rk}

\DeclareMathOperator{\ad}{ad}

\DeclareMathOperator{\GL}{GL}
\DeclareMathOperator{\Sym}{Sym}
\DeclareMathOperator{\sh}{sh}
\DeclareMathOperator{\Alb}{Alb}

\DeclareMathOperator{\ra}{\rightarrow}

\DeclareMathOperator{\Spec}{Spec}

\DeclareMathOperator{\N}{||}
\DeclareMathOperator{\dvol}{dvol}

\newcommand{\newpar}[1]{\subsection{\texorpdfstring{}{}}}
\newcommand{\parref}[1]{\hyperref[#1]{\S\ref*{#1}}}
\newcommand{\chapref}[1]{\hyperref[#1]{Chapter~\ref*{#1}}}

\counterwithin{equation}{subsection}

\newcounter{intro}

\newtheorem{intro-conjecture}[intro]{Conjecture}
\newtheorem{intro-corollary}[intro]{Corollary}
\newtheorem{intro-theorem}[intro]{Theorem}

\usepackage{enumitem}

\title[The linear Shafarevich conjecture for quasiprojective varieties]{The linear Shafarevich conjecture for quasiprojective varieties and algebraicity of Shafarevich morphisms}
\date{\today}
 \author[B. Bakker]{Benjamin Bakker}
\address{\noindent B. Bakker:  Dept. of Mathematics, Statistics, and Computer Science, University of Illinois at Chicago, Chicago, USA.}
\email{bakker.uic@gmail.com}

\author[Y. Brunebarbe]{Yohan Brunebarbe}
\address{\noindent Y. Brunebarbe: CNRS, Universit\'e de Bordeaux, Talence, France.}
\email{yohan.brunebarbe@math.u-bordeaux.fr}

\author[J. Tsimerman]{Jacob Tsimerman}
\address{\noindent J. Tsimerman:  Dept. of Mathematics, University of Toronto, Toronto, Canada.}
\email{jacobt@math.toronto.edu}

\begin{document}

\begin{abstract}
    We prove that the universal cover of a normal complex algebraic variety admitting a faithful complex representation of its fundamental group is an analytic Zariski open subset of a holomorphically convex complex space.  This is a non-proper version of the Shafarevich conjecture.  More generally we define a class of subset of the Betti stack for which the covering space trivializing the corresponding local systems has this property.  Secondly, we show that for any complex local system $V$ on a normal complex algebraic variety $X$ there is an algebraic map $f:X\to Y$ contracting precisely the subvarieties on which $V$ is isotrivial.
\end{abstract}

\maketitle
\setcounter{tocdepth}{1}
\tableofcontents



\section{Introduction}

In an attempt to describe which analytic varieties arise as the universal covers of complex algebraic varieties, Shafarevich asked \cite[IX.4.3]{shafarevich} whether the universal cover $\tilde X$ of a smooth projective variety $X$ is always holomorphically convex, meaning that $\tilde X$ admits a proper holomorphic map to a Stein space.  Stein spaces are the analytic analog of affine schemes; more precisely, they are characterized by the existence of a finite holomorphic mapping to some $\bC^n$.  This is a difficult question in general, but the following special case suggests that Hodge theory might be used to answer it.  If $X$ supports a variation of integral pure Hodge structures, then the corresponding period map $X^\an\to \bG(\bZ)\backslash D$ (possibly after partially compactifying $X$) factors as a proper algebraic map $X\to Y$ followed by a closed embedding $Y^\an\hookrightarrow \bG(\bZ)\backslash D$ \cite{bbt}.  Stein spaces can also be characterized by the existence of a strictly plurisubharmonic exhaustion function; in this case, at least when the image is smooth (or in general with some argument), such an exhaustion function can be obtained by restricting an appropriate function on the period domain $D$ to any component of the inverse image of $Y$ in $D$.  It follows that the universal cover $\tilde Y$ is Stein and the base-change $X^\an\times_{Y^\an}\tilde Y$ is holomorphically convex.

 The Shafarevich question was first taken up for surfaces by Napier \cite{napier} and Gurjar--Shastri \cite{gurjar}.  A general approach was investigated by Campana \cite{campana94} and Koll\'ar \cite{kollar93,kollar95}, who proved the existence of a rational Shafarevich map, see below.  A  different strategy using techniques from non-abelian Hodge theory was developed by Katzarkov, Ramachandran, and Eyssidieux \cite{Knilpotent, KRsurfaces, Eyssidieux} using ideas from Corlette and Simpson \cite{Corlette92,simpsonhiggs,Corlette-Simpson}, Gromov--Schoen \cite{Gromov-Schoen}, Mok \cite{Mok_factorization}, and Zuo \cite{Zuo96}.  This line of attack culminated in the proof of Eyssidieux--Katzarkov--Pantev--Ramachandran \cite{EKPR} that a smooth projective variety admitting an almost faithful representation of its fundamental group has holomorphically convex universal cover.  Subsequent developments have been achieved in \cite{mok_stein,campanareps,eyssidieuxkahler,liuconvexity}.

Allowing $X$ to be quasiprojective makes the theory substantially more difficult due to the presence of the boundary.  The extension of the archimedean harmonic theory has been worked out by Simpson, Biquard, Sabbah, and Mochizuki \cite{Simpson_noncompact, Biquard,Sabbah_twistor_D_modules,Mochizuki-AMS2,mochizukimixedtwistor}, and more recently the nonarchimedean harmonic theory has been generalized by Brotbek, Daskalopoulos, Deng, and Mese \cite{Daskalopoulos-Mese, BDDM}.  This has led to a number of recent developments in this setting \cite{aguilar_campana, Green_Griffiths_Katzarkov, brunebarbeshaf, DengShaf}.

Our first main result is the following quasiprojective version of the linear Shafarevich conjecture.

\begin{thm}\label{main baby}
Let $X$ be a connected normal algebraic space whose fundamental group admits an almost faithful finite-dimensional complex linear representation.  Then there is a partial compactification $X\subset \bar X$ by a connected normal Deligne--Mumford stack\footnote{There will be a finite \'etale cover of $\bar X$ which is an algebraic space.  In fact, all of our results generalize immediately to the setting where $X$ is a stack quotient of a normal algebraic space by a faithful action of a finite group, and this is the most natural setting.  Such stacks are considered by \cite{eyssidieuxkahler}.} with almost isomorphic fundamental group such that the universal cover of $\bar X$ is a holomorphically convex complex space.  In particular, the universal cover of $X$ is a dense Zariski open subset of a holomorphically convex complex space.
\end{thm}

Here by an almost faithful representation we mean a representation $\rho:\pi_1(X^\an,x)\to\bGL_r(\bC)$ with finite kernel, and by the fundamental groups of $X\subset \bar X$ being almost isomorphic we mean $\pi_1(X^\an,x)\to\pi_1(X^{\prime\an},x)$ has finite kernel and cokernel, although the cokernel is automatically trivial since $\bar X$ is normal.  Passing to a partial compactification in \Cref{main baby} may be necessary for trivial reasons:  $\bA^2\setminus\{0\}$ is simply connected and not holomorphically convex. 

As a simple example of \Cref{main baby} we have the following consequence of \Cref{corDMpullback} and \Cref{main cor 2} below, but which we state now for concreteness:  if the fundamental group of $X$ admits a complex linear representation with infinite image, then the universal cover of $X$ admits a nonconstant global holomorphic function.

In fact we prove a much more precise statement.  For a connected normal algebraic space $X$, let $\cM_B(X)$ be the stack of local systems on $X^\an$, so that $\cM_B(X)(\bC)$ is the groupoid of complex local systems.  For any set $\Sigma\subset\cM_B(X)(\bC)$ of (isomorphism classes of) complex local systems, we denote by $\tilde X^\Sigma\to X^\an$ the minimal cover on which all the local systems in $\Sigma$ are trivialized.

\begin{thm}\label{main result}
Let $X$ be a connected normal complex algebraic space and $\Sigma \subset \cM_B(X)(\bC)$ a nonextendable absolute Hodge subset.  Then the complex analytic space $\tilde{X}^{\Sigma}$ is holomorphically convex.
\end{thm}

Note that this improves on the result of Eyssidieux--Katzarkov--Pantev--Ramachandran in the case of compact $X$ as it allows us to take $\Sigma\subsetneq \cM_B(X)(\bC)$.  Absolute Hodge subsets have a good functorial behavior (see \Cref{basic prop abs hodge substack}) and in particular include all subsets that can be made with functorial operations (intersections, unions, $\bQ$-irreducible components, inverse-images under pull-back along arbitrary $f \colon X\to Y$, images under pull-back along dominant $f \colon X\to Y$ or the inclusion of a Lefschetz curve) starting from all of $\cM_B(X,r)(\bC)$ (that is, local systems of rank $r$) or a trivial local system.  The nonextendability hypothesis says that for any partial compactification $X\subset \bar X$ by a connected normal Deligne--Mumford stack, some element of $\Sigma$ does not extend.  As noted above, this condition is necessary, and can always be achieved by replacing $X$ with a partial compactification to which $\Sigma$ extends and is nonextendable.  The ``absolute Hodge'' condition should be thought of as saying $\Sigma$ is a non-abelian Hodge substructure of $\cM_B(X)$, and that it is algebraic in the Betti stack of every Galois conjugate of $X$.  In particular, $\cM_B(X, r)(\bC)$ for every $r\geq 1$ itself is an absolute Hodge subset.  This condition is important for the proof, but not clearly necessary.  Note that some condition is needed however, since there exist rank one local systems on abelian varieties whose trivializing covers are not holomorphically convex.

\Cref{main result} implies (by taking Stein factorization) that $\tilde X^\Sigma$ admits a proper surjective holomorphic map with connected fibers to a normal Stein space $\tilde Y$ which is unique (called the Cartan-Remmert reduction).  The fundamental group acts properly discontinuously on $\tilde Y$, and so the map descends to a proper surjective analytic map $X^\an\to\cY $ (with connected fibers) which contracts precisely those subvarieties on which the restriction of the local systems in $\Sigma$ have uniformly finite monodromy. 
 This observation led Campana and Koll\'ar to introduce the notion of a Shafarevich morphism.  Precisely, given any subset $\Sigma\subset \cM_B(X)(\bC)$, an algebraic $\Sigma$-Shafarevich morphism is a morphism $s \colon X\to Y$ to a generically inertia-free connected normal Deligne--Mumford stack $Y$ such that:
\begin{enumerate}
\item
$s \colon X\to Y$ is dominant and $K(Y)$ is algebraically closed in $K(X)$.
\item $\Sigma$ is the pull back of a large nonextendable $\Sigma_Y\subset\cM_B(Y)(\bC)$ and for every point $y\in Y(\bC)$ the inertia of $y$ acts faithfully on $\bigoplus_{V\in \Sigma_Y} i^*_yV$. 
\item If a morphism $g \colon Z\to X$ from a connected $Z$ has the property that $g^*V$ is trivial for every $V\in\Sigma$, then the composition $Z\to X\to Y$ factors through $Z\to\Spec\bC$.
\end{enumerate}
Recall that a subset $\Sigma\subset \cM_B(X)(\bC)$ is large if for any non-constant map $g \colon Z\to X$ from a connected normal variety, some local system in $g^*\Sigma$ has infinite monodromy.

Provided a $\Sigma$-Shafarevich morphism exists, $Y$ will always be a global quotient of an algebraic space by a finite group action, and after passing to a finite \'etale cover $p \colon X'\to X$ there will be a $p^*\Sigma$-Shafarevich morphism whose target is an algebraic space.  The above conditions ensure the Shafarevich morphism is unique and functorial provided it exists (see \Cref{shaf functo}).  If $\Sigma$ consists of a single local system underlying an integral variation of Hodge structures, then the Stein factorization of the period map provides the Shafarevich morphism. 

For projective $X$, a rational Shafarevich map was constructed by Campana \cite{campana94} and Koll\'ar \cite{kollar93,kollar95}; their construction more generally produces a rational map contracting subvarieties $Z$ through a very general point with finite image (up to normalization) in $\pi_1(X,x)/\Gamma$ for \emph{any} normal subgroup $\Gamma$ (not just those cut out by linear representations).  In the case that $\Sigma$ consists of semisimple local systems of bounded rank, a $\Sigma$-Shafarevich morphism was constructed for projective $X$ by Eyssidieux \cite{Eyssidieux} and analytically for quasiprojective $X$ by the second author \cite{brunebarbeshaf} and Deng--Yamanoi \cite{DengShaf}, who also observed that the map is algebraic after a modification.  

Our second main result is the existence of algebraic Shafarevich morphisms in general:

\begin{thm}\label{mainShaf}
For $X$ a connected normal algebraic space and $\Sigma \subset \cM_B(X)(\bC)$ a set of local systems of bounded rank, there is a unique algebraic $\Sigma$-Shafarevich morphism $sh_\Sigma(X):X\to \Sh_\Sigma(X)$, which is proper if and only if $\Sigma$ is nonextendable.  Moreover, if $\Sigma$ consists of semisimple local systems then the coarse space of $\Sh_\Sigma(X)$ is quasiprojective.
\end{thm}
As mentioned above, there is always a partial compactification $ \bar X$ of $X$ to which $\Sigma$ extends and is nonextendable.

\Cref{mainShaf} says that any collection of local systems is pulled back from a large collection on an algebraic space (after a finite \'etale cover):
\begin{cor}\label{corDMpullback}
    For $X$ a connected normal complex algebraic space and $\Sigma \subset \cM_B(X)(\bC)$ any subset of bounded rank, there is a morphism $s:X\to Y$ to a generically inertia-free Deligne--Mumford stack such that $\Sigma$ is the pull-back of a large collection of local systems on $Y$.
\end{cor}
In the large case, the statement of \Cref{main result} is particularly simple---note that any covering space of a Stein complex space is Stein \cite{Stein}:

\begin{cor}\label{main cor}
    For $X$ a connected normal complex algebraic space and $\Sigma\subset \cM_B(X)(\bC)$ a large nonextendable absolute Hodge subset, $\tilde X^{\Sigma}$ (and therefore also any covering space thereof) is Stein.  In particular, if $X$ admits a large nonextendable representation of its fundamental group, then the universal cover of $X$ is Stein.
\end{cor}

\begin{cor}\label{main cor 2}
    Let $X$ be a connected normal complex algebraic space admitting 
 a large almost faithful representation of its fundamental group.  Then the universal cover of $X$ is an analytic Zariski open subset of a Stein complex space.
\end{cor}

We now summarize the main ingredients to the proofs of \Cref{main result} and \Cref{mainShaf}.  A salient theme throughout is that while the presence of boundary results in many complications in non-abelian Hodge theory, everything is much better behaved in the case of quasiunipotent local monodromy, and this is sufficient to ``see'' all of the Betti stack.
\subsection{Definable Stein factorization}As noted above, if $V$ underlies an admissible variation of integral graded-polarized mixed Hodge structures, the $V$-Shafarevich morphism is the Stein factorization of the period map $X^\an\to \bG(\bZ)\backslash D$, and as such \Cref{mainShaf} is a version of Griffiths' conjecture on the quasiprojectivity of the images of period maps.  This conjecture was proven in \cite{bbt}, \cite{bbt2} using definable GAGA \cite{bbt} and the definability of period maps \cite{bkt}, \cite{bbkt}.

We now briefly summarize the analytic construction of the $\Sigma$-Shafarevich morphism for a nonextendable collection of semisimple local systems in \cite{brunebarbeshaf}.  Importantly, any $\Sigma$ will have the same Shafarevich morphism as the smallest closed absolute Hodge subset containing it, so we may assume $\Sigma$ is a closed absolute Hodge subset.  Techniques from non-abelian Hodge theory then ensure $\Sigma$ contains complex variations of Hodge structures, and combining their period maps with the non-archimedean reductions of the $\bar \bQ$-points of $\Sigma$ yields an analytic map
\[\phi:\tilde X^\Sigma\to M^\an\]
to an algebraic variety $M$ which contracts exactly those subvarieties on which $\Sigma$ is uniformly finite.  The nonextendability hypothesis implies the connected components of the fibers of this map are compact, and the Stein factorization $\sigma:\tilde X^\Sigma\to\tilde\cY$ then descends to an analytic Shafarevich morphism $s:X^\an\to \cY$.

Definable GAGA implies that algebraizing $X^\an\to \cY$ is equivalent to giving it the structure of a morphism of definable analytic spaces.  The main obstacle to generalizing the proof in the case of period maps of integral variations is that $\phi$ does not descend to a reasonable analog of the period map $X\to \bG(\bZ)\backslash D$, since the monodromy may act highly nondiscretely on $M^\an$.  Instead, we rely on the following:
\begin{thm}[\Cref{thm:Stein}]\label{steinfactor}
    Let $f:X\to Y$ be a proper morphism of seminormal definable analytic spaces.  Then the Stein factorization of $f$ exists in the category of definable analytic spaces. 
\end{thm}
By lifting a definable cover of $X$ to $\tilde X^\Sigma$ and applying \Cref{steinfactor} to $\phi$, $\cY$ is endowed with a definable analytic space structure as desired.  The usual proofs of Stein factorization in the algebraic and analytic categories do not apply to the setting of \Cref{steinfactor} since cohomology is poorly behaved in the definable analytic category.  Indeed, as coverings in the definable site are required to be finite, there do not seem to be acyclic coverings in general.  Interestingly, \Cref{steinfactor} together with definable GAGA gives a new proof of the algebraicity of period maps $X^\an\to\bG(\bZ)\backslash D$ of integral variations, at least up to a Stein factorization, which does not depend on the hard input of the global definability of the period map, and instead only uses the relatively cheap ``local'' definability---that is, the regularity of the Hodge filtration (or the nilpotent orbit theorem). 

\subsection{Hodge and twistor structures on miniversal local systems}As in \cite{Eyssidieux, EKPR}, to construct the Shafarevich morphism in general, a key observation is that $\Sigma$ may be replaced with its absolute Hodge closure $\Sigma^{\mathrm{abs}}$, which can be shown to contain local systems underlying variations of Hodge structures using techniques from non-abelian Hodge theory.  The following result will essentially say these variations are Zariski dense in $\Sigma^\mathrm{abs}$, provided we allow artinian thickenings.
\begin{thm}[\Cref{thm:versal}]\label{thmES}
    Let $X$ be a connected algebraic space, $V_0\in\cM_B(X)(\bC)$ a complex local system underlying an admissible complex variation of mixed Hodge structures (resp. an admissible tame purely imaginary variation of mixed twistor structures), and $ \hat V$ a miniversal $\hat\cO_{\cM_B(X),V_0}$-local system for $\cM_B(X)$ with closed point $V_0$, where $\hat\cO_{\cM_B(X),V_0}$ is a complete local $\bC$-algebra.  Then:
    \begin{enumerate}
        \item The miniversal deformation algebra $\hat\cO_{\cM_B(X),V_0}$ admits a canonical pro-complex mixed Hodge structure (resp. pro-mixed twistor structure) which is functorial with respect to pull-back morphisms $f^*:\cM_B(Y)\to\cM_B(X)$ along algebraic maps $f:X\to Y$, direct sums, and tensor products (in a sense to be made precise in \Cref{thm:versal frame} and \Cref{thm:versal}).
        \item The miniversal family $\hat V$ can be equipped with a pro-admissible complex variation of mixed Hodge structure (resp. pro-admissible tame purely imaginary variation of mixed twistor structures) for which the $\hat\cO_{\cM_B(X),V_0}$-action is compatible with the Hodge structures (resp. twistor structures).
    \end{enumerate}
\end{thm}
We also give versions of both theorems for the framed spaces (see \Cref{thm:versal frame}).  Note that by definition a complex variation of mixed Hodge structure satisfies an admissibility condition and in particular has quasiunipotent local monodromy (see \Cref{sect:MHSdefn}).  On the other hand, every local system underlies an admissible tame purely imaginary variation of mixed twistor structures.  \Cref{thmES} in the Hodge case generalizes the results of \cite{ES} when $X$ is smooth projective, and results of \cite{lefevrei,lefevreii} endowing the deformation algebra with a pro-mixed Hodge structure in general.  The twistor case was explained in the compact case by Simpson \cite{simpsontwistor}; see also \cite{Simpsonrank2twistor,simpsonHitchinDeligne} for some results in the case of quasiprojective curves.  Our approach is quite different from these:  rather than using Goldman--Millson theory \cite{goldmanmillson}, we employ elementary deformation theory together with Saito's theory of mixed Hodge modules and Sabbah/Mochizuki's theory of mixed twistor $\mathscr{D}$-modules to equip certain cohomology groups with Hodge/twistor structures.  This allows us to avoid some of the complications in the harmonic theory in the non-proper case (though see \Cref{rem:GoldmanMillson}). 

Both the Hodge and twistor cases of \Cref{thmES} are important. 
In the compact case, Simpson famously described a $\bC^*$-action on the good moduli space $M_B(X)$ whose fixed points are precisely complex variations of Hodge structures; as the image of $\Sigma^\mathrm{abs}$ in $M_B(X)$ is $\bC^*$-stable and closed, limits of $\bC^*$-orbits will produce variations of Hodge structures in the semisimplification of $\Sigma^{\mathrm{abs}}$.  On the other hand, there might be irreducible components of $\cM_B(X)$ which are not generically semisimple and whose semisimplification is strictly contained in an irreducible component of $M_B(X)$---it is less clear how to use the $\bC^*$-action to produce variations of Hodge structures there.  The twistor structures in \Cref{thmES} allow for the local germs of $\cM_B(X)$ to be identified with the local germs of the Higgs bundle stack, where the $\bC^*$-action can be understood, via the Deligne--Hitchin space of $\lambda$-connections (see \Cref{DeligneHitchin sect}).  Together with \Cref{intro comparison} below, this allows us for example to establish the following:
\begin{cor}[\Cref{basic prop abs hodge substack} and \Cref{every component contains Hodge}]
    Every irreducible component of the Betti stack $\cM_B(X)$ contains a point underlying a variation of Hodge structures.
\end{cor}
In general the twistor case of \Cref{thmES} is used to show every irreducible component of an absolute Hodge subset contains points underlying variations of Hodge structures.  Note that in the non-proper case the Deligne--Hitchin space is complicated by the monodromy of the residual eigenvalues, but in a neighborhood of the quasiunipotent local monodromy locus it is well-behaved, and by \Cref{intro qu density} below this is sufficient.

\subsection{Algebraic integrability of the Katzarkov--Zuo foliation and graded nearby cycles}  The general strategy to prove \Cref{main result} is roughly the same as in \cite{Eyssidieux} and \cite{EKPR} with some important differences.  Note that by \Cref{mainShaf} it suffices to prove \Cref{main cor}.  Recall that in the projective setting, the semisimple case is handled in \cite{Eyssidieux} using pure period maps and non-archimedean reductions, and the general case of \cite{EKPR} builds on \cite{Eyssidieux} by adding mixed period maps associated to miniversal families. 
 The correct division in the quasiprojective setting is to first prove \Cref{main result} in the case the local systems in $\Sigma$ are semisimple \emph{with quasiunipotent local monodromy}, and then to prove the general case.

 To adapt Eyssidieux's proof to accomplish the first step, the quasiunipotence of local monodromy is important for two reasons:

    \subsubsection*{Properness of pluriharmonic maps} 
 Ultimately the proof that $\tilde X^\Sigma$ is Stein rests on the construction of a strictly plurisubharmonic exhaustion function.  For a nonextendable $\bQ$-local system $V$ with reductive algebraic monodromy $\bG$, one obtains a candidate by putting together the pluriharmonic maps at all places $X^\an\to \bG(\bQ)\backslash \prod_v\Delta_v(\bG)$ and pulling back an exhaustion function.  Such a function will be plurisubharmonic by general theory, but will be more easily shown to be proper in the case that $V$ has quasiunipotent local monodromy, since in this case the asymptotics of the pluriharmonic maps (specifically the archimedean one) are better behaved.
    \subsubsection*{Boundary behavior of Katzarkov--Zuo foliation} To show that the exhaustion function described in the previous paragraph is \emph{strictly} plurisubharmonic, one must generalize Eyssidieux's inductive proof (using Simpson's Lefschetz-type result \cite{Simpsonlefschetz}) of the algebraic integrability of the Katzarkov--Zuo foliation.  

    \begin{thm}[\Cref{thm KZ integrable}]\label{intro KZ}
  Let $X$ be a connected normal algebraic space, $\Sigma\subset \cM_B(X)(\bC)$ an absolute $\bar\bQ$-constructible set of semisimple local systems with quasiunipotent local monodromy, and $v$ a nonarchimedean valuation on $\bar\bQ$.  Then the $v$-adic Katzarkov--Zuo foliation of $X$ associated to the $\bar\bQ$-points of $\Sigma$ is algebraically integrable.   
\end{thm}
    
In the quasiprojective setting the induction step requires the foliation to restrict to a Katzarkov--Zuo foliation in the boundary, even when the local system does not extend.  To do this, the key observation is that in the nonarchimedean setting, a local system with quasiunipotent local monodromy locally has an integral structure at the boundary, so the pluriharmonic map associated to a pluriharmonic norm extends even though the local system may not.  We show that this can be interpreted as equipping the graded nearby cycles functor with a pluriharmonic norm.  The Katzarkov--Zuo foliation of the graded nearby cycles local system on the boundary is then cut out by the restrictions of the symmetric forms cutting out the Katzarkov--Zuo foliation on the open part (see \Cref{characteristic_polynomial_nearby_cycle}).

 \vskip1em
We further introduce two simplifications to Eyssidieux's proof:
\begin{enumerate}
    \item We use the Ax--Schanuel theorem \cite{axsemiab} for abelian varieties in place of the higher-dimensional Castelnuovo--de Franchis theorem.
    \item We use a Steinness criterion due to Mok \cite{mok_stein} instead of the Demailly--Pa\u{u}n theorem \cite{Demailly-Paun}.
\end{enumerate}

\subsection{Density of quasiunipotent local systems and the Simpson--Mochizuki correspondence}  To go beyond semisimple local systems with quasiunipotent local monodromy, a new strategy to prove Steinness is needed.  Part of this is for the same reason as in \cite{EKPR}:  if $\bG$ is not reductive, the pluriharmonic maps described above need not exist.  But even when $\bG$ is reductive, if the local monodromy is not quasiunipotent it is unclear how to proceed.  The mixed variations provided by \Cref{thmES} at points of $\Sigma$ underlying complex variations of Hodge structure with quasiunipotent local monodromy handle both issues (the first as in \cite{EKPR}), since their relative period maps are affine. 

Crucially, to ensure there are enough points of $\Sigma$ underlying variations of Hodge structures with quasiunipotent local monodromy, we need two results.  The first is a generalization of a result of Esnault--Kerz \cite{esnaultkerz}:

\begin{thm}[{\Cref{qu dense in abs}}]\label{intro qu density}Let $X$ be a connected normal complex algebraic space and $\Sigma\subset \cM_{B}(X)(\bC)$ be an absolute $\bar\bQ$-constructible subset.  Let $\Sigma^{\qu}$ be the locus of points with quasiunipotent local monodromy.  Then $\Sigma^{\qu}$ is Zariski dense in $\Sigma$.
\end{thm}
   The idea behind \Cref{intro qu density} is also present in \cite{Budur-Lerer-Wang}.  Here, ``absolute $\bar\bQ$-constructible'' means every Galois conjugate of $\Sigma$ is $\bar\bQ$-algebraic, \emph{where the Galois action is made sense of via the Riemann--Hilbert correspondence}, see \Cref{sect:abssets}. In fact, $\bar\bQ$-absoluteness is equivalent to $\bar\bQ$-bialgebraicity with respect to the Riemann--Hilbert correspondence (see \Cref{bialg=abs}).  With the correct setup, \Cref{intro qu density} follows rather easily from the Riemann--Hilbert correspondence and the Gelfond--Schneider and Ax--Schanuel theorems for the exponential.

   The second is to prove that the non-abelian Hodge correspondence of Simpson and Mochizuki yields a homeomorphism between the moduli spaces of polystable logarithmic Higgs bundles with nilpotent residues and vanishing rational Chern classes, semisimple logarithmic connections with nilpotent residues, and semisimple complex local systems with unipotent local monodromy, at least in the case of curves.

\begin{thm}[\Cref{BettiDeRhamcomparison} and \Cref{comparison}]\label{intro comparison}
    Let $(\bar X,D)$ be a log smooth proper curve.  Then the correspondence via purely imaginary tame harmonic bundles yields identifications
    \[M_{Dol}^\nilp(\bar X,D)(\bC)\cong M_{DR}^\nilp(\bar X, D)(\bC)\cong M_B^\unip(X)(\bC)\]
    where the first is a homeomorphism and the second is a complex analytic isomorphism.
\end{thm}
This leads to finer control on the $\bC^*$-action on the Betti stack with unipotent local monodromy, whose fixed points underlie variations of Hodge structures.

\subsection{Acknowledgements}
The authors benefited from discussions with S. Boucksom, J. Daniel, C. Sabbah, C. Schnell, C. Simpson,  and T. Mochizuki.  We are particularly grateful to C. Sabbah for communicating the proof of \Cref{lemma_Sabbah}.  B.B. was partially supported by NSF grant DMS-2131688. Y.B. is partially supported by the French Agence Nationale de la Recherche (ANR) under reference ANR-23-CE40-0011 (CYCLADES).

\subsection{Outline} In \Cref{sect:def Stein} we prove the existence of Stein factorization in the definable analytic category, \Cref{steinfactor}. 
 In \Cref{sect:quasiproj} we give an algebraization criterion for coherent sheaves coming from a flat principal bundle.  In \Cref{sect:miniversal} we prove the existence of functorial mixed Hodge structures on the local rings of the Betti stack at points underlying variations of mixed Hodge structures and functorial mixed twistor structures at arbitrary points to prove \Cref{thmES}.  In \Cref{sect:bialg} we develop some general properties of bialgebraic subsets of a stack with respect to an analytic correspondence and prove \Cref{intro qu density} using the Riemann--Hilbert correspondence.  In \Cref{sect:harmonic} we review the general theory of (pluri-)harmonic maps to NPC spaces.  In \Cref{sect:Gm action} we discuss the Dolbeault realization and prove \Cref{intro comparison}. 
 In \Cref{DeligneHitchin sect} we construct the Deligne--Hitchin space and use the results of \Cref{sect:miniversal} to understand its completion along preferred sections.  In \Cref{sect:constructible} we assemble some properties of constructible subsets of the Betti stack including: nonextendability, the equivalence between $\bar\bQ$-bialgebraic (with respect to the Riemann--Hilbert correspondence) and $\bar\bQ$-absolute subsets, the graded nearby cycles functor.  We also define and prove some basic properties of absolute Hodge substacks of the Betti stack.  In \Cref{sect:norms} we develop some properties of pluriharmonic norms on non-archimedean local systems to show that the graded nearby cycles functor is compatible with the Katzarkov--Zuo foliation for semisimple local systems with quasiunipotent local monodromy, \Cref{characteristic_polynomial_nearby_cycle}. 
 In \Cref{sect:KZ} we prove the algebraic integrability of the Katzarkov--Zuo foliation of an absolute $\bar\bQ$-absolute set of semisimple local systems with quasiunipotent local monodromy (\Cref{intro KZ}).  In \Cref{sect:stein} we give a simple characterization of Stein complex spaces which is implicit in the work of Mok.  In \Cref{sect:qu stein} we prove \Cref{mainShaf} in the semisimple case and \Cref{main result} in the semisimple quasiunipotent local monodromy case using the results of \Cref{sect:quasiproj}, \Cref{sect:KZ}, and \Cref{sect:stein}. In \Cref{sect:proofs} we upgrade \Cref{main result} and \Cref{mainShaf} to the general case using \Cref{steinfactor}, \Cref{thmES}, \Cref{intro qu density}, and \Cref{intro comparison}.
 
\subsection{Notation}
Unless otherwise indicated, all algebraic spaces (resp. definable analytic spaces, resp. analytic spaces) $X$ are separated algebraic spaces (resp. definable analytic spaces, resp. analytic spaces) of finite type over $\bC$.  Definability will always be meant with respect to a fixed o-minimal structure for general results, and with respect to $\bR_{\an,\exp}$ in applications.

The only stacks we will consider for the most part (namely the Betti stack of any connected algebraic space, and the De Rham and Dolbeault stacks of a log smooth projective variety in the unipotent case) are global quotients of quasiprojective schemes by a reductive group.  In several places (specifically the De Rham stack of a log smooth algebraic space in the non-quasiunipotent case and in the construction of the Deligne--Hitchin space of a log smooth algebraic space), we will consider stacks which are the quotient of possibly non-separated algebraic spaces (or analytic spaces) which admit a countable open cover by finite-type algebraic spaces (or analytic spaces) by a reductive group. 

Throughout we use the following notation.  For a subset $\Sigma\subset\cM_B(X)(\bC)$ of complex points of the Betti stack, we will denote:  $\Sigma^\ss\subset\Sigma$ the subset of semisimple local systems; $\Sigma^\qu\subset \Sigma$ the subset of local systems with quasiunipotent local monodromy (see \Cref{sect:Betti}); for any $n\in\bN$, $\Sigma^{\qu|n}\subset\Sigma$ the subset of local systems for which the eigenvalues of the local monodromy have order dividing $n$.  For a substack $\cZ\subset \cM_B(X)$, we therefore refer to $\cZ(\bC)^\ss$ and $ \cZ(\bC)^\qu$, but in the last case we denote $\cZ(\bC)^{\qu|n}=\cZ^{\qu|n}(\bC)$ since $\cZ(\bC)^{\qu|n}$ will underlie a closed substack $\cZ^{\qu|n}\subset\cZ$.


\section{Definable analytic Stein factorization}\label{sect:def Stein}
Recall that we say a map $f:X\to Y$ of algebraic spaces, definable analytic spaces, or analytic spaces is a fibration if it is proper and the pullback map $\cO_Y\to f_*\cO_X$ is an isomorphism.  Recall also that if $f:X\to Z$ is a morphism of algebraic (resp. analytic) spaces such that the fibers of $f$ have compact connected components, there is a diagram in the category of algebraic (resp. analytic) spaces 

\begin{equation}\notag\begin{tikzcd}
X\ar[rd,"g",swap]\ar[rr,"f"]&&Z\\
&Y\ar[ru,"h",swap]&
\end{tikzcd}\end{equation}
such that $g$ is a fibration and $h$ has discrete fibers.  Such a diagram is uniquely determined up to isomorphism (compatible with $f$) and is called the Stein factorization of $f$.  In both categories, the map $g:X\to Y$ is the quotient of $X$ by the equivalence relation $x\sim x'$ if $x,x'$ are contained in a connected component of a fiber of $f$.  We call this equivalence relation in the topological/definable topological/definable analytic/analytic categories the \emph{connected fiber equivalence relation}.

The goal of this section is to extend the existence of the Stein factorization to the definable analytic category, at least for seminormal varieties (we say a reduced definable analytic space $X$ is (semi)normal if the analytification is):

\begin{thm}\label{thm:Stein}
Let $X$ be a seminormal definable analytic space, $Z$ a definable analytic space, and $f:X\to Z$ a morphism with compact fibers. Then there is a diagram, unique up to isomorphism (compatible with $f:X\to Z$) 
\begin{equation}\label{eq:Stein}\begin{tikzcd}
X\ar[rd,"g",swap]\ar[rr,"f"]&&Z\\
&Y\ar[ru,"h",swap]&
\end{tikzcd}\end{equation}
where $g$ is a fibration and $h$ is quasifinite.  Moreover, $g$ is the quotient by the connected fiber equivalence relation in the definable analytic category, and  \eqref{eq:Stein} analytifies to the analytic Stein factorization.

\end{thm}
Note that allowing for $X$ to be seminormal will be important in the following section.

\subsection{Definable Spec and (semi)normalization}

\begin{prop}\label{prop:Spec}Let $Z$ be a definable analytic space and $\cA$ a definable coherent sheaf of $\cO_Z$-algebras.  Then there is a unique definable analytic space $h:Y\to Z$ equipped with an isomorphism $\phi:h_*\cO_Z\to \cA$ which is finite over $Z$, up to isomorphism as spaces over $Z$ that are compatible  with $\phi$.  We define $Y=\Spec \cA$.   
\end{prop}
\begin{proof}The uniqueness is clear.  For the existence, we may therefore assume we have a surjection $\cO_Z^n\to\cA$.  Let $a_1,\ldots,a_n$ be the corresponding sections of $\cA$.  Let $\pi:\bC^n\times Z\to Z$ be the second projection and consider the surjection $\cO_{\bC^n\times Z}\to\pi^*\cA$ defined by sending the $i$th coordinate $z_i$ to $a_i$.  The kernel is an ideal sheaf for a subspace $Y\subset \bC^n\times Z$ together with a natural map $\phi:h_*\cO_Z\to \cA$ with the required properties. 
\end{proof}

\begin{prop}\label{prop:normal}Let $Z$ be a reduced definable analytic space.  There is a definable analytic space $h:Y\to Z$ uniquely determined by the property that it analytifies to the analytic normalization.
\end{prop}
\begin{proof}The uniqueness is clear by \cite[Proposition 2.45]{bbt}, as for any other such morphism $h':Y'\to Z$ the analytic isomorphism $Y^\an\to Y^{\prime\an}$ is the closure of the identity on the (definable) dense Zariski open subset of $Z$ where both $h,h'$ are isomorphisms, hence definable.  

For the existence, recall the Oka criterion for normality \cite[Chap. 6 \S5.2]{GR}:  a reduced space $\cY$ is normal if and only if the natural map $\cO_\cY\to\cEnd(I_{\cY^\sing})$ is an isomorphism, where $I_{\cY^\sing}$ is the ideal of the reduced singular locus.  Moreover, $\cEnd(I_{\cY^\sing})$ is naturally a coherent sheaf of $\cO_X$-algebras (in particular commutative, see \cite[Chap. 6 \S5.1]{GR}).  Thus, defining $Z_0=Z$ and $Z_i=\Spec \cEnd(I_{(Z_{i-1})^\sing})$ inductively, we have a tower of finite morphisms $h_i:Z_i\to Z$.  The supports of the cokernels of the maps
\[\cO_Z=\cO_{Z_0}\to h_{1*}\cO_{Z_1}\to\cdots\to h_{i*}\cO_{Z_i}\to\cdots \]
yield a decreasing chain of subspaces which can only stablize when $Z_i$ is normal, and by noetherian induction the claim is proven.
\end{proof}

\begin{cor}\label{cor:seminormal}Let $Z$ be a reduced definable analytic space.  There is a definable analytic space $h:Y\to Z$ uniquely determined by the property that it analytifies to the analytic seminormalization.
\end{cor}
Recall that the analytic seminormalization $\cY\to\mathcal{Z}$ is uniquely characterized by the property that its regular functions are meromorphic functions on $\mathcal{Z}$ which extend continuously \cite{gtsn}.  This means that if $\cY'\to \mathcal{Z}$ is the normalization, $\cY$ is the quotient by the equivalence relation $\cY'\times_\mathcal{Z}\cY'\subset\cY'\times\cY'$ with the reduced structure.
\begin{proof}[Proof of \Cref{cor:seminormal}]The uniqueness is again clear.  Let $h':Y'\to Z$ be the normalization, $\cY\to Z^\an$ the analytic seminormalization, and $R\subset Y'\times Y'$ be the reduced definable subspace whose underlying set is the support of $Y'\times_Z Y'\subset Y'\times Y'$.  Let $r_1,r_2:R\to Y'$ be the two finite projections; the compositions with $h'$ are equal, $t=h'\circ r_i$.  Let $\cA$ be the kernel of $r_1^*-r_2^*:h'_*\cO_{Y'}\to t_*\cO_R$.  It follows from the above that $\cA^\an$ is $h_*\cO_\cY$ and therefore $\Spec\cA$ is the required $Y$.
\end{proof}

\begin{cor}\label{cor:sn and n} \Cref{thm:Stein} holds if and only if it holds additionally assuming $X$ is normal.
\end{cor}
\begin{proof}The forward implication is clear.  For the backward implication, let $f':X'\to Z'$ be the maps on normalizations, and $X'\to Y'\to Z'$ the Stein factorization. 
 The equivalence relation in $Y'\times Y'$ defining $Y$ is the image of the equivalence relation $X'\times_XX'\subset X'\times X'$, and as in the previous corollary we obtain $h:Y\to Z$ analytifying to the finite part of the Stein factorization.  The composition $X'\to Y'\to Y$ factors through $X$ analytically, so by \cite[Proposition 2.55]{bbt} we obtain the morphism $g:X\to Y$.
\end{proof}

\subsection{Stein factorization of definable topological spaces}
Recall the following important result of Van den Dries on equivalence relations in the category of definable topological spaces.
\begin{thm}[{Van den Dries \cite[Chap. 10, (2.15) Theorem]{Dries}}]\label{thm vdd quotient}  Let $X$ be a definable topological space and $R\subset X\times X$ a proper definable equivalence relation.  Then the quotient $q:X\to X/R$ exists as a definable topological space.  Moreover the map $q$ is proper.
\end{thm}
The quotient is both a categorical quotient (in the sense of being universal with respect to morphisms $X\to Y$ which are equal on the relation) and a geometric quotient (in the sense that the topology is the quotient topology and fibers are equivalence classes).

\begin{prop}\label{prop:def Stein}Let $f:X\to Z$ be a morphism of definable topological spaces with compact fibers.  Let $R\subset X\times X$ be the equivalence relation $x\sim x'$ if $x,x'$ are in the same connected component of a fiber of $X$.  Then the quotient $g:X\to Y$ by $R$ exists in the category of definable topological spaces. 
 Moreover, there is a factorization 
\[\begin{tikzcd}
X\ar[rd,"g",swap]\ar[rr,"f"]&&Z\\
&Y\ar[ru,"h",swap]&
\end{tikzcd}\]
uniquely determined by the property that $g$ has connected fibers and $h$ is quasifinite.  Finally $g$ is proper.
\end{prop}
\begin{proof}Connected components are always closed and compact subsets of $X$ are separated by a nonzero distance, so $R$ is a closed proper equivalence relation.  There is a definable stratification of $Z$ by locally closed subspaces over which $f$ is definably trivializable \cite[Chap. 9, (1.2) Theorem]{Dries}, hence $R$ is definable.  Thus, by \Cref{thm vdd quotient} the quotient $Y=X/R$ exists, and the quotient map $g:X\to Y$ is proper.  The remaining statements are clear. 
\end{proof}

\begin{lem}\label{lem:qf and f}Let $f:Y\to Z$ be a quasifinite morphism of regular definable topological spaces.  Then up to a definable cover of $Z$, $f$ factors as $Y\to Z'\to Z$ where $Y\to Z'$ is finite and $Z'\to Z$ is an open embedding on each component. 
\end{lem}
\begin{proof}First observe that we may assume $f$ factors as $Y\to\bar Y\to Z$ where $j:Y\to\bar Y$ is an open embedding and $\bar f:\bar Y\to Z$ is finite.  Indeed, we may assume $Y$ is embedded in $\bR^n$, and therefore that $f$ factors as $Y\to \bR^n\times Z\to Z$ where the first map is a locally closed embedding and the second is the projection.  After a definable open cover of $Z$, we may assume the last coordinate separates points in the fiber, and thus that $n=1$.  We then take $\bar Y$ to be the closure of the image of $Y$ in $\bP^1_\bR\times Z$.  Let $\partial Y=\bar Y\setminus Y$, which is closed in $\bar Y$.

The argument of \cite[Proposition 2.4]{bbt} produces the following, possibly after a definable cover:  triangulations $\{C\}$ and $\{D\} $ of $\bar Y$ and $Z$ such that 
\begin{enumerate}
\item $\partial Y$ is a subcomplex of $\bar Y$.
\item for each open simplex $D$ of $Z$, $\bar f^{-1}(D)$ is a disjoint union of open simplices of $\bar Y$, each mapping isomorphically to $D$;
\item the closure of each simplex in $\bar Y$ injects into $Z$.
\end{enumerate}
As in \cite{bbt}, for any simplex $D$ of $Z$ take $Z(D)$ to be the star of $D$ and likewise for $\bar Y(C)$.  Then $\{Z(D)\}$ is a definable open cover of $Z$ for which $\bar f^{-1}(Z(D))$ is the disjoint union of the stars $Y(C)$ of the lifts $C$ of $D$.  Moreover, for any $D$, and any lift $C$ of $D$, $\bar Y(C)\cap \partial Y\neq\varnothing$ if and only if $C\subset \partial Y$.  Finally, if $C\subset \partial Y$, $Y(C):=\bar Y(C)\cap Y$ is finite over $Z(D)\setminus \bar f(\partial Y)$, which is open in $Z(D)$.  The cover $\{Z(D)\}$ then satisfies all the required properties. 
\end{proof}
\begin{rem}
    Note that the definable topological spaces underlying definable analytic spaces are locally isomorphic to a locally closed subset of $\bR^n$, and are therefore regular (see \cite{Dries}). 
\end{rem}

\subsection{Proof of \Cref{thm:Stein}} 

The Stein factorization \eqref{eq:Stein} exists if and only if the quotient $g:X\to Y$ by the connected fiber equivalence relation exists in the category of definable analytic spaces, hence the uniqueness.  The Stein factorization \eqref{eq:Stein} exists both in the analytic category and the definable topological space category, by \Cref{prop:def Stein}.  In both categories the map $g:X\to Y$ is the quotient by the connected fiber equivalence relation, hence the uniqueness, and for both the map $|g|:|X|\to| Y|$ on underlying topological spaces is also the quotient by the connected fiber equivalence relation, so the diagrams on underlying topological spaces are isomorphic.  It follows by \cite[Proposition 2.45]{bbt} that if a morphism $g:X\to Y$ exists in the definable analytic category which gives the quotient in the analytic and definable topological space categories then it is the quotient in the definable analytic category.

We now show the existence.  By \Cref{cor:sn and n} we may assume $X$ is normal. By the uniqueness statement we may freely pass to definable open covers of $Z$. In particular, by \Cref{prop:def Stein} and \Cref{lem:qf and f} we may assume $f$ is proper. We may also assume $Z$ is a basic definable analytic space---that is, the zero-locus of finitely many definable analytic functions on a definable open subset of $\bC^n$.  By definable Noether normalization \cite[Theorem 2.14]{peterzilstarchenko}, we may assume there is a finite linear projection $Z\to A$ for $A\subset \bC^m$ open, so we may replace $Z$ with $A$ and thereby assume $Z$ is a definable open subset of $\bC^n$. The existence will be a consequence of the following result.

\begin{prop}\label{finite_locally_affine}
Let $Z$ be a definable open subset of $\bC^n$. Let $f\colon X \to Z$ be a proper surjective morphism of normal definable analytic spaces. Let $X^{\an} \to \cY \to Z^{\an}$ be the Stein factorization of its analytification. Then, up to passing to a definable cover of $Z$, there exists $V = Z \setminus T \subset Z$ the complement of a nowhere-dense Zariski-closed definable subset $T \subset Z$, a positive integer $N$ and a morphism $g\colon X_V:=X\times_Z V \to \bA^N$ such that there is a commutative diagram
\[\begin{tikzcd}
    X_V^\an\ar[dr]\ar[rr,"(f\times g)^{\an}"]&&V^\an\times\bC^N\\
    &\cY_{V^\an}\ar[ur]&
\end{tikzcd}\]
where the right diagonal map is a closed immersion of $\cY_{V^\an}:=\cY\times_{Z^\an}V^\an$.
\end{prop}

\begin{cor}\label{cor almost Stein}Assume the hypothesis of \Cref{finite_locally_affine}.  Then, up to passing to a definable open cover of $Z$, there is a reduced irreducible definable analytic space $Y'$ and a factorization $X\to Y'\to Z$ with $X\to Y'$ proper surjective and $Y'\to Z$ finite which analytifies to the Stein factorization of $X^\an\to Z^\an$ over a dense definable analytic Zariski open subset $V\subset Z$.
\end{cor}
\begin{proof}[Proof of \Cref{cor almost Stein} assuming \Cref{finite_locally_affine}]
    After passing to a definable cover of $Z$, we may assume the ideal sheaf of $T$ is globally generated, and by multiplying $g$ by sufficiently high powers of these generators we may assume the function $g$ extend to $X$, by \cite[Lemma 3.2]{bbt2}.  Consider the definable analytic map $f\times g \colon X\to Z\times\bC^N$. Being a proper map, it follows from Remmert theorem that its image is a closed analytic subvariety which is clearly definable, hence by \cite[Proposition 2.45]{bbt} it is a definable closed analytic subvariety $Y'$ of $Z\times\bC^N$.  This variety is finite over $Z$ and its analytification contains $\cY_{V^\an}$ as a dense Zariski open subset by construction.
\end{proof}
\begin{proof}[Proof of \Cref{thm:Stein} assuming \Cref{cor almost Stein}]
Let $Y\to Y'$ be the normalization of $Y'$, as guaranteed by \Cref{prop:normal}.  Since $X$ is normal, the analytic Stein factorization $\cY$ is normal, hence equal to the normalization of $Y^{\prime\an}$.  By \cite[Proposition 2.45]{bbt} we therefore have a factorization $X\to Y\to Z$ in the definable analytic category.  
\end{proof}

We now proceed with the proof of Proposition \ref{finite_locally_affine}. Let $X^\an\to\cY\to Z^\an$ be the analytic Stein factorization. Observe that there is a nowhere dense reduced closed definable subspace $T\subset Z$ such that $\cY\to Z^\an$ is \'etale in corestriction to the complement $V:=Z\setminus T$.  Indeed, we may take $T$ to be closure of the locus of points where the fiber of $X\to Z$ is everywhere nonreduced on some connected component. 
 We may as well assume $T$ has pure dimension $n-1$. By definable Riemann existence (which follows from the existence of quotients by closed \'etale equivalence relations \cite[Proposition 2.57]{bbt}), this means there is a finite \'etale cover $U\to V$ which analytifies to the restriction of the map $\cY\to Z^\an$ to $V^\an$.

Proposition \ref{finite_locally_affine} is clearly true when $\cY \to Z^{\an}$ is finite étale, so in particular true over $V$. Finally, we may further assume there is a linear projection $Z\to B\subset \bC^{n-1}$ for an open definable $B\subset \bC^{n-1}$ which is finite on $T$ by another application of definable Noether normalization, and we therefore think of $Z$ as a definable open in $B\times \bC$.  The map $p:T\to B$ is \'etale over some dense definable Zariski open $B_0\subset B\subset\bC^{n-1}$.


\begin{lem}Let $B\subset\bC^{n-1}$ be a definable open subset, $T\subset B\times\bC$ a reduced definable analytic subspace of pure dimension $n-1$ which is finite over $B$, $T\subset Z\subset B\times \bC$ a definable open neighborhood.  Let $U\to V$ be a finite connected \'etale cover of $V:=Z\setminus T$.  Then up to passing to a definable open cover of $B$, shrinking $Z$ as a neighborhood of $T$, and taking connected components of $Z$, there is a dense Zariski open subset $B_0\subset B$ such that, setting $U_0\subset U$ to be the preimage of $B_0$ in $U$, we have a commutative diagram
\[\begin{tikzcd}
U_0\ar[rd]\ar[rr,"j"]&&C_0\setminus q_0 \ar[ld]\\
&B_0&
\end{tikzcd}\]
where $j$ is an open embedding and $C_0\to B_0$ is a smooth proper definable analytic family of curves with a definable analytic Cartier divisor $q_0$ which is finite \'etale over $B_0$.  
\end{lem}
\begin{proof}
We denote by $T_b\subset Z_b\subset \bC$ the fibers over $b\in B$.  Note that $T$ is the closure of $T\cap U_0$.  For $z\in \bC$ we denote by $B_R(z)$ the open disk of radius $R$ centered at $z$ and for a subset $\Sigma\subset\bC$ we denote by $K_R(\Sigma)$ the convex hull of $\bigcup_{p\in\Sigma}B_R(p)$. Assume first that for a continuous definable function $r$ on $B$, $Z_b=K_{r(b)}(T_b)$.  Then $V_0':=B_0\times \bC\setminus T$ deformation retracts onto $V_0:=B_0\times\bC\cap V$.  It follows that $U_0\to V_0$, the restriction of $U\to V$ to $V_0$, extends to a finite \'etale cover $U_0'\to V_0'$, which then extends to a finite ramified cover $C\to B_0\times\bP^1$.  The divisor $q_0$ is simply the reduced preimage of the point at infinity (which is \'etale over $B_0$ possibly after shrinking $B_0$), and this proves the proposition in this case.

It therefore remains to show that after a definable cover of $B$ and shrinking and taking components of $Z$ we are in the above case.  In particular we may assume $T$ is connected.  Let $d$ be the degree of $T\to B$.  There is a natural morphism $B\to\Sym^d\bC$ such that $T$ is contained in the pullback of the universal $d$-tuple.  Let $\bC^d\to\Sym^d\bC$ be the obvious quotient map, and $\nu:B'\to B$ the base-change to $B$.  There are then sections $t_1,\ldots,t_d$ of the projection $B'\times\bC\to B'$ whose union is the base-change of $T$ generically.  For any partition $\pi$ of $[d]$, consider the set $B'_\pi\subset B'$ of points $b'$ such that there is some $\epsilon>0$ satisfying:
\begin{enumerate}
\item For each $\Sigma\in\pi$ we have $K_{\epsilon}(t_\Sigma(b'))\subset Z_{\nu (b')}$, where we denote $t_\Sigma(b')=\{t_i(b')\mid i\in\Sigma\}$;
\item The closed convex hulls $\overline {K_\epsilon(t_\Sigma(b'))}$ are pairwise disjoint for distinct parts $\Sigma$ of $\pi$.
\end{enumerate}
Note that if a set $B'_\pi$ contains a point $b'$, then the orbit decomposition of $[d]$ under the stabilizer of $b'$ in $S_d$ yields a partition $\pi_{b'}$ which necessarily refines $\pi$, since for any part $\Sigma\in\pi_{b'}$ the sections $t_i(b')$ are equal for $i\in\Sigma$.  

The $\{B'_\pi\}$ are a definable cover of $B'$.  By \cite[Proposition 2.4]{bbt} there is a definable open cover $\{W\}$ of $B$ such that for each $W$, every connected component of $\nu^{-1}(W)$ is contained in some $B'_\pi$.  Moreover, because of the way this cover is constructed, the stabilizer of a component in $S_d$ is equal to the stabilizer of a point in that component.  It follows that $\pi$ is invariant under this stabilizer.   

Passing to this cover, there is a partition $\pi_T$ of the irreducible components of $T$ and a definable continuous function $r:B\to \bR_{>0}$ such that
\begin{enumerate}
\item For each $\Sigma\in \pi_T$ we have $K_r(\Sigma):=\bigcup_{b\in B}K_{r(b)}(\Sigma_b)$, where $\Sigma_b$ is the set of points of $T_b$ contained in a component in $\Sigma$;
\item The closures $ \overline{K_r(\Sigma)}$ are pairwise disjoint for distinct parts $\Sigma$ of $\pi_T$.
\end{enumerate}
We may therefore replace $Z$ with $K_r(\Sigma)$, in which case we are in the above situation.

\end{proof}
Returning to the proof, we make the following:
\begin{claim}After passing to an open definable cover of $T$ in $Z$ and a dense definable Zariski open of $B_0$, $U_0$ has a definable divisor $D_0\subset U_0$ which is finite \'etale over $B_0$ and such that the generic fiber of the pair $(C_0,D_0)$ has no automorphisms.
\end{claim}
\begin{proof}$X$ is covered by basic definable analytic open subspaces $X_i$.  As $f$ is generically smooth, each $f(X_i)$ has nonempty interior, and further $Z$ is covered by the interior closures of the images $f(X_i)$.  Thus we may assume by passing to a cover that one basic definable analytic subspace $X_i$ has dense image in $Z$.

\begin{lem}Let $Z\subset\bC^n$ be a definable open subset.  Let $X\subset \bC^m$ be an irreducible basic definable analytic space and $f:X\to Z$ a morphism with dense image.  Let $Z_0\subset Z$ be a dense definable Zariski open where $f$ is equidimensional and $X_0=f^{-1}(Z_0)$.  Fix $k\geq\dim X-n$ and let $W$ be a dense definable open subset of the space of affine codimension $k$ subspaces $H\subset \bC^m$ which meet $X$.  Then there is a definable open cover $Z_i$ of $Z$ such that for each $i$ there is a $H_i\in W$ for which $H_i\cap f^{-1}(Z_i\cap Z_0)$ is finite over $Z_i\cap Z_0$.
\end{lem}
\begin{proof}Let $\partial X=\bar X\setminus X$ be the boundary of $X$ in $\bP^n$; $\partial X$ is closed in $\bP^n$.  Note that the projection $H\cap X\to Z$ will be finite on the open set $Z\setminus f(H\cap\partial X)$, and if $f(H\cap\partial X)$ has no interior then, the desired property will be satisfied on the interior closure of $Z\setminus f(H\cap\partial X)$ in $Z$.  Observe that for any closed definable $\Sigma\subset Z$, we have 
\[\Sigma\setminus f(H\cap \partial X\cap f^{-1}(\Sigma))\subset Z\setminus f(H\cap\partial X).\]
Setting $\Sigma_0=\Sigma\cap Z_0$ we also have $\dim_\bR \partial X\cap f^{-1}(\Sigma_0)\leq \dim_\bR\Sigma_0+2(\dim X-n)$.  There is then a $H\in W$ for which $H\cap\partial X$ is not saturated with respect to $f$ and further
\[\dim_\bR H\cap \partial X\cap f^{-1}(\Sigma_0)< \dim_\bR\Sigma_0+2(\dim X-n)-2k\leq\dim_\bR\Sigma_0.\]
Thus, starting with $\Sigma=f(\partial X)$ by induction on $\dim_\bR \Sigma_0$ we inductively build the desired cover by adding the interior closure of $Z\setminus f(H\cap\partial X)$ to our cover and replacing $\Sigma$ by $\overline{f(H\cap\partial X\cap f^{-1}(\Sigma))}$ at each step.
\end{proof}
Applying the lemma with $k=\dim X-n+1$, pushing forward and using \cite[Proposition 2.45]{bbt} we have a divisor $D_0$ on $U_0$ and after shrinking $B_0$ and taking the divisor sufficiently generally, we may ensure the claim holds. 
\end{proof}

Let $g$ be the genus of the general fiber of $C_0$, $m$ the degree of $D_0$ over $B_0$, $d$ the degree of $q_0$ over $B_0$, and let $\mathcal{M}:=S_{m,d}\backslash \mathcal{M}_{g,m+d}$ be the moduli stack of genus $g$ curves with two sets of unordered points of cardinality $m$ and $d$.  Let $\mathcal{C}\to\mathcal{M}$ be the universal curve.  Up to passing to a dense definable Zariski open of $B$, we may assume there is an open affine subscheme $A\subset \mathcal{M}$ consisting of curves with no automorphisms and a definable analytic morphism $B_0\to A^\df$ such that $C_0$ is the pullback of the restriction of the universal curve $\cC_A\to A$.  

Note that the family $\cC^\circ_A\to A$ obtained by puncturing $\cC_A$ at the order $d$ set is affine, as therefore is $\cC_A^\circ$.  Let $x_i$ be coordinate functions on $\cC^\circ_A$.  Replacing $T$ with $T\cup \pi_1^{-1}(B\setminus B_0)$ where $\pi_1:Z\to B$ is the projection, we have a map $X_V\to U\to \cC_A^\circ$ and pulling back the functions $x_i$ yields the required morphism $g:X_V\to\bA^N$.
\qed

\begin{rem}The proof of \Cref{thm:Stein} also gives an elementary proof of Stein factorization in the analytic category assuming Remmert's proper mapping theorem, at least for morphisms of seminormal spaces (whose fibers have compact connected components).
\end{rem}

\section{Quasiprojectivity of complex period maps}\label{sect:quasiproj}

In this section we prove a general algebraicity result which among other things should be viewed as showing that the image of the period map associated to a complex variation of mixed Hodge structures is an algebraic space (quasiprojective in the pure case), up to a Stein factorization and assuming a weak properness assumption.  In this sense it is a version of the Griffiths conjecture \cite{bbt,bbt2}.
\subsection{Algebraicity}
The setup is as follows:
\begin{setup}\label{setup:twisted map}For an algebraic space $X$ suppose we have:
\begin{enumerate}
\item A finitely generated group $\Gamma$.
\item A definable analytic space $\scrM$ with an action by $\Gamma$ such that each $\gamma\in\Gamma$ acts by a definable analytic automorphism.
\item A normal covering space $\pi:\tilde X\to X^\an$ with deck transformation group $\Gamma$. 
\item A $\pi$-definable $\Gamma$-equivariant morphism $\phi:\tilde X\to \scrM^\an$.

\end{enumerate}

\end{setup}

Here we have used the following definition:

\begin{defn}Let $\scrX$ be a definable analytic space and $\pi:\tilde\scrX\to\scrX^\an$ a covering space.  Let $\scrM$ be a definable analytic space and $\phi:\tilde\scrX\to \scrM^\an$ a morphism of analytic spaces.  We say $\phi$ is $\pi$-definable if for any definable analytic space $\scrT$ and any commutative diagram
\begin{equation}\begin{tikzcd}
    & && &\tilde\scrX\ar["\pi",d]\\
\scrT\ar[r,"t",swap]&\scrX && \scrT^\an\ar[r,"t^\an",swap]\ar[ur,"\tilde\tau"]&\scrX^\an
\end{tikzcd}\label{lifting}\end{equation}
the resulting map $\phi\circ\tilde \tau:\scrT^\an\to \scrM^\an$ is the analytification of a morphism $\mu:\scrT\to \scrM$ of definable analytic spaces.  Note that by definable triangulation it suffices to check the property for open definable subspaces $\cT$. 
\end{defn}

In practice, \Cref{setup:twisted map} will arise when $\scrM$ has an action of $\pi_1(X^\an,x)$ (for some choice of baepoint $x$) by definable analytic automorphisms and $\phi$ will be $\pi_1(X^\an,x)$-equivariant, in which case it is sufficient to check the $\pi$-definability on a fundamental set of $\tilde\scrX$.  Note that if we are in \Cref{setup:twisted map} for some algebraic space $Y$ and $g:X\to Y$ is a morphism, then $X$ naturally inherits all of the structures in Setup \ref{setup:twisted map}.

 Now if we have \Cref{setup:twisted map} for an algebraic space $X$ and in addition the fibers of $\phi$ have compact connected components, recall that there is then a factorization in the category of analytic spaces
\[\begin{tikzcd}
\tilde X\ar[rd," \psi",swap]\ar[rr,"\phi"]&&\scrM^\an\\
&\tilde \cY\ar[ur,"\chi",swap]&
\end{tikzcd}\]
where $\psi$ is a fibration and $\chi$ has discrete fibers.  Moreover, $\Gamma$ acts properly discontinuously on $\tilde \cY$ so the map $ \psi$ descends to a fibration $\sigma: X^\an\to \cY$.  We refer to $\sigma$ as the Stein factorization of $\phi$.

\begin{thm}\label{thm:alg}Assume Setup \ref{setup:twisted map} for a seminormal algebraic space $X$ with the additional assumption that the fibers of $\phi$ have compact connected components.  Then the Stein factorization $\sigma:X^\an\to\cY$ of $\phi$ is the analytification of an algebraic map $g:X\to Y$, uniquely determined up to isomorphism (as a map with fixed source).
\end{thm}
\begin{proof}  By definable GAGA an algebraic structure $X\to Y$ on $\sigma:X^\an\to\cY$ is equivalent to a definable analytic space structure $X^\df\to Y$, and moreover either is uniquely determined.  Thus, it suffices to find a finite open analytic cover $\cY_i$ of $\cY$ such that $\sigma^{-1}(\cY_i)=X_i^\an$ for a definable open cover $X_i\subset X^\df$ and definable analytic space morphisms $X_i\to \mathscr{Y}_i$ which analytify to $X_i^\an\to \cY_i$.  This has the following consequence:
\begin{lem}If the underlying map on topological spaces $|\sigma|:|X^\an|\to |\cY|$ has a definable topological space structure $|X^\df|\to \mathfrak{Y}$, then $\sigma$ has a definable analytic space structure.
\end{lem}
\begin{proof}First observe that by definable triangulation there is a finite open cover $\cY_i$ of $\cY$ by contractible open sets such that the $X_i:=\sigma^{-1}(\cY_i)$ form a definable open cover of $X^\df$.

Now, there is a finite map $\cY_i'\to\cY_i$ which is a quotient by a finite group $G_i$ such that $\cY_i'$ lifts to $\tilde \cY$.  The preimage of the lift in $\tilde X$ is a finite \'etale $G_i$-cover $X_i'\to X_i$.  The resulting map $X_i'\to \scrM$ is a morphism of definable analytic spaces with compact fibers.  By \Cref{thm:Stein} there is a Stein factorization $X_i'\to \mathscr{Y}_i'\to \scrM$, where $\mathscr{Y}_i'$ analytifies to $\cY_i'$.  Taking the quotient by $G_i$ using \cite[Proposition 2.63]{bbt}, we obtain a definable analytic morphism $X_i\to \mathscr{Y}_i$ where $\mathscr{Y}_i$ analytifies to $\cY_i$ as required.  
\end{proof}

Returning to the proof of the theorem, we first reduce to the case that $X$ is normal.  Let $X'\to X$ be the normalization of $X$, $\cY'\to\cY$ the normalization of $\cY$, and suppose $X^{\prime\an}\to \cY'$ is algebraized by $X'\to Y'$.  The topological space underlying $\cY$ is the quotient of $|Y^{\prime\an}|$ by the image of the equivalence relation $|X'\times_X X'|\subset |X'\times X'|$ in $|Y'\times Y'|$.  This equivalence relation is clearly definable, closed, and proper, so by \Cref{thm vdd quotient} the quotient exists in the category of definable topological spaces, and by the lemma we are done.  

Thus we may assume $X$ is normal, and we proceed by induction on $\dim X$, the base case being trivial.  By a Hilbert scheme argument as in \cite{bbt} we have the following:

\begin{prop}[Compare with {\cite[Proposition III]{Sommese}}]\label{algebraizing_fibration}
Let $X$ be a connected normal algebraic space and $f \colon X^{\an} \to \cY$ a holomorphic fibration onto a normal analytic space $\cY$. Then there exists a fibration $g \colon X^\prime \to Y^\prime$ between two connected normal algebraic spaces, a proper modification $ X^\prime \to X$ and a holomorphic proper modification $(Y^\prime)^\an \to \cY$ such that the diagram
\[\begin{tikzcd}
(X^\prime)^\an \ar[r]\ar[d, "g^\an"]& X^\an \ar[d, "f"]\\
 (Y^\prime)^\an \ar[r]& \cY.
\end{tikzcd}\]
is commutative.
\end{prop}
\begin{proof}
One can assume that $X$ is smooth quasiprojective. Let $\mathrm{Hilb}(X)$ be the Hilbert scheme of proper algebraic subspaces of $X$. Its analytification is the Douady space of compact analytic subspaces of $X^{\an}$ (see \cite[p.200]{bbt}). Let $\cU \subset \cY$ be a non-empty Stein open subset over which $f$ is smooth. By Remmert proper mapping theorem, every irreducible compact complex subspace of $X^\an$ contained in $f^{-1}(\cU)$ is contained in a fiber of $f$. Therefore, the induced holomorphic map $\cU \to \mathrm{Hilb}(X)^{\an}$ is an open immersion. Let $H \subset \mathrm{Hilb}(X)$ be a irreducible component such that $H^\an$ contains $\cU$. The compact complex analytic subspaces of $X^\an$ that are contained in a fiber of $f$ form a closed analytic subset of $\mathrm{Hilb}(X)^\an$ that contains the image of $\cU$. Since $\cU$ is open in $H^{\an}$ for the Euclidean topology, it follows
that (the analytification of) every proper algebraic subspace of $X$ corresponding to an element of $H$ is contained in a fiber of $f$. Therefore, letting $Y^\prime$ be the normalization of $H$, we get a holomorphic map $(Y^\prime)^{\an} \to \cY$, which is a biholomorphism in corestriction to $\cU$. If we denote by $g \colon X^\prime  \to Y^\prime$ the normalization of the base-change of the universal family to $Y^\prime$, we get a commutative diagram 
\[\begin{tikzcd}
(X^\prime)^\an \ar[r]\ar[d, "g^\an"]& X^\an \ar[d, "f"]\\
 (Y^\prime)^\an \ar[r]& \cY.
\end{tikzcd}\]
Since the induced map $(X^\prime)^\an \to \cY$ is proper and its image contains $\cU$, it is surjective. A fortiori, the map $(Y^\prime)^\an \to \cY$ is proper and surjective. Finally, $g^\an$ and $f$ coincide over $\cU$. The remaining claims follow.
\end{proof}

Thus we may assume $\sigma$ is a proper modification.  The \'etale locus of $\phi$ is definable and $\rho$-equivariant, so there is a Zariski open subset $U\subset X$ on which $\sigma$ is an isomorphism.  Let $S:=X\setminus U$ and $\mathcal{T}\subset \cY$ the image of $S^\an$ under $\sigma$.  By the inductive hypothesis, if $S'\to S$ is the seminormalization, and $S^{\prime\an}\to \mathcal{T}'\to\mathcal{T}$ the Stein factorization of $S^{\prime\an}\to \mathcal{T}$, then $S^{\prime\an}\to\mathcal{T}'$ has an algebraic structure $S'\to T'$.  Since $S$ is saturated with respect to the map $\psi$, the fibers are connected, and therefore $T^{\prime\an}=\mathcal{T}'\to\mathcal{T}$ is the seminormalization.  Since seminormalizations are homeomorphisms on the underlying topological space and $\sigma$ is an isomorphism on $U$, it follows that the set-theoretic equivalence relation of $\sigma$ is definable, hence again by \Cref{thm vdd quotient} the quotient as a definable topological space exists, and by the lemma the claim is proven.
\end{proof}

\subsection{Quasiprojectivity}
We now address the question of when line bundles on $\scrM$ induces semiample bundles in the context of the previous section.  We first describe when bundles pulled back from $\scrM$ have a natural algebraic structure, as in \cite[\S 6.2]{bbt} (see also \cite[\S 2.3]{bbt2}).  

\begin{setup}\label{setup:alg}Assume \Cref{setup:twisted map} for an algebraic space $X$.  Suppose further that, for a chosen basepoint $x$ of $X$, we have
\begin{enumerate}
    \item $\G$ a linear algebraic group acting definably on $\scrM$.
    \item A homomorphism $\rho:\pi_1(X^\an,x)\to\G(\bC)$ with norm one eigenvalues of local monodromy.
    \item The covering map $\pi$ is the covering $\pi:\tilde X^{\rho}\to X^\an$ corresponding to the subgroup $\ker\rho\subset\pi_1(X^\an,x)$.
    \item The map $\phi$ is $\pi_1(X^\an,x)$-equivariant\footnote{In this case we say $\phi$ is $\rho$-equivariant.}, where the action of $\pi_1(X^\an,x)$ on $\scrM$ is via $\rho$.
\end{enumerate}

\end{setup}
Here we have used the following:
\begin{defn}We say $\rho:\pi_1(X^\an,x)\to\G(\bC)$ has norm one eigenvalues of local monodromy if for some log smooth compactification $\bar X'$ of a resolution $r:X'\to X$ and some faithful complex representation $V$ of $\G$, taking $V$ to be the local system on $X^\an$ induced by $V$ via $\rho$, the the local monodromy of $r^*V$ around any divisor has (complex) norm one eigenvalues.  Note that if this is satisfied, the same will be true for any log smooth compactification of any resolution and any complex representation of $\G$.
\end{defn}

Recall that by \cite{BMull}, the eigenvalues of local monodromy having norm one implies that the two natural structures as a definable coherent sheaf on $\cO_{X^\df}\otimes_{\bZ_{X^\df}}V$ coming from the flat sections and the canonical algebraic structure of the associated flat vector bundle are equivalent.

Observe that in Setup \ref{setup:alg} for an algebraic space $X$ and for any $\G$-equivariant definable analytic vector bundle bundle $\scrE$ on $\scrM$, there is a natural analytic vector bundle $\cE_X$ on $X^\an$ together with a natural $\pi_1(X^\an,x)$-equivariant isomorphism $\alpha:\pi^*\cE_X^\an\xrightarrow{\cong} \phi^*\scrE^\an$.  Here the naturality means that for any morphism $g:X\to Y$ there is an isomorphism $\lambda:(g^\an)^*\cE_Y\to \cE_{X}$ fitting into a commutative diagram
\[\begin{tikzcd}
\pi_X^*\cE_{X}\ar[d,"\pi_X^*\lambda",swap]\ar[rrd,"\alpha_X"]&&\\
(\tilde g^\an)^*\pi_Y^*\cE_Y\ar[rr,"(\tilde g^\an)^*\alpha_Y",swap]&&\phi_X^*\scrE^\an=(\tilde g^\an)^*\phi_Y^*\scrE^\an
\end{tikzcd}\]
with the obvious notation, where $\tilde g^\an:X^{K_X}\to Y^{K_Y}$ is the map induced by $f$ on covers.

In the following, by a constant $\G$-equivariant vector bundle on $\scrM$ we mean a $\G$-equivariant vector bundle of the form $\cO_\scrM\otimes_\bC V$ for a $\G$-representation $V$.

\begin{lem}\label{alg bundle}Assume Setup \ref{setup:alg} for an algebraic space $X$, and suppose $\scrE$ is a subquotient of a constant $\G$-equivariant vector bundle on $\scrM$.  Then:
\begin{enumerate}
\item
There is an algebraic vector bundle $E_X$ analytifying to $\cE_X$ which is uniquely determined by the property that for any diagram \eqref{lifting} (with $\scrX=X^\df$) and the induced map $\mu:\scrT\to \scrM$, the isomorphism $\tau^{\prime*}\alpha:(t^\an)^*E_X^\an\to (\mu^\an)^*\scrE^\an$ is the analytification of an isomorphism $t^*E_X^\df\to \mu^*\scrE$.
\item For any morphism $g:X\to Y$, there is an isomorphism $g^*E_Y\to E_{X}$ analytifying to the above isomorphism $\lambda:(g^\an)^*\cE_Y\to \cE_{X}$.
\end{enumerate}
\end{lem}

\begin{proof}There is clearly a definable analytic vector bundle $\scrE_Y$ with all the required properties.  By the assumption on $\scrE$, $\scrE_Y$ is a subquotient of the flat definable vector bundle corresponding to a local system induced by $\rho$, so by \cite{BMull} and definable GAGA $\scrE_Y$ has a unique algebraic structure which moreover satisfies all the required properties.
\end{proof}
\begin{setup}\label{setup:qp}Assume Setup \ref{setup:alg} for an algebraic space $X$, and further suppose we have:
\begin{enumerate}
\item A line bundle $\scrL$ on $\scrM$ which is a subquotient of a constant $\G$-equivariant vector bundle.
\item For any morphism $g:Z\to X$ from a smooth algebraic variety $Z$ with a log smooth compactification $(\bar Z,D)$, if $\phi_Z$ has generically discrete fibers then $L_Z$ extends to a big and nef line bundle $L_{\bar Z}$ on $\bar Z$.  Moreover, this extension is functorial with respect to morphisms $(\bar Z',D')\to  (\bar Z,D)$ of log smooth pairs which are isomorphisms on the open part.  
\end{enumerate}
\end{setup}
\begin{prop}\label{prop:qp}Assume \Cref{setup:qp} for an algebraic space $X$ and assume $\phi$ has discrete fibers.  Then $L_X$ is ample.  In particular, $X$ is a quasiprojective variety.
\end{prop}
\begin{proof}Immediate by \cite[Theorem 5.4]{bbt}.
\end{proof}

Critically, complex variations of pure Hodge structures fall within the framework of \Cref{setup:qp}.  The following is a consequence of the main theorem \Cref{mainShaf}, but we include it now as an example.

\begin{thm}\label{baby qp image}
Let $(\bar X, D)$ be a smooth proper log-pair and $X := \bar X \setminus D$. 
Let $V$ be a pure $\bC$-VHS on $X$ whose underlying local system has discrete monodromy $\Gamma$. Assume that $V$ has infinite local monodromy. Then the Stein factorization of the associated period map $X^\an \to \Gamma \backslash \scrM$ is the analytification of an algebraic morphism $X \to Y$, with $Y$ normal quasiprojective.
\end{thm}
    
\begin{proof}In the notation of \Cref{setup:twisted map}, we take $\scrM$ to be the period domain and $\phi:\tilde X^V\to\scrM$ to be the period map.  The definability condition follows from \cite{BMull} and the fact that the Hodge filtration extends to the Deligne extension (see the proof of Lemma 7.4 in \cite{brunebarbesemipos}) and the algebraicity follows from \Cref{thm:alg}.  Note the fibers of $\phi$ have compact connected components since the condition on the local monodromy implies by the Griffiths criterion \cite[Proposition 9.11]{Giii} that $X^\an\to\Gamma\backslash \scrM$ is proper.

It remains to show the quasiprojectivity.  By replacing $X$ with $Y$ we may assume $\phi$ has discrete fibers.  We take $\scrL$ to be the Griffiths bundle; then \Cref{setup:alg} is satisfied, and \Cref{alg bundle} applies to $\scrL$.  By the following lemma, \Cref{setup:qp} applies.  

    \begin{lem}\label{lem:big and nef}Let $X$ be a smooth algebraic variety, $V$ a $\bC$-VHS with unipotent local monodromy, and $N_X$ the Griffiths bundle.  Then there is a functorial nef extension $N_{\bar X}$ to any log smooth compactification.  Moreover, if the period map $\phi$ of $V$ has generically discrete fibers, then $N_{\bar X}$ is big.
\end{lem}
\begin{proof}
    The Griffiths bundle of the $\bR$-VHS $V\oplus\bar V$ is a power of that of $V$, so it suffices to consider a $\bR$-VHS and this is \cite[Lemma 6.15]{bbt}.
\end{proof}
\end{proof}
\begin{rem}
    There is a version of \Cref{baby qp image} in the mixed case, but the condition of infinite local monodromy must be replaced with the nonextendability of the local system (see \Cref{sect:nonextend}).
\end{rem}
\section{Hodge and twistor theory of miniversal local systems}\label{sect:miniversal}
In this section we construct functorial mixed Hodge structures on the miniversal deformation rings at points of the Betti stack $\cM_B(X)$ (see \Cref{sect:Betti}) underlying complex variations of mixed Hodge structure, as well as complex variations of mixed Hodge structures on the miniversal families themselves.  In fact, we will need the same statement in the category of variations of mixed twistor structures (see \Cref{sect:MHSdefn} for definitions), which gives additional structure to the miniversal family of $\cM_B(X)$ at every point.

We leave some relevant definitions for \Cref{sect:pro stuff}, but the precise statements are \Cref{thm:versal frame} and \Cref{thm:versal} below.  Before stating it, we make the following notational convention:

\begin{notn}\label{notn ms}
    Starting in \Cref{sect:saito/mochizuki}, we develop the theory in the Hodge and twistor cases largely in parallel.  We therefore use the phrase ``complex mixed structure'' ($\bC$-MS) as a stand-in to mean either a complex mixed Hodge structure ($\bC$-MHS) or complex mixed twistor structure ($\bC$-MTS), and $\bC$-VMS ($\bC$-AVMS) to mean an (admissible) variation of complex mixed Hodge structures or complex mixed twistor structures.  In either category, we denote by $\bT(0)$ the weight 0 Tate object, that is, the unit for the tensor structure. 
\end{notn}

We first state the result for the framed Betti moduli space, where the lack of automoprhisms makes the universal properties better behaved.

\begin{thm}\label{thm:versal frame} Let $X$ be a connected algebraic space, $x \in X$ a basepoint, and $V\in\cM_B(X)(\bC)$ a complex local system equipped with an admissible\footnote{Here, the admissibility hypothesis in the Hodge case includes the condition that the local monodromy is quasiunipotent, see \Cref{defn:avmhs}} (graded polarizable) complex variation of mixed structures ($\bC$-AVMS).  Let $\phi:V_{x}\to\bC^r$ be a framing of $V$ at $x$.  Let $M$ be $\bC^r$ equipped with the mixed structure induced via $\phi$.  Let $(\hat\cO_{R_B(X,x),(V,\phi)},\hat V)$ be the universal local system at $(V,\phi)$ for $R_B(X,x)$ with universal framing $\hat \phi:\hat V_x\to \hat\cO_{R_B(X,x),(V,\phi)}\otimes_\bC \bC^r$.
        \begin{enumerate}
            \item There exists a pro-$\tate$-MS-algebra structure on $\hat {\cO}_{R_B(X,x),(V,\phi)}$ and a pro-$\hat {\cO}_{R_B(X,x),(V,\phi)}$-AVMS structure on $\hat V$ that is compatible with $V$ and such that the framing $\hat\phi:\hat V_x\to \hat\cO_{R_B(X,x),(V,\phi)}\otimes_{\tate} M$ is a morphism of pro-$\hat {\cO}_{R_B(X,x),(V,\phi)}$-MS-modules.
    \item For a fixed $\bC$-AVMS structure on $V$, the pair $(\hat {\cO}_{R_B(X,x),(V,\phi)},(\hat V,\hat\phi))$ is uniquely determined by the following universal property.  For any artinian local $\tate$-MS-algebra $A$ and any $A$-AVMS $U$ which is equipped with a framing $\psi:U_x\to A\otimes_{\tate}M$ which restricts to $(V,\phi)$ mod $\mathfrak{m}_A$, there is a unique morphism $\hat {\cO}_{R_B(X,x),(V,\phi)}\to A$ of local pro-$\tate$-MS-algebras such that $(A,(U,\psi))$ is isomorphic to $A\otimes (\hat {\cO}_{R_B(X,x),(V,\phi)},(\hat V,\hat \phi))$.

    \item These structures are functorial.  For any morphism $f:(Y,y)\to (X,x)$ of connected algebraic spaces respecting basepoints, the induced pullback morphism $f^*:R_B(X,x)\to R_B(Y,y)$ induces a morphism $\hat {\cO}_{R_B(Y,y),(f^*V,f^*\phi)}\to \hat {\cO}_{R_B(X,x),(V,\phi)}$ of pro-$\tate$-MS-algebras.  In the same way, these structures are compatible with the direct sum and tensor product morphisms.
        \end{enumerate}

\end{thm}
On the other hand, the Hodge/twistor structures on $\cM_B(X)$ are more intrinsic because they do not depend on the choice of basepoint.  The functoriality properties are better understood in terms of the framed space, and the representability of the diagonal in part (3) of the next theorem allows us to pass from the framed to the unframed space and vice versa. 
 
\begin{thm}\label{thm:versal}Let $X$ be a connected algebraic space, and $V\in\cM_B(X)(\bC)$ underlying a $\bC$-AVMS.  Then for any miniversal family $(\hat\cO_{\cM_B(X),V},\hat V)$ of $\cM_B(X)$ at $V$, the following are true: 

\begin{enumerate}
    \item There exists a pro-$\tate$-MS-algebra structure on $\hat {\cO}_{\cM_B(X),V}$ and a pro-$\hat\cO_{\cM_B(X),V}$-AVMS structure on $\hat V$ that is compatible with $V$.  
    \item For a fixed $\bC$-AVMS structure on $V$, the pro-$\tate$-MS-algebra structure on $\hat {\cO}_{\cM_B(X),V}$ is uniquely determined (up to isomorphism) by the existence of a pro-$\hat\cO_{\cM_B(X),V}$-AVMS structure on $\hat V$ compatible with $V$.  If $V$ is a simple local system, the pro-$\hat\cO_{\cM_B(X),V}$-AVMS structure on $\hat V$ is unique as well (up to isomorphism).

    \item The diagonal $\Delta:\cM_B(X)\to\cM_B(X)\times\cM_B(X)$ is representable by mixed structures in the following sense.  Suppose $\Lambda$ is a pro-$\tate$-MS-algebra and $U_1,U_2$ two $\Lambda$-AVMSs.  Let $f_0:\tate\otimes_\Lambda U_1\to \tate\otimes_\Lambda U_2$ be an isomorphism of $\bC$-AVMS.  Form the fibered 2-product
    \[\begin{tikzcd}
        \Spec\hat\cO\ar[r,dashed]&\Spec\cO\ar[r]\ar[d]\ar[dr, phantom, "\square"]&\Spec\Lambda\ar[d]\\
        &\cM_B(X)\ar[r,"\Delta"]&\cM_B(X)\times\cM_B(X)
    \end{tikzcd}\]
    and let $\hat \cO$ be the completion of $\cO$ at the point $f_0$.  Then $\hat\cO$ admits a unique structure of pro-$\Lambda$-MS-algebra such that the induced isomorphism $\hat f:\hat\cO\otimes_\Lambda U_1\to \hat\cO\otimes_\Lambda U_2$ is a morphism of pro-$\hat\cO$-AVMS.
 
\end{enumerate} Finally, if $X$ is normal and $V$ underlies a complex variation of pure structures, then the maximal ideal of $\hat \cO_{\cM_B(X),V}$ is $W_{-1}\hat \cO_{\cM_B(X),V}$. 
\end{thm}

The Hodge cases of \Cref{thm:versal} and \Cref{thm:versal frame} were proven for $X$ smooth projective and $V$ a complex variation of Hodge structures in \cite{ES}, and part (1) of \Cref{thm:versal} has been addressed in the quasiprojective case in \cite{lefevrei,lefevreii}. The twistor case has been investigated by Simpson in the compact case \cite{Simpson-Hodge-filtration}, and some partial results given in the case of a quasiprojective curve \cite{Simpsonrank2twistor,simpsonHitchinDeligne}. In all of these cases, the approach is by the Goldman--Millson description of the deformation ring \cite{goldmanmillson}.  Our approach is instead by elementary deformation theory, relying on the results of Saito/Sabbah and Mochizuki to equip various ext groups with functorial Hodge/twistor structures, which allows us to handle Hodge/twistor structures on the deformation ring and the universal family simultaneously.  

\begin{rem}\label{rem:GoldmanMillson}
    The Goldman--Millson approach can be carried out as in \cite{ES} without modification for the substacks of $\cM_B(X)$ where we fix the conjugacy class of the local monodromies.  These are in some sense pure non-abelian Hodge substructures; their deformation theory is governed by the intersection cohomology, so in particular pure, and the harmonic theory carries through.
\end{rem}

\begin{rem}
    The category of complex mixed Hodge structures is equivalent to the category of $\bG_m$-equivariant mixed twistor structures.  While this perspective might streamline some of our constructions and proofs, for simplicity of exposition we prefer to understand mixed Hodge structures in terms of the classical description.
\end{rem}

 \subsection{The Betti stack}\label{sect:Betti}

 Let $X$ be a connected algebraic space.  We denote by $\cM_B(X)$ the algebraic stack of local systems on $X^\an$.  We can think of its points over an affine scheme $\Spec A$ as the groupoid of local systems of finite-rank free $A$-modules on $X^\an$, which we just refer to as free $A$-local systems on $X$.  It is the disjoint union of the stacks $\cM_B(X,r)$ of rank $r$ free local systems.  Concretely, after choosing a basepoint $x\in X(\bC)$, we may identify $\cM_B(X,r):=[\bGL_r\backslash R_B(X,x,r)]$, where $R_B(X,x,r):=\Hom(\pi_1(X^\an,x),\bGL_r)$ is the affine scheme whose $A$-points are homomorphisms $\pi_1(X^\an,x)\to\bGL_r(A)$ and where $\bGL_r$ acts by conjugation.  We can also think of an $A$-point of $R_B(X,x,r)$ as a free $A$-local system $V$ on $X^\an$ equipped with a framing at $x$---an isomorphism $V_x\xrightarrow{\cong}A^r$.  The algebraic stack $\cM_B(X)$ is naturally defined over $\bZ$ but we usually consider it over $\bQ$.  
 
 The stack $\cM_B(X)$ admits a quasiprojective good moduli space $M_B(X)$ in the sense of \cite{alpergood} (or GIT \cite{mumford}).  This in particular means there is an affine scheme $M_B(X)$ and a morphism $c_X:\cM_B(X)\to M_B(X)$ which is surjective on $\bC$-points and universally closed (see \cite[\href{https://stacks.math.columbia.edu/tag/0513}{Tag 0513}]{stacks-project}) in the Zariski topology (see \Cref{sect:stackZariski} for more discussion on the Zariski topology for stacks).  We may construct it as follows.  After choosing a basepoint $x$, observe that $R_B(X,x,r)=\Spec H^0(R_B(X,x,r), \cO_{R_B(X,x,r)})$ is affine.  Then $M_B(X,r)=\Spec(H^0(R_B(X,x,r), \cO_{R_B(X,x,r)})^{\bGL_r})$ is the spectrum of the invariant ring and the $\bC$-points of $M_B(X)$ are naturally identified with isomorphism classes of semisimple complex local systems \cite[Proposition 6.1]{SimpsonmoduliII}. Denote by $\cM_B(X,r)(\bC)^{\ss}\subset\cM_B(X,r)(\bC)$ the set of rank $r$ semisimple local systems, which is easily seen to be $\bQ$-constructible. We call the restriction $p_X:\cM_B(X)(\bC)^{\ss}\to M_B(X)(\bC)$ and the resulting constructible retraction $\ss_X=p_X^{-1}\circ c_X:\cM_B(X)(\bC)\to \cM_B(X)(\bC)$ the semi-simplification.  

  \begin{lem}\label{ss in closure}
For any (Zariski/euclidean) closed $\Sigma\subset\cM_B(X)(\bC)$, $\ss_X(\Sigma)\subset\Sigma$.

\end{lem}
\begin{proof}
    The proof is standard:  for a representation $V$, taking a filtration with semisimple subquotients $V_i$ and choosing a splitting $V=\bigoplus_{i\geq 0} V_i$ as complex vector spaces, conjugation by the operator $\bigoplus_{i\geq 0} t^{i}$ limits to the semisimplification as $t\to 0$.  
\end{proof}

For any morphism $f:X\to Y$ of algebraic spaces there is a representable $\bQ$-morphism of algebraic stacks $f^*:\cM_B(Y)\to\cM_B(X)$ given by pull-back, which is in fact the disjoint union of the quotients of the natural pull-backs $f^*:R_B(Y,y,r)\to R_B(X,x,r)$ of framed local systems.

\begin{rem} More generally, observe that for any finitely generated group $\Gamma$, there is a representation stack $\cM_B(\Gamma)$ with affine good moduli space $M_B(\Gamma)$.  Moreover, these structures are functorial with respect to $\Gamma$.
\end{rem}

It will be useful throughout to restrict to curves, so we introduce the following:
\begin{defn}\label{defn lefschetz}
    Let $X$ be a connected normal algebraic space.  A \emph{Lefschetz curve} is a locally closed immersion $i:C\to X$ of a connected smooth affine curve such that $i_*:\pi_1(C,c)\to \pi_1(X,x)$ is surjective (for compatibly chosen basepoints) and for some projective resolution $\pi:X'\to X$ for $X'=\bar X'\setminus D$ with $(\bar X',D')$ a log smooth projective variety, there is a factorization $i':C\to X'$ of $i$ through $\pi$ such that the closure of $i'(C)$ in $\bar X'$ meets every irreducible component of the regular locus of $D'$.
\end{defn}
\begin{lem}\label{lefschetz curve}
    For any connected normal algebraic space $X$ a Lefschetz curve $i:C\to X$ exists.
\end{lem}
\begin{proof}
    Since $X$ is normal, $\pi_1(X',x')\to\pi_1(X,x)$ is surjective for $\pi:X'\to X$ as in the definition, so it suffices to take $X=\bar X\setminus D$ for $(\bar X,D)$ a log smooth projective variety, which is given by \cite[\S II.5.1]{stratmorse}.
\end{proof}
\begin{rem}\label{qu lefschetz}Recall that a local system $V$ on a connected normal algebraic space $X$ has quasiunipotent local monodromy if for some (hence any) projective resolution $\pi:X'\to X$ by $X'=\bar X'\setminus D$ for $(\bar X',D')$ a log smooth projective variety, the local monodromy of $\pi^*V$ is quasiunipotent.  It follows that for any Lefschetz curve $i:C\to X$, $i^*V$ has quasiunipotent local monodromy if and only if $V$ does.
\end{rem}

\subsection{Mixed twistor/Hodge structures and their variations:  definitions}
\label{sect:MHSdefn}In this section we recall the definitions of complex mixed Hodge structures ($\bC$-MHSs), mixed twistor structures ($\bC$-MTSs), admissible complex variations of mixed Hodge structures ($\bC$-VMHSs), and admissible complex variations of mixed twistor structures ($\bC$-VMTSs).  The main references are Deligne \cite{Delignehodgeii} and for example \cite{ES} (and the references therein) in the Hodge case, and \cite{simpsontwistor,mochizukimixedtwistor} in the twistor case.

\subsubsection{Over a point:  Hodge case}

\begin{defn}  Let $k\in\bZ$.  A complex pure Hodge structure ($\bC$-HS) of weight $k$ is a triple $(V,F^\bullet,F'^\bullet)$ where $V$ is a finite-dimensional $\bC$-vector space and $F^\bullet,F'^\bullet$ are $k$-opposed decreasing filtrations.  Recall that this means that $\gr_{F}^p\gr_{F'}^qV=0$ if $p+q\neq k$, and implies that there is a splitting $V=\bigoplus_{p+q=k} V^{p,q}$ such that $F^p=\bigoplus_{i\geq p} V^{i,k-i}$ and $F'^q=\bigoplus_{i\geq  q}V^{k-i,i}$ given by $V^{p,q}=F^p\cap F'^{q}$.  We often refer to the $\bC$-HS by just $V$.  

Morphisms of $\bC$-HS are filtered morphisms.  A polarization of $V$ is a hermitian form $h$ on $V$ such that the splitting $V=\bigoplus_{p+q=w} V^{p,q}$ is orthogonal and $\bigoplus_{p+q=k}(-1)^ph|_{V^{p,q}}$ is positive definite.  
    
\end{defn}
Note that any $\bC$-HS is polarizable.

\begin{defn} A complex mixed Hodge structure ($\bC$-MHS) is a quadruple $(V,W_\bullet,F^\bullet,F'^\bullet)$ where $V$ is a finite-dimensional $\bC$-vector space, $W_\bullet$ is an increasing filtration (called the weight filtration), and $F^\bullet,F'^\bullet$ are decreasing filtrations, such that $(\gr^W_kV,F^\bullet\gr^W_kV,F'^\bullet\gr^W_kV)$ is a weight $k$ $\bC$-HS for all $k$.  Equivalently, we ask that $\gr^p_F\gr^q_{F'}\gr^W_kV=0$ if $p+q\neq k$.

Morphisms of $\bC$-MHS are morphisms compatible with all three filtrations.  A graded polarization $h_\bullet$ is a polarization $h_k$ on the graded object $\gr^W_k V$ for each $k$.
    
\end{defn}

Note that any $\bC$-MHS is graded-polarizable.  Morphisms of $\bC$-MHS are automatically strict for each filtration, and so the category of $\bC$-MHS is abelian.

\subsubsection{Over a point:  twistor case}
\begin{defn}  A complex mixed twistor structure ($\bC$-MTS) is a pair $(V,W_\bullet)$ where $V$ is a locally free coherent $\cO_{\bP^1}$-module on $\bP^1$ and $W_\bullet$ is an increasing locally split filtration (called the weight filtration) by $\cO_{\bP^1}$-submodules such that for each $k$, $\gr^W_kV\cong \cO_{\bP^1}(k)^{n_k}$ for some $n_k$.  A $\bC$-MTS is pure of weight $k$ if $\gr_j^WV=0$ for all $j \neq k$.  A morphism of $\bC$-MTS is a filtered morphism of $\cO_{\bP^1}$-modules.
\end{defn}
Morphisms of mixed twistor structure are automatically strict with respect to the weight filtration.

\subsubsection{Splittings of $\bC$-MTS}
Deligne \cite{Delignehodgeii} shows the existence of functorial splittings of either $(W_\bullet,F^\bullet)$ or $(W_\bullet, F'^\bullet)$ in the category of $\bC$-MHS.  We do the same for $\bC$-MTS.

It is first useful to have the following version of the Rees construction, which is the twistor version of \cite{penacchio}.  Let $\bL$ be the total space of the line bundle $\cO_{\bP^1}(1)$ on $\bP^1$ with the scaling action by $\bG_m$, $\mathbf{0}_\bL\subset\bL$ the zero section, and $\pi:\bL\to \bP^1$ the projection.  There is a natural equivalence of categories between coherent sheaves $V$ on $\bP^1$ and $\bG_m$-equivariant coherent sheaves $\cV$ on $\bL\setminus \mathbf{0}_\bL$ given by pullback along the quotient $\bL\setminus \zero_\bL\to \bP^1$ by the $\bG_m$-action.  There is moreover an equivalence of categories between filtered $\bG_m$-equivariant coherent sheaves $(\cV,\cW_\bullet)$ on $\bL\setminus\zero_\bL$ and $\bG_m$-equivariant coherent sheaves $\bar \cV$ on $\bL$ with no embedded points along the zero section. 
 In the forward direction, we associate to a filtered equivariant sheaf $(\cV,\cW_\bullet)$ (corresponding to a filtered sheaf $(V,W_\bullet)$ on $\bP^1$) the equivariant sheaf on $\bL$ generated by $\pi^*W_k(-k\zero_\bL)$, or equivalently generated by the sections of $\cW_k$ with torus weights $\geq k$.  Note that $\bar V|_{\zero_\bL}$ is canonically isomorphic to $\bigoplus_k \pi^*\gr^W_k V(-k)$, and $\pi^*\gr^W_k(-k)$ has pure torus weight $-k$.  An equivariant sheaf $\bar \cV$ on $\bL$ has a natural filtration $\bar\cW_k$ by sections of torus weight $\geq -k$, and we associate the restriction to $\bL\setminus \zero_\bL$.  We have therefore proven the following:

 \begin{lem}\label{tw equiv}
     There is a natural equivalence of categories as above between the category of complex mixed twistor structures and the category of $\bG_m$-equivariant locally free sheaves on $\bL$ whose restriction to $\zero_\bL$ is trivial.
 \end{lem}
 As a consequence, we obtain a version of the Deligne splitting.  
 \begin{lem}\label{lem tw splitting}Fix a point $\lambda\in\bP^1$.  Then every mixed twistor structure admits a functorial splitting of its weight filtration in restriction to $\bP^1\setminus \{ \lambda \}$.  The splitting is moreover compatible with tensor products and duals.
     
 \end{lem}
 \begin{proof}
     Let $s\subset\bL$ be a section of $\cO_{\bP^1}(1)$ vanishing at $\lambda$.  To any mixed twistor structure $(V,W_\bullet)$, consider the corresponding $\bG_m$-equivariant sheaf $\bar \cV$ on $\bL$.  Observe that since $\bar\cV|_{\zero_\bL}$ is trivial, the same is true for $\bar\cV|_s$.  Indeed, this is true in the pure case, since it is true for line bundles on $\bP^1$, and so $\bar\cV|_s$ is an iterated extension of trivial vector bundles, hence trivial.  It further follows that the filtration restricted to $s$ is by trivial subbundles.  If $\zero_x$ is the zero of $s$ (which maps to $\lambda$), then the fiber of $\bar\cV$ at $\zero_x$ has a canonical grading which splits the weight filtration.  Since $\bar\cV|_s\cong H^0(\bar\cV|_s)\otimes\cO_s$ and the filtration is induced by the filtration on $H^0(\bar\cV|_s)$, the claim follows.  Note that the splitting only depends on $\lambda$ and not $s$, since there is a unique $s$ vanishing at $\lambda$ up to scale, and the splitting is functorial and compatible with tensor products and duals.
 \end{proof}
 
In the case of $\bC$-MHS (that is, the $\bG_m$-equivariant $\bC$-MTS case), we obtain the two Deligne splittings by taking sections vanishing at the two torus fixed points, since the fibers there are canonically $\gr_F\gr^W$ and $\gr_{F'}\gr^W$.

\subsubsection{Over a base:  Hodge case}

\begin{defn}
    Let $X$ be a complex manifold.  A graded polarizable complex variation of mixed Hodge structures ($\bC$-VMHS) on $X$ is a quadruple $(V,W_\bullet ,F^\bullet,F'^\bullet)$ where 
    \begin{itemize}
        \item $V$ is a complex local system;
        \item $W_\bullet $ is a flat increasing filtration of $V$ (called the weight filtration);

        \item $F^\bullet $ is a locally split decreasing filtration of $\cO_{X}\otimes_{\bC_{X}}V$ such that $\nabla F^p\subset F^{p-1}\otimes\Omega_{X}$ for each $p$, where $\nabla$ is the natural flat connection;
        \item $F'^\bullet $ is a locally split decreasing filtration of $\bar\cO_{X}\otimes_{\bC_{X}}V$ such that $\bar\nabla F'^p\subset F'^{p-1}\otimes\bar \Omega_{X}$ for each $p$, where $\bar \nabla$ is the natural flat connection;
       \item There exists a flat hermitian form $h_k$ on each $\gr_k^WV$ such that for each $x\in X$, $(V_x,(W_\bullet)_x,(h_\bullet)_x, F_x^\bullet,F'^\bullet_x)$ is a graded polarized $\bC$-MHS.

    \end{itemize}
    A $\bC$-VMHS for which $\gr_k^WV=0$ for all but one $k$ is a polarizable complex variation of pure Hodge structures ($\bC$-VHS).  A morphism of $\bC$-VMHS is a morphism of local systems which is compatible with all three filtrations. 
\end{defn}

For background on the definitions of (pre)-admissibility, see \cite{sz,kashiwara}. 
\begin{defn}\label{defn:avmhs}
    Let $(V,W_\bullet,F^\bullet,F'^\bullet)$ be a $\bC$-VMHS on the punctured disk $\Delta^*$.  Assume $V$ has unipotent monodromy.  We say $(V,W_\bullet,F^\bullet,F'^\bullet)$ is pre-admissible if:
    \begin{itemize}
        \item $F^\bullet$ (resp. $F'^\bullet$) extends to a filtration of the Deligne extension $\cV$ (resp. $\bar \cV$) of $(\cO_X\otimes_{\bC_X}V,\nabla)$ (resp. $(\bar \cO_X\otimes_{\bC_X}V,\bar \nabla)$) such that each $\gr_F^p\gr_k^{\cW}\cV$ (resp. $\gr_{F'}^p\gr_k^{\bar \cW}\bar\cV$) is locally free. 
        \item There exists a relative monodromy-weight filtration.
    \end{itemize}
\end{defn}

\begin{defn}
    Let $(\bar X,D)$ be a log smooth compact manifold with $X=\bar X\setminus D$.  An admissible graded polarizable complex variation of mixed Hodge structures ($\bC$-AVMHS) on $X$ is a $\bC$-VMHS $(V,W_\bullet ,F^\bullet,F'^\bullet)$ such that

\begin{itemize}
\item $V$ has quasiunipotent local monodromy.
            \item For any map $f:(\Delta,0)\to (\bar X,D)$ such that $f^*V$ has unipotent local monodromy, $f^*(V,W_\bullet ,F^\bullet,F'^\bullet)$ is pre-admissible.
        
\end{itemize}
\end{defn}

\subsubsection{Over a base:  twistor case}  
\begin{exa}\label{exa defining can twistor}
Let $X$ be a complex manifold.  Let $\bfV=(\cV,h,\nabla)$ be a tame harmonic bundle on $X$ with underlying $C^\infty$ bundle $\cV$ and $\nabla=\nabla_h+\theta+\theta^*$ (see \Cref{sect:harmonic bundles} for background).  Let $\scrA_X:=C^\infty_X\boxtimes \cO_{\bP^1}$ be the sheaf of $C^\infty$ functions on $X_{\bP^1}=X\times \bP^1$ which are holomorphic in the $\bP^1$ direction.  Let $\scrV:=\cV\boxtimes \cO_{\bP^1}$, which is naturally an $\scrA_X$-module on $X_{\bP^1}$.  Then, choosing generating sections $x,y$ of $\cO_{\bP^1}(1)$ vanishing at $0$ and $\infty$ respectively, there is a natural $x\partial_{X_{\bP^1}/\bP^1}+y\bar\partial_{X_{\bP^1}/\bP^1}$-connection $\mathscr{D}:\scrV\to\scrV\otimes\Omega_{X_{\bP^1}/\bP^1}(1)$ given by $\scrD=xD'+yD''$ where $D'=\nabla_h^{1,0}+\theta^*$ and $D''=\nabla_h^{0,1}+\theta$ which satisfies the integrability condition $0=\scrD^2=x^2D'^2+xy(D'D''+D''D')+y^2D''^2$.  The resulting $(\scrV,\scrD)$ is the variation of pure twistor structures ($\bC$-VTS) of weight 0 associated to the tame harmonic bundle $\bfV$.  If the harmonic bundle is tame and purely imaginary, the harmonic metric is unique up to flat automorphism, and such an automorphism induces an isomorphism of the associated $\bC$-VTS.  Thus, any semisimple complex local system underlies a $\bC$-VTS which is unique up to isomorphism.
\end{exa}
\begin{defn}
    Let $X$ be a complex manifold.  A graded polarized variation of mixed twistor structures ($\bC$-VMTS) on $X$ is a triple $(\scrV,W_\bullet\scrV,\scrD)$ where 
    \begin{itemize}
        \item $\scrV$ is a $\scrA_X$-module on $X_{\bP^1}$;
        \item $\scrD:\scrV\to \scrV\otimes\Omega_{X_{\bP^1}/\bP^1}(1)$ is an integrable $x\partial_{X_{\bP^1}/\bP^1}+y\bar\partial_{X_{\bP^1}/\bP^1}$-connection as above;
        \item $W_\bullet\scrV $ is a $\scrD$-flat increasing filtration of $\cV$ (called the weight filtration);
         \item There exists a harmonic hermitian metric $h_k$ on each $\gr_k^{W}\scrV$ such that $(\gr^W_k\scrV,\gr_k^W \scrD)$ is the $\bC$-VTS of weight $k$ associated to a tame purely imaginary harmonic bundle by shifting by $\cO_{\bP^1}(k)$ (see \Cref{sect:harmonic bundles}).
    \end{itemize}
    A morphism of $\bC$-VMTS is a morphism of $\scrA_X$-modules which is compatible with the filtration and the operator $\scrD$. 
\end{defn}

\begin{rem}
Mochizuki \cite{mochizukimixedtwistor} develops the theory quite generally, but we will only consider tame purely imaginary variations of mixed twistor structure.  In fact, only the results of this section in the quasiunipotent local monodromy case will be used.
\end{rem}
The admissibility conditions are more complicated in the twistor case, so we make the following definition without fully explaining what twistor $\mathscr{D}$-modules are.  See \cite{mochizukimixedtwistor} for details.
\begin{defn}\label{defn:avmts}Let $(\bar X,D)$ be a log smooth projective variety with $X=\bar X\setminus D$.  A $\bC$-AVMTS on $X$ is a mixed twistor $\mathscr{D}$-module on $\bar X$ whose restriction to $X$ underlies a tame purely imaginary graded polarized variation of mixed twistor structures as above.
\end{defn}

\subsection{Ext groups over a point}
We employ \Cref{notn ms}, so most results are stated for ``mixed structures''.  For notational simplicity we denote ``$\Hom_{\CMS}$'' and ``$\Hom_{D^b(\CMS)}$'' by ``$\Hom_{\MS}$'', and ``$\Ext^i_{\CMS}$'' by ``$\Ext^i_{\MS}$'' 
 
\begin{lem}[{Beilinson \cite[Cor. 1.10]{beilinson}}$+\epsilon$]\label{lem beilinson}
    Let $M$ and $N$ be two $\bC$-MS.  Then $\Ext^i_{\MS}(M,N)=0$ for $i>1$.  Thus, for any object $E$ of $D^b(\bC\mhyphen\mathrm{MS})$, there is a noncanonical isomorphism  $E\cong \bigoplus_i\mathscr{H}^i(E)[-i]$. 
\end{lem}
\begin{proof}
The result of Beilinson is in the case of $\bR$-MHSs, and since the functor $V\mapsto V\oplus \bar V$ is an exact functor to the category of $\bR$-MHS, so we obtain the result for $\bC$-MHSs.

In the twistor case, morphisms of mixed twistor structure are strict with respect to the weight filtration, so $W_0$ is an exact functor.  Clearly $H^0(\bP^1,W_0\Hom(M,N))=\Hom_\MTS(M,N)$, and the claim follows.
\end{proof}

As usual, the splitting is not canonical, but there is a canonical filtration which is noncanonically split.
\begin{cor}
    For any object $M$ of $D^b(\CMS)$, there is a canonical short exact sequence

\[0\to \Ext^1_{\MS}(\tate,\mathscr{H}^{i-1}(M))\to \Hom_{\MS}(\tate,M[i])\to \Hom_{\MS}(\tate,\mathscr{H}^i(M))\to 0.\]
\end{cor}

 In fact, the proof of \cite{beilinson} shows that, for any object $M$ of $D^b(\CMHS)$ (thought of as a complex), there is a natural identification
 \[R\Hom_{\MHS}(\tate,M)=[F^0W_0M\oplus F'^0W_0M\xrightarrow{\iota-\iota'}W_0M]\]
 where $\iota:F^0W_0M\to W_0M$ and $\iota':F'^0W_0M\to W_0M$ are the inclusions, since the functors $F^0W_0,F'^0W_0,W_0$ are exact by strictness of the morphisms in $(\CMHS)$.  Putting this together with the proof of \Cref{lem beilinson} in the twistor case, we obtain:

\begin{cor}\label{what is ext1}\hspace{1in}
\begin{enumerate}
    \item For any $\bC$-MHS $M$, we have a natural identification
    \[\Ext^1_{\MHS}(\tate,M)\cong W_{0}M/F^0W_0M+F'^0W_0M.\]
    \item For any $\bC$-MTS $M$, we have a natural identification
    \[\Ext^1_{\MTS}(\tate,M)\cong H^1(\bP^1,W_{0}M).\]
\end{enumerate}

\end{cor}
We can think of how an element $\alpha\in\Ext^1_\MS(M,N)$ acts on an extension
\[0\to N\to E\to M\to 0\]
in the two cases as follows.  In the Hodge case we represent $\alpha$ by $f\in W_0\Hom(M,N)$ and twist $F'^\bullet$ by the endomorphism $1+f$ of $E$.  In the twistor case, we may represent $\alpha$ by a \v{C}ech 1-cocyle $f$ valued in $W_0\Hom(M,N)$, and twist the gluing map for the filtered bundle $E$ by $1+f$.  The result is a mixed twistor structure by the twistor version of \cite[Lemma 1.6]{ES} using \Cref{tw equiv}.

\subsection{The results of Saito, Sabbah and Mochizuki}\label{sect:saito/mochizuki}In this section we review the results of Saito, Simpson, Sabbah, and Mochizuki putting functorial mixed Hodge/twistor structures on the cohomology groups of admissible variations of mixed Hodge/twistor structures.  

\subsubsection{Hodge case}  Cohomology groups of $\bC$-AVMS are equipped with functorial $\bC$-MSs by using Saito's theory of mixed Hodge modules in the Hodge case and Mochizuki's theory of mixed twistor modules in the twistor case.  The latter is developed in the case of complex coefficients, so we may quote the necessary results.  In the former case, most of the results in the literature are stated in the case of $\bQ$ or $\bR$ coefficients, so we take some time to precisely define the required notions for complex coefficients.  For background, see \cite{saito88,saito90}, and specifically \cite{saitodef} for the case of real coefficients.  The extension to complex coefficients is discussed in \cite[\S3.2]{sabbahCMHM}.

For any algebraic space $X$, there is a category $\MHM(X,\bR)$ of algebraic real mixed Hodge modules.  The algebraicity condition means they extend to an algebraic compacification.  For smooth $X$, the category of pure Hodge modules $\HM(X,\bR)$ is a subcategory of the category $\MF(X,\bR)$ of filtered regular holonomic $D$-modules with real structure $(M,F^\bullet,\alpha,V)$, where $(M,F^\bullet)$ is a filtered regular holonomic $D$-module, $V$ a $\bR$-perverse sheaf on $X$, and $\alpha: \DR(M)\to \bC_{X}\otimes_{\bR_{X}} V$ an isomorphism in $D^b(\bC_{X})$ between the De Rham complex of $M$ and $\bC_X\otimes_{\bR_{X}}V$.  The category of mixed Hodge modules is then a subcategory of the category $\MFW(X,\bR)$ of filtered objects $(\cM, W_\bullet)$ in $\MF(X,\bR)$.
 
 Again for smooth $X$, the subcategory of smooth objects in $\MHM(X,\bR)$ (namely those for which $V$ is a shifted local system) is equivalent to the category of admissible graded polarizable real variations of mixed Hodge structures on $X$ \cite{saitodef}.  For arbitrary $X$, we can still make sense of the category $\MFW(X,\bR)$ locally using an embedding in an ambient smooth variety and gluing, or globally for projective $X$ using a projective embedding, and this is how $\MHM(X,\bR)$ is defined.  Note that the category $\MFW(X,\bC)$ can be naturally defined as well, although in this case the complex perverse sheaf is redundant.
 
 The derived category $D^b\MHM(X,\bR)$ admits a faithful exact functor \[\mathrm{coeff}:D^b\MHM(X,\bR)\to D^b_c(\bR_{X})\] to the derived category of constructible $\bR_{X}$-modules by extracting the underlying $\bR$-perverse sheaf.  The restriction to $\MHM(X,\bR)$ is faithful and exact.  For algebraic maps $f:X\to Y$, the categories $D^b\MHM(X,\bR)$ admit natural exact functors $f_*,f^*,f_!,f^!,\mathbb{D}_X,\otimes$ which are compatible with the corresponding functors on $D^b(\bR_{X^\an})$ via $\mathrm{coeff}$. 

 As in \cite[\S3.2]{sabbahCMHM}, we make the following definition. 
\begin{defn}Let $X$ be an algebraic space.  An algebraic complex mixed Hodge module is a direct factor of an algebraic real mixed Hodge module in $\MFW(X,\bC)$.  A morphism of algebraic complex mixed Hodge modules $f:M\to N$ is a morphism in $\MFW(X,\bC)$ which arises from restriction to direct factors of a morphism of algebraic real mixed Hodge modules.  We denote the category of algebraic complex mixed Hodge modules by $\MHM(X,\bC)$, and the derived category by $D^b\MHM(X,\bC)$.
\end{defn}

We have the following version of the six functor formalism.  
\begin{thm}[Saito \cite{saito88,saito90}]\label{lem: six functors}
Let $X$ be an algebraic space.
\begin{enumerate}
    \item There is a natural faithful exact functor 
\[\mathrm{coeff}:D^b\MHM(X,\bC)\to D^b_c(\bC_{X})\]
by passing to the underlying perverse sheaf.
\item There are natural functors $f_*,f^*,f_!,f^!,\mathbb{D}_X,\otimes$ (the first four associated to any algebraic map $f:X\to Y$) lifting the corresponding functors on $D^b(\bC_X)$.  The pair $(f^*,f_*)$ is adjoint, and $f^!=\mathbb{D}_Xf^*\mathbb{D}_X, f_!=\mathbb{D}_Xf_*\mathbb{D}_X$.

\item For smooth $X$, the category of smooth objects of $\MHM(X,\bC)$ is equivalent to the category of $\bC$-AVMHS.
\end{enumerate}  
\end{thm}
\begin{proof}
The functors are defined on the level of filtered $D$-modules and compatible with taking direct summands, hence follow from the case of real coefficients.
\end{proof}
\subsubsection{Twistor case}

We denote by $\MTM(X,\bC)$ the full subcategory of the category of mixed twistor $\mathscr{D}$-modules consisting of tame purely imaginary algebraic mixed twistor $\mathscr{D}$-modules.  The algebraicity condition means the twistor $\mathscr{D}$-module extends to $\bar X$.
\begin{thm}[Sabbah \cite{Sabbah_twistor_D_modules}, Mochizuki \cite{mochizukiams1,Mochizuki-AMS2,mochizukimixedtwistor}]\label{lem: six functors tw}
Let $X$ be an algebraic space.
\begin{enumerate}
    \item For each $\lambda\in\bG_m$ there are natural exact specialization functors 
\[\begin{tikzcd}
    D^b\MTM(X,\bC)\ar[r,"\mathrm{sp}_\lambda^{DR}"]\ar[rd,"\mathrm{sp}_\lambda^{B}",swap]& D^b\Hol(X)\ar[d]\\
    &D^b_c(\bC_X)
\end{tikzcd}\]
by scaling the $\lambda$-connection to the derived category of regular holonomic $D$-modules and the derived category of constructible sheaves, where the right functor is the Riemann--Hilbert functor.  
\item For $\lambda=1$, $\sp^{B}_1$ is faithful, and there are natural functors $f_*,f^*,f_!,f^!,\mathbb{D}_X,\otimes$ (the first four associated to any algebraic map $f:X\to Y$) commuting with the corresponding functors on $D^b(\bC_X)$ via $\sp^{B}_1$.  The pair $(f^*,f_*)$ is adjoint, and $f^!=\mathbb{D}_Xf^*\mathbb{D}_X, f_!=\mathbb{D}_Xf_*\mathbb{D}_X$.
\end{enumerate}

\end{thm}
\subsubsection{Consequences}
We now deduce some consequences in both the Hodge and twistor case.  By a ``mixed module'' we mean either an algebraic complex mixed Hodge module or a tame purely imaginary algebraic mixed twistor $\scrD$-module. We denote the category of mixed modules by $\MM(X,\bC)$ and the derived category by $D^b\MM(X,\bC)$.  We make the following definition (see also \cite{bbt2}):
\begin{defn}
    For a connected algebraic space $X$, we define a $\bC$-AVMS to be a smooth object of $D^b\MM(X,\bC)$ whose image under $\mathrm{coeff}$ in the Hodge case (or the specialization at $\lambda=1$ in the twistor case) is supported in degree 0 (with respect to the standard t-structure)---that is, an object whose image under $\mathrm{coeff}$ is a local system in degree 0.
\end{defn}

We abusively denote ``$\Hom_{D^b\MM(X,\bC)}$'' by ``$\Hom_{\MM}$''.  Note however that if $X$ is not smooth, smooth objects of $D^b\MM(X,\bC)$ will not generally be elements of $\MM(X,\bC)$, even up to shifts.  We also denote by $(\bC_X\text{-}\mathrm{Mod})$ the category of constructible $\bC_X$-modules and we likewise denote ``$\Hom_{D^b_c(\bC_X)}$'' by ``$\Hom_{\bC_X}$''.  Note that if $\LS(X,\bC)$ is the category of complex local systems, then the natural functor $\LS(X,\bC)\to (\bC_X\text{-}\mathrm{Mod})$ is fully faithful with extension closed image.

\begin{lem}\label{lem esnault}Let $X$ be an algebraic space and $M,N$ two $\bC$-AVMSs on $X$ (supported in the same degree).
\begin{enumerate}
    \item The groups $\Ext^i_{\bC_X}(M,N)$ carry functorial $\bC$-MSs.
    \item  There is a natural short exact sequence
\begin{align*}0\to \Ext^1_{\MS}(\tate,\Hom_{\bC_X}(M,N))&\to \Hom_{\MM}(M,N[1])\\
&\to \Hom_{\MS}(\tate,\Ext^1_{\bC_X}(M,N))\to 0
\end{align*}

\end{enumerate}

\end{lem}
\begin{cor}[{cf. \cite[Lemma 4.3]{esnaultdaddezio}}]\label{lifting MHM ext}
    For $X,M,N$ as above, an extension of local systems
    \[0\to N\to E\to M\to 0\]
    can be lifted to a distinguished triangle of smooth objects in $D^b\MM(X,\bC)$ if and only if the extension class in $\Ext^1_{\bC_X}(M,N)$ is a Tate class.  In this case, the set of isomorphism classes of such lifts is a torsor under $\Ext^1_{\MS}(\tate,\Hom_{\bC_X}(M,N))$, which acts naturally on the filtrations $F^\bullet,F'^\bullet$ in the underlying $\bC$-AVMS via the identification of \Cref{what is ext1}.
\end{cor}
\begin{proof}[Proof of \Cref{lem esnault}]
    (1) follows from \Cref{lem: six functors}, since for $M,N$ smooth we have $\Ext^i_{\bC_X}(M,N)\cong H^i(X,\cHom(M,N))$. 

For (2), let $\pt_X:X\to \Spec\bC$ be the map to a point. 
 Because $M$ and $N$ are smooth we may naturally identify
\[\Hom_{\MM}(M,N[1])\cong \Hom_{\MM}(\pt_X^*\tate,\cHom(M,N)[1])\]
and thus assume $M=\pt_X^*\tate$.
  There is an adjunction morphism $\pt_X^*\pt_{X*}\to 1$ in $D^b\MM(X,\bC)$.  Thus, we have a natural identification
  \[\Hom_{\MM}(\pt_X^*\tate,N[1])\cong\Hom_{\MS}(\tate,\pt_{X*}N[1]).\]
Since $\mathscr{H}^i\pt_{X*}N=H^i(X,N)$, by \Cref{lem beilinson}, there is a natural short exact sequence
\begin{align*}0\to \Hom_{\MS}(\tate,H^0(X,N)[1])&\to \Hom_{\MS}(\tate,\pt_{X*}N[1])\\&\to\Hom_{\MS}(\tate,H^1(X,N))\to 0
\end{align*}
which yields the claim.
\end{proof}

\subsection{Recollections from deformation theory}
We briefly recall some elements of deformation theory as in \cite[\href{https://stacks.math.columbia.edu/tag/06G7}{Tag 06G7}]{stacks-project}.  In practice, one can often get by with the more classical deformation functor of isomorphism classes, but the entire category of deformations (which is naturally fibered in groupoids over the category of artinian affine schemes) is more natural.

Let $\Lambda$ be a complete local noetherian $\bC$-algebra with residue field $\bC$.  Let $(\Art/\Lambda)$ be the category of local artinian $\Lambda$-algebras.  We identify $(\Art/\Lambda)^\mathrm{op}$ with the category of artinian local $\Lambda$-schemes. A deformation category is a category cofibered in groupoids $\cD\to(\Art/\Lambda)$ (meaning $\cD^\op\to(\Art/\Lambda)^{\op}$ is fibered in groupoids) for which $\cD(\bC)$ is equivalent to a point and such that the Rim--Schlessinger gluing condition holds:  for any morphisms $A_1\to A$ and $A_2\to A$ in $(\Art/\Lambda)$ with $A_1\to A$ surjective, the natural functor $\cD(A_1\times_AA_2)\to \cD(A_1)\times_{\cD(A)}\cD(A_2)$ is an equivalence of categories\footnote{The fiber product is a fiber product of groupoids, so the objects are pairs $(x_1,x_2)$ of objects in $\cD(A_1)$ and $\cD(A_2)$ with a choice of isomorphism of the pullbacks to $A$.}.  The condition in particular implies the set of isomorphism classes in $\cD(\bC\oplus\bC\epsilon)$ (where $\epsilon^2=0$ and $\m_\Lambda$ kills $\epsilon$) has the structure of a $\bC$-vector space, called the tangent space $\ft_\cD$.  It also implies that there is a $\bC$-vector space $\fa_\cD$ (the infinitesimal automorphisms) which is identified with the group of automorphisms of any object in $\cD(\bC\oplus\bC\epsilon)$.  We assume both to be finite-dimensional. 

For any $A\in(\Art/\Lambda)$, we define $(\Art/\Lambda)_A$ to be the category of $A'\in(\Art/\Lambda)$ equipped with a morphism $A'\to A$.  For $x\in \cD(A)$, we define $\cD_x$ to be the category cofibered in groupoids over the category $(\Art/\Lambda)_A$ whose objects are morphisms $x'\to x$ in $\cD$ lying over $A'\to  A$ in $(\Art/\Lambda)_A$ and whose morphisms are morphisms of $\cD$ commuting with the map to $x$.  For $A'\in (\Art/\Lambda)_A$ the groupoid $\cD_x(A')$ is then the category of lifts of $x$ to $A'$.

For a complete noetherian local $\Lambda$-algebra $(B,\m_B)$ with residue field $\bC$, a formal point $\hat x\in\cD(B):=\lim \cD(B/\m_B^n)$ consists of an assignment to each 
$n$ of objects $x_n$ of $\cD$ lying over $B/\m_B^{n+1}$ and morphisms $x_{n+1}\to x_{n}$ lying over $B/\m_B^{m+2}\to B/\m_B^{n+1}$.  The formal point $\hat x$ is versal if for any solid diagram in $\cD$
\[
\begin{tikzcd}
    x_n\ar[rrdd]&&&& B/\m_B^{n+1}\ar[rrdd]&&\\
    &&&\mbox{lying above}&&&\\
    x_m\ar[r,dashed]\ar[uu]&y'\ar[r]&y && B/\m_B^{m+1}\ar[uu]\ar[r,dashed] & A'\ar[r]&A
\end{tikzcd}
\]
where $A'\to A$ is surjective, there exists an $m\geq n$ and the dashed arrows making the diagram commute, where the vertical arrows are the canonical ones.  The formal object $\hat x$ is miniversal if in addition the natural map on $\Lambda$-tangent spaces $\ft_{B/\Lambda}\to \ft_\cD$ is an isomorphism, where $\ft_{B/\Lambda}:=\Der_\Lambda(B,\bC)\cong (\mathfrak{m}_B/\mathfrak{m}_B^2+\mathfrak{m}_\Lambda B)^\vee$.

Let $\cX$ be an algebraic stack over $\bC$ (which we think of as a category fibered in groupoids over the category $(\Aff/\bC)$ of affine $\bC$-schemes).  For any $A\in(\Art/\bC)$ and any $A$-point $x\in \cX(A)$, we define $\cX_x$ to be the category fibered in groupoids over the category $(\Art/\bC)_A^{\mathrm{op}}$ whose objects are morphisms $x\to x'$ in $\cX$ lying over $\Spec A\to\Spec A'$ in $(\Art/\bC)_A^{\mathrm{op}}$ and whose morphisms are morphisms of $\cX$ commuting with the map from $x$.  The groupoid $\cX_x(A')$ is then the category of lifts of $x$ to $A'$, and if $x_0\in\cX(\bC)$, $(\cX_{x_0})^\op$ is a deformation category.

We pause to make these notions more concrete for $\cX=\cM_B(X)$.  We use the following terminology.
\begin{defn}
Let $A$ be a $\bC$-algebra and $X$ a topological space homeomorphic to a finite CW-complex.  By a (free) $A$-local system on $X$ we mean a local system of finitely generated (free) $A$-modules on $X$. 
\end{defn}

\subsubsection{}An object $V$ of $\cM_B(X)$ lying over $\Spec(A)\in(\Aff/\bC)$ is a free $A$-local system on $X$, and a morphism $V\to V'$ in the category $\cM_B(X)$ lying over a morphism $\Spec(A)\to \Spec (A')$ in $\Aff_\bC$ consists of a free $A$-local system $V$ (resp. a free $A'$-local system $V'$) and a morphism $V'\to V$ of $A'$-local systems which is an isomorphism upon tensoring with $A$ (we say it is an isomorphism over $A$).
\subsubsection{}For $V_0\in \cM_B(X)(\bC)$, the category $\cM_B(X)_{V_0}$ is the category of free $A$-local systems $V$ for $A\in(\Art/\bC)$ equipped with a morphism $V\to V_0$ of $A$-local systems which is an isomorphism over $\bC$.  Morphisms in $\cM_B(X)_{V_0}$ are morphisms (of the $V$ factor) in $\cM_B(X)$ commuting with the structure map to $V_0$.
\subsubsection{}More generally, for $A\in \Art/\bC$ and a free $A$-local system $V$, $\cM_B(X)_{V}$ is the category whose objects are pairs $(A'\to A,V'\to V)$ where $A'\to A$ is in $(\Art/\bC)_A$, $V'$ is a free $A'$-local system, and $V'\to V$ is a morphism of $A'$-local systems which is an isomorphism over $A$.  Morphisms in $\cM_B(X)_V$ are morphisms in $\cM_B(X)$ (of the $V'$ factor) commuting with the structure map to $V$.  Note that if $A'\to A$ is surjective with ideal $J$, then $V'\to V $ is an isomorphism over $A$ if and only if it is surjective with kernel is $JV'$.
\subsubsection{}For a complete noetherian ring $\hat\cO$, a formal point of $\cM_B(X)$ over $\hat\cO$ arises from an ordinary $\hat\cO$-point, namely a free $\hat\cO$-local system $ V$. 
 Then $ V$ is versal if the morphism $\Spec \hat\cO\to \cM_B(X)$ is formally smooth, meaning that for any:
 \begin{itemize}
     \item  surjection $A'\to A$ of artinian $\bC$-algebras;
     \item $\bC$-algebra homomorphism $\hat\cO\to A$;
     \item free $A$-local system $U$ with a morphism $V\to U$ of $\hat\cO$-local systems which is an isomorphism over $A$;
     \item free $A'$-local system $U'$ with a morphism $U'\to U$ of $A'$-local systems which is an isomorphism over $A$; 
 \end{itemize}
there is a lift $\hat\cO\to A'$ of $\hat\cO\to A$ and a lift $ V\to U'$ of $ V\to U$ which is an isomorphism over $A'$.  

In particular, for any $A\in \Art/\bC$, any free $A$-local system $U$, and any isomorphism $V/\m_\cO  V\to U/\m_AU$, there is a $\bC$-algebra homomorphism $\hat\cO\to A$ and a morphism $ V\to U$ of $\hat\cO$-local systems which is an isomorphism over $A$ and equal to $V/\m_\cO  V\to U/\m_AU$ over $\bC$.  If $ V$ is miniversal then the map $\hat\cO\to A$ is uniquely determined to first order.

\vskip1em

Returning to the general setup, we now briefly describe Schlessinger's construction of a miniversal formal point of $\cD$.  A canonical first order object $y_1$ is constructed over $\cO_1=\bC\oplus \ft_\cD^\vee$ (with the trivial $\Lambda$-algebra structure), using the gluing condition.  Let $S=\Lambda[[\ft_\cD^\vee]]=\Lambda_\bC\otimes\Sym^*\ft_\cD^\vee$ be the formal power series ring over $\Lambda$ (over a basis of $\ft_\cD^\vee$) and let $S\to \cO_1$ be the canonical quotient which is the identity on $\Lambda$-tangent spaces.  Suppose we have inductively constructed a quotient $S\to \cO_n$ as well as an object $y_n$ in $\cD(\cO_n)$ lifting $y_1$.  Let $I$ be the ideal of $\cO_n$ in $S$, and form the maximal small extension of $\cO_n$ which is an isomorphism on tangent spaces, explicitly given by
\begin{equation}\label{eq:bigsmall}0\to I/\m_SI\to S/\m_SI\to \cO_n\to 0.\end{equation}
Then (by the gluing condition again) there is a smallest ideal $I'$ with $\m_SI\subset I'\subset I$ such that $y_n$ lifts to $S/I'$, and we take $y_{n+1}$ to be any such lift and $\cO_{n+1}=S/I'$.  
\begin{thm}[{Schlessinger \cite{schlessinger}, see also \cite[\href{https://stacks.math.columbia.edu/tag/06IX}{Tag 06IX}]{stacks-project}}]\label{lem Schlessinger}
    The formal point $ \hat y=\lim y_n$ over $\hat\cO=\lim \cO_n$ is miniversal for $\cD$.
\end{thm}

\subsection{Pro-Hodge/twistor structures and their variations}\label{sect:pro stuff}
In this section we collect some straightforward definitions that will be needed for the sequel.  
\begin{defn}\hspace{1in}\label{MHS algebra stuff}
\begin{enumerate}
   \item By a pro-$\bC$-MS we mean a pro-object of the category of $\bC$-MS.
    \item By a pro-$\tate$-MS-algebra we mean a $\tate$-algebra object in the category of pro-$\bC$-MS.  Likewise, for any pro-$\tate$-MS-algebra $\Lambda$, by a pro-$\Lambda$-MS-algebra we mean a pro-$\tate$-MS-algebra $A$ equipped with a morphism $\Lambda\to A$ of pro-$\tate$-MS-algebras.
    \item We say a pro-$\tate$-MS-algebra is (complete/noetherian/artinian/local) if the underlying $\bC$-algebra is. 
    \item An ideal $I$ of a pro-$\tate$-MS-algebra is a $\bC$-MS-ideal if it is a sub-pro-$\bC$-MS.
    \item For $\Lambda$ a pro-$\tate$-MS-algebra, we define a pro-$\Lambda$-MS-module as a pro-$\Lambda$-module object in the category of $\CMS$.
\end{enumerate}
\end{defn}
In practice, all of the pro-$\tate$-MS-algebras we will consider will be complete noetherian local rings $\Lambda$, where each $\Lambda/\fm_\Lambda^k$ is equipped with a $\tate$-MS-algebra structure such that the quotient maps $\Lambda/\fm_\Lambda^{k+1}\to\Lambda/\fm_\Lambda^k$ are morphisms of $\bC$-MS.  These algebras are particularly well-behaved if $\fm_\Lambda=W_{-1}\Lambda$.  If $\Lambda$ is local complete noetherian and $\fm_\Lambda=W_{-1}\Lambda$, then each $\gr^W_k\Lambda$ is finite-dimensional and $W_0\Lambda=\Lambda$. 

\begin{lem}\label{tw loc triv}
    Let $A\to B$ be a morphism of pro-$\tate$-MTS-algebras.  Then $A\to B$ is locally trivial over $\bP^1$ as a morphism of sheaves of algebras.
\end{lem}
\begin{proof}
    It suffices to prove that artinian $\tate$-MTS-algebras $\Lambda$ are functorially trivializable on any (fixed) affine open subset of $\bP^1$, and this is what we'll show.  First, observe that $\gr^W\Lambda$ is trivial, since $R=\bigoplus_kH^0(\gr^W_k\Lambda(-k))$ has a natural graded ring structure and we can identify $\gr^W\Lambda=\bigoplus_k R_k\otimes_\bC\cO_{\bP^1}(k)$ as sheaves of algebras.  Now, for any $\lambda\in \bP^1$, $\Lambda|_{\bP^1\setminus \lambda}$ has a splitting of the weight filtration compatible with the algebra structure thanks to \Cref{lem tw splitting}, hence isomorphic to $\gr^W\Lambda|_{\bP^1\setminus \lambda}$ as a sheaf of algebras, and in particular trivial.
\end{proof}

\begin{lem}\label{MHS ring components}
    Let $\Lambda$ be a local complete noetherian pro-$\tate$-MS-algebra with $\fm_\Lambda=W_{-1}\Lambda$.  Then the nilradical and minimal prime ideals are $\bC$-MS-ideals.
\end{lem}
\begin{proof}
We first treat the Hodge case.  First, by Deligne \cite{Delignehodgeii} there is a functorial splitting of $(W_\bullet,F^\bullet)$ for any $\bC$-MHS, and likewise for $(W_\bullet,F'^\bullet)$.  Observe that for a $\bC$-MHS $V$ and a subspace $U\subset V$, for $U$ to be a sub-$\bC$-MHS it suffices to show it is strict with respect to each splitting.  By functoriality, the splittings for $\Lambda/\fm_\Lambda^k$ are compatible with the algebra structure and the quotient maps, so in particular any element $x\in\Lambda$ can be written as $x=\sum x^{p,q}$.  

Now, take nonzero $x,y\in \Lambda$ with splittings $x=\sum x^{p,q}$ with respect to $(W_\bullet,F^\bullet)$, and likewise for $y$.  Since $\fm_\Lambda=W_{-1}\Lambda$, there is a nonzero term $x^{p_0,q_0}$ with $(p_0+q_0,p_0)$ lexicographically maximal.  If $xy=0$, then it follows by induction that a power of $x^{p_0+q_0,p_0}$ annihilates each component of $y$.  This first shows that if $x$ is nilpotent, each component of $x$ is nilpotent.  Second, for any minimal prime $\mathfrak{p}$, if $y$ is contained in the intersection of all other minimal primes and not in $\mathfrak{p}$ while $x\in\mathfrak{p}$, then $xy$ is nilpotent, and it follows that each component of $x$ is contained in $\mathfrak{p}$.  This completes the proof.  

For the twistor case, by the same proof the nilradical and minimal prime ideals are graded with respect to the splittings of \Cref{lem tw splitting} on two open sets $\bP^1\setminus\lambda,\bP^1\setminus\lambda'$.  Since the sheaf of algebras is Zariski locally trivial, each yields a subsheaf of ideals which is strictly compatible with the weight filtration.  Moreover, each such ideal $I$ maps isomorphically to an ideal of the corresponding type (either the nilradical or a minimal prime) of $\gr^W\Lambda$.  The nilradical and minimal primes of the ring $R$ from the proof of \Cref{tw loc triv} are graded, so the corresponding ideals of $\gr^W\Lambda$ are twistor ideals, and this completes the proof.  
\end{proof}

\begin{defn}Let $X$ a connected algebraic space and $\Lambda$ a pro-$\tate$-MS-algebra whose underlying algebra is a complete noetherian local $\bC$-algebra.  By a pro-$\Lambda$-AVMS on $X$ we mean a pro-object of the category of $\bC$-AVMS on $X$ of the form $V=\lim V_k$ where
\begin{enumerate}
    \item $V$ has an underlying $\Lambda$-local system,
    \item each $V_k:=V/\fm_\Lambda^{k+1}V$ has the structure of a $\bC$-AVMS for which the natural map $V_{k+1}\to V_k$ is a morphism of $\bC$-AVMS and for which $\Lambda/\fm_\Lambda^{k+1}\subset \End(V_k)$ is a sub-$\bC$-MS.
\end{enumerate}
We say a pro-$\Lambda$-AVMS is (free/miniversal) if the underlying $\Lambda$-local system is.
\end{defn}

\subsection{Hodge/twistor enhancements of deformation categories}
Let $\Lambda$ be a complete local noetherian $\bC$-algebra with residue field $\bC$ and let $\cD$ be a deformation category over $(\Art/\Lambda)$.  An obstruction theory for $\cD$ consists of a complex vector space $\fo_\cD$ and an assignment, to each small extension $J\to A'\to A$ in $(\Art/\Lambda)$\footnote{Meaning that $A' \to A$ is a surjective morphism in $(\Art/\Lambda)$ whose kernel $J$ is annihilated by $\fm_{A'}$.} and $x\in \cD(A)$, of an element $\obs_{A/A'}(x)\in J\otimes \fo_\cD$ which detects obstructions in the sense that $\obs_{A/A'}(x)=0$ if and only if $x$ lifts to $A'$ and which is functorial in the sense that for any morphism of small extensions $J_1\to A'_1\to A_1$ to $J_2\to A_2'\to A_2$ and an element $x_1\in \cD(A_1)$ we have that $\obs_{A_2/A_2'}(x_2) $ is the image of $\obs_{A_1/A_1'}(x_1)$ under the natural map $J_1\otimes_\bC\fo_\cD\to J_2\otimes_\bC\fo_\cD$, where $x_2$ is the push-forward of $x_1$.  See for example \cite[Definition 6.1.21]{fgaexplained} for details.

Assume now $\Lambda$ is further equipped with a $\tate$-MS-algebra structure and let $(\MSArt/\Lambda)$ be the category of artinian local $\Lambda$-MS-algebras.  Denote by $|-|:(\MSArt/\Lambda)\to(\Art/\Lambda)$ the forgetful functor in the Hodge case, or the fiber at $1\in\bP^1$ in the twistor case, and likewise for $\Lambda$-MS-modules. 

Let $A$ be an object of $(\MS\mhyphen\Art/\Lambda)$. Then $A=\tate\oplus \fm_A$ as $\CMS$, and the projection $A \to \fm_A$ induces a natural group homomorphism  $\Ext^1_{\MS}(\fm_A,J) \to \Ext^1_{\MS}(A,J)$. It follows that via the natural group homomorphism $\Ext^1_{\MS}(\ft_{A/\Lambda}^\vee,J)\to \Ext^1_{\MS}(\fm_A,J)$, the group $\Ext^1_{\MS}(\ft_{A/\Lambda}^\vee,J)$ acts on the set of extensions of $A$ by $J$ as $\CMS$.

\begin{lem}\label{twists of MHS alg}For any small extension $J \to A' \to A$ in $(\MS\mhyphen\Art/\Lambda)$, the group $\Ext^1_{\MS}(\ft_{A/\Lambda}^\vee,J)$ acts transitively on the set of small extensions $ J\to B\to A$ in $(\MS\mhyphen\Art/\Lambda)$ whose underlying small extension in $(\Art/\Lambda)$ is $|J| \to |A'|\to |A|$, up to isomorphisms of the diagram which are the identity on $J,A$.
\end{lem}

\begin{proof}We have a commutative diagram with exact rows
\[
\begin{tikzcd}
0\ar[r]&    J\ar[r]&\fm_{A'}\ar[r]&\fm_A\ar[r]&0\\
0\ar[r]&    J\ar[r]\ar[u,equals]&J+\fm_{A'}^2+\fm_\Lambda A'\ar[r]\ar[u]&\fm_A^2+\fm_\Lambda A\ar[r]\ar[u]&0
\end{tikzcd}
\]
where the bottom row is the pullback extension.  The bottom middle term is the image of $J\oplus (\fm_A\otimes_A\fm_A)\oplus \fm_\Lambda A$ under the natural maps associated to the $\Lambda$-algebra structure, and therefore has a fixed $\CMS$.  Moreover, the map from $J\oplus (\fm_A\otimes_A\fm_A)\oplus \fm_\Lambda A$ determines the $\Lambda$-algebra structure.  Thus, the twist by an element of $\Ext^1_{\MS}(\fm_A,J)$ is compatible with the $\Lambda$-algebra structure if and only if the image in $\Ext^1_{\MS}(\fm_A^2+\fm_\Lambda A,J)$ is trivial, which is the case if and only if it is in the image of the group $\Ext^1_{\MS}(\ft_{A/\Lambda}^\vee,J)$.
\end{proof}

\begin{prop}\label{lem MHS smooth quotient}
    Let $A$ be an object of $(\MS\mhyphen\Art/\Lambda)$ with $\Lambda$-tangent space $\frak{t}_{A/\Lambda}$.  Let $S$ be a pro-$\Lambda$-MS-algebra with $|S|\cong\widehat\Sym^*_{|\Lambda|}|\frak{t}_{A/\Lambda}^\vee|$ admitting a quotient map $q:S\to A$ of pro-$\Lambda$-MS-algebras.  Then for any small extension $J\to A'\to A$ in $(\MSArt/\Lambda)$ with $\ft_{A'/\Lambda}^{\vee}\xrightarrow{\cong}\ft_{A/\Lambda}^\vee$, there is a possibly different pro-$\Lambda$-MS-algebra structure $S'$ on $|S|$ and a quotient map $q':S'\to A'$ of pro-$\Lambda$-MS-algebras lifting $|q|$.
\end{prop}

\begin{cor}\label{smooth algebra quotients}
    Any object $A$ of $(\MSArt/\Lambda)$ is a pro-$\Lambda$-MS-algebra quotient of a pro-$\Lambda$-MS-algebra $S$ with $|S|\cong \widehat\Sym^*_{|\Lambda|}|\frak{t}_{A/\Lambda}^\vee|$.
\end{cor}
\begin{proof}[Proof of \Cref{lem MHS smooth quotient}]
We first discuss pro-$\Lambda$-MHS-algebra structures on smooth $\Lambda$-algebras in the Hodge case.  Let $\ft^\vee$ be a $\CMHS$ and consider $S_0=\widehat\Sym^*_{|\Lambda|}|\frak{t}^\vee|$.  Given a lift $V\subset\fm_{S_0}$, by lifting either Deligne splitting of $\ft^\vee$ we uniquely obtain a bigraded $\Lambda$-algebra structure on $S_0$ (using the corresponding bigrading on $\Lambda$).  Suppose $V,V'$ are two lifts; by lifting the Deligne splitting of $(W_\bullet\ft^\vee,F^\bullet\ft^\vee)$ (resp. $(W_\bullet\ft^\vee,F'^
\bullet\ft^\vee)$), we therefore obtain filtrations $(W_\bullet S_0,F^\bullet S_0)$ (resp. $W_\bullet'S_0,F'^\bullet S_0$), giving $S_0$ the structure of a bifiltered $\Lambda$-algebra, where $\Lambda$ is given the filtrations $(W_\bullet\Lambda,F^\bullet\Lambda)$ (resp. $(W_\bullet\Lambda,F'^\bullet\Lambda)$).
\begin{claim}
    We have $W_\bullet V'=W'_\bullet V'$ if and only if $W_\bullet S_0=W_\bullet'S_0$ and $(W_\bullet S_0,F^\bullet S_0,F'^\bullet S_0)$ give $S_0$ the structure of a pro-$\Lambda$-MHS-algebra.
\end{claim}
\begin{proof}
    The reverse implication is clear.  Consider the isomorphism $f:V\to V'$ given as the composition $V\xrightarrow{\cong} \ft^\vee\xleftarrow{\cong} V'$, which maps $W_\bullet V$ to $W_\bullet'V'$.  There is a unique derivation $\delta\in \Der_\Lambda(S_0,\fm_{S_0}^2)^\wedge$ such that $1+\delta$ agrees with $f$ on $V$.  The resulting algebra automorphism $e^\delta:S_0\to S_0$ clearly maps $W_\bullet S_0$ to $W_\bullet'S_0$.  Since $e^\delta$ is trivial on $\gr_{\fm_0}S_0$, it is nilpotent in each $S_0/\fm_0^k$, so $e^\delta$ preserves $W_\bullet S_0$ iff $\delta\in W_0\Der_\Lambda(S_0,\fm_{S_0}^2)^\wedge$ iff $f\in W_0\Hom(V,V')$ (equipping $V'$ with $W_\bullet V'$) iff $W_\bullet V'=W_\bullet' V'$.  Since the filtrations $(W_\bullet,F^\bullet,F'^\bullet)$ induce $\CMHS$s on $\gr_{\fm_0}S_0$ and are compatible with the algebra structure, it follows by \cite[Lemma 1.6]{ES} that they give $S_0$ the structure of a pro-$\Lambda$-MHS-algebra.   
\end{proof}

We now handle the twistor case of the above claim, which is largely the same.  Let $\ft^\vee$ be a $\bC$-MTS, choose $\lambda,\lambda'\in\bP^1\setminus \lambda_0$, and set $U^{(\prime)}=\bP^1\setminus \lambda^{(\prime)}$.  For lifts $V^{(\prime)}\subset \fm_{S_0}$ of $|\ft^\vee|$, we obtain gradings on $S_0\otimes \cO_{U^{(\prime)}}$ (and filtrations $W_\bullet,W'_\bullet$) such that $e^\delta:S_0\otimes\cO_{U\cap U'}\to S_0\otimes\cO_{U\cap U'}$ as in the previous proof maps $W_\bullet S_0\otimes \cO_{U\cap U'}$ to $W_\bullet S_0\otimes  \cO_{U\cap U'}$.  If we have $W_\bullet (V'\otimes\cO_{U\cap U'})=W_\bullet'V'\otimes\cO_{U\cap U'}$, then gluing the two filtrations produces a pro-$\Lambda$-MTS-algebra structure lifting $S_0$ by the twistor version of \cite[Lemma 1.6]{ES} using \Cref{tw equiv}.

We now prove the proposition.  Let $I$ be the ideal of $A$.  Any lift $q'_0:|S|\to |A'|$ of $|q|$ yields a map $|I|\to |J|$ which is independent of the choice of lift.  Since a lift can be chosen to be compatible with either Deligne splitting, it follows that $I\to J$ is a morphism of pro-$\CMS$.  It therefore follows that the ideal $I'$ of $|A'|$ in $|S|$ can be lifted to a pro-$\CMS$-ideal, and the quotient $A''=S/I'$ equips $|A'|$ with a pro-$\Lambda$-MS-algebra structure which is a small extension of $A$ by $J$.  By \Cref{twists of MHS alg}, the pro-$\Lambda$-MS-algebra structures $A''$ and $A'$ differ by an element of $\Ext^1_\MS(\ft_{A/\Lambda}^\vee,J)$.  Concretely, up to automorphisms, we can take this to mean in the Hodge case (resp. the twistor case) that there is an element $f\in W_0\Hom(\ft_{A/\Lambda}^\vee,J)$ (resp. a 1-cocycle) such that $A''$ and $A'$ differ by the twist by $\id +f$ as described after \Cref{what is ext1}.  Any such $f$ can be lifted to an element $\tilde f\in W_0\Hom(\ft_{A/\Lambda}^\vee,I)^\wedge$ (resp. a 1-cocycle), and any lift $V'$ as above can be modified by $\id +\tilde f$.  The claim (and its twistor version) then shows there is a pro-$\Lambda$-MS-algebra structure $S'$ on $|S|$ for which $q':S'\to A'$ is a morphism of pro-$\Lambda$-MS-algebras and which lifts $q$. 
\end{proof}

\begin{defn}\label{def Hodge enhancement}
    Let $\cD$ be a deformation category over $(\Art/\Lambda)$ equipped with an obstruction theory (with obstruction space $\fo_\cD$).  A $\MS$-enhancement consists of:
    \begin{enumerate}
    \item The structure of a $\tate$-MS-algebra on $\Lambda$.
        \item A deformation category $ \cD^\MS$ over $(\mbox{MS-}\Art/\Lambda)$, meaning a category cofibered in groupoids over the category $(\mbox{MS-}\Art/\Lambda)$ for which $\cD^\MS(\tate)$ is equivalent to a point and such that the natural analog of the Rim--Schlessinger condition holds.
        \item A commutative diagram of functors
        \begin{equation}\label{forgetful diagram}\begin{tikzcd}
             \cD^\MS\ar[r,"|-|"]\ar[d]&\cD\ar[d]\\
            (\mbox{MS-}\Art/\Lambda)\ar[r,"|-|"]&(\Art/|\Lambda|)
        \end{tikzcd}\end{equation}
        where the vertical functors are the structural ones and the bottom functor is the forgetful functor.
        \item An object $M$ of $D^b(\CMS)$ together with isomorphisms $\mathscr{H}^1(M)\cong \ft_\cD$, $\mathscr{H}^2(M)\cong \fo_\cD$.
    \end{enumerate}
    satisfying conditions (5) and (6) below: 
    \begin{enumerate}\setcounter{enumi}{4}
    \item The universal first order object in $\cD(\bC\oplus \ft_\cD^\vee)$ lifts to an object in $\cD^\MS(\tate\oplus\ft_\cD^\vee)$.
    \item     For $J\to  A'\to  A$ a small extension in $(\mbox{MS-}\Art/\Lambda)$ together with an object $ x\in \cD^\MS(A)$, there is an obstruction $\obs_{A/A'}^\MS(x)\in\Hom_{\MS}(\tate,J\otimes_{\tate} M[2])$, functorial with respect to morphisms of small extensions as above, such that: 
\begin{enumerate}
    \item $\obs_{A/A'}^\MS(x)=0$ if and only if $x$ lifts to an object $x'\in \cD^\MHS(A')$.
    \item In the natural short exact sequence

            \begin{align}
        0\to\Ext^1_{\MS}(\tate,J\otimes_{\tate} \ft_\cD)&\to\Hom_{\MS}(\tate,J\otimes_{\tate} M[2])\label{defo obs}\\
       &\to\Hom_{\MS}(\tate,J\otimes_{\tate} \fo_\cD)\to0\notag,
    \end{align}
    the class $\obs_{A/A'}^\MS(x)$ maps to $\obs_{|A|/|A'|}(|x|)$, via the inclusion $\Hom_{\MS}(\tate,J\otimes_{\tate} \fo_\cD) \to J\otimes_{\tate} \fo_\cD$.
    \item The class $\obs_{A/A'}^\MS(x)$ is equivariant with respect to the natural action of $\Ext^1_{\MS}(\tate,J\otimes_{\tate} \ft_{A/\Lambda})$ on the small extension $J\to A'\to A$ (by changing the $\CMS$ on $A'$) and on the space $\Hom_{\MS}(\tate,J\otimes_{\tate} M[2])$ (via the natural map induced by the map on tangent spaces $\ft_{A/ \Lambda}\to\ft_\cD$).
\end{enumerate}
\end{enumerate}
We say the enhancement is \emph{precise} if it furthermore satisfies:
\begin{enumerate}\setcounter{enumi}{6}
\item There is an isomorphism $\mathscr{H}^0(M)\cong \fa_\cD$.
\item Suppose given $J\to  A'\to  A$ a small extension in $(\mbox{MS-}\Art/\Lambda)$ and $ x\in \cD^\MHS(A)$ admitting a lift to $\cD^\MS(A')$.  Then the set of isomorphism classes of the category of lifts $\cD^\MS_x(A')$ admits a transitive action by $\Hom_{\MS}(\tate,J\otimes_{\tate} M[1])$, functorially in $J$.  The natural functor
\[\cD_x^\MS(A')\to \cD_x(|A'|)\]
on isomorphism classes is equivariant with respect to the two natural actions via the second map in the exact sequence    
        \begin{align*}
        0\to\Ext^1_{\MS}(\tate,J\otimes_{\tate} \fa_\cD)&\to\Hom_{\MS}(\tate,J\otimes_{\tate} M[1])\\
       &\to\Hom_{\MS}(\tate,J\otimes_{\tate} \ft_\cD)\to0
    \end{align*}
and the inclusion $\Hom_{\MS}(\tate,J\otimes_{\tate} \ft_\cD)\to J\otimes_{\tate} \ft_\cD$.
\item In the context of (8), the automorphisms of an object $x'\in \cD^\MS_x(A')$ are identified with $\Hom_{\MS}(\tate,J\otimes_{\tate} \fa_\cD)$ (functorially in $J$), and the natural map from the automorphisms of $x'$ to the automorphisms of $|x'|\in\cD_{|x|}(|A'|)$ is identified with the natural map $\Hom_{\MS}(\tate,J\otimes_{\tate} \fa_\cD)\to J\otimes_{\tate} \fa_\cD$.
    \end{enumerate}
\end{defn}
We typically refer to the $\MS$-enhancement simply by $\cD^\MS\to \cD$.  The following lemma follows immediately from the definitions and \Cref{twists of MHS alg}.
\begin{lem}\label{lem unique lift}Let $\cD^\MS\to \cD$ be a $\MS$-enhancement.  Let $(A,x)$ be an element of $\cD^\MS$, $J\to A'\to A$ a small extension in $(\MSArt/\Lambda)$, and $(|A'|,y)$ a lift of $(|A|,|x|)$ in $\cD$.  Then:
\begin{enumerate}
    \item  The set of isomorphism classes of lifts of $|J|\to|A'|\to|A|$ to a small extension of $A$ by $J$ which admits a lift of $(A,x)$, if nonempty, is a torsor for
    \[\ker\left(\Ext^1_{\MS}(\tate,J\otimes_{\tate}\ft_{A/\Lambda})\to\Ext^1_{\MS}(\tate,J\otimes_{\tate}\ft_\cD)\right).\]
    In particular, if $\ft_{A/\Lambda}\xrightarrow{\cong}\ft_\cD$, then there is at most one such lift.
\end{enumerate}
    If in addition $\cD^\MS\to \cD$ is precise, then:
    \begin{enumerate}\setcounter{enumi}{1}
        \item The set of isomorphism classes of common lifts of $(A,x)$ and $(|A'|,y)$ to $A'$, if nonempty, admits a transitive action by the inverse image of the stabilizer of $y$ in $\Hom_{\MS}(\tate,J\otimes_{\tate}\ft_\cD)$ in $\Hom_{\MS}(\tate,J\otimes_{\tate}M[1])$.  In particular, if $\fa_\cD=0$, then there is at most one such lift.
    \end{enumerate}
\end{lem}
\begin{prop}\label{MHS Schlessinger}
    Let $\cD$ be a deformation category over $(\Art/\Lambda)$ equipped with an obstruction theory and a $\MS$-enhancement $\cD^\MS\to\cD$.  Then:
    \begin{enumerate}
        \item Any miniversal point $(\hat\cO,\hat y)$ of $\cD$ lifts to $(\hat\cO^\MS,\hat x)$ in $\cD^\MS$.  The lift $\hat\cO^\MS$ is uniquely determined up to isomorphism.
        \item 
        If the enhancement is precise and $\fa_\cD=0$, then $\cD^\MS$ is equivalent to the category $(\MSArt/\hat\cO^\MS)$.  In particular, any object $(A,x)$ of $\cD^\MS$ is pulled back from $(\hat\cO^\MS, \hat x^\MS)$ via a unique morphism $\hat\cO^\MS\to A$ of pro-$\Lambda$-MS-algebras.
    \end{enumerate}
\end{prop}
\begin{proof}We begin with part (1).  The uniqueness statement of (1) follows immediately from \Cref{lem unique lift}, so we need only show such a lift $(\hat\cO^\MS,\hat x)$ exists.  We prove that the miniversal element from Schlessinger's construction described in the previous section lifts to each order.  By condition (5) the first order object lifts, and we may take any such lift.  Assume inductively there is a lift $(\cO_n^\MS,x_n)$ of the pair $(\cO_n,y_n)$ constructed in the $n$th step of the proof of \Cref{lem Schlessinger}.  Assume also that we have equipped $\widehat\Sym^*_{|\Lambda|}|\frak{t}_{A/\Lambda}^\vee|$ with the structure of a pro-$\Lambda$-MS-algebra $S$ and that we have a quotient $S\to \cO_n^\MS$ of pro-$\Lambda$-MS-algebras.   

The resulting maximal small extension $A\to \cO^\MS_n$ with ideal $J_0$ from \eqref{eq:bigsmall} is a small extension of $\Lambda$-MS-algebras, so the obstruction $\obs_{\cO_n/|A|}(y_n)\in J_0\otimes_{\tate}\fo_\cD$ is a Tate class, which we interpret as a morphism $J_0^\vee\to\fo_\cD$ of $\bC$-MS.  The image $N$ of this map yields a sub-$\bC$-MS $N^\vee\subset J_0$, and we define $B=A/N^\vee$.  Let $J=J_0/N^\vee$ be the ideal of $\cO_n^\MS$ in $B$.  The resulting obstruction in $\cD$ vanishes, so by \eqref{defo obs} the class $\obs^\MS_{\cO_n^\MS/B}(x_n)$ is an element of $\Ext^1_{\MS}(\tate,J\otimes_{\tate} \ft_\cD)$.  Since the map on tangent spaces $\ft_{\cO_n^\MS/\Lambda}\to \ft_\cD$ is an isomorphism, it follows that after changing the mixed structure on $B$ (as a small extension in $(\MS\mhyphen\Art/\Lambda)$ without changing the underlying algebra structure) by an element of $\Ext^1_{\MS}(\tate,J\otimes_{\tate}\ft_{\cO_n^\MS/\Lambda})$, the obstruction vanishes and there is a lift $x_{n+1}$ of $x_n$.  We let $\cO_{n+1}^\MS$ be $B$ with this new mixed structure and choose any such lift $x_{n+1}$, as we may take any lift of $y_n$ to $\cO_{n+1}$ in Schlessinger's construction.  Finally, we must show that the pro-$\Lambda$-MS-algebra structure on $S$ can be changed to $S'$ so that the quotient $S'\to \cO_{n+1}^\MS$ is a morphism of pro-$\Lambda$-MS-algebras, and this is \Cref{lem MHS smooth quotient}.

We turn now to the proof of part (2).  It suffices to show that for a small extension $J\to A'\to A$ in $(\MSArt/\Lambda)$ and an object $(A',x')$ lifting $(A,x)$ for which $x$ is pulled back via $f:\hat\cO^\MS\to A$ from $\hat x^\MS$, there is a unique lift $f':\hat\cO^\MS\to A'$ by which $x'$ is pulled back.  The uniqueness follows from the representability of $\cD$.  Moreover, since $\Hom_\MS(\tate,J\otimes_{\tate} \ft_\cD)$ acts simply transitively on the lifts of $(A,x)$ to $A'$, it is enough to construct a morphism $\hat\cO^\MS\to A'$.

Let $S$ be a smooth pro-$\Lambda$-MS-algebra with a quotient $S\to \hat\cO^\MS$ of pro-$\Lambda$-MS-algebras that is an isomorphism on tangent spaces.  We have the solid diagram with exact rows
\[\begin{tikzcd}
    0\ar[r]&I\ar[d,dashed,"h"]\ar[r]&S\ar[r]\ar[d,dashed,"g"]&A\ar[r]\ar[d,equals]&0\\
    0\ar[r]&J\ar[r]&A'\ar[r]&A\ar[r]&0
\end{tikzcd}\]
The lift $g$ exists forgetting the mixed structures, as therefore does the left vertical map $h$.  Note that $h$ is independent of the choice of $g$.  In particular, since $g$ can be chosen to be compatible with either splitting, it follows that $I\to J$ is a morphism of pro-MS-ideals.  Taking the quotient of $S$ by the kernel, we obtain a morphism $S\to A''$ for some choice $A''$ of $\Lambda$-MS-algebra structure on $|A'|$.  Moreover, since $S\to A$ factors through a finite level quotient $\hat\cO^\MS_n\to A$, and since $\obs_{|A|/|A'|}(x)=0$, it follows as in the usual proof of Schlessinger's theorem (see \cite[\href{https://stacks.math.columbia.edu/tag/06IX}{Tag 06IX}]{stacks-project}) by considering $S\to A'\times_A\hat\cO_{n}^\MS$ that there is a morphism $\hat\cO^\MS_{n+1}\to A''$ lifting it.  In particular, there is a lift $f'':\hat\cO^\MS\to A''$ of $f$.

The $\Lambda$-MS-algebra structure on $A'$ differs from that of $A''$ by twisting by an element $\xi\in \Ext^1_\MS(\ft_{A/\Lambda}^\vee,J)$.  Let $E$ be the image of $\ft_{A/\Lambda}\to\ft_{\hat\cO^\MS/\Lambda}$ and consider the short exact sequence
\[0\to E\to\ft_{\hat\cO^\MS/\Lambda}\to E'\to 0 \]
which leads to the exact sequence
\[\Hom_\MS(E'^\vee,J)\to \Ext^1_\MS(E^\vee,J )\to\Ext^1_\MS(\ft_{\hat\cO^\MS/\Lambda}^\vee,J).\]Since both $A'$ and $A''$ admit lifts of $x$, it follows that the image of $\xi $ in $\Ext^1_\MS(\ft_{\hat\cO^\MS/\Lambda}^\vee,J)$ vanishes.  Thus, the image in $\Ext^1_\MS(E^\vee,J )$ may be lifted to $c\in \Hom_\MS(E'^\vee,J)$.  In the Hodge case (the twistor case is similar), let $\tilde c$ be a lift of $c$ to $F^0W_0\Hom(\ft_{\hat\cO^\MS/\Lambda}^\vee,J)$ and $\tilde c'$ a lift to $F'^0W_0\Hom(\ft_{\hat\cO^\MS/\Lambda}^\vee,J)$.  We can then think of $\xi$ as a lift $b\in W_0\Hom(\ft_{A/\Lambda}^\vee,J)/F^0W_0+F'^0W_0$ of $\tilde c-\tilde c'\in W_0\Hom(E'^\vee,J)/F^0W_0+F'^0W_0$.  It follows that by modifying the lift to $|f''|+c:|\hat\cO^\MS|\to |A''|$, we obtain the required morphism $\hat\cO^\MS\to A'$.  Indeed, $|f''|+c$ respects the weight filtration and $F^\bullet$, while \[(|f''|+c)F'^\bullet\hat\cO^\MS=(|f''|+\tilde c')(1+\tilde c-\tilde c')F'^\bullet\hat\cO^\MS=(1+b)(|f''|+\tilde c')F'^\bullet A''=(1+b)F'^\bullet A''.\]
\end{proof}

\subsection{Deformation theory of local systems}Let $X$ be a topological space homeomorphic to a finite CW complex and $(\bC_X\mhyphen\mathrm{Mod})$ be the abelian category of contructible $\bC_X$-modules on $X$.  The deformation theory of the Betti stack $\cM_B(X)$ is well-known; we first briefly summarize the straightforward approach in terms of group cocycles.
\begin{rem}\label{easy defo theory}
 Given a rank $r$ local system $V_0\in\cM_B(X)(\bC)$ with monodromy representation $\rho_{V_0,x}:\pi_1:=\pi_1(X,x)\to\GL(V_{0,x})$ for a choice of basepoint $x\in X$ and a small extension $J\to A'\to A$ of artinian $\bC$-algebras, a deformation $V$ of $V_0$ over $A$ will have monodromy representation $\rho_{V,x}:\pi_1\to \GL(V_x)$.  Identifying $V_x\cong A\otimes V_{0,x}$, we write $\rho_{V,x}=(1+f)\rho_{V_0,x}$ for $f\in \Hom_{\mathrm{Sets}}(\pi_1,\fm_A\otimes \End(V_{0,x}))$.  Then choosing an arbitrary lift $f'\in\fm_{A'}\otimes \End(V_{0,x})$, we have
 \[(1+f'(\gamma_1\gamma_2))\rho_{V_0,x}(\gamma_1\gamma_2)=(1+f'(\gamma_1))\rho_{V_0,x}(\gamma_1)(1+f'(\gamma_2))\rho_{V_0,x}(\gamma_2)+o(\gamma_1,\gamma_2)\] for a $2$-cocycle $o$ with values in $J\otimes \End(V_{0,x})$, and the corresponding element of $J\otimes H^2(\pi_1,\End(V_{0,x}))$ is the obstruction for lifting both $V$ and $\rho_{V,x}$ to $A'$-points of $\cM_B(X)$ and $R_B(X,x,r)$.  Note that since we have a natural injection $H^2(\pi_1,\End(V_{0,x}))\to \Ext^2_{\bC_X}(V_0,V_0)$, we may equally well use the image of $o$ as an obstruction.

 The first order deformations of $V_0$ (resp. $\rho_{V_0,x}$) are naturally identified with $H^1(\pi_1,\End(V_{0,x}))=\Ext^1_{\bC_X}(V_0,V_0)$ (resp. the group $1$-cocycles $Z^1(\pi_1,\End(V_{0,x}))=H^1((X,x),\cEnd(V_0))$).  The first order obstruction map is also easily computed to be the self-commutator
 \[[\;,\;]:H^1(\pi_1,\End(V_{0,x}))\to H^2(\pi_1,\End(V_{0,x})).\]
\end{rem}

We will instead need a treatment of the deformation theory from which the existence of Hodge enhancements will easily follow in the next section.  Essentially this means we need to construct the obstruction classes functorially in terms of functors that lift to the category of mixed Hodge/twistor modules.

Let $J\to A'\to A$ be a small extension of artinian local $\bC$-algebras.  Let $V$ be a free $A$-local system and set $V_0=V/\m_AV$.  By tensoring the sequence of $A$-modules
\[0\to J\to \m_{A'}\to \m_{A}\to 0\]
with $V$, we obtain an exact sequence of $A$-local systems
\begin{equation}\label{extm}
    0\to J\otimes_\bC V_0\to \m_{A'}\otimes_A V\to \m_A\otimes_A V\to 0.
\end{equation}
By tensoring the sequence of $A$-modules
\[0\to \m_A\to A\to \bC\to 0\]
we also have an exact sequence of $A$-local systems
\begin{equation}\label{ext3}0\to \m_{A}\otimes_A V\to V\xrightarrow{\pi} V_0\to 0\end{equation}
and therefore an exact sequence of $\bC$-vector spaces
\begin{equation}\label{LES}
\begin{tikzcd}
   0\ar[r]&\Hom_{\bC_X}(V_0,J\otimes_\bC V_0)\ar[r,"\pi_0^*"]& \Hom_{\bC_X}(V,J\otimes_\bC V_0)\ar[r]&\Hom_{\bC_X}(\m_{A}\otimes_A V, J\otimes_\bC V_0)\\[-15pt]
   \ar[r,"\partial_1"]&\Ext^1_{\bC_X}(V_0,J\otimes_\bC V_0)\ar[r,"\pi_1^*"]& \Ext^1_{\bC_X}(V,J\otimes_\bC V_0)\ar[r]&\Ext^1_{\bC_X}(\m_{A}\otimes_A V, J\otimes_\bC V_0)\\[-15pt]
   \ar[r,"\partial_2"]&\Ext^2_{\bC_X}(V_0,J\otimes_\bC V_0).&& 
\end{tikzcd}\end{equation}
Let $\eta_{V,A/A'}\in \Ext^1_{\bC_X}(\m_{A}\otimes_A V, J\otimes_\bC V_0)$ be the class of the extension \eqref{extm}.

\begin{lem}\label{lem cat of lifts}

The category of commutative diagrams of $\bC$-local systems with exact rows and columns
        \[
    \begin{tikzcd}
    &&0\ar[d]&0\ar[d]&\\
        0\ar[r]&J\otimes_\bC V_0\ar[r]\ar[d,equals]&\m_{A'}\otimes_A V\ar[r]\ar[d,"\mu"]&\m_A\otimes_A V\ar[r]\ar[d]&0\\
        0\ar[r]&J\otimes_\bC V_0\ar[r,"\alpha"]&V'\ar[r,"\beta"]\ar[d,"\pi"]&V\ar[r]\ar[d]&0.\\
        &&V_0\ar[d]\ar[r,equals]&V_0\ar[d]&\\
        &&0&0&
    \end{tikzcd}
    \]where all maps but $\alpha,\beta,\mu$ are the canonical ones and whose morphisms are morphisms of diagrams of $\bC$-local systems which are the identity on all but the $V'$ factor is equivalent to the category of lifts $\cM_B(X)_V(A')$.  
\end{lem}

\begin{proof}Given such a diagram, $\beta$ induces a morphism $\m_{A'}\otimes_{\bC} V' \to \m_{A'}\otimes_A V$, whose composition with the morphism $\m_{A'}\otimes_A V\xrightarrow{\mu} V' $  defines the structure of a $A'$-local system on $V'$ for which $V'\to V$ is a surjective morphism of $A'$-local systems with kernel $JV'$.  Thus, $V'$ is a free $A'$-local system, and the adjoint $A\otimes_{A'}V'\to V$ is an isomorphism.  A morphism of diagrams in the above sense is clearly equivalent to a morphism of $A'$-local systems which commutes with the projection to $V$, and any free $A'$-local system $V'$ with an isomorphism $A\otimes_{A'}V'\to V$ comes from such a diagram.
\end{proof}

\begin{cor}\label{cor def abs}\hspace{1in}

    \begin{enumerate}

        \item $V$ can be lifted to $A'$ if and only if $\partial_2\eta_{V,A/A'}=0\in J\otimes_\bC\Ext^2_{\bC_X}(V_0, V_0)$.
        \item The group $J\otimes_\bC \Ext^1_{\bC_X}(V_0,V_0)$ acts transitively on the set of isomorphism classes of $\cM_B(X)_{V}(A')$, provided it is nonempty.
            \item $T_{V_0}\cM_B(X)$ is naturally identified with $\Ext^1_{\bC_X}(V_0,V_0)$.
        \item The automorphisms of $V'\to V$ in $\cM_B(X)_{V}(A')$ are given by $\id+\alpha\sigma\pi$ for $\sigma\in J\otimes_\bC\Hom_{\bC_X}(V_0,V_0)$.
    \end{enumerate}
\end{cor}
\begin{proof}
    By \Cref{lem cat of lifts} and the long exact sequence \eqref{LES}.
\end{proof}

We now give a slightly more flexible version of \Cref{lem cat of lifts}.  Consider the natural two-term complex
\[N=[\overset{-1}{\fm_{A'}\otimes_A V}\to \overset{0}{V}]\]
where the map is induced by the composition $\fm_{A'}\to \fm_A\to A$.  We have a natural diagram whose rows and columns are distinguished triangles
\[\begin{tikzcd}
&V\ar[d]\ar[r,equals]&V\ar[d]&\\
        J\otimes_{\bC} V_0[1]\ar[r]\ar[d,equals]&N\ar[r]\ar[d]&V_0\ar[r]\ar[d]&J\otimes_{\bC} V_0[2]\ar[d,equals]\\
        J\otimes_{\bC} V_0[1]\ar[r]&\fm_{A'}\otimes_A V[1]\ar[r]\ar[d]&\fm_A\otimes_A V[1]\ar[r]\ar[d]&J\otimes_{\bC} V_0[2].\\
            &V[1]\ar[r,equals]&V[1]&
\end{tikzcd}\]

\begin{lem}\label{lem cat of lifts 2}
    There is a lift $V'\to V$ if and only if $N$ splits, and isomorphism classes in $\cM_B(X)_{V}(A')$ are naturally identified with splittings.
\end{lem}
\begin{proof}
 A lift gives a splitting via the diagram   
\[\begin{tikzcd}
\fm_{A'}\otimes_A V\ar[d]\ar[r]&V\ar[d,equals]\\
    V'\ar[r]&V
\end{tikzcd}\]
and isomorphic lifts give the same splitting (as a morphism in the derived category).

A splitting $N\to J\otimes_\bC V_0[1]$ gives a lift by taking cones, from the diagram
\[\begin{tikzcd}
&J\otimes_\bC V_0\ar[d]\ar[r,equals]&J\otimes_\bC V_0\ar[d]&\\
        \fm_{A'}\otimes_A V\ar[r]\ar[d,equals]&V'\ar[r]\ar[d]&V_0\ar[r]\ar[d]&\fm_{A'}\otimes_{A} V[1]\ar[d,equals]\\
        \fm_{A'}\otimes_{A} V\ar[r]&V\ar[r]\ar[d]&N\ar[r]\ar[d]&\fm_{A'}\otimes_{A} V[1].\\
            &J\otimes_\bC V_0[1]\ar[r,equals]&J\otimes_\bC V_0[1]&
\end{tikzcd}\]
While the cone $V'$ is not unique, another choice of cone fits into an isomorphic diagram (via an isomorphism of diagrams which is the identity at every other node), and therefore gives a well-defined isomorphism class of lift.  These two maps are easily seen to be mutually inverse.
\end{proof}

We now treat the absolute and relative deformation theory of some related stacks, including the framed moduli space, the pullback morphism, and the fixed local monodromy leaves, as this will give us the desired functoriality properties of the Hodge/twistor structures.

  \subsubsection{Pullback}\label{pullback section}
For a continuous map $f:X\to Y$ of topological space (with the above finiteness assumptions), consider the pullback $f^*:\cM_B(Y)\to\cM_B(X)$.  Let $V_0\in \cM_B(Y)(\bC)$ and let $(\Lambda,\hat U)$ be a $\Lambda$-local system for $\cM_B(X)$ with closed point $f^*V_0$.  Let $\cD$ be the category cofibered in groupoids over $(\Art/\Lambda)$ whose objects over a local artinian $\Lambda$-algebra $A$ are pairs $(V,b)$ where $V\in\cM_B(Y)(A)$ and $b:\hat U\to f^*V$ is a morphism of $\Lambda$-local systems which is an isomorphism over $A$.  A morphism $(V_1,b_1)\to(V_2,b_2)$ lying over a morphism $\alpha:A_1\to A_2$ of $(\Art/\Lambda) $ consists of a morphism $g:V_1\to V_2$ of $A_1$-local systems which is an isomorphism over $A_2$ and which form a commutative diagrams
\begin{equation}\label{morphism for rel}\begin{tikzcd}
    A_1\otimes_\Lambda \hat U\ar[r,"b_1"]\ar[d,"\alpha\otimes\id",swap]&f^*V_1\ar[d,"f^*g"]\\
    A_2\otimes_\Lambda \hat U\ar[r,"b_2"]&f^*V_2.
\end{tikzcd}\end{equation}
Then $\cD_{(V_0,b_0)}$ is a deformation category over $(\Art/\Lambda)$, where $b_0:\hat U\to f^*V_0$ is induced by the chosen identification.  The family $(\Lambda,\hat U)$ yields a morphism $\Spec\Lambda\to\cM_B(X)$, and $\cD_{(V_0,b_0)}$ is easily seen to be the deformation category associated to the fiber 2-product $\Spec\Lambda\times_{\cM_B(X)}\cM_B(Y)$ at the point $(V_0,b_0)$.  In particular, since $f^*:\cM_B(Y)\to\cM_B(X)$ is representable, $\cD_{(V_0,b_0)}$ will be equivalent to $\Spec\hat\cO$ for a complete $\Lambda$-algebra $\hat\cO$.  If $(\Lambda,\hat U)$ is miniversal, the morphism $\Spec\hat\cO\to\cM_B(Y)$ will only be versal in general.

The relative deformation theory will be governed by the natural exact triangle
 \[\bC_Y\to Rf_*\bC_X\to M_{X/Y}\to \bC_Y[1].\]
As before, the splitting of $f^*N$ gives a morphism $N\to Rf_*f^*(J\otimes_\bC V_0)[1]$ and a commutative diagram
\[
\begin{tikzcd}
    J\otimes_\bC V_0[1]\ar[r]\ar[d,equals]&N\ar[r]\ar[d]&V_0\ar[r]\ar[d,dashed]&J\otimes_\bC V_0[2]\ar[d,equals]\\
    J\otimes_\bC V_0[1]\ar[r]&Rf_*f^*(J\otimes_\bC V_0)[1]\ar[r]&J\otimes_\bC M_{X/Y}\otimes_\bC V_0[1]\ar[r]&J\otimes_\bC V_0[2].
\end{tikzcd}
\]

\begin{lem}\label{lem rel splitting}\hspace{.5in}
\begin{enumerate}
    \item The resulting element of $\Hom_{\bC_X}(V_0,J\otimes_\bC M_{X/Y}\otimes_\bC V_0[1])$ is uniquely determined and vanishes if and only if there is a relative lift $V'\to V$, $\hat U\to f^*V'$.
         \item The group $J\otimes_\bC\Hom_{\bC_X}(V_0, M_{X/Y}\otimes_{\bC}V_0)$ acts simply transitively on the set of isomorphism classes of such lifts $\cD_{(V,b)}(A')$, provided it is nonempty.
         \item The relative tangent space is $\Hom_{\bC_X}(V_0, M_{X/Y}\otimes_\bC V_0)$.
         \item There are no first order automorphisms.
\end{enumerate}
    
\end{lem}
\begin{proof}For (1), the dashed arrow is uniquely determined up to $\Hom_{\bC_X}(J\otimes_\bC V_0[2],Rf_*f^*(J\otimes_\bC V_0)[1])=0$, and if it vanishes there is a splitting $N\to J\otimes_\bC V_0[1]$ compatible with the splitting of $f^*N$.  Part (4) is clear from the definition.  Since $J\otimes_\bC\Hom_{\bC_X}(V_0,V_0\otimes_\bC M_{X/Y})$ acts transitively on the splittings of $N$ compatible with that of $f^*N$, this together with (4) yields (2) and therefore (3).
\end{proof}

\subsubsection{Framing}\label{section framing absolute}\label{section framing basic}

Note that for $X=\pt$ we have $\cM_B(\pt,r)=[\bGL_r\backslash\pt]$ and the miniversal family is the quotient morphism $\pt=R_B(\pt,\pt,r)\to \cM_B(\pt,r)$.  For general $X$, $R_B(X,x)$ is the fibered 2-product of $i^*:\cM_B(X)\to \cM_B(x)$ and $R_B(x,x)\to \cM_B(x)$ where $i:\{x\}\to X$ is the inclusion, so the deformation theory of $R_B(X,x)$ can be thought of as the relative deformation theory of $i^*$.

\subsubsection{The diagonal}\label{the diagonal} Let $\Lambda$ be a complete noetherian local ring, $U_1,U_2$ two free $\Lambda$-local systems, and $f_0:U_{1,0}\to U_{2,0}$ an isomorphism over the closed points, $U_{i,0}:=\bC\otimes_\Lambda U_i$.  The deformation category $\cD_{f_0}$ of the fibered 2-product of the resulting morphisms $\Spec\Lambda\to \cM_B(X)$ at the point $(U_{1,0},U_{2,0},f_0)$ is the category cofibered in groupoids over $(\Art/\Lambda)$ consisting of pairs $(A,f)$ where $A$ is a $\Lambda$-algebra and $f:A\otimes_\Lambda U_1\to A\otimes_\Lambda U_2$ is an isomorphism of $A$-local systems, with the obvious notion of isomorphism.  

Let $J\to A'\to A$ be a small extension in $(\Art/\Lambda)$ and $(A,f)$ an object of $\cD_{f_0}$.  We have a short exact sequence of $A'$-local systems
\begin{align}0\to \cHom(U_{1,0},J\otimes_\bC U_{2,0})&\to \cHom_{A'}(A'\otimes_\Lambda U_1,A'\otimes_\Lambda U_2)\notag\\
&\to\cHom_{A}(A\otimes_\Lambda U_1,A\otimes_\Lambda U_2)\to0.\label{diagonal ses} \end{align}
Taking cohomology, there is an exact sequence
\begin{align*}0&\to \Hom_{\bC_X}(U_{1,0},J\otimes_\bC U_{2,0})\to \Hom_{A'_X}(A'\otimes_\Lambda U_1,A'\otimes_\Lambda U_2)\to\Hom_{A_X}(A\otimes_\Lambda U_1,A\otimes_\Lambda U_2)\\
&\xrightarrow{\partial} \Ext^1_{\bC_X}(U_{1,0},J\otimes_\bC U_{2,0})\end{align*}
\begin{lem}\label{diagonal defo}\hspace{.5in}
    \begin{enumerate}
        \item The element $\partial f\in J\otimes_\bC\Ext^1_{\bC_X}(U_{1,0},J\otimes U_{2,0})$ is an obstruction for lifting $(A,f)$.
        \item The tangent space of $\cD_{f_0}$ is naturally identified with $\Hom_{\bC_X}(U_{1,0}, U_{2,0})$.
        \item There are no infinitesimal automorphisms.
    \end{enumerate}
\end{lem}
\subsubsection{Fixed local monodromy}\label{FLM}  For simplicity, let $(\bar X,D)$ be a log smooth projective curve with $X=\bar X\setminus D$  and $j:X\to \bar X$ the inclusion.  Note that in this case, for a complex local system $W$ on $X$, the intermediate extension $j_{!*}(W[1])$ is just $(j_*W)[1]$, and that $Rj_*$ is perverse-exact.  For each $\bar x\in D$ take a small embedded punctured disk $j_{\bar x}:\bD^*\to \bar X$ whose interior closure is a small disk neighborhood of $\bar x$ and consider the representable morphism of algebraic stacks
\begin{equation}\label{restrictP}
\psi_D:=\prod_{\bar x\in D} j_{\bar x}^*:\cM_B(X)\to\prod_{\bar x\in D}\cM_B(\hat\bD^*).
\end{equation}
where by $\hat\bD^*$ we mean the germ of $\bD^*$ around the origin.  For $V_0\in\cM_B(X)(\bC)$ we define the fiber of this morphism over the image of $V_0$, namely $FM(V_0):=(\psi_D)^{-1}(\psi_DV_0)$, to be the fixed local monodromy leaf at $V_0$.  Here by $j^*_{D^\mathrm{loc}}V_0$ we mean the locally closed substack supported at $V_0$.  As the terminology suggests, this is the locally closed\footnote{Note: For every $X$, every $\bC$-point of $\cM_B(X)$ is locally closed:  every orbit $O$ in $R_B(X,x)$ is constructible and irreducible, hence contains a point $\rho$ which is in the interior of the closure $\bar O$, and so the same is true for every point of $O$.} substack of $\cM_B(X)$ of local systems with the same local monodromy as $V_0$ up to conjugation (separately for each point of $D$).  It does do not depend on the choice of the cover.

We now consider the relative deformation theory of the embedding $FM(V_0)\subset \cM_B(X)$.  Fix a miniversal local system $(\Lambda,\hat V)$ at $V_0$ and let $\cD$ be the full subcategory of $(\Art/\Lambda)$ of $A$ for which $A\otimes_\Lambda \hat V$ has fixed local monodromy.  Let $J\to A'\to A$ be a small extension of local artinian $\Lambda$-algebras and suppose $V:=A\otimes_{\Lambda}\hat V$ has fixed local monodromy. 
 Set $V'=A'\otimes_\Lambda \hat V$.  According to \Cref{easy defo theory}, the obstruction at each point $\bar x\in D$ is naturally an element of the stalk $J\otimes (R^1j_*\cHom(V_0,V_0))_{\bar x}$, but we want to interpret this obstruction in terms of functors that lift to Hodge/twistor modules.  Note first that we can interpret this group as $J\otimes H^1(\bZ,\psi_{\bar x}j_{!*}\cHom(V_0,V_0))$, where $\psi_{\bar x}:=\psi_{q_{\bar x}}$ is the nearby cycles functor for a choice of coordinate $q_{\bar x}$ at $\bar x$ equipped with its natural $\bC[T,T^{-1}]$-module structure given by the local monodromy operator.

 Since $\psi_{\bar x}j_{!*}$ is exact on local systems, we have a short exact sequence
\[0\to \psi_{\bar x}j_{!*}\cHom(V_0,V_0)\to \psi_{\bar x}j_{!*}\cHom(V_0,\fm_{A'}\otimes_A V)\to \psi_{\bar x}j_{!*}\cHom(V_0,\fm_A\otimes_A V)\to 0\]
of $\bC[\bZ]=\bC[T,T^{-1}]$-modules, and because $V$ has fixed local monodromy, this sequence is exact after applying $H^i(\bZ,-)$.  Thus there is a natural short exact sequence
    \begin{align*}
        0\to H^1(\bZ,\psi_{\bar x}j_{!*}\cHom(V_0,V_0))&\to H^1(\bZ,\psi_{\bar x}j_{!*}\cHom(V_0,\fm_{A'}\otimes_A V))\\
       &\to H^1(\bZ,\psi_{\bar x}j_{!*}\cHom(V_0,\fm_A\otimes_A V))\to 0
    \end{align*}

Consider $E'\subset\cHom(V_0,V')$ the sub $A'$-local system of $\bC$-linear morphisms for which the composition $V_0\to V'\to V_0$ is a multiple of the identity.  Define $E$ similarly for $V$.  Then we have a commutative diagram with exact rows and columns:
\begin{equation}\label{sesmono2}\notag
\begin{tikzcd}
&0\ar[d]&0\ar[d]&\\
&J\otimes\cHom(V_0,V_0)\ar[d]\ar[r,equals]&J\otimes\cHom(V_0,V_0)\ar[d]&&\\
    0\ar[r]&\cHom(V_0, \fm_{A'}\otimes_AV)\ar[d]\ar[r]&E'\ar[d]\ar[r]&\bC_X\ar[d,equals]\ar[r]&0\\
    0\ar[r]&\cHom(V_0,\fm_A\otimes_A V)\ar[d]\ar[r]&E\ar[r]\ar[d]&\bC_X\ar[r]&0\\
    &0&0&
\end{tikzcd}
\end{equation}
After applying $\psi_{\bar x}j_{!*}$, the bottom row splits as an extension of $\bC[T,T^{-1}]$-modules, so the extension class of the middle row yields a class $\obs_{A/A'}\in J\otimes H^1(\bZ,\psi_{\bar x}j_{!*}\cHom(V_0,V_0))$.  This class vanishes if and only if $V'$ has fixed local monodromy, so we have proven:

\begin{lem}\label{FLM obs}\hspace{.5in}
    \begin{enumerate}
    \item There are no relative automorphisms or deformations for $FM(V_0)\subset \cM_B(X)$.

        \item The class $\obs_{A/A'}\in J\otimes H^1(\bZ,\psi_{\bar x}j_{!*}\cHom(V_0,V_0))$ is a relative obstruction for $FM(V_0)\subset \cM_B(X)$.
            \end{enumerate}
\end{lem}
\begin{proof}
    The first part is clear since $FM(V_0)\subset \cM_B(X)$ is a substack, and the second part is the above argument.
\end{proof}
We can interpret the obstruction naturally in terms of the intersection cohomology $IH^1(\cHom(V_0,V_0)):=H^0\pt_*j_{!*}(\cHom(V_0,V_0)[1])$ using the canonical short exact sequence of perverse sheaves
\[0\to j_{!*}W[1]\to Rj_*W[1]\to i_*i^!j_{!*}(W[1])[1]\to 0.\]

\begin{lem}\label{IH and weights}Let $i_{\bar x}:\bar x\to \bar X$ be the inclusion.  For any local system $W$ on $X$ the natural morphism of functors $\psi_{\bar x}\to i_{\bar x}^!$ induces an isomorphism $ H^1(\bZ,\psi_{\bar x}j_{!*}W[1])\to i_{\bar x}^!(j_{!*}W[1])[1]$.   In particular, there is a natural exact exact sequence
\begin{align*}0&\to IH^1(X,\cHom(V_0,V_0))\to H^1(X,\cHom(V_0,V_0))\to H^1(\bZ,\psi_{\bar x}j_{!*}\cHom_{\bC_X}(V_0,V_0))\\
&\to IH^2(X,\cHom(V_0,V_0))\to 0.\end{align*}
    Moreover, if $V_0$ underlies a pure $\bC$-VMS, then $j_{!*}(\cHom(V_0,V_0)[1])=W_1(\cHom(V_0,V_0)[1])$ is pure, as is $IH^1(X,\cHom(V_0,V_0))=W_1H^1(X,\cHom(V_0,V_0))$. 
\end{lem}
\begin{proof}
    The first part is standard from the description of perverse sheaves on the disk.  For the second part, the mixed module $\cHom(V_0,V_0)[1]$ has pure weight one, as does $j_{!*}(\cHom(V_0,V_0)[1])$, so $i^!j_{!*}(\cHom(V_0,V_0)[1])[1]$ has weights $\geq 2$, and proper push-forward preserves weights. 
\end{proof}

\subsection{Hodge/twistor enhancements for local systems} Let $X$ be a connected algebraic space. 
 Let $\cD^\MS\to (\mbox{MS-}\Art/\tate)$ be the category cofibered in groupoids whose objects over $A$ are free $A$-AVMSs $V$.  A morphism $V'\to V$ lying over a morphism $A'\to A$ in $(\mbox{MS-}\Art/\bC)$ is a morphism $V'\to V$ of $A'$-AVMS which is an isomorphism over $A$.  We define $\cD^{\MS}_V$ for $V$ in $\cD^{\MS}$ as before to be the category obtained by including the data of a morphism to $V$.

\begin{prop}\label{prop mini MB(X)}
    Suppose $V_0\in\cM_B(X)(\bC)$ underlies a $\bC$-AVMS which we fix.  Then $\cD^\MS_{V_0}\to (\MS\textrm{-}\Art/\tate)$ is part of a precise $\MS$-enhancement of $\cM_B(X)_{V_0}^\op\to (\Art/\bC)$ with its standard obstruction theory.
\end{prop}  
\begin{proof}We verify the conditions in \Cref{def Hodge enhancement}.  The cofibered category $\cD^\MS_{V_0}\to (\MS\textrm{-}\Art/\tate)$ clearly satisfies condition (2) and in condition (3) we take \eqref{forgetful diagram} to be the obvious forgetful diagram.  Condition (1) is trivial ($\Lambda=\tate$).

For (4), let $\pt_X:X\to\Spec\bC$ be the map to a point.  Taking $M=\pt_{X*}\cHom(V_0,V_0)$ (as in \Cref{lem esnault}), we have natural isomorphisms $\mathscr{H}^i(M)\cong \Ext^i_{\bC_X}(V_0,V_0)$, equipped with their natural mixed structures.

For (5), the universal extension
\begin{equation}\label{onestep}0\to\Ext^1_{\bC_X}(V_0,V_0)^\vee\otimes_{\tate} V_0\to V_1\to V_0\to 0\end{equation}
can be lifted to the category of $\bC$-AVMS (with the fixed lift of $V_0$) by \Cref{lem esnault} since the classifying map is the identity $\Ext^1_{\bC_X}(V_0,V_0)\to \Ext^1_{\bC_X}(V_0,V_0)$ and in particular a morphism of $\bC$-MS.

For (6), suppose $(A,V)$ is an element of $\cD^\MS_{V_0}(A)$, which comes equipped with a morphism of $V\to V_0$ in $\cD^\MS$ which becomes an isomorphism over $\bC$.  Then \eqref{extm} and \eqref{ext3} live in the category of $\bC$-AVMS, so there is an exact sequence (all Homs are in $D^b\MM(X,\bC)$)

\begin{equation}
\begin{tikzcd}[sep=small, scale cd=1]
   0\ar[r]&\Hom(V_0,J\otimes_{\tate} V_0)\ar[r,"\pi_0^*"]& \Hom(V,J\otimes_{\tate} V_0)\ar[r]&\Hom(\m_{A}\otimes_A V, J\otimes_{\tate} V_0)\\[-15pt]
   \ar[r,"\partial_1"]&\Hom(V_0,J\otimes_{\tate} V_0[1])\ar[r,"\pi_1^*"]& \Hom(V,J\otimes_{\tate} V_0[1])\ar[r]&\Hom(\m_{A}\otimes_A V, J\otimes_{\tate} V_0[1])\\[-15pt]
   \ar[r,"\partial_2"]&\Hom(V_0,J\otimes_{\tate} V_0[2]).&& 
\end{tikzcd}\label{LESMHM}\end{equation}
and the class of the extension \eqref{extm} yields $\eta_{V,A/A'}^\MS\in \Hom_
{\MM}(\m_{A}\otimes_A V, J\otimes_{\tate} V_0[1])$.  We take $\obs^\MS_{A/A'}(V)=\partial_2\eta_{V,A/A'}^\MS\in\Hom_{\MM}(V_0,J\otimes_{\tate} V_0[2])$.

\begin{lem}\label{lem cat of lifts MHM}

The category of commutative diagrams of smooth objects of $D^b\MM(X,\bC)$ with distinguished rows and columns
        \begin{equation}
    \begin{tikzcd}
    &V_0[-1]\ar[d]\ar[r,equals]&V_0[-1]\ar[d]&\\
        J\otimes_{\tate} V_0\ar[r]\ar[d,equals]&\m_{A'}\otimes_A V\ar[r]\ar[d,"\mu"]&\m_A\otimes_A V\ar[r]\ar[d]&J\otimes_{\tate} V_0[1]\ar[d,equals]\\
        J\otimes_{\tate} V_0\ar[r,"\alpha"]&V'\ar[r,"\beta"]\ar[d,"\pi"]&V\ar[r,"\gamma"]\ar[d]&J\otimes_{\tate} V_0[1].\\
            &V_0\ar[r,equals]&V_0&
    \end{tikzcd}\label{MHM lift}
    \end{equation}where all morphisms but $\alpha,\beta,\gamma,\mu$ are the canonical ones and whose morphisms are morphisms of diagrams in $D^b\MM(X,\bC)$ which are the identity on all but the $V'$ factor is equivalent to the category of lifts $\cD_V^\MS(A')$.  
\end{lem}
\begin{proof}As in \Cref{lem cat of lifts}, using that an exact triangle of smooth objects in $D^b\MM(X,\bC)$ (supported in the same degree) has an underlying short exact sequence of local systems. 
\end{proof}

From the definition, if $\obs^\MS_{A/A'}(V)=0$, then there is a diagram \eqref{MHM lift}, so $V$ lifts, thus giving condition (a).  Condition (b) is clear.  For condition (c), the composition  \[\Ext_{\MS}^1(\tate,J\otimes_{\tate}\ft_A)=\Hom_{\MS}(\ft_A^\vee,J[1])\to\Hom_{\MS}(\ft_\cD^\vee,J[1])\to \Hom_{\MM}(V_0,J\otimes_{\tate}V_0[2])\] can be thought of as precomposing $(\ft^\vee_A\to J[1])\otimes_{\tate}V_0$ with the composition
\[
V_0\to\ft^\vee_\cD\otimes_{\tate}V_0[1]\to \ft^\vee_A\otimes_{\tate}V_0[1]
\]
where the first morphism is the universal extension.  This composition is equal to the composition
\[V_0\to \fm_A\otimes_A V[1]\to \ft_A^\vee\otimes_{\tate} V_0[1]\]
where the first map is the one associated to $\fm_A\to\fm_A/\fm_A^2=\ft_A^\vee$.  The obstruction $\obs^\MS_{A'/A}(V)$ can be thought of as the composition of $V_0[-1]\to \fm_A\otimes_A V$ and $\fm_A\otimes_A V\to J\otimes_{\tate} V_0[1]$.  A map $\ft_A^\vee\to J[1]$ acts on the extension $J\to A'\to A$ and modifies the corresponding map $\fm_A\otimes_A V\to J\otimes_{\tate}V_0[1]$ by adding the composition $\fm_A\otimes_A V\to \ft_A^\vee\otimes_{\tate}V_0\to J\otimes_{\tate}V_0[1]$ while fixing $V_0[-1]\to\fm_A\otimes_AV$.  The obstruction element is then modified by adding the composition $V_0[-1]\to \fm_A\otimes_A V\to \ft_A^\vee\otimes_{\tate}V_0\to J\otimes_{\tate}V_0[1]$, so by the above (c) is proven.

The remaining conditions (7)-(9) are easily verified using \Cref{lem cat of lifts MHM} and \eqref{LESMHM} as in \Cref{lem cat of lifts}.
\end{proof}

\subsubsection{The diagonal}  We continue the discussion from \Cref{the diagonal}.

Let $\Lambda$ be a pro-$\tate$-MS-algebra, $U_1,U_2$ two pro-$\Lambda$-AVMS, and $f_0:U_{1,0}\to U_{2,0}$ an isomorphism of closed points as $\bC$-AVMS.  We let $\cD_{f_0}^\MS$ be the category cofibered in groupoids over $(\MSArt/\Lambda)$ consisting of pairs $(A,f)$ where $A$ is a pro-$\Lambda$-MS-algebra and $f:A\otimes_\Lambda U_1\to A\otimes_\Lambda U_2$ an isomorphism of $A$-AVMS.

  \begin{prop}\label{diagonal enhancement}
    $\cD_{f_0}^\MS\to(\MSArt/\Lambda)$ is part of a precise MS-enhancement of $\cD_{|f_0|}\to(\Art/|\Lambda|)$.
\end{prop}
\begin{proof}
    We take $M=\pt_{X*}\cHom(U_{1,0},U_{2,0})[-1]$ and set $\ft=\mathscr{H}^1(M)=\Hom_{\bC_X}(U_{1,0},U_{2,0})$, $\fo=\mathscr{H}^2(M)=\Ext^1_{\bC_X}(U_{1,0},U_{2,0})$.  The first order element is given by the universal homomorphism 
    \[\tate[\ft^\vee]\otimes_\Lambda U_1\to \tate[\ft^\vee]\otimes_\Lambda U_2 \]
    where $\ft^\vee$ is square zero and $\tate[\ft^\vee]$ has the trivial $\Lambda$-algebra structure.

    Suppose $J\to A'\to A$ is a small extension in $(\MSArt/\Lambda)$ and $(A,f)$ an object of $\cD^\MS_{f_0}$.  Then \eqref{diagonal ses} is a short exact sequence of $A'$-AVMS, and as in \eqref{diagonal defo} the element $\partial f\in \Hom_{\MM}(U_{1,0},J\otimes_{\tate}U_{2,0}[1
])$ obstructs the existence of a lift.  Moreover, twisting the mixed structure on $A'$ by an element of $\Ext^1_\MS(\tate,J\otimes_{\tate} \ft_{A/\Lambda})$ will act in the required way on the obstruction class.  Finally, the set of lifts is clearly a torsor over $\Hom_\MS(\tate,J\otimes_{\tate}\ft)$.  
\end{proof}

\begin{proof}[Proof of \Cref{thm:versal}]  
    Parts (1) and (2) follow from \Cref{prop mini MB(X)} and \Cref{MHS Schlessinger}, while \Cref{diagonal enhancement} and \Cref{MHS Schlessinger} give part (3).  The last statement follows since if $X$ is normal and $V_0$ underlies a $\bC$-VHS then the Hodge structure on $\Ext^1_{\bC_X}(V_0,V_0)$ has weights $\geq 1$, for instance by the computation in \Cref{IH and weights} (after reducing to a curve by Lefschetz).
\end{proof}

\subsubsection{Pullback}We continue the discussion\footnote{We will in fact only need the relative deformation theory of the pullback to a point.  The deformation theory of the framed space can be carried out directly as in the treatment of $\cM_B(X)$; \Cref{prop rel mini map} below will however show that the pullback morphism $f^*:\cM_B(Y)\to\cM_B(X)$ is representable by mixed structures as in \Cref{thm:versal}.} from \Cref{pullback section}.

    Let $f:X\to Y$ be an algebraic map, $f^*:\cM_B(Y)\to\cM_B(X)$ the pullback, $V_0\in \cM_B(Y)(\bC)$ underlying a $\bC$-AVMS, and suppose $\Lambda$ (resp. $\hat U$) is equipped with a pro-$\tate$-$\MS$-algebra (resp. pro-$\Lambda$-AVMS) structure.  Let $\cD^\MS$ be the category cofibered in groupoids over $(\MSArt/\Lambda)$ whose objects over $A$ are pairs $(V,\beta)$ where $V$ an $A$-AVMS on $Y$ and $\beta:\hat U\to f^*V$ a morphism of pro-$\Lambda$-AVMS which is an isomorphism over $A$.  Let $\beta_0:\hat U\to f^*V_0$ be the morphism coming from the chosen identification of closed points.  The notion of morphism in $\cD^\MS$ is as before but now respecting the mixed structures.

\begin{prop}\label{prop rel mini map}
 $\cD^\MS_{(V_0,\beta_0)}\to (\MS\textrm{-}\Art/\Lambda)^\op$ is part of a precise Hodge enhancement of $\cD\to(\Art/|\Lambda|)^\op$. 
\end{prop}

\begin{proof}
We take $M=\pt_{X*}\cHom(V_0,M_{X/Y}\otimes_{\tate}V_0)[-1]$.  The discussion leading up to \Cref{lem rel splitting} applies to $D^b\MM(X,\bC)$ as well, in particular giving an obstruction element in $\Hom_{\MS}(\tate,J\otimes_{\tate}M[2])$.  The verification of most parts of the definition is as in \Cref{prop mini MB(X)}, except (5) and (6)(c).  

    For (5), let $\ft:=\Hom_{\bC_X}(V_0,M_{X/Y}\otimes_{\tate}V_0)=\mathscr{H}^1(M)$ with its natural mixed structure.  The natural morphism $V_0\to \ft^\vee\otimes_{\tate}M_{X/Y}\otimes_{\tate}V_0$ yields by taking cones a lift of the universal first order element:
    \[\begin{tikzcd}
        \ft^\vee\otimes_{\tate}V_0\ar[r]\ar[d,equals]&V_1\ar[r]\ar[d,dashed]&V_0\ar[r]\ar[d]&\ft^\vee\otimes_{\tate}V_0[1]\ar[d,equals]\\
        \ft^\vee\otimes_{\tate}V_0\ar[r]&\ft^\vee\otimes_{\tate}Rf_*f^*V_0\ar[r]&\ft^\vee\otimes_{\tate} M_{X/Y}\otimes_{\tate}V_0\ar[r]&\ft^\vee\otimes_{\tate}V_0[1].
    \end{tikzcd}\]
    
 For (6)(c), for a small extension $J\to A'\to A$ in $(\MSArt/\Lambda)$ and an object $(A,V)$ of $\cD^\MS$, unwrapping the definitions, the obstruction class is given by the morphism induced on total complexes by
 \[
 \begin{tikzcd}
     J\otimes_{\tate}V_0\ar[r]\ar[d,equals]& \fm_{A'}\otimes_A V\ar[d]\ar[r]&V\ar[d]\\
     J\otimes_{\tate}V_0\ar[r] & Rf_*(A'\otimes_\Lambda \hat U)\ar[r]&Rf_*(A\otimes_\Lambda \hat U)
 \end{tikzcd}
 \]
 where the bottom left map comes from the composition $J\otimes_{\tate}V_0\to Rf_*(J\otimes_{\Lambda}\hat U)\to Rf_*(A'\otimes_\Lambda \hat U)$ using the identification on closed points induced by the map $\hat U\to f^*V$.
 Twisting by an element $t$ of $\Ext^1_{\MS}(\ft^\vee_{A/\Lambda},J)$ we obtain a new small extension $J\to A''\to A$, with the same underlying small extension of $|\Lambda|$-algebras.  As in \Cref{twists of MHS alg}, the two extensions are isomorphic on $J+\fm_{A'}^2+\fm_\Lambda A'=J+\fm_{A''}^2+\fm_\Lambda A''$; the element $t$ can be thought of in the derived category as the map of complexes
 \[
 \begin{tikzcd}
     J\ar[r," i"]\ar[d,equals]&\left( \fm_{A'}\oplus_{\fm_A}\fm_{A''}\right)/(J+\fm_{A'}^2+\fm_\Lambda A')_{ \Delta}=:B\ar[d]\\
     J\ar[r]&0
 \end{tikzcd}
 \]
 where $i$ is the inclusion $j\mapsto (-j,0)=(0,j)$, and $(J+\fm_{A'}^2+\fm_\Lambda A')_{ \Delta}$ is the diagonal embedding of $J+\fm_A^2+\fm_\Lambda A'$.

 The difference of the obstruction classes $\obs_{A/A''}^\MS(V)-\obs_{A/A'}^\MS(V)$ is the element in $\Hom_{\MM}(V_0,J\otimes_{\tate}M_{X/Y}\otimes_{\tate}V_0[1])$ given by

  \[
 \begin{tikzcd}
         J\otimes_{\tate}V_0\ar[r,"i"]\ar[d,equals]& (\fm_{A'}\oplus_{\fm_A}\fm_{A''}/(J+\fm_{A'}^2+\fm_\Lambda A')_{\Delta})\otimes_A V\ar[d]\ar[r]&V/(\fm_A^2+\fm_\Lambda A)V\ar[d]\\
     J\otimes_{\tate}V_0\ar[r,"i"] & Rf_*((A'\oplus_{ A} A'')/(J+\fm_{A'}^2+\fm_\Lambda A')_{ \Delta}\otimes_\Lambda \hat U)\ar[r]&Rf_*((A/\fm_A^2+\fm_\Lambda A)\otimes_\Lambda \hat U)  
 \end{tikzcd}
 \]
 where $i$ is the antidiagonal embedding as above and the vertical maps are the natural ones.  This is identified with the image of $t$ in $\Hom_{\MM}(V_0,J\otimes_{\tate}M_{X/Y}\otimes_{\tate}V_0[1])$, which is the composition of the universal evaluation $V_0\to \ft\otimes M_{X/Y}\otimes_{\tate}V_0$ with the map $J^\vee\to \ft[1]$ given by $t$ tensored with $M_{X/Y}\otimes_{\tate}V_0$.  Indeed, it is identified with the following composition in the derived category
 \[
 \begin{tikzcd}
     0\ar[r]\ar[d]&\ft_{A/\Lambda}^\vee\otimes V_0\ar[r]\ar[d]&V_1\ar[d]&& V_0\ar[d]\\
     0\ar[r]&\ft^\vee\otimes V_0\ar[r]&\ft^\vee\otimes Rf_*\hat U&&\ft\otimes_{\tate} M_{X/Y}\otimes_{\tate} V_0\arrow[d,sloped, phantom,"\cong"] \\
     J\otimes V_0\ar[r]\ar[u]\ar[d]&J\otimes Rf_*\hat U\oplus B\otimes V_0\ar[u]\ar[d]\ar[r]&B\otimes Rf_*\hat U\ar[u]\ar[d]&& \ft\otimes_{\tate} M_{X/Y}\otimes_{\tate} V_0\ar[d]\\
     J\otimes V_0\ar[r]&J\otimes Rf_*\hat U\ar[r]&0&&J\otimes_{\tate} M_{X/Y}\otimes_{\tate} V_0
 \end{tikzcd}
 \]
\end{proof}

\begin{proof}[Proof of \Cref{thm:versal frame}]
    Since the deformation theory of $R_B(X,x)$ is identified with the relative deformation theory of the pullback to a point (see the discussion in \Cref{section framing basic}), \Cref{prop rel mini map} and \Cref{MHS Schlessinger} give (1) and (2), using that the universal property from \Cref{MHS Schlessinger}(2) applies.  The universal property immediately implies (3), since pullbacks, direct sums, and tensor products of AVMSs are AVMSs.
\end{proof}

\subsubsection{Fixed local monodromy}  We continue the discussion of \Cref{FLM}.  

If $V_0$ underlies a $\bC$-AVMS, then we may equip $\Lambda$ (resp. $\hat V$) with a pro-$\tate$-$\MS$-algebra (resp. pro-$\Lambda$-AVMS) structure by \Cref{thm:versal}.  Let $\cD^\MS$ be the full subcategory of $(\MSArt/\Lambda)$ (considered as a category cofibered in groupoids over $(\MSArt/\Lambda)$) whose objects are $A$ for which $A\otimes_\Lambda \hat V$ has fixed local monodromy as a local system at each point of $D$.

 The following proposition demonstrates that showing a full sub-deformation category of a deformation category admitting an enhancement itself admits an enhancement is much easier---we must just verify the relative obstructions are Tate classes in a mixed structure. 
\begin{prop}\label{prop rel FLM}
 $\cD^\MS\to (\MS\textrm{-}\Art/\Lambda)^\op$ is part of a precise Hodge enhancement of $\cD\to(\Art/|\Lambda|)^\op$. 
\end{prop}
\begin{proof}
    As there are no infinitesimal automorphisms or deformations, we need only show the relative obstruction class for a small extension $J\to A'\to A$ in $(\MSArt/\Lambda)$ for which $A\otimes_\Lambda U$ has fixed local monodromy is a Tate class of $J\otimes i^!(j_{!*}\cHom(V_0,V_0)[1])$, which is clear from the construction leading up to \Cref{FLM obs}.
\end{proof}

\subsubsection{Fixed residual eigenvalues}\label{FRE}  Continuing the setup of \Cref{FLM}, consider instead the map to the good moduli space
\[\cM_B(X,r)\to \prod_{\bar x\in D}M_B(\hat \bD^*,r)\cong \prod_{\bar x\in D}\Sym^r\bG_m\]
where the isomorphism is by taking the characteristic polynomial.  We call the fiber through $V_0\in \cM_B(X)(\bC)$ the fixed residual eigenvalues leaf $FE(V_0)$.  If $V_0$ underlies a $\bC$-AVMS and we again equip $\Lambda$ (resp. $\hat V$) with a pro-$\tate$-$\MS$-algebra (resp. pro-$\Lambda$-AVMS) structure by \Cref{thm:versal}, then for each $\bar x\in D$, $\psi_{\bar x}\hat V$ is a pro-$\Lambda$-MS-module, and comes equipped with an endomorphism $T_{\bar x}$ compatible with this structure.  Then the ideal of $FE(V_0)$ is cut out by the coefficients of the characteristic polynomial $p_{T_{\bar x}}$ of $T_{\bar x}$ for each $\bar x\in D$, which is naturally a pro-MS-ideal.

\subsection{Hodge/twistor substacks of $\cM_B(X)$}

Observe for the following definition that for a morphism of algebraic stacks $f:\cM\to\cN$ and miniversal points $\hat m:\Spec\hat\cO_{\cM,m}\to\cM$, $\hat n:\Spec\hat\cO_{\cN,n}\to\cN$ at closed points $m\in\cM(\bC)$, $n:=f(m)$, respectively, the composition $\Spec\hat\cO_{\cM,m}\to\cN$ lifts to $\Spec\hat\cO_{\cM,m}\to\Spec\hat\cO_{\cN,n}$ and so there is a pullback map $\hat\cO_{\cN,n}\to\hat\cO_{\cM,m}$ which is noncanonical but determined up to first order.  

We have the following compatibilities:

\begin{lem}\label{lem mini compat}
With the above notation, assume we have a closed immersions $g:\cZ\to \cN$ and $z\in\cZ(\bC)$ with $g(z)\cong n$ (with a choice of isomorphism).  Then: 
\begin{enumerate}
    \item $\Spec\hat\cO_{\cZ,z}:=\cZ\times_\cN\Spec\hat\cO_{\cN,n}\to\cZ$ is miniversal at $z$.
    \item $\Spec\hat\cO:=\Spec(\hat\cO_{\cZ,z}\otimes_{\hat\cO_{\cN,n}}\hat\cO_{\cM,m})\to \cZ\times_\cN\cM$ is miniversal at $(z,m)$.

\end{enumerate}
\end{lem}
\begin{proof}
    For any group $G$ and any $G$-representation $V$, the only closed substacks of $[G\backslash\Spec(\bC\oplus V)]$ are of the form $[G\backslash \Spec(\bC\oplus U)]$ for a subrepresentation $U\subset V$.  This implies that the base-change of the first order neighborhood $[G\backslash\Spec(\bC\oplus \ft_{\cN,n}^\vee)]$ of $n$ is the first order neighborhood $[G\backslash\Spec(\bC\oplus \ft_{\cZ,z}^\vee)]$ of $z$, where $G$ is the inertia of $n $ (and thus $z$).  Since the base-change of $\hat n$ is formally smooth, this gives (1).

    For (2), we have the following diagram
    \[\begin{tikzcd}
        &\Spec\hat\cO\ar[rr]\ar[dd]\ar[ld]&&\Spec\hat\cO_{\cM,m}\ar[dd]\ar[ld]\\
        \cZ\times_\cN\cM\ar[dd]\ar[rr, crossing over]&&\cM&\\
        &\Spec\hat\cO_{\cZ,z}\ar[rr]\ar[ld]&&\Spec\hat\cO_{\cN,n}\ar[ld]\\
        \cZ\ar[rr]&&\cN\ar[from=uu, crossing over]&
    \end{tikzcd}\]
    with Cartesian front, bottom, and back faces.  It follows that the top face is Cartesian, so by (1) the conclusion follows.
\end{proof}

In the following definition we only consider the formal twistor geometry at semisimple points with quasiunipotent local monodromy:  the former condition ensures there is a canonical choice of variation of twistor structures, and the latter, while not necessary here, allows for a simpler construction of the Deligne--Hitchin space in \Cref{DeligneHitchin sect} and will be sufficient for our purposes.
\begin{defn}\label{defn formally Hodge}Let $X$ be a connected algebraic space and $\cZ\subset \cM_B(X)$ a locally closed algebraic substack.
\begin{enumerate}
    \item We say $\cZ$ is \emph{formally Hodge} at a $\bR_{>0}$-fixed point $V\in \cZ(\bC)^{\qu,\ss}$ (see \Cref{sect:R* on Betti}) if, equipping $V$ with its canonical weight zero $\bC$-VHS\footnote{By Deligne \cite[Proposition 1.13]{delignefiniteness}, for any $\bR_{>0}$-fixed semisimple complex local system $V$, there is a unique weight zero $\bC$-VHS structure on $V$ for which any simple sub-local system is a sub-$\bC$-VHS.}, there is a surjective morphism of pro-$\bC(0)$-MHS-algebras $\hat \cO_{\cM_B(X),V}\to\hat\cO_{\cZ,V}$ for which $\hat\cO_{\cZ,V}$ is miniversal for $\cZ$ at $V$.
    \item We say $\cZ$ is a \emph{formally twistor at a point} $V\in \cZ(\bC)^{\qu,\ss}$ if, equipping $V$ with its canonical (tame purely imaginary) variation of twistor structures (as in \Cref{exa defining can twistor}), there is a surjective morphism of pro-$\tate$-MTS-algebras $\hat \cO_{\cM_B(X),V}\to\hat\cO_{\cZ,V}$ for which $\hat\cO_{\cZ,V}$ is miniversal for $\cZ$ at $V$.  
    \item The substack $\cZ$ is \emph{formally twistor} if $\cZ(\bC)$ is closed under semisimplification, formally twistor at every point of $\cZ(\bC)^{\qu,\ss}$, and formally Hodge at every point underlying a $\CVHS$ with quasiunipotent local monodromy.
\end{enumerate}
We likewise define formally twistor subvarieties of $R_B(X,x)$.  

\end{defn}
  Note that by \Cref{ss in closure} any closed subset of $\cM_B(X)(\bC)$ is closed under semisimplification.

\begin{lem}\label{check on versal}
    Let $\cZ\subset\cM_B(X)$ be a locally closed algebraic substack, $V\in\cZ(\bC)$, $\Spec \Lambda\to\cM_B(X)$ a versal family at $V$ underlying a $\Lambda$-AVMS, and $\Spec \Lambda_\cZ := \cZ\times_{\cM_B(X)} \Spec\Lambda \to \Spec \Lambda$ the base-change. Then $\cZ$ is formally Hodge (resp. twistor) at $V$ if and only if $\Lambda\to \Lambda_\cZ$ is a (surjective) morphism of pro-$\bC(0)$-MHS-algebras (resp. pro-$\tate$-MTS-algebras).
\end{lem}

\begin{proof}
    Let $\Spec \hat \cO_{\cM_B(X), V}\to\cM_B(X)$ be the miniversal family at $V$, and consider the completion of the fibered 2-product at an identification of the closed points of $\Spec \Lambda$ and $\Spec\hat\cO_{\cM_B(X), V}$ in the solid diagram below
    \[\begin{tikzcd}
        &\Spec \cS\ar[rr]\ar[dd]&&\Spec\Lambda\ar[dd]\\
        \Spec\cS_\cZ\ar[rr,dashed]\ar[ur,dashed]\ar[dd,dashed]&&\Spec\Lambda_\cZ\ar[dd,dashed]\ar[ur,dashed]&\\
        &\Spec\hat\cO_{\cM_B(X),V}\ar[rr]&&\cM_B(X).\\
        \Spec\hat \cO_\cZ\ar[rr,dashed]\ar[ru,dashed]&&\cZ\ar[ru,dashed]&
    \end{tikzcd}\]
    Let the diagonal arrows be the base-change of $\cZ\to\cM_B(X)$.  By \Cref{thm:versal}(3), the solid left and solid top arrows correspond to morphisms of pro-$\tate$-MS-algebras.  The diagonal morphisms are surjective on algebras, and the horizontal and vertical ones are injective on algebras (as they are formally smooth), except those involving $\cZ$.  Thus, if $\cS\to\cS_\cZ$ is a morphism of pro-$\tate$-MS-algebras, then so are the composition $\hat\cO_{\cM_B(X)} \to \cS \to \cS_\cZ$ and its corestriction to its image  $\hat\cO_{\cM_B(X),V}\to \hat\cO_\cZ$. The converse is clear since $\cS \to \cS_\cZ$ is the tensor product of $\hat\cO_{\cM_B(X),V}\to \hat\cO_\cZ$ by $\hat\cO_{\cM_B(X),V}\to \cS$.  By the same argument, $\cS\to\cS_\cZ$ is a morphism of pro-$\tate$-MS-algebras if and only if $\Lambda\to\Lambda_\cZ$ is.
\end{proof}

\begin{cor}\label{cor use framing}
    Let $\cZ\subset \cM_B(X)$ be a locally closed algebraic substack and $R\cZ\subset R_B(X,x)$ the base-change by the quotient map.  Then $\cZ$ is formally twistor if and only if $R\cZ$ is.
\end{cor}

\crefformat{footnote}{#2\footnotemark[#1]#3}

\begin{prop}\label{lem: prop Hodge sub of Betti}
        Let $X$ be connected normal algebraic space.

    \begin{enumerate}
        \item\label{lem: prop Hodge sub of Betti 1} For any $r$, $\cM_B(X,r)$ and $\{\triv_r\}$ are formally twistor substacks of $\cM_B(X)$.
        \item\label{lem: prop Hodge sub of Betti 2} Assume $X$ is a curve.  For any $V\in \cM_B(X)(\bC)^{\qu,\ss}$, the fixed local monodromy leaf $FM(V)\subset\cM_B(X)$ is formally twistor at $V$.  For any $V\in \cM_B(X)(\bC)$ the fixed residual eigenvalues leaf $FE(V)\subset\cM_B(X)$ is a closed formally twistor substack.
        \item \label{lem: prop Hodge sub of Betti 3} Intersections and reductions of formally twistor substacks are formally twistor substacks.  Irreducible components\footnote{\label{foot1}These operations are defined by taking the quotients of the corresponding operations on the base-change to the framed moduli space---see \Cref{sect:stackZariski}.} of reduced twistor substacks are formally twistor substacks.  Finite unions of closed twistor substacks are twistor substacks.   
        \item \label{lem: prop Hodge sub of Betti 4} Let $f:X\to Y$ be a morphism of connected normal algebraic spaces and $f^*:\cM_B(Y)\to\cM_B(X)$ the pullback morphism.  Then:
        \begin{enumerate}
                    \item\label{lem: prop Hodge sub of Betti 4a} For any formally twistor substack $\cZ\subset \cM_B(X)$, $(f^*)^{-1}(\cZ)\subset \cM_B(Y)$ is a formally twistor substack.
        \item\label{lem: prop Hodge sub of Betti 4b} For any formally twistor substack $\cZ\subset \cM_B(Y)$, $f^*\cZ\subset \cM_B(X)$ is a formally twistor substack provided $f$ is dominant or a Lefschetz curve (see \Cref{defn lefschetz}).
        \end{enumerate}
 \item\label{lem: prop Hodge sub of Betti 5} Inverse images of formally twistor substacks under the direct sum and tensor product morphisms $\oplus,\otimes:\cM_B(X)^2\to\cM_B(X)$ are formally twistor substacks\footnote{With the obvious definition of formally twistor substack of $\cM_B(X)^2$.}. 
    \end{enumerate}
\end{prop}

\begin{proof}
    The first part of \eqref{lem: prop Hodge sub of Betti 1} is clear and the second follows from \eqref{lem: prop Hodge sub of Betti 4b} via pullback from a point.  In parts (2) and (3), the closure under semisimplification follows since the semisimplification of a point is contained in the closure of that point, by \Cref{ss in closure}.  Part \eqref{lem: prop Hodge sub of Betti 2} is then \Cref{prop rel FLM} and the discussion in \Cref{FRE}.  The reduction and irreducible component part of \eqref{lem: prop Hodge sub of Betti 3} follows from \Cref{MHS ring components}, since by \Cref{thm:versal} and the normality of $X$ the maximal ideal of the twistor algebras are the weight -1 part.
    
    For the remaining parts of \eqref{lem: prop Hodge sub of Betti 3}-\eqref{lem: prop Hodge sub of Betti 5}, all of the operations are compatible with the quotient map from the framed moduli space, so by \Cref{cor use framing} it suffices to prove the corresponding statements in the framed case.  The intersection part of \eqref{lem: prop Hodge sub of Betti 3} is clear.  The union part follows since the completed local ring of the union is the image of the completed local ring of the product. 
    
    For part (4b), under the hypotheses on $f$, $f^*$ is immersive\footnote{By immersive we mean a locally closed immersion locally on the source.} by \Cref{lem immersive} below. The closure under semisimplification condition in parts \eqref{lem: prop Hodge sub of Betti 4} and \eqref{lem: prop Hodge sub of Betti 5} follows from the fact that semisimplification commutes with pullback, direct sum, and tensor product---in the pullback case we must use \Cref{pullback ss}.  This in particular implies that semisimple points map to semisimple points and semisimple points in an image lift to semisimple points under these functors.  Moreover, the pullback of a local system with quasiunipotent local monodromy has quasiunipotent local monodromy, and in part (4b) under the assumptions on $f$ the converse also holds by \Cref{qu lefschetz}.  Altogether, this implies that in part (4a) we have $((f^*)^{-1}\cZ)(\bC)^{\qu,\ss}\subset (f^*)^{-1}(\cZ(\bC)^{\qu,\ss})$ and in part (4b) we have $(f^*\cZ)(\bC)^{\qu,\ss}=f^*(\cZ(\bC)^{\qu,\ss})$.  Finally, we must use the same property for $\CVHS$ (which follows from the $\bR_{>0}$-action---see \Cref{sect:R* on Betti}).

    Part \eqref{lem: prop Hodge sub of Betti 4a} is then immediate from \Cref{lem mini compat}(2) and \Cref{thm:versal frame}(3).  Part \eqref{lem: prop Hodge sub of Betti 4b} follows from \Cref{thm:versal frame}(3) as well since the completed local ring of the image is the image of the pullback map on completed local rings.  For part \eqref{lem: prop Hodge sub of Betti 5}, we consider the fiber product
    \[\begin{tikzcd}
        S\ar[rr]\ar[d]&&R_B(X,x)\ar[d]\\
        R_B(X,x)^2\ar[r]&\cM_B(X)^2\ar[r,"\oplus\text{,}\otimes"]&\cM_B(X).
    \end{tikzcd}\]
    Since sums and tensor products of AVMS are AVMS, by \Cref{thm:versal}(3) the top induces a morphism of pro-$\tate$-MS-algebras on completed local rings, and so images and inverse images of formally Hodge/twistor subsets along the top map are formally Hodge/twistor.  By \Cref{check on versal} this is enough. 
\end{proof}

\begin{lem}\label{lem immersive}Let $f \colon X\to Y$ be a morphism of algebraic spaces.
\begin{enumerate}
    \item If $f_* \colon \pi(X,x)\to \pi_1(Y,y)$ is surjective, then $f^* \colon \cM_B(Y)\to \cM_B(X)$ is a closed immersion.
    \item If $f$ is dominant and $Y$ is normal, then $f^* \colon \cM_B(Y)\to \cM_B(X)$ is immersive.

\end{enumerate}
    
\end{lem}
\begin{proof}
The first part is clear on representation spaces.  For the second part, there is a locally closed subspace $i \colon Z\to X$ such that $f|_Z \colon Z\to Y$ is generically finite, and if $f|_{Z}^*=i^*\circ f^*$ is immersive, then so is $f^*$. 
 Thus we reduce to $f$ generically finite.  There is a dense Zariski open $U\subset Y$ for which the base-change $f_U \colon X_U\to U$ is finite \'etale, and since by normality and part (1) the restrictions $\cM_B(X)\to \cM_B(X_U)$ and $\cM_B(Y)\to\cM_B(U)$ are closed immersions, we reduce to the case $f$ is finite \'etale.  In this case, $\bC_Y\to f_*\bC_{X}$ splits, so the map on tangent spaces $\Ext^1_{\bC_Y}(V,V)\to \Ext^1_{\bC_{X}}(f^*V,f^*V)$ is split as well, hence injective.
\end{proof}

\section{Bialgebraic sets of local systems and density of the quasiunipotent locus}\label{sect:bialg}
In this section we define constructible, bialgebraic, and $\bar \bQ$-bialgebraic sets of local systems, and prove \Cref{intro qu density}.  Once the basic definitions are made, the result will follow rather easily from known transcendence theorems (namely Ax--Lindemann for the exponential map and the Gelfond--Schneider theorem).  The main subtlety is that for $X$ non-proper, some care must be taken on the De Rham side, as the naive De Rham analytic space does not have an algebraic structure.  This is essentially because the condition that the singularities be regular implicitly makes reference to the existence of an extension as a connection with log poles to a log smooth compactification, and there is a non-finite-type choice of such an extension.  

Much of the general formalism of constructible subsets of stacks and bialgebraic subsets of the Betti stack have been treated elsewhere (Simpson \cite{Simpsonrankone} in the projective case, and Budur, Wang, and Lerer \cite{Budur-Wang,Budur-Lerer-Wang} in the quasiprojective case) in roughly equivalent ways.  We direct the reader to the corresponding sections along the way.
  
 \subsection{Constructible subsets of algebraic stacks}\label{sect:stackZariski}
  For the remainder, unless otherwise indicated, $X$ will be a connected smooth algebraic space defined over an algebraically closed subfield $k\subset\bC$.  We use the embedding to define the analytification $X^\an$.  We say a subset $\Sigma\subset X(\bC)$ has nonempty interior in each component of $X$ if, for each irreducible component $X_0$ of $X$, the intersection $\Sigma\cap X_0$ has nonempty interior in $X_0$ with respect to the euclidean topology.  
 
 Given an algebraic stack $\cM$ over $\bC$, it will be useful to have a reasonable notion of constructible subsets of $\cM(\bC)$.  We have the following easy lemma:

 \begin{lem}
     Let $f:V\to U$ be a smooth morphism of finite type complex algebraic spaces which is surjective on points.  Then $\Sigma\subset U(\bC)$ is Zariski (open/closed/constructible), (open/closed) in the euclidean topology, or analytically constructible\footnote{That is, in the Boolean algebra generated by closed analytic subvarieties.} if and only if $f^{-1}(\Sigma)$ is.
 \end{lem}
 \begin{proof}
     All the claims follow from the existence of sections locally on $U$ (using surjectivity) in the euclidean and \'etale topologies and Chevalley's theorem.
 \end{proof}
For some smooth atlas $\pi:U\to\cM$ by an algebraic space, we may equip $\cM(\bC)$ with the quotient Zariski or euclidean topology via the surjective map $\pi:U(\bC)\to\cM(\bC)$, which by the lemma is independent of the atlas.  We say a subset $\Sigma\subset \cM(\bC)$ is (Zariski/analytically) constructible\footnote{We often just say ``constructible'' if we mean Zariski constructible if there's no chance of confusion.} if $\pi^{-1}(\Sigma)$ is, which by the lemma is again independent of the atlas, and we may likewise define both the Zariski closure $\Sigma^\Zar$ and the euclidean closure $\bar\Sigma$.  Note that $\pi^{-1}(\Sigma)^\Zar=\pi^{-1}(\Sigma^\Zar)$ and $\overline{\pi^{-1}(\Sigma)}=\pi^{-1}\left(\overline \Sigma\right)$ for any atlas.  We say $\Sigma$ is irreducible if it is irreducible with respect to the Zariski topology, meaning every nonempty open subset of $\Sigma$ is dense.  This is equivalent to there being a morphism $f:Z\to\cM$ from an irreducible algebraic space $Z$ whose image is contained and dense in $\Sigma^\Zar$, since on the one hand images of irreducible topological spaces are irreducible, and on the other hand any irreducible component of the preimage of $\Sigma$ in an atlas has image in $\Sigma$ with nonempty interior.  Finally, observe that for any Zariski closed $\Sigma\subset\cM$, there is a unique closed immersion of algebraic stacks $\cN\subset\cM$ with $\Sigma=\cN(\bC)$.  Indeed, given a presentation in groupoids $R\rightrightarrows U\xrightarrow{\pi} \cM$, we take the closed subspace $V=\pi^{-1}(\Sigma)$ and the induced full presentation in groupoids $R_V\rightrightarrows V$ where $R_V=R\times_{U\times U}V\times V$.
\begin{rem}\label{rem on images}
    In the case of a global quotient $\cM=[\bG\backslash X]$ of an algebraic space by an algebraic group, the Zariski/euclidean topology on $\cM(\bC)$ is the quotient topology, and the algebra of constructible subsets is equivalent to the algebra of invariant constructible subsets of $X$. 
\end{rem}
\begin{cor}Let $\cM$ be an algebraic stack over $\bC$.  The algebra of constructible subsets of $\cM(\bC)$ is the Boolean algebra generated by finite type Zariski closed subsets. 
\end{cor}

\begin{cor}
    Let $f:\cM\to\cN$ be a $\bC$-morphism of algebraic stacks over $\bC$.  For any constructible subset $\Sigma\subset\cM(\bC)$ (resp. constructible subset $\mathrm{T}\subset \cN(\bC)$), the image $f(\Sigma)\subset\cN(\bC)$ (resp. preimage $f^{-1}(\mathrm{T})\subset\cM(\bC)$) is constructible.
\end{cor}

 \begin{lem}\label{pull-push-contained-nonempty-interior}
 Let $f\colon \cM \to \cN$ be a $\bC$-morphism between two finite type algebraic stacks over $\bC$. 
 \begin{enumerate}
     \item if $\Sigma\subset \cM(\bC)$ has nonempty interior in each component of its Zariski closure, then $f(\Sigma) \subset \cN(\bC)$ has nonempty interior in each component in its Zariski closure.
     \item Assume that $f$ is universally open (e.g. flat). If $\mathrm{T}\subset \cN(\bC)$ has nonempty interior in each component in its Zariski closure, then $f^{-1}(\mathrm{T}) \subset \cM(\bC)$ has nonempty interior in each component in its Zariski closure.
 \end{enumerate}  
 \end{lem}
 \begin{proof}
One reduces immediately to the case where both $\cM$ and $\cN$ are algebraic spaces of finite type. We start with $(1)$. Let $Z$ be an irreducible component of the Zariski closure $\Sigma^{\Zar}$ of $\Sigma$ in $\cM$. Then $\Sigma \cap Z(\bC)$ is Zariski dense in $Z$ and has nonempty interior in $Z(\bC)$ with respect to the euclidean topology. Therefore, one is reduced to consider the case where both $\cM$ and $\cN$ are irreducible algebraic spaces and $f$ is dominant. In particular, every Zariski closed strict subset of $\cM(\bC)$ or $\cN(\bC)$ has empty interior with respect to the euclidean topology. Let $U \subset \cN$ be a dense Zariski-open subset such that $\cM_U \to U$ is flat. The induced holomorphic map $\cM_U^\an \to U^\an$ is therefore flat too, hence in particular open. The set $\Sigma \cap \cM_U(\bC)$ has nonempty interior in $\cM_U(\bC)$, hence its image $f(\Sigma \cap \cM_U(\bC)) = f(\Sigma) \cap U(\bC)$ has nonempty interior in $U^\an$. A fortiori, $f(\Sigma)$ has non-empty interior in $\cN^\an$. This shows $(1)$.

Let us now prove $(2)$. Base-changing $f$ along the open inclusion $\mathrm{T}^\Zar \subset \cN$, one can assume that $\mathrm{T}$ is Zariski-dense in $\cN$. Moreover, working component by component, it is sufficient to treat the case where $\cN$ is irreducible. Since $f$ is universally open, one has $f^{-1}(\mathrm{T}^\Zar)=f^{-1}(\mathrm{T})^\Zar$. Let $Z \subset \cM$ be an irreducible component of $f^{-1}(\mathrm{T})^\Zar$. Since $f$ is open, its image $f(Z)$ is a dense constructible subset of $\cN$, hence it contains a non-empty Zariski open subset $U \subset \cN$. Let $V$ be a nonempty euclidean open subset of $\cN^\an$ contained in $\Sigma$. Since $U^\an$ is dense in $\cN^\an$, the intersection $U^\an \cap V$ is nonempty. Therefore, by shrinking $V$, one can assume that $V \subset U^\an$. It follows that $Z^\an$ contains the nonempty open subset $f^{-1}(V)$. Since $f^{-1}(V) \subset f^{-1}(\mathrm{T})$, this completes the proof. \end{proof}

 If $\cM$ is defined over $k\subset\bC$ we likewise define the notions of $k$-constructible subsets of $\cM(\bC)$ and their $k$-Zariski closures by using a smooth atlas defined over $k$.  For any $k$-morphism $f:\cM\to \cN$ of $k$-algebraic stacks, there is an induced map $f(\bC):\cM(\bC)\to \cN(\bC)$ under which images and preimages of $k$-constructible subsets are $k$-constructible.

\subsection{Correspondence stacks and morphisms} 
 By a correspondence stack $(\cM_{DR}, \cM_B ,\cZ)$ we mean a pair of algebraic stacks $\cM_{DR}, \cM_B$ together with an analytic stack $\cZ$ and a diagram
 \[\begin{tikzcd}
     &\cZ\ar[rd,"\pi_B"]\ar[ld,"\pi_{DR}"']&\\
     \cM_{DR}^\an&&\cM_B^\an.
 \end{tikzcd}\]
 If $\cM_{DR},\cM_B$ are algebraic spaces (in which case we can always take $\cZ$ to be an analytic space, see below), we call $(\cM_{DR}, \cM_B ,\cZ)$ a correspondence space.  At the moment the subscripts $DR$ and $B$ are merely suggestive, and the abstract definition is symmetric in the two ``realizations''.  We will be mainly interested in correspondence stacks where $\cZ$ is the graph of a holomorphic map $\cM_{DR}^\an \to \cM_{B}^\an$. We often refer to the correspondence stack as just the pair $(\cM_{DR},\cM_{B})$ when the correspondence $\cZ$ is clear.

A morphism of correspondence stacks $f:(\cM_{DR},\cM_{B})\to(\cM_{DR}',\cM_{B}')$ consists of algebraic morphisms $f_{DR}:\cM_{DR}\to\cM_{DR}'$ and $f_B:\cM_{B}\to\cM_B'$ such that the image of the natural map $\cZ\to \cM_{DR}'^\an\times\cM_B'^\an$ is contained in the image of $\cZ'$ on the level of points---that is, if the natural map $\cZ\times_{(\cM_{DR}'^\an\times\cM_B'^\an)}\cZ'\to\cZ$ is surjective on points.  Observe that with this notion of morphism, any correspondence stack $(\cM_{DR}, \cM_B ,\cZ)$ is naturally isomorphic to the correspondence stack $(\cM_{DR},\cM_B, \tilde \cZ)$ obtained by replacing the analytic correspondence by $\tilde \cZ$ for any morphism $\tilde \cZ\to\cZ$ of finite type analytic stacks which is surjective on points.  A morphism $(\cM_{DR},\cM_{B})\to(\cM_{DR}',\cM_{B}')$ is then given by compatible morphisms $f_{DR}:\cM_{DR}\to\cM_{DR}'$, $f_B:\cM_{B}\to\cM_{B}'$, $\phi:\cZ\to\cZ'$ up to replacing $\cZ$ in this way.

For every algebraic stack $\cM$ defined over $\bC$ and every field automorphism $\sigma\in\Aut(\bC/\bQ)$, we write $\cM^\sigma := \cM \otimes_\sigma \bC$.  For a subfield $K\subset\bC$, a $K$-correspondence stack $(\cM_{DR},\{\cM_{B,\sigma}\})$ consists of an algebraic stack $\cM_{DR}$ defined over $\bC$ and correspondence stacks $(\cM_{DR}^\sigma,\cM_{B,\sigma},\cZ_\sigma)$ for each $\sigma\in\Aut(\bC/\bQ)$, where each $\cM_{B,\sigma}$ is defined over $K$. 
 A morphism of $K$-correspondence stacks $f:(\cM_{DR},\{\cM_{B,\sigma}\})\to(\cM_{DR}',\{\cM_{B,\sigma}'\})$ consists of a morphism $f:\cM_{DR}\to\cM_{DR}'$ and $K$-morphisms $f_{B,\sigma}: \cM_{B,\sigma}\to\cM_{B,\sigma}'$ such that $(f^\sigma,f_{B,\sigma}):(\cM_{DR}^\sigma,\cM_{B,\sigma})\to (\cM_{DR}^{\prime\sigma},\cM_{B,\sigma}')$ is a morphism of correspondence stacks for each $\sigma$.  Given a correspondence stack $(\cM_{DR},\cM_B)$ with $\cM_{DR}$ defined over $\bQ$ and $\cM_B$ defined over $K$, we get a natural $K$-correspondence stack $(\cM_{DR},\{\cM_B\})$ by taking all the correspondences to be the same as that of $(\cM_{DR},\cM_B)$.

\begin{exa}
    Let $\Lambda$ be a finitely generated free $\bZ$-module.  There is a natural correspondence space $\Exp(\Lambda):=(\G_a\otimes\Lambda,\G_m\otimes\Lambda)$ whose correspondence is
    \[
\begin{tikzcd}
&\bC\otimes\Lambda\ar[ld,equals]\ar[rd,"\exp\otimes 1"]&\\
\bC\otimes\Lambda&&\bC^*\otimes\Lambda.
\end{tikzcd}
\]
As above, there is a naturally associated $\bQ$-correspondence space $\Exp_{\bQ}(\Lambda):=(\G_a\otimes\Lambda,\{\G_m\otimes\Lambda\})$.  The construction is functorial in that a homomorphism $\Lambda\to\Lambda'$ yields a morphism $(\G_a\otimes\Lambda,\G_m\otimes\Lambda)\to(\G_a\otimes\Lambda',\G_m\otimes\Lambda')$ of correspondence spaces, as well as a morphism $(\G_a\otimes\Lambda,\{\G_m\otimes\Lambda\})\to(\G_a\otimes\Lambda',\{\G_m\otimes\Lambda'\})$ of $\bQ$-correspondence spaces.

Taking $\Lambda=\bZ^n$, denote $\Exp(n):=\Exp(\bZ^n)$. Taking the quotient of both sides by the symmetric group $\mathfrak{S}_n$, we have a natural correspondence space $\Sym\Exp(n):=(\Sym^n\G_a,\Sym^n\G_m)$ as well as its powers.  There is a natural quotient morphism $\Exp(n)\to\Sym\Exp(n)$.  Likewise for the $K$-analogs, $\Exp_{K}(\Lambda),\Exp_{K}(n)$ and $\Sym\Exp_{K}(n)$.
\end{exa}

\subsection{Bialgebraic subsets}Let $(\cM_{DR},\cM_B)$ be a correspondence stack with correspondence 
 \[
\begin{tikzcd}
&\cZ\ar[ld,"\pi_{DR}",swap]\ar[rd,"\pi_{B}"]&\\
\cM_{DR}^\an&&\cM_B^\an.
\end{tikzcd}
\]
\begin{defn}
$(\cM_{DR},\cM_{B})$ is \emph{irreducible bialgebraic} if $\cM_{DR},\cM_B,\cZ$ are irreducible and $\img\pi_{DR}$ and $\img\pi_{B}$ have nonempty interior.  We say a pair $(\Sigma_{DR},\Sigma_B)\subset (\cM_{DR},\cM_B)$ of constructible subsets $\Sigma_{DR}\subset \cM_{DR}(\bC)$ and $\Sigma_B\subset \cM_B(\bC)$ is \emph{bialgebraic} if there is a correspondence stack $(\cN_{DR},\cN_{B})$ which is a finite disjoint union of irreducible bialgebraic correspondence stacks and a morphism $f:(\cN_{DR},\cN_B)\to (\cM_{DR},\cM_B)$ such that $f_{DR}(\cN_{DR}(\bC))^\Zar=\Sigma_{DR}^\Zar$ and $f_B(\cN_B(\bC))^\Zar=\Sigma_B^\Zar$.  We say a constructible $\Sigma_{DR}\subset\cM_{DR}(\bC)$ (resp. $\Sigma_B\subset\cM_B(\bC)$) is \emph{bialgebraic} if it can be completed to a bialgebraic pair $(\Sigma_{DR},\Sigma_B)$.

\end{defn}
By definition, bialgebraicity is a property of the Zariski closures.  We immediately have the following:
\begin{lem}
A constructible subset $\Sigma_{DR}\subset\cM_{DR}(\bC)$ (resp. $\Sigma_{B}\subset\cM_{B}(\bC)$) is bialgebraic if and only if for every irreducible component $\Sigma_0$ of $\Sigma_{DR}^\Zar$ (resp. $\Sigma_B^\Zar$) there is an irreducible component $\cZ_0$ of $\pi_{DR}^{-1}(\Sigma_0)$ (resp. $\pi_{B}^{-1}(\Sigma_0)$) such that $\pi_B(\cZ_0)$ (resp. $\pi_{DR}(\cZ_0)$) has nonempty interior in $\pi_B(\cZ_0)^\Zar$ (resp. $\pi_{DR}(\cZ_0)^\Zar$).
\end{lem}
We have the following consequence of \Cref{pull-push-contained-nonempty-interior}.  

\begin{lem}
    Let $f:(\cM_{DR},\cM_B)\to (\cM_{DR}',\cM_B')$ be a morphism of correspondence stacks.
    \begin{enumerate}
        \item If $(\Sigma_{DR},\Sigma_B)\subset (\cM_{DR},\cM_B)$ is bialgebraic, then so is $(f_{DR}(\Sigma_{DR}),f_B(\Sigma_B))\subset (\cM_{DR}',\cM_B')$.
        \item Assume $f_{DR},f_B$ are both universally open, and that the natural map
        \begin{equation}\label{dumb surjection}\cZ\times_{(\cM_{DR}'^\an\times\cM_B'^\an)}\cZ'\to (\cM_{DR}^\an\times\cM_B^\an)\times_{(\cM_{DR}'^\an\times\cM_B'^\an)}\cZ'\end{equation}
        is surjective on points.  Then if $(\Sigma_{DR}',\Sigma_B')\subset (\cM_{DR}',\cM_B')$ is bialgebraic, so is $(f_{DR}^{-1}(\Sigma_{DR}'),f_B^{-1}(\Sigma_B'))\subset (\cM_{DR},\cM_B)$.
    \end{enumerate}

\end{lem}
\begin{proof}(1) is immediate from \Cref{pull-push-contained-nonempty-interior}.  For (2), observe that the condition on the correspondence in \eqref{dumb surjection} is preserved under base-change.  For $(\cN_{DR},\cN_B)$ a disjoint union of irreducible bialgebraic correspondence stacks and $g:(\cN_{DR},\cN_B)\to (\cM_{DR}',\cM_B')$ a morphism, we may therefore reduce to the base-change
\[(\cN_{DR},\cN_B)=(\cN_{DR}'\times_{\cM_{DR}'}\cM_{DR},\cN_{B}'\times_{\cM_{B}'}\cM_{B})\to (\cN_{DR}',\cN_B')\]
and so assume $(\cM_{DR}',\cM_B')$ is irreducible bialgebraic and $(\Sigma_{DR},\Sigma_B)=(\cM_{DR}'(\bC),\cM_B'(\bC))$.  By the universal openness, there is an open set $U_{DR}'\subset\cM_{DR}'$ with $U_{DR}'(\bC)\subset \Sigma'_{DR}$ such that $f^{-1}_{DR}(U_{DR}'(\bC))$ is dense in $f^{-1}_{DR}(\Sigma_{DR}')$ and is a disjoint union of irreducible constructible sets each dominating $U_{DR}'$; likewise for $\cM_B'$.  Thus, by restricting to a dense open set on the target and connected components on the source, we may assume each of $\cM_{DR},\cM_B$ is a disjoint union of irreducible stacks and $f_{DR},f_B$ are surjective on each component.  Finally, the condition on the map implies $f_{DR}^{-1}(\img\pi'_{DR})\subset \img\pi_{DR}$ and $f_{B}^{-1}(\img\pi'_{B})\subset\img\pi_{B}$.  By \Cref{pull-push-contained-nonempty-interior}, for each irreducible component $\cN_{DR}$ of $\cM_{DR}$ there is some irreducible component $\cN_{B}$ of $\cM_B$ and an irreducible component $\cW$ of the correspondence such that $(\cN_{DR},\cN_B,\cW)$ is irreducible bialgebraic and $(\cN_{DR},\cN_B)\to(\cM_{DR}',\cM_B')$ is a morphism (and likewise for irreducible components of $\cM_B$), so the conclusion follows.
\end{proof}

 Let $(\cM_{DR},\{\cM_{B,\sigma}\})$ be a $K$-correspondence stack with correspondences
  \[
\begin{tikzcd}
&\cZ_\sigma\ar[ld,"\pi_{DR,\sigma}",swap]\ar[rd,"\pi_{B,\sigma}"]&\\
(\cM_{DR}^\sigma)^\an&&\cM_{B,\sigma}^\an.
\end{tikzcd}
\]
\begin{defn}
$(\cM_{DR},\{\cM_{B,\sigma}\})$ is \emph{irreducible bialgebraic} if $(\cM_{DR},\cM_{B,\sigma})$ is irreducible bialgebraic for all $\sigma$.  We say a pair $(\Sigma_{DR},\{\Sigma_{B,\sigma}\})\subset (\cM_{DR},\{\cM_{B,\sigma}\})$ consisting of a constructible subset $\Sigma_{DR}\subset \cM_{DR}(\bC)$ and $K$-constructible subsets $\Sigma_{B,\sigma}\subset \cM_{B,\sigma}(\bC)$ is \emph{$K$-bialgebraic} if there is a $K$-correspondence stack $(\cN_{DR},\{\cN_{B,\sigma}\})$ which is a finite disjoint union of irreducible bialgebraic $K$-correspondence stacks and a morphism $f:(\cN_{DR},\{\cN_{B,\sigma}\})\to (\cM_{DR},\{\cM_{B,\sigma}\})$ such that $f_{DR}(\cN_{DR}(\bC))^\Zar=\Sigma_{DR}^\Zar$ and $f_B(\cN_{B,\sigma}(\bC))^\Zar=\Sigma_{B,\sigma}^\Zar$ for all $\sigma$.
\end{defn}

 \begin{cor}\label{Qpushpull}
    Let $f:(\cM_{DR},\cM_{B,\sigma})\to (\cM_{DR}',\{\cM_{B,\sigma}'\})$ be a morphism of $K$-correspondence stacks.
    \begin{enumerate}
        \item\label{Qpushpull1} If $(\Sigma_{DR},\{\Sigma_{B,\sigma}\})\subset (\cM_{DR},\{\cM_{B,\sigma}\})$ is $K$-bialgebraic, then so is $(f_{DR}(\Sigma_{DR}),\{f_{B,\sigma}(\Sigma_{B,\sigma})\})\subset (\cM_{DR}',\{\cM_{B,\sigma}'\})$.
        \item\label{Qpushpull2} Assume $f_{DR},f_{B,\sigma}$ are all universally open, and that the natural map
        \[\cZ_\sigma\times_{(\cM_{DR}'^\an\times\cM_{B,\sigma}'^\an)}\cZ_\sigma'\to (\cM_{DR}^\an\times\cM_{B,\sigma}^\an)\times_{(\cM_{DR}'^\an\times\cM_{B,\sigma}'^\an)}\cZ_\sigma'\]
        is surjective on points for each $\sigma$.  Then if $(\Sigma_{DR}',\{\Sigma_{B,\sigma}'\})$ is $K$-bialgebraic, so is $(f_{DR}^{-1}(\Sigma_{DR}),\{f_B^{-1}(\Sigma_{B,\sigma})\})$.
    \end{enumerate}

\end{cor}

\begin{thm}[{Ax--Schanuel \cite{axtorus}, Gelfond--Schneider \cite{Gelfond}}]\label{gelfond}Let $\Lambda$ be a finitely-generated free $\bZ$-module.
\begin{enumerate}
\item For any bialgebraic pair $(\Sigma_{DR},\Sigma_{B})$ of $\Exp(\Lambda)$, $\Sigma_{DR}^\Zar$ is a finite union of $\bC$-translates of $\bQ$-linear subspaces.
    \item For any $\bar\bQ$-bialgebraic pair $(\Sigma_{DR},\{\Sigma_{B,\sigma}\})$ of $\Exp_{\bar\bQ}(\Lambda)$, $\Sigma_{DR}^\Zar$ is a finite union of $\bQ$-translates of $\bQ$-linear subspaces.
\end{enumerate}
     
\end{thm}

\begin{proof}Ax--Schanuel implies (1).  For (2) we are reduced to the statement for a point by taking closures, components, and quotients, using (1).  That point must be defined over $\bar\bQ$ on the Betti side by definition, as well as on the De Rham side, or else its $\Aut(\bC/\bar\bQ)$ orbit would be uncountable.  Then this is the Gelfond--Schneider theorem.  
\end{proof}

\begin{cor}\label{gelfond cor}For any $\bar\bQ$-bialgebraic pair $(\Sigma_{DR},\{\Sigma_{B,\sigma}\})$ of $\Sym\Exp_{\bar\bQ}(n)$, $\Sigma_{DR}^\Zar$ is the image of a finite union of $\bQ$-translates of $\bQ$-linear subspaces under the natural quotient map
\[\Exp_{\bar\bQ}(n)\to \Sym\Exp_{\bar\bQ}(n).\]
In particular, the $\Sigma_{B,\sigma}^\Zar$ are all the same, and images of $\bQ$-points (resp. torsion points) under the quotient map are Zariski dense in $\Sigma_{DR}$ (resp. each $\Sigma_{B,\sigma}$).
\end{cor}
\begin{proof}
    The quotient map is flat and satisfies the condition of \Cref{Qpushpull}\ref{Qpushpull2}.
\end{proof}

\subsection{The De Rham stack}\label{subsect:DR stack}In the next few sections, it will be useful to introduce the notion of a countably finite type algebraic (or analytic) stack, by which we mean a stack with a countable open cover by finite type algebraic (or analytic) stacks.  One easily adapts the notions of constructible subsets (which are constructible on some finite type substack), correspondence stacks, and bialgebraic pairs to this setting.  Note by the Baire category theorem, any irreducible bialgebraic pair in the sense of countably finite type stacks is irreducible bialgebraic in the ordinary sense.

There is a natural analytic stack $\cM_{DR}(X^\an)$ of analytic flat vector bundles on $X^\an$.  In fact, by solving the connection, there is a natural biholomorphism 
\[RH_{X^\an}:\cM_{DR}(X^\an)\xrightarrow{\cong}\cM_B(X)^\an.\]
Unfortunately $\cM_{DR}(X^\an)$ does not in general have a well-behaved algebraic structure capturing the flat vector bundle.

\begin{exa}\label{ex Gm 1}Let $X=\G_m$ in either the algebraic or the analytic category and consider the stack of rank 1 flat vector bundles $(L,\nabla)$ on $X$.  The underlying line bundle is necessarily $L\cong \cO_X$, so the natural action by $H^0(X,\Omega_X)$ on $\cM_{DR}(X)$ sending $(L,\nabla)$ to $(L,\nabla+\alpha)$ is transitive, while the choice of identification with $\cO_X$ is a torsor for $H^0(X,\cO^\times_X)$.  Thus, we can realize $\cM_{DR}(X)$ as $[H^0(X,\cO_X^\times)\backslash H^0(X,\Omega_X)]$, where an invertible function $f$ acts on a form $\alpha $ as $\alpha\mapsto d\log f+\alpha$.  In the analytic category, integrating $\alpha$ around the unit circle and exponentiating provides a natural coordinate and provides the isomorphism with the analytification of $\cM_B(X)\cong [\bG_m\backslash\bG_m]$, where $\bG_m$ acts trivially.  In the algebraic category, $H^0(X,\cO_X^\times)$ consist of powers $q^n$ up to scaling, $H^0(X,\Omega_X^1)=\bC[q,q^{-1}]dq$, and $\lambda q^n$ acts by adding $n\frac{dq}{q}$ so the quotient is not of finite type. 
\end{exa}
What is missing in the algebraic description of the De Rham stack is the requirement that the singularities be regular, as indeed Deligne's version of the Riemann-Hilbert correspondence is an equivalence of categories between the category of local systems on $X^\an$ and the category of algebraic flat vector bundles on $X$ with regular singularities. 
 This condition makes reference to the existence of a logarithmic extension to a log smooth compactification.

Let us now introduce some terminology.  Let $(\bar X,D)/k$ be a proper log smooth algebraic space with $X=\bar X\setminus D$.  For an algebraic space $T$, let $\mathcal{D}_{\bar X_T}$ be the sheaf of algebraic $\pi_T^{-1}\cO_T$-linear differential operators on $\bar X_T:=\bar X\times T$ where $\pi_T:\bar X_T\to T$ is the projection and let $\mathcal{D}_{\bar X_T/T}[\log D] \subset \mathcal{D}_{\bar X_T/T}$ be the $\mathcal{O}_{\bar X_T}$-subalgebra generated by germs of $\pi_T^{-1}\cO_T$-linear derivations of $\cO_{\bar X_T}$ which preserve $\mathcal{O}_{\bar X_T}(-D_T)$.  A family of rank $r$ logarithmic connections on $(\bar X,D)$ parametrized by an algebraic space $T$ is a $\mathcal{D}_{\bar X_T/T}[\log D]$-module $E$ on $\bar X_T$ which is locally free of rank $r$ as a $\mathcal{O}_{\bar X_T}$-module.  Given a point $x\in X(\bC)$, a family of framed (at $x$) rank $r$ logarithmic connections on $(\bar X,D)$ parametrized by a scheme $T$ is a family of logarithmic connections $E$ together with a trivialization $E|_{x\times T}\cong \mathcal{O}_T^r$ as $\cO_T$-modules. 

\begin{thm}[{\cite[Proposition 3.3]{Nitsure}$+\epsilon$}] \label{nitsure}
Let $T$ be an algebraic space and let $\mathcal{E}$ be a coherent sheaf on $\bar X_T$ which is flat over $T$. Consider the functor on $T$-spaces which associates to
$f : T^\prime \rightarrow T$ the set of all $\mathcal{D}_{\bar X_{T^\prime}/T'}[\log D]$-module structures on $(1_X \times f )^\ast \mathcal{E}$. Then this functor is represented by a an algebraic space of finite type over $T$.
\end{thm}
\begin{proof}
    The statement of \cite{Nitsure} assumes $(\bar X,D)$ is projective.  We may take a log-smooth projective modification $\pi:(\bar X',D')\to (\bar X,D)$ and apply the theorem to $\pi^*E$.  To check that a connection $\nabla$ descends, it suffices to check for a fixed meromorphic basis $s_i$ of $E$ that the sections $\nabla s_i$ are pulled back from $\bar X$.
\end{proof}

The stack of locally free sheaves on $\bar X$ is countably finite type, and it follows that the stack $\cM_{DR}(\bar X, D)$ of logarithmic connections is countably finite type.  The analytification $\cM_{DR}(\bar X,D)^\an$ is naturally identified with the analytic stack $\cM_{DR}(\bar X^\an,D^\an)$ of analytic logarithmic connections by GAGA.

\begin{exa}\label{ex Gm 2}Continuing \Cref{ex Gm 1}, let $X=\G_m$ again and take $\bar X =\bP^1$.  The De Rham stack $\cM_{DR}(\bar X,D,1)$ of rank one logarithmic connections is a union over $k\in\bZ$ of $[\G_m\backslash H^0(X,\Omega_X(\log D))]$.  On the $k$th factor we associate to a log one-form $\alpha$ the logarithmic connection $(\cO_{\bar X}(k),d+\alpha)$ where we consider $\cO_{\bar X}(k)$ as the sheaf of functions with at worst a $k$-order pole at $\infty$ and $\G_m$ acts trivially. 
\end{exa}

\subsection{The Riemann--Hilbert correspondence stack}\label{sect:RH}  See \Cref{sect:Betti} for the basic properties of the Betti stack.  Let $(\bar X,D)$ be a proper log smooth algebraic space with $X=\bar X\setminus D$. We define the Riemann--Hilbert correspondence stack $\cM_{RH}(\bar X,D)$ of $(\bar X,D)$ to be the correspondence stack $(\cM_{DR}(\bar X,D), \cM_B(X))$ with correspondences
\[
\begin{tikzcd}
    &\cM_{DR}(\bar X^\an,D^\an)\ar[ddl,equal]\ar[rd]&&\\
    &&\cM_{DR}(X^\an)\ar[rd,"RH_{X^\an}"]&\\
    \cM_{DR}(\bar X, D)^\an&&&\cM_B(X)^\an
\end{tikzcd}
\]
where the unlabelled arrow is restriction to $X$.  We denote the resulting analytic morphism $RH_{(\bar X,D)}:\cM_{DR}(\bar X,D)^\an\to \cM_B(X)^\an$.  We define the Riemann--Hilbert $\bQ$-correspondence stack $\cM_{RH,\bQ}(\bar X,D)$ to be the $\bQ$-correspondence stack $(\cM_{DR}(X),\{\cM_B(X^\sigma)\})$ whose correspondences are given by the graph of $RH_{(\bar X^\sigma,D^\sigma)}$.

\begin{thm}[{Simpson \cite[Theorem 3.1(c)]{Simpsonrankone}, \cite[Theorem 6.1]{Simpsonrankone}}]\label{simpsonbialg}Let $A$ be a complex abelian variety.
\begin{enumerate}
\item For any bialgebraic pair $(\Sigma_{DR},\Sigma_{B})$ of $\cM_{RH}(A,1)$, $\Sigma_B^\Zar$ is a finite union of translates of pullbacks of $\cM_B(A',1)$ for abelian variety quotients $A\to A'$.
\item For any $\bar\bQ$-bialgebraic pair $(\Sigma_{DR},\Sigma_{B})$ of $\cM_{RH,\bQ}(A,1)$, $\Sigma_B^\Zar$ is a finite union of torsion translates of pullbacks of $\cM_B(A',1)$ for abelian variety quotients $A\to A'$.

\end{enumerate}
     
\end{thm}

Note that by the Remark in \cite[\S6]{Simpsonrankone}, the hypotheses of \Cref{simpsonbialg} agree with those of \cite[Theorem 6.1]{Simpson_noncompact}, except the latter requires algebraicity in the Dolbeault realization but doesn't need it.

\begin{rem}
    If $\bD\subset\bC$ is the unit disk centered at the origin, we can interpret $\Sym\Exp_\bQ(n)$ as the coarse Riemann--Hilbert stack $M_{RH,\bQ}(\bD,0,n)$, the rank $n$ substack of $M_{RH,\bQ}(\bD,0)$.
\end{rem}

\subsection{The residue map}
Let $(\bar X,D)$ be a proper log-smooth algebraic space with $X=\bar X\setminus D$ and denote by $\{D_i\}_{i\in I}$ the irreducible components of the boundary $D$.  Fix a $k$-point $x\in X(k)$.  For each $i$ we have a conjugacy class of loops $[\gamma_i]$ around $D_i$ and a representation $\rho\in R_B(X,x,r) (\bC)$ is said to have quasiunipotent local monodromy if $\rho(\gamma_i)$ is quasiunipotent for each $i$---that is, the eigenvalues of $\rho(\gamma_i)$ are roots of unity for each $i$.

On the Betti side, for any $\gamma\in\pi_1(X^\an,x)$, we get a morphism $\mathrm{Char}_B(\gamma) : R_B(X,x,r) \rightarrow \bA^r_{\bZ} $ defined over $\bZ$, given by the composition of the evaluation map and taking the characteristic polynomial. Clearly, this map depends only on the conjugation class of $\gamma$. Since every $D_i$ provides a conjugacy class of simple loops around it, we obtain a morphism $\mathrm{Char}_B(\bar X,D,x) : R_B(X,x,r) \rightarrow (\bA^r_{\bZ})^I $, which clearly descends to a $\bar\bQ$-morphism
\[\mathrm{Char}_B(\bar X,D,r):\cM_B(X,r)\to (\Sym^r\G_a)^I\]
where we have identified $\img\mathrm{Char}_B(\gamma)\cong \Sym^r\G_a$ over $\bar\bQ$ by taking roots of the associated polynomial.  We say a local system $V\in\cM_B(X)(\bC)$ has quasiunipotent local monodromy if some (hence any) framing has quasiunipotent local monodromy, or equivalently if $\mathrm{Char}_B(\bar X,D)(V)$ is the image of a torsion point of some $((\G_m)^r)^I$.

 On the De Rham side, suppose we have a family of framed rank $n$ log-connections on $(\bar X, D)$ parametrized by a $k$-scheme $T$.  Then we obtain a $k$-morphism $\mathrm{Char}_{DR}(\infty): T\to (\bA^r_k)^I $ by taking the characteristic polynomials of the residues along the boundary components.  In particular, we obtain a $k$-morphism
 \[\mathrm{Char}_{DR}(\bar X,D,r):\cM_{DR}(\bar X,D,r)\to (\Sym^r\G_a)^I.\]
 We say a point $\cV\in\cM_{DR}(\bar X,D)(\bC)$ has quasiunipotent local monodromy if $\mathrm{Char}_{DR}(\bar X,D)(\cV)$ is the image of a $\bQ$-point of $((\G_a)^r)^I$.

 \begin{lem}Let $\cV\in \cM_{DR}(\bar X,D)(\bC)$.  The following are equivalent.
\begin{enumerate}
    \item $\cV$ has quasiunipotent local monodromy;
    \item $RH_{X^\an}(\cV^\an)$ has quasiunipotent local monodromy;
    \item $RH_{(X^\sigma)^\an}((\cV^\sigma)^\an)$ has quasiunipotent local monodromy for every $\sigma\in\Aut(\bC/\bQ)$.
\end{enumerate}
\end{lem}

\begin{lem}\label{qu is bialg}
    There is a morphism
    \[\mathrm{Char}_{\bQ}(\bar X,D,r):\cM_{RH,\bQ}(\bar X,D,r)\to \Sym\Exp_{\bQ}(r)^{ I}\]
    of $\bQ$-correspondence stacks whose underlying maps on algebraic stacks are $$(\mathrm{Char}_{DR}(\bar X,D,r),\{\mathrm{Char}_B(\bar X^\sigma,D^\sigma,r)\}).$$
\end{lem}
\begin{proof} The residue is an algebraic map, and therefore compatible with the action of $\Aut(\bC/\bQ)$.  In a polydisk neighborhood of the boundary, the eigenvalues of the local monodromy are easily seen to be the exponential of the residues, and the characteristic polynomial is independent of the basepoint. 
\end{proof}

\subsection{Normal algebraic spaces}
We now make sense of the Riemann--Hilbert stack for a normal complex algebraic space $X$.  On the Betti side this is straightforward:  setting $i:X^\reg\to X$ to be the regular locus, since $\pi_1((X^{\reg})^\an,x)\to \pi_1(X^\an,x))$ is surjective for any basepoint $x\in X^\reg$, we have a closed immersion $i^*:\cM_B(X)\to \cM_B(X^\reg)$.

Let $\pi:X'\to X$ be a log resolution which is an isomorphism over $X^\reg$ and let $E$ be the exceptional locus.  Then there is a lift $i':X^\reg\to X'$ and for the same reason $\pi^*:\cM_B(X)\to\cM_B(X')$ and $i'^*:\cM_B(X')\to\cM_B(X^\reg)$ are closed immersions.  We may further assume there is a log smooth compactification $(\bar X',D')$ such that $(\bar X',D'':=D'\cup\bar E)$ is a log smooth compactification of $X^\reg$. 

\begin{lem}
    Every irreducible component of the reduced preimage $RH_{(\bar X',D'')}^{-1}((i^*\cM_B(X))^\an)$ is algebraic.
\end{lem}
\begin{proof}
    The preimage is an analytic substack with countably many irreducible components.  There is a natural closed immersion
    \[b:\cM_{DR}(\bar X',D')\to\cM_{DR}(\bar X',D'')\]
    realizing $\cM_{DR}(\bar X',D')$ as the substack of log connections with vanishing residue along the irreducible components of $E$.  Since there is a commutative diagram
    \[
    \begin{tikzcd}
        \cM_{DR}(\bar X',D')^\an\ar[r,"RH_{(\bar X',D')}"]\ar[d,"b"']&\cM_B(X')^\an\ar[d,"(i'^*)^\an"]\\
        \cM_{DR}(\bar X',D'')^\an\ar[r,"RH_{(\bar X',D'')}"]&\cM_B(X^\reg)^\an
    \end{tikzcd}
    \]

    It therefore suffices to show every irreducible component of $RH_{(\bar X',D')}^{-1}((\pi^*\cM_B(X))^\an)$ is algebraic.
    
    For any fiber $F$ of $\pi:X'\to X$, let $q_F:F'\to F$ be a resolution and consider the induced map
    \[q_F^*:\cM_{DR}(\bar X',D')\to\cM_{DR}(F').\]
    Then on the level of points, $RH_{(\bar X',D')}^{-1}((\pi^*\cM_B(X))^\an)$ is the intersection of $((q^*_F)^\an)^{-1}(RH_F^{-1}\triv_F)$ over all $F$, where $\triv_F\subset \cM_B(F)$ is the set of trivial local systems on $F$.  Since $RH_F^{-1}(\triv_F)$ is a closed analytic subset and a countable union of algebraic subsets (namely, those flat bundles admitting a flat meromorphic frame), it follows that each irreducible component of $RH_F^{-1}(\triv_F)$ is Zariski closed, and therefore the same is true of $((q^*_F)^\an)^{-1}(RH_F^{-1}\triv_F)$.  By noetherianity, each irreducible component of the intersection is therefore algebraic.
\end{proof}
\begin{defn}
    Given $\pi:X'\to X$ and $(\bar X',D')$ as above, we define $\cM_{DR}(\bar X',D',E)$ to be the reduced preimage $RH_{(\bar X',D'')}^{-1}((i^*\cM_B(X))^\an)$ with its natural structure as a substack, and we define $RH_{(\bar X',D',E)}:\cM_{DR}(\bar X',D',E)^\an\to\cM_B(X)^\an$ to be the restriction of $RH_{(\bar X',D'')}$.  We define $(\cM_{DR}(\bar X',D',E),\cM_B(X),RH_{(\bar X',D',E)})$ as the Riemann-Hilbert correspondence stack $\cM_{RH}(\bar X',D',E)$.

    Both $\pi$ and $(\bar X',D')$ can be chosen to be defined over the field of definition of $X$, and we define $(\cM_{DR}(\bar X',D',E),\{\cM_B(X^\sigma)\},\{RH_{(\bar X'^\sigma,D'^\sigma,E^\sigma)}\})$ to be the Riemann--Hilbert $\bQ$-correspondence stack $\cM_{RH,\bQ}(\bar X',D',E)$.

\end{defn}
    Note that by definition we have morphisms of ($\bQ$-)correspondence stacks
        \[i^*:\cM_{RH}(\bar X',D',E)\to\cM_{RH}(\bar X',D'') \]
    \[i^*:\cM_{RH,\bQ}(\bar X',D',E)\to\cM_{RH,\bQ}(\bar X',D'') \]
which are injective on points in the De Rham realization and closed immersions in all of the Betti realizations.

\subsection{Proof of \Cref{intro qu density}}
We now state a more precise version of \Cref{intro qu density}:
\begin{thm}\label{qu dense in bialg}
    Let $X$ be a connected normal algebraic space and let $(\Sigma_{DR},\{\Sigma_{B,\sigma}\})$ be a bialgebraic constructible pair of the Riemann--Hilbert $\bQ$-correspondence stack $\cM_{RH,\bQ}(\bar X',D',E)$ of a log smooth compactification of a log resolution.  Then the local systems with quasiunipotent local monodromy are Zariski dense in $\Sigma_{DR}$ and each $\Sigma_{B,\sigma}$.
\end{thm}
\begin{proof}
    The restriction to the regular locus is a bialgebraic pair so by \Cref{Qpushpull} we may assume $X$ is smooth.  We may further assume $(\Sigma_{DR},\{\Sigma_{B,\sigma}\})$ is the image of an irreducible bialgebraic correspondence stack.  Let $(\Sigma_{DR}',\{\Sigma_{B,\sigma}'\})$ be the image in $\Sym\Exp_{\bQ}^I$ under the residue map.  According to \Cref{Qpushpull} and \Cref{gelfond cor}, image of rational points (resp. torsion points) under the quotient map are Zariski dense in $\Sigma_{DR}'$ (resp. each $\Sigma_{B,\sigma}'$), and it follows that the preimages of such points are Zariski dense in $\Sigma_{DR}$ (resp. each $\Sigma_{B,\sigma}$), which are exactly the points with quasiunipotent local monodromy.
\end{proof}

    In the \Cref{bialg=abs} we will show that the notion of $K$-bialgebraic constructible subsets of the Riemann--Hilbert stacks is equivalent to the notion of absolute $K$-constructible subsets (i.e. $K$-constructible in every Betti realization, with no condition on the De Rham realization) for countable $K\subset\bC$.

\section{Harmonic maps towards NPC spaces}\label{sect:harmonic}
We will need to use pluriharmonic bundles in both the archimedean and non-archimedean setting.  Recall this means equipping a local system $V$ on an algebraic variety $X$ with a pluriharmonic metric in the former case and a pluriharmonic norm in the latter case.  Either way, the pluriharmonic structure is encoded in an equivariant map  $\tilde X^V\to\Delta$ for an appropriate classifying space $\Delta$---a symmetric space in the former case, and a Euclidean building in the latter.  

The main goal of this section is to formulate the version of the existence theorem for harmonic maps we will need (\Cref{Gromov-Schoen-regular}).  This follows from known results but does not appear in this form in the literature, so we first give the necessary background from \cite{KoSc93}.  We also state some related results that will be needed in subsequent sections.
\subsection{Sobolev spaces}
Let $(\Omega,g)$ be a Riemannian domain, i.e. a connected open subset of a Riemannian manifold $(M, g)$, having the property that its metric completion $\bar \Omega$ is a compact subset of $M$. Let $(\Delta, d)$ be a complete metric space. The space $L^2(\Omega, \Delta)$ is by definition
the set of Borel-measurable maps $u\colon \Omega \to \Delta$ having separable range for which
\[  \int_\Omega d^2(u(x),Q) \dvol_g(x) < \infty \]
for some $Q \in \Delta$. The space $L^2(\Omega, \Delta)$ equipped with the distance function $D$ given by 
\[ D(u,v) = \left( \int_\Omega  \left( d(u(x), v(x) \right)^2 \dvol_g(x) \right)^{1/2}\]
is a complete metric space.

For every $\epsilon > 0$, let $\Omega_\epsilon = \{ x \in \Omega \, | \, \text{dist}(x, \partial \Omega) > \epsilon \}$, let $B_\epsilon(x)$ be the geodesic ball centered at $x \in M$ and $S_\epsilon(x) = \partial B_\epsilon(x)$. Let $d\sigma_{S_\epsilon(x)}$ be the $(\dim M - 1)$-dimensional surface measure on $S_\epsilon(x)$. For every $u \in L^2(\Omega, \Delta)$ and every $\epsilon > 0$, let $e^u_\epsilon \in L^1(\Omega, \bR)$ be the function defined by
\[  e^u_\epsilon(x) := \int_{y \in S_\epsilon(x)} \frac{d^2(u(x), u(y))}{\epsilon^2} \frac{d\sigma_{S_\epsilon(x)}(y)}{\epsilon^{n-1}}\]
for $x \in \Omega_\epsilon$ and zero otherwise.

For $\epsilon > 0$ and $f \in C_c(\Omega)$ (the set of continuous real-valued functions with compact support), let
\[ E_\epsilon^u(f) := \int_\Omega f \cdot e^u_\epsilon(x) \dvol_g(x).\]
We say that $u \in L^2(\Omega, \Delta)$ has finite energy, and we write $u \in W^{1,2}(\Omega, \Delta )$, if
\[ E(u) :=  \sup_{f \in C_c(\Omega), 0 \leq f \leq 1}  \limsup_{\epsilon \to 0} E_\epsilon^u(f)  < \infty .\]
In particular, $W^{1,2}(\Omega, \Delta )$ contains all Lipschitz continuous maps from $\Omega$ to $\Delta$. If $u \in W^{1,2}(\Omega, \Delta )$, then the measures $e^u_\epsilon(x) \dvol_g(x)$ converge weakly to a limiting measure which is absolutely continuous with respect to Lebesgue measure, and hence may
be written as $e^u(x) \dvol_g(x)$ for a function $e^u \in L^1(\Omega, \bR)$ called the energy density of $u$. 

\begin{thm}[{\cite[Theorem 1.6.1]{KoSc93}}]\label{KoSc_convexity_energy}
Let $(u_k)$ be a sequence of maps in $W^{1,2}(\Omega, \Delta )$ that converges in $L^2(\Omega, \Delta)$ to a map $u$. If $\sup_k E(u_k) < \infty$, then $u \in W^{1,2}(\Omega, \Delta )$ and $E(u) \leq \liminf_k E(u_k)$.  
\end{thm}

\begin{rem}
 If there is an isometric embedding $i \colon (\Delta,d) \hookrightarrow (\bR^N, d_E)$ in an Euclidean space, then $W^{1,2}(\Omega,\Delta ) =\{ u\in W^{1,2}(\Omega, \bR^N) \, | \, u(x) \in X \text{ a.e. } x \in \Omega \}$ and 
\[ E(u) = \int_\Omega |\nabla(i \circ u)|^2.\]
These definitions are the working definitions adopted in \cite{Gromov-Schoen}, where harmonic maps towards locally compact Euclidean buildings were first studied, by locally embedding isometrically Euclidean buildings in Euclidean spaces. In particular, this shows that Korevaar-Schoen definition of the Sobolev space specializes to the classical definition when $(\Delta,d)$ is the real line equipped with the Euclidean metric. 
\end{rem}

When the boundary $\partial \Omega$ of $\Omega$ in $M$ is locally Lipschitz, there is a well-defined trace map $W^{1,2}(\Omega, \Delta ) \to L^{2}(\partial \Omega, \Delta )$.

\begin{thm}[{\cite[Corollary 1.6.3 and Theorem 1.12.2]{KoSc93}}]\label{KoSc-trace}
Let $(\Omega,g)$ be a Lipschitz Riemannian domain and let $(\Delta,d)$ be a complete metric space.
\begin{enumerate}
    \item If $u,v \in W^{1,2}(\Omega, \Delta )$, then the function $d_\Delta(u,v)$ is in $W^{1,2}(\Omega, \bR )$, and the condition $tr(u) = tr(v)$ is equivalent to $tr(d(u,v)) = 0$ in $L^{2}(\partial \Omega, \bR )$.
    \item If the sequence $(u_i)$ of maps in $W^{1,2}(\Omega, \Delta )$ has uniformly bounded energies $E(u_i)$, and if $(u_i)$ converges in $L^2(\Omega, \Delta)$ to a map $u$, then $u \in W^{1,2}(\Omega, \Delta )$ and the trace functions of the $u_i$ converge in $L^{2}(\partial \Omega, \Delta )$ to the trace of $u$. 
\end{enumerate}
\end{thm}

\subsection{NPC spaces}
\begin{defn}
A complete metric space $(\Delta, d)$ is said to be non-positively curved (NPC), or CAT(0) in Gromov terminology, if for any pair of points $x,y \in \Delta$, there exists a point $m \in \Delta$ such that
\[ d(z, m)^2 + \frac{d(x, y)^2}{4} \leq \frac{d(z, x)^2 + d(z, y)^2}{2} \]
for any $z \in \Delta$.
\end{defn}
The point $m$ satisfies $d(x, m) = d(y, m) = \frac{d(x,y)}{2}$, and it is uniquely determined by this property. It is called the midpoint of x and y. 

Crucially, the energy on $W^{1,2}(\Omega, \Delta )$ is convex When $(\Delta, d)$ is NPC.
\subsection{Harmonic maps}

When the boundary $\partial \Omega$ of the source $\Omega$ in $M$ is locally Lipschitz and the target $(\Delta,d)$ is a NPC complete metric space, the Dirichlet problem can be solved as follows.

\begin{thm}[Korevaar-Schoen {\cite[Theorem 2.2]{KoSc93}}]\label{Dirichlet-NPC}
Let $(\Omega,g)$ be a Lipschitz Riemannian domain and let $(\Delta,d)$ be a NPC complete metric space. Let $\phi \in W^{1,2}(\Omega, \Delta )$. Then there exists a unique $u \in W^{1,2}(\Omega, \Delta )$ which minimizes the energy among all maps in $W^{1,2}(\Omega, \Delta )$ with the same trace as $\phi$. The map $u$ is locally Lipschitz continuous in the interior of $\Omega$.
\end{thm}

\begin{defn}
Let $(M,g)$ be a Riemannian manifold and $(\Delta,d)$ be a complete metric space. A map $u\colon M \to \Delta$ is called harmonic if it is a local energy minimizer: every point in $M$ admits an open neighborhood with Lipschitz boundary such that all maps in $W^{1,2}(\Omega,\Delta )$ which agree with $u$ outside of this neighborhood have no less energy.
\end{defn}

Thanks to Theorem \ref{Dirichlet-NPC}, a harmonic map is locally Lipschitz continuous. In particular, to check that a locally Lipschitz continuous map $u\colon M \to \Delta$ is harmonic, it is sufficient to check that every point in $M$ admits an open neighborhood with Lipschitz boundary such that every locally Lipschitz continuous map $v\colon M \to \Delta$ which agrees with $u$ outside of this neighborhood has no less energy.

It will be important to know that the local Lipschitz constant of a harmonic map is controlled by the energy as follows.

\begin{thm}[Zhang-Zhong-Zhu, see {\cite[Theorem 1.4 and Remark 1.5 ii)]{ZZZ}}]\label{energy_control_Lipschitz}
Let $\Omega$ be a bounded domain (with smooth boundary) of an $n$-dimensional Riemannian
manifold $(M, g)$ with $\mathrm{Ric}_M \geq - K$ for some $K \geq 0$, and let $(\Delta, d)$ be a NPC complete metric space. For every $R > 0$, there exists a constant $C = C(n, \sqrt{K} R)$ such that for every harmonic map $u \colon \Omega \to \Delta$ and every $x \in \Omega$ with $B_{2R}(x) \Subset \Omega$, the restriction of $u$ to $B_R(x)$ is Lipschitz continuous and the Lipschitz constant is bounded above by $C\int_{B_R(x)} e^u(x) \dvol_g(x)$.
\end{thm}

Recall that a metric space is called proper if all of its closed balls are compact; in particular it is locally compact. By the Hopf-Rinow Theorem, a locally compact NPC metric space is proper if and only if it is complete, see \cite[Theorem I.2.4]{Ballmann95}. Symmetric spaces of noncompact type and locally compact Euclidean Bruhat-Tits buildings are examples of proper CAT(0) spaces.

\begin{thm}[Corlette, Gromov-Schoen, Korevaar-Schoen]\label{existence_harmonic_finite_energy}
Let $M$ be a complete Riemannian manifold with a finitely generated fundamental group, $\Delta$ be a proper NPC metric space and $\rho \colon \pi_1(M) \to \mathrm{Isom}(\Delta)$ a group homomorphism. Assume that
\begin{itemize}
    \item the action of $\rho$ does not have a fixed point on $\partial \Delta$;
    \item there exists a finite energy\footnote{Since the energy density of a $\rho$-equivariant map $u \colon \tilde{M} \to \Delta$ is $\rho$-equivariant, it is well-defined on $M$; the map $u$ is said to have finite energy if the integral on $M$ of its energy density is finite.} $\rho$-equivariant map $\tilde{M} \to \Delta$.
\end{itemize}
Then there exists a $\rho$-equivariant locally Lipschitz continuous map $u:\tilde{M} \to \Delta$ of least (finite) energy.  In particular, $u$ is harmonic.
\end{thm}

\begin{proof}
When $\Delta$ is a complete simply connected manifold with nonpositive sectional curvature, this is \cite[Theorem 2.2]{Corlette92}. When $\Delta$ is a locally compact Euclidean Bruhat-Tits building, this is \cite[Theorem 7.1]{Gromov-Schoen}. In general, this is the conjonction of Theorem 2.1.3, Remark 2.1.5 and Theorem 2.2.1 in \cite{KoSc97}.     
\end{proof}

\subsection{Averaging harmonic maps}

\begin{prop}[{\cite[Lemma 2.5.1]{KoSc93}}]
Let $(\Delta, d)$ be a NPC complete metric space. Let $S \subset \Delta$ be a finite subset. Then, there exists a unique point $c(S)$ such that the number $\sum_{P \in S} d(c(S),P)^2$ is minimal.
\end{prop}
The point $c(S)$ is called the center of mass of the finite set $S$. It belongs to the closed convex hull of $S$ \cite[Proposition 2.5.4]{KoSc93}.

\begin{prop}[{\cite[Proposition 2.5.2]{KoSc93}}]\label{Lipschitz_centerofmass}
Let $(\Delta, d)$ be a NPC complete metric space. Let $n$ be a positive integer. Let $P, Q\colon \{1, \ldots, n\} \to \Delta$ be two functions. Let $c_P$ and $c_Q$ be the centers of mass of the sets $\{P(1), \ldots, P(n)\}$ and $\{Q(1), \ldots, Q(n)\}$ respectively. Then
\[    d(c_P, c_Q)^2 \leq  \frac{1}{n} \sum_{i= 1}^n d(P(i), Q(i))^2 .\]    
\end{prop}

\begin{prop}\label{centerofmass_leastenergy}
Let $(M,g)$ be a complete Riemannian manifold, $(\Delta,d)$ be a proper NPC space and $\rho \colon \pi_1(M) \to \mathrm{Isom}(\Delta)$ a group homomorphism.    
Let $u_i \colon \tilde{M} \to \Delta, 1 \leq i \leq n$, be $\rho$-equivariant locally Lipschitz continuous maps of finite energy. Then the map $u \colon x \mapsto c(\{u_i(x)\})$ is a $\rho$-equivariant locally Lipschitz continuous map of finite energy. Moreover, if the $u_i$'s have least finite energy, then $u$ has least finite energy.
\end{prop}
\begin{proof}
Since the maps $u_i$ are locally Lipschitz continuous, the map $u$ is locally Lipschitz continuous thanks to Proposition \ref{Lipschitz_centerofmass}. Thanks to Equation $2.5.iii, p. 641$ in \cite{KoSc93}, the energy $E(u)$ of the map $u$ is finite and satisfies $E(u) \leq \frac{1}{n} \sum_{i= 1}^n E(u_i)$.  
\end{proof}

\begin{prop}\label{centerofmass_finite_energy}
Let $(M,g)$ be a complete Riemannian manifold, $(\Delta,d)$ be a proper NPC space and $\rho \colon \pi_1(M) \to \mathrm{Isom}(\Delta)$ a group homomorphism. Let $M^\prime \to M$ be a finite Galois covering space and $\rho^\prime \colon \pi_1(M^\prime) \to \mathrm{Isom}(\Delta)$ be the induced homomorphism. If there exists a $\rho^\prime$-equivariant locally Lipschitz continuous map $\tilde{M} \to \Delta$ of finite energy, then there exists a $\rho$-equivariant locally Lipschitz continuous map $\tilde{M} \to \Delta$ of finite energy.
\end{prop}
\begin{proof}
Let $u\colon \tilde{M^\prime} \to \Delta$ be a $\rho^\prime$-equivariant locally Lipschitz continuous map of finite energy. The center of mass of its Galois conjugates is a $\rho^\prime$-equivariant locally Lipschitz continuous map of finite energy by Proposition \ref{centerofmass_leastenergy}. Since it is Galois invariant, it descends to a $\rho$-equivariant locally Lipschitz continuous map $\tilde{M} \to \Delta$ of finite energy.
\end{proof}

\subsection{Constructing subharmonic functions}
 
\begin{prop}[see {\cite[Lemma 10.2]{Eells-Fuglede}}]
Let $(\Delta,d)$ be a NPC complete metric space. Let $\phi \colon \Delta \to \bR$ be a function which is convex when restricted to geodesics. If the map $u \colon M \to \Delta$ is harmonic, then the function $\phi \circ u$ is subharmonic.
\end{prop}

\begin{prop}\label{distance_subharmonic}
Let $(\Delta,d)$ be a NPC complete metric space. Let $u, u^\prime \colon M \to \Delta$ be two harmonic maps. Then the function $x \mapsto d(u(x), u^\prime(x))$ is subharmonic.    
\end{prop}
\begin{proof}
This is a consequence of the previous Proposition, since the map $(u, u^\prime)\colon M \to \Delta \times \Delta$ is harmonic and the function $(y, y^\prime) \mapsto d(y, y^\prime)$ is convex, the space $\Delta$ being NPC.
\end{proof}

\subsection{Limit of sequence of harmonic maps}

\begin{prop}\label{limit_of_harmonic_is_sometime_harmonic}
Let $(\Omega,g)$ be a Lipschitz Riemannian domain and let $(\Delta,d)$ be a NPC complete metric space. Let $\{\phi_k\}_k$ be a sequence of harmonic maps in $W^{1,2}(\Omega, \Delta )$. Assume that the sequence of energy $\{E(\phi_k)\}_k$ is bounded. Assume also that the set $\{\phi_k(x) | k \geq 0\}$ is relatively compact in $\Omega$ for some $x \in \Omega$. Then $\{\phi_k\}_k$ has a subsequence that converges to a harmonic map uniformly on every compact subset.
\end{prop}
\begin{proof}
A harmonic map is locally Lipschitz continuous. Moreover, since the $\{\phi_k\}_k$ have uniformly bounded energy, the local Lipschitz constants are uniformly bounded thanks to \Cref{energy_control_Lipschitz}. It follows from Arzel\`a-Ascoli theorem that a subsequence of $\{\phi_k\}_k$ converges uniformly on every compact subset of $\Omega$. The limit $\phi$ belongs to $W^{1,2}(\Omega, \Delta )$ thanks to Theorem \ref{KoSc-trace}. Let $u \in W^{1,2}(\Omega, \Delta )$ be the unique map which minimizes the energy among all maps in $W^{1,2}(\Omega, \Delta )$ with the same trace as $\phi$, see Theorem \ref{Dirichlet-NPC}. The function $d_\Delta(u, \phi)$ is the uniform limit of (a subsequence of) the functions $d_\Delta(u, \phi_k)$. Since the latter are subharmonic thanks to Proposition \ref{distance_subharmonic}, $d_\Delta(u, \phi)$ is subharmonic. But it is zero on the boundary of $\Omega$, hence it is zero everywhere by the maximum principle.
\end{proof}

\begin{cor}\label{limit_of_harmonic_is_often_harmonic}
Let $(\Omega,g)$ be a Lipschitz Riemannian domain and let $(\Delta,d)$ be a NPC complete metric space. Let $\{\phi_k\}_k$ be a converging sequence of harmonic maps in $W^{1,2}(\Omega, \Delta )$. Then the limit is harmonic.
\end{cor}
\begin{proof}
Since the sequence converges in $W^{1,2}(\Omega, \Delta )$, the energy are uniformly bounded, so that one can apply the preceding Proposition.
\end{proof}


\subsection{Poincaré-type metric}
\begin{defn}
Let $D$ be a normal crossing divisor in a complex manifold $X$ of dimension $n$. An admissible polydisk is an open subset $U \subset X$ equipped with a biholomorphism $U \simeq \bD^n $ such that $D \cap U$ corresponds to the union of some of the $\{z_i = 0\}$.
\end{defn}
We equip $\bD$ and $\bD^\ast$ with their Poincaré metrics $\omega_P$ that are respectively given by
\[ \frac{i dz \wedge d\bar z}{\left(1 - |z|^2 \right) ^2} \text{ and } \frac{i dz \wedge d\bar z}{ |z|^2 \left(- \log |z|^2 \right) ^2}. \]
We denote also by $\omega_P$ the induced metric on $\bD^l \times (\bD^\ast) ^m$ for every $l,m \geq 0$.

\begin{defn}
Let $D$ be a normal crossing divisor in a complex manifold $X$ of dimension $n$. A metric $\omega$ on $X \setminus D$ is of Poincaré type if for every admissible polydisk $\bD^n \subset X$ there exists $C > 0$ such that the restriction of $\omega$ to $\bD^n \setminus D$ satisfies $C^{-1} \omega_P \leq \omega \leq C \omega_P$ in the neighborhood of $0$. (Here $\omega_P$ denotes the Poincaré metric on $\bD^n \setminus D$.)
\end{defn}

As is well-known, if $D$ is a normal crossing divisor in a compact Kähler manifold $\bar X$, then $\bar X \setminus D$ admits a complete Kähler metric of Poincaré type.


\subsection{Euclidean buildings}

Let $K$ be a non-archimedean local field. Let $\G$ be a connected, reductive group over $K$. We denote by $\Delta(\G, K)$ the extended Bruhat-Tits building of $\G$, as introduced in \cite{Bruhat-TitsI, Bruhat-TitsII}. See also \cite{Landvogt}. It is a (poly-)simplicial complex equipped with an action of $\G(K)$. One can define a metric $d$ on $\Delta(\G, K)$ such that $(\Delta(\bG, K), d)$ is a NPC complete metric space and $\bG(K)$ acts by isometries.\\

We will need the following functoriality results from \cite{Landvogt}.
\begin{prop}
Let $f\colon \G \to \H$ be a homomorphism of connected, reductive $K$-groups. Then there exists a map $f_\ast : \Delta(\bG, K) \to \Delta(\H, K)$ which is $\bG(K)$-equivariant and, in case $f$ is injective, an isometrical inclusion.
\end{prop}

\begin{prop}[{\cite[Proposition 2.1.6]{Landvogt}}]\label{Bruhat-Tits-product}
Let $\bG_1, \ldots, \bG_n$ be connected, reductive $K$-groups. Then there is a canonical bijection,
\[ \Delta(\bG_1, K) \times \ldots \times \Delta(\bG_n, K) \to \Delta(\bG_1 \times \ldots \times \bG_n, K)\]
which is $\bG_1(K) \times \ldots \times \bG_n(K) = (\bG_1 \times \ldots \times \bG_n)(K)$-equivariant. After a
suitable normalization of the metrics on $\Delta(\bG_i, K)$, for $i= 1, \ldots , n$, this map becomes isometrical.
\end{prop}

\begin{prop}[{\cite[Theorem 2.2.1]{Landvogt}}]
Let $i\colon \G\to \H$ be an inclusion of connected, reductive $K$-groups. Then, there exists an inclusion $f_\ast : \Delta(\bG, K) \to \Delta(\H, K)$ which is $\bG(K)$-equivariant and isometrical (after a suitable normalization of the metric on $\Delta(\bG, K)$).
\end{prop}

\begin{prop}\label{Bruhat-Tits-isogeny}
Let $\pi \colon \bG^\prime \to \bG$ be a central isogeny of connected, reductive $K$-groups. Then there is a canonical bijection $\Delta(\bG^\prime, K) \to \Delta(\bG, K)$ which is compatible with the $\bG^\prime(K)$-action on the left hand side and the $\bG(K)$-action on the right hand side. After a suitable normalization of the metric on $\Delta(\bG^\prime, K)$, this map becomes isometrical.
\end{prop}
\begin{proof}
See \cite[4.2.18]{Bruhat-TitsII} or \cite[Proposition 2.1.7]{Landvogt}.     
\end{proof}

\subsection{Harmonic maps towards Euclidean buildings}

\begin{defn}
Let $(M,g)$ be a Riemannian manifold. Let $\Delta$ be a locally compact Euclidean building. Let $u \colon M \to \Delta$ be a map. Its regular locus is by definition the open subset $\Reg(u) \subset M$ consisting of points $x \in M$ admitting a neighborhood $U \subset M$ such that $u(U) \subset \Delta$ is contained in a single apartment of $\Delta$. The complementary $\Sing(u) := M \setminus \Reg(u)$ in $M$ is called the singular locus of $u$.     
\end{defn}

\begin{thm}[{\cite[Theorem 6.4]{Gromov-Schoen}}] \label{Gromov-Schoen-regular}
Let $M$ be a Riemannian manifold. Let $\Delta$ be a locally compact Euclidean building. Let $u \colon M \to \Delta$ be a harmonic map. Then the singular locus of $u$ has Hausdorff codimension at least $2$.
\end{thm}

\begin{thm}\label{existence_harmonic_nonarchimedean_quasiunipotentmonodromies}
Let $(\bar X,D)$ be a projective log smooth variety and set $X=\bar X\setminus D$. Let $K$ be a nonarchimedean local field. Let $\G$ be a connected reductive group over $K$. Let $\rho \colon \pi_1(X) \to \G(K)$ be a representation with Zariski dense image and quasiunipotent local monodromy. Let $\Delta(\G,K)$ be the (extended) Euclidean Bruhat-Tits building associated to $\G(K)$. Fix a Poincaré-type complete Kähler metric on $X$. Then there exists a $\rho$-equivariant harmonic map $\tilde{X} \to \Delta(\G,K)$ of finite energy.    
\end{thm}

When $\G$ is simple, a proof of this result is given in \cite{Jost-Zuo00}. For $\G = \mathrm{SL}_2$, a detailed proof appears in \cite{Corlette-Simpson}. The case where $\G$ is semisimple has been recently addressed in \cite{Daskalopoulos-Mese, BDDM}, in the more general setting of harmonic maps with possibly infinite energy.
\begin{proof}

When $\G$ is semisimple, the existence of such a map follows from \cite{BDDM}. When $\G = \T$ is a torus, the building $\Delta(\T,K) = \left( X_K^\ast(G) \otimes_\bZ \bR \right) ^\ast$ is a finite-dimensional real vector space on which $\T(K)$ acts via the translation action of the canonical lattice. Since the representation $\rho$ has quasiunipotent local monodromy, it factorizes through the surjection $\pi_1(X) \to \pi_1(\bar X)$. The existence of a $\rho$-equivariant harmonic map $\widetilde{\bar X} \to \Delta(\T,K)$ follows from classical Hodge theory for $H^1(\bar X, \bZ)$. The map $\widetilde{\bar X} \to \Delta(\T,K)$ has finite energy since $\bar X$ is compact, and its pull-back to $\tilde{X}$ is therefore a $\rho$-equivariant harmonic map of finite energy.
For a general connected reductive group $\G$, let $Z(\G)$ and $\cD(\G)$ denote the center and the derived subgroup of $\G$ respectively. Since $\G$ is reductive, $\T := \G \slash \cD(\G)$ is a torus, $\G^{ad} := \G \slash Z(\G)$ is a semisimple group and the canonical morphism $\G \to \T \times \G^{ad}$ is a central isogeny. Thanks to \Cref{Bruhat-Tits-product} and \Cref{Bruhat-Tits-isogeny}, this induces an isometrical bijection $\Delta(\bG, K) = \Delta(\T, K) \times \Delta(\G^{ad}, K)$ which is compatible with respect to the $\bG(K)$-action on the left hand side and the $\T(K) \times \G^{ad}(K)$-action on the right hand side.
Let $u_1 \colon \tilde{X} \to \Delta(\T,K)$ (resp. $u_2 \colon \tilde{X} \to \Delta(\G^{ad},K)$) be a harmonic map  of finite energy which is 
equivariant with respect to the induced representation $\pi_1(X) \to \T(K)$ (resp. $\pi_1(X) \to \G^{ad}(K)$). Then $(u_1, u_2) \colon \tilde{X} \to \Delta(\bG, K) = \Delta(\T, K) \times \Delta(\G^{ad}, K)$ is a 
$\rho$-equivariant harmonic map  of finite energy.    
\end{proof}

\subsection{Pluriharmonic maps}

\begin{defn}\label{def-pluriharmonicity} Let $M$ be a complex analytic space and $(\Delta,d)$ be a complete metric space. A map $u\colon M \to \Delta$ is called pluriharmonic if for any Kähler manifold $(N,h)$ equipped with a holomorphic map $f\colon N \to M$ the composite map $u \circ f \colon N \to \Delta$ is harmonic.
\end{defn}

In particular, a pluriharmonic map $M \to \Delta$ is harmonic for any choice of a Kähler metric on $M$.

\begin{thm} \label{pluriharmonic}
In the setting of Theorem \ref{existence_harmonic_nonarchimedean_quasiunipotentmonodromies}, every equivariant harmonic map $\tilde{X} \to \Delta$ of finite energy is pluriharmonic.
\end{thm}

\begin{proof}
This is due to Gromov-Schoen \cite[Theorem 7.3]{Gromov-Schoen}. Note however that the definition of pluriharmonicity in  \cite{Gromov-Schoen} is a priori weaker than our Definition \ref{def-pluriharmonicity}, but turns out to be equivalent by \cite[Proposition 1.3.6]{Eyssidieux} or Theorem \ref{extension_pluriharmonic_map} below.
\end{proof}

\begin{prop}[{See \cite[Proposition 1.3.3]{Eyssidieux}}]\label{Singular_locus_pluripolar}
    Let $M$ be a complex manifold. Let $\Delta$ be a locally compact Euclidean building. Let $u \colon M \to \Delta$ be a pluriharmonic map. Then the singular locus of $u$ is contained in a (strict) complex analytic subset of $M$.
\end{prop}

\begin{thm} \label{extension_pluriharmonic_map}
Let $(M,g)$ be a complex manifold. Let $(\Delta, d)$ be a NPC complete metric space. Let $Z \subset M$ be a closed pluripolar subset. Let $u \colon M \to \Delta$ be a locally Lipschitz continuous map.  Then $u \colon M \to \Delta$ is pluriharmonic as soon as it is pluriharmonic in restriction to $M \setminus Z$.
\end{thm}
\begin{proof}
Let $(N,h)$ be a Kähler manifold equipped with a holomorphic map $f\colon N \to M$, and let us prove that the composite map $u \circ f \colon N \to \Delta$ is harmonic. Since the map $u \circ f \colon N \to \Delta$ is locally Lipschitz continuous, it is sufficient to check that every point in $N$ admits an open neighborhood with Lipschitz boundary such that every locally Lipschitz continuous map $v\colon N \to \Delta$ which agree with $u$ outside of this neighborhood has no less energy. Therefore one can assume that $N$ is a ball $\bB^r$. Moreover, up to shrinking $\bB^r$, one can assume that the holomorphic map $f \colon \bB^r \to M$ is a uniform limit of holomorphic maps $g_k \colon \bB^r \to M$ such that $g_k^{-1}(Z)  \subsetneq \bB^r$ is a closed pluripolar subset. Since $u$ is locally Lipschitz continuous, the maps $u \circ g_k$ converges to $u \circ f$ in $W^{1,2}(\bB^n , \Delta)$. Since the limit in $W^{1,2}(\bB^n , \Delta)$ of a sequence of harmonic maps is harmonic (Corollary \ref{limit_of_harmonic_is_often_harmonic}), it is sufficient to treat the case where $f^{-1}(Z) \subsetneq \bB^r$ is a closed polar subset.

Let $h \colon \bB^r \to \Delta$ be the unique harmonic map that coincides with $u \circ f \colon \bB^r \to \Delta$ on $\partial \bB^r$, see Theorem \ref{Dirichlet-NPC}. Then the function $z \mapsto d(h(z), u \circ f(z))$ is locally Lipschitz continuous and its restriction to $\bB^r \setminus f^{-1}(Z)$ is subharmonic by Proposition \ref{distance_subharmonic}.

Therefore it is subharmonic on $\bB^r$ thanks to Brelot extension theorem \cite[Theorem 5.18]{Hayman-Kennedy}. Since it is zero on $\partial \bB^r$, it is zero everywhere by the maximum principle. Therefore, we have proved that $u \circ f \colon \bB^r \to \Delta$ is harmonic, and this finishes the proof.
\end{proof}

\section{The Betti--Dolbeault correspondence and the \texorpdfstring{$\bG_m$}{Gm}-action}\label{sect:Gm action}

Work of Corlette \cite{Corlette_JDG} and Mochizuki \cite[Part 5]{Mochizuki-AMS2} shows that every semisimple complex local system admits an essentially unique pluriharmonic metric.  Similarly, Simpson \cite{Simpson_noncompact} and Mochizuki \cite{mochizukikobayashi} characterize the existence of pluriharmonic metrics on logarithmic Higgs bundles.  This yields a correspondence between semisimple complex local systems and certain polystable logarithmic Higgs bundles.  Strictly speaking parabolic structures are needed to make this precise, but in the case of unipotent local monodromy on the Betti side (and nilpotent residues of the Higgs field on the Dolbeault side) there is a canonical such choice.  In this section, we show the bijection in this special case is a homeomorphism (for curves), and use it to transport the natural $\bG_m$-action on the Dolbeault side to the unipotent local monodromy locus on the Betti side.  This together with the density result of \Cref{qu dense in bialg} will allow us to produce $\bR_{>0}$-fixed points (and in particular variations of Hodge structures) in every irreducible component of a $\bR_{>0}$-stable Zariski closed $\bar\bQ$-bialgebraic subset of $M_B(X)$.

\subsection{Moduli spaces of logarithmic Higgs bundles}\label{moduli Lambda algebra}

Let $f \colon \barX \to S$ be a smooth projective morphism to a scheme of finite type over $\bC$, and $D \subset \barX$ a relative normal crossing divisor. Let $\Omega_{\barX/S}(\log D)$ denote the sheaf of logarithmic differentials, and $T_{\barX/S}(-\log D) \subset T_{\barX/S}$ be the subsheaf of the tangent sheaf dual to $\Omega_{\barX/S}(\log D)$.
\begin{defn}
 A logarithmic Higgs sheaf on $(\barX,D)$ over $S$ is a coherent $\cO_{\barX}$-module $E$ on $\barX$ together with a morphism of $\cO_{\barX}$-modules $\theta \colon E \to \Omega_{\barX/S}(\log D) \otimes_{\cO_{\barX}} E$ such that $\theta \wedge \theta = 0$. Equivalently, it is a (left) $\Sym T_{{\barX}/S}(-\log D)$-module which is coherent as a $\cO_{\barX}$-module. A Higgs bundle is a Higgs sheaf $(E, \theta)$ such that $E$ is a locally free $\cO_{\barX}$-module.    
\end{defn}

Fix $\cO_{\barX}(1)$ a relatively very ample line bundle on ${\barX}$. A logarithmic Higgs sheaf $(E, \theta)$ on $({\barX},D)$ over $S$ is Gieseker semistable if the underlying $\cO_{\barX}$-module is flat over $S$, and if for every geometric point $s$ of $S$ the restriction of $E$ to the fibre ${\barX}_s$ is pure and Gieseker semistable as a log-Higgs sheaf on $({\barX}_s, D_s)$ (the Hilbert polynomial, the rank, and the slope of $E_s$ are defined to be those of the underlying sheaf of $\cO_{{\barX}_s}$-module.) One defines similarly Gieseker stability, slope stability and slope semistability.  It follows from \cite[Theorem 4.3]{Langer-JEMS} that for logarithmic Higgs sheaves with vanishing Chern classes, Gieseker-semistability (resp. Gieseker-stability) is equivalent to slope-semistability (resp. slope-stability).
(Since ${\barX}_s$ is smooth, a coherent sheaf of pure dimension $d = \dim({\barX}_s)$ is the same thing as a torsion-free coherent sheaf.)

\begin{thm}
Let $f \colon {\barX} \to S$ be a smooth projective morphism to a scheme of finite type over $\bC$, and $D \subset {\barX}$ a relative normal crossing divisor. Fix $\cO_{\barX}(1)$ a relatively very ample line bundle on ${\barX}$, a section $\xi \colon S \to {\barX}$ and a positive integer $r$. Then,
\begin{enumerate}
\item The functor which associates to any $S$-scheme $T$ the set of isomorphism classes
of pairs $((E, \theta), \beta)$, where $(E, \theta)$ is a rank $r$ Gieseker-semistable logarithmic Higgs bundle on $({\barX}_T, D_T)$ over $T$, such that the Chern classes $c_i(E_t)$ vanish in $H^{2i}({\barX}_t, \bC)$ for all closed points $t \in T$, and $\beta \colon \xi^\ast E \simeq \cO_T^r$ is a framing, is representable by a quasiprojective complex algebraic variety $R_{Dol}(({\barX},D)/S,\xi,r)$.  
\item Changing the framing $\beta$ yields an action of the group $\bGL_r$ on $R_{Dol}(({\barX},D)/S,\xi,r)$, and every point is semistable for this action (with respect to the an appropriate linearized bundle obtained from $\cO_{\barX}(1)$). The GIT quotient is a good quotient $M_{Dol}(({\barX},D)/S,r) := \bGL_r\backslash\hspace{-.2em}\backslash R_{Dol}(({\barX},D)/S,\xi,r)$ that universally corepresents the functor which to an $S$-scheme $T$ associates the set of isomorphism classes of Gieseker semistable logarithmic Higgs bundles $(E, \theta)$ on ${\barX}_T$ over $T$ of Hilbert polynomial $r P_0$, such that the Chern classes $c_i(E_t)$ vanish in $H^{2i}({\barX}_t, \bC)$ for all closed points $t \in T$.
\item The closed orbits in $R_{Dol}(({\barX},D)/S,r)$ correspond to direct sums of stable logarithmic Higgs bundles with vanishing rational Chern classes.
\end{enumerate}
\end{thm}
\begin{proof}
Let $P_0$ denote the Hilbert polynomial of $\cO_{\barX}$. By applying \cite[Theorem 4.10]{SimpsonmoduliI} to the sheaf of rings of differential operators $\Lambda =\Sym T_{{\barX}/S}(-\log D)$ and the Hilbert polynomial $P= r P_0$, it follows that the functor which associates to any $S$-scheme $T$ the set of isomorphism classes of pairs $((E, \theta), \beta)$, where $(E, \theta)$ is a Gieseker-semistable logarithmic Higgs sheaf on $({\barX}_T, D_T)$ over $T$ with Hilbert polynomial $P$, satisfying condition $LF(\xi)$, and $\beta \colon \xi^\ast E \simeq \cO_T^n$ is a framing, is representable by a quasiprojective complex algebraic variety $R_{Dol}(({\barX},D)/S,\xi, r P_0)$. Let $R_{Dol}(({\barX},D)/S,\xi,r)$ be the disjoint union of the connected components of $R_{Dol}(({\barX},D)/S, n P_0)$ such that the rational Chern classes of the universal family vanish in cohomology for all closed points. Since a slope semistable logarithmic Higgs sheaf with vanishing Chern classes is locally free \cite[Theorem 4.3]{Langer-JEMS}, and Gieseker-semistability implies slope semistability, a Gieseker-semistable logarithmic Higgs sheaf $E$ with vanishing rational Chern classes on $({\barX}_T, D_T)$ over $T$ is necessarily a logarithmic Higgs bundle thanks to \cite[Lemma 1.27]{SimpsonmoduliI}. It follows that the condition $LF(\xi)$ is always satisfied, and this finishes the proof of part $(1)$. Parts $(2)$ and $(3)$ then follow from \cite[Theorem 4.10]{SimpsonmoduliI}.
\end{proof}

Likewise, the stack quotient $\cM_{Dol}(({\barX},D)/S,r):=[\bGL_r\backslash R_{Dol}(({\barX},D)/S,\xi,r)]$ is naturally identified with the stack of rank $r$ Gieseker-semistable logarithmic Higgs bundles with vanishing rational Chern classes.  In particular, for $({\barX},D)$ a projective log smooth pair, taking $S=\Spec\bC$ we let $\cM_{Dol}({\barX},D,r)$ be the stack of Gieseker-semistable logarithmic Higgs bundles with vanishing rational Chern classes, which has as good moduli space $M_{Dol}({\barX},D,x,r)$.  We denote the disjoint unions over all $r$ by $R_{Dol}(\barX,D,x)$, $\cM_{Dol}(\barX,D)$ and $M_{Dol}(\barX,D)$.  

The stack of Gieseker-semistable logarithmic Higgs bundles with vanishing rational Chern classes and whose Higgs field has nilpotent residues is easily seen to be a closed substack which we denote $\cM_{Dol}^\nilp(({\barX},D)/S,r)=[\bGL_r\backslash R_{Dol}^\nilp(({\barX},D)/S,\xi,r)]$, with good moduli space $M_{Dol}^\nilp(({\barX},D)/S,r)=\bGL_r\backslash\hspace{-.2em}\backslash R_{Dol}^\nilp(({\barX},D)/S,\xi,r)$, in the relative case, and $\cM^\nilp_{Dol}({\barX},D,r)$ and $M_{Dol}^\nilp({\barX},D,r)$ for $S=\Spec\bC$.

\subsubsection{The $\bG_m$-action}

The algebraic group $\bG_m$ acts on the complex algebraic variety $R_{Dol}(({\barX},D)/S,\xi,r)$ by scaling the Higgs field. This action commutes with the action of $\bGL_r$ on $R_{Dol}(({\barX},D)/S,\xi,r)$, hence it induces an action of $\bG_m$ on the moduli stack $\cM_{Dol}(({\barX},D)/S,r)$ and on its good moduli space $M_{Dol}(({\barX},D)/S,r)$ such that the morphisms $R_{Dol}(({\barX},D)/S,\xi,r) \to \cM_{Dol}(({\barX},D)/S,r) \to M_{Dol}(({\barX},D)/S,r)$ are $\bG_m$-equivariant.  The closed subvarieties/substacks corresponding to nilpotent residues are $\bG_m$-stable, so we also obtain $\bG_m$-actions on $R_{Dol}^\nilp(({\barX},D)/S,\xi,r),\cM_{Dol}^\nilp(({\barX},D)/S,r)$ and $M_{Dol}^\nilp(({\barX},D)/S,r)$.

The fixed points of the action of $\bG_m$ (or any infinite subgroup thereof) correspond to the logarithmic Higgs bundles which come from complex variations of Hodge structure, cf. \cite[Theorem 8]{Simpson_noncompact} and the proof of \cite[Proposition 10.3]{mochizukitame}.

\begin{prop}\label{existence_of_Gm_limit}
For every $x \in  M_{Dol}(({\barX},D)/S,r)$, the morphism 
\[ \bG_m \to M_{Dol}(({\barX},D)/S,r), t \mapsto t \cdot x \]
extends uniquely to a morphism $\bA^1 \to M_{Dol}(({\barX},D)/S,r)$. The image of $0 \in \bA^1$ is fixed by the action of $\bG_m$ and therefore corresponds to a $\bC$-VHS.
\end{prop}

\begin{proof}
The morphism $\bG_m \to M_{Dol}(({\barX},D)/S,r)$ lifts to a morphism $\bG_m \to \cM_{Dol}(({\barX},D)/S,r)$, and it is sufficient to prove that the later extends uniquely to a morphism $\bA^1 \to \cM_{Dol}(({\barX},D)/S,r)$. Assume that $x \in  M_{Dol}(({\barX}_s,D_s),r)$. It follows from Langton's theory (see \cite[Proposition 10.1]{Simpson-Hodge-filtration} or \cite[Theorem 5.1]{Langer-documenta}) that there exists a rank $r$ Gieseker-semistable logarithmic Higgs sheaf $(E , \theta)$ on $({\barX}_s \times \bA_1, D_s \times \bA^1)$ over $\bA^1$, whose restriction to $\bG_m$ corresponds to the morphism $\bG_m \to \cM_{Dol}(({\barX},D)/S,r)$. By flatness of $E$ over $\bA^1$, the rational Chern classes of $E_0$ are also zero. Therefore $(E_0, \theta_0)$ is a rank $r$ Gieseker-semistable logarithmic Higgs bundle on $({\barX}_s, D_s)$ by \cite[Theorem 4.3]{Langer-JEMS}, and $E$ is locally-free by \cite[Lemma 1.27]{SimpsonmoduliI}.    
\end{proof}

\subsection{The Riemann--Hilbert correspondence revisited}\label{sect:goodlocus}
As already described in \Cref{sect:RH}, the Riemann--Hilbert correspondence does not provide an isomorphism between $\cM_{DR}(\bar X,D)^\an$ and $\cM_B(X)^\an$ in the non-proper case, but it will in restriction to the locus with unipotent local monodromy. 

Let $(\bar X,D)$ be a proper log smooth algebraic space and set $X=\bar X\setminus D$. Let $\cM_{DR}(\bar X, D)^{\good}$ denote the open substack of $\cM_{DR}(\bar X, D)^{\an}$ consisting of logarithmic connections such that for each irreducible component $D_i$ of $D$ no two eigenvalues of the residue map differ by a nonzero integer \cite{Budur-Lerer-Wang}.

Observe that for every $\sigma \in \Aut(\bC/\bQ)$, the good locus of $\cM_{DR}(\bar X, D)^{\an}$ is sent bijectively to the good locus of $\cM_{DR}(\bar X^\sigma, D^\sigma)^{\an}$.

\begin{prop}\label{local iso on good}
The restriction of the Riemann-Hilbert morphism  $RH_{(\bar X,D)}^\good:\cM_{DR}(\bar X, D)^{\good}\to \cM_B(X)^\an$ is a surjective local isomorphism of analytic stacks.

\end{prop}
\begin{proof}
    The existence of the Deligne canonical extension gives surjectivity (in fact surjectivity if we restrict the eigenvalues to lie in a chosen fundamental domain for the exponential $\exp:\bC\to\bC^*$).

    To prove that that it is a local isomorphism, we analyze the relative obstruction theory.  We first show the relative deformation functor is unramified.  For a small extension $J\to A'\to A$ of artinian $\bC$-algebras, suppose $V'\to V$ is an extension of $A'$-local systems on $X^\an$ lying over $A'\to A$, and suppose $(E',\nabla')$ is a lift of $V'$ to $\cM_{DR}(\bar X,D)^{\good}$.  Letting $j:X := \bar X \setminus D\to \bar X$ be the open inclusion, we think of $E'$ as a subsheaf $E'\subset E'(*D)\subset j_*(\cO_{X^\an}\otimes_{\bC_{X^\an}}V')$ of the meromorphic extension (as a $A'\otimes_\bC\cO_{\bar X}$-module), which is uniquely determined by $V'$.  It is straightforward to show that the set of lifts $(E',\nabla')$ of $V'$ lying above a fixed lift of $(E,\nabla)$ of $V$ is naturally a torsor under global sections of the kernel of
    \[\begin{tikzcd}
    \cHom_{\cO_{\bar X}}(E_0,E_0(*D)/E_0)\ar[r,"\ad\nabla"]&\cHom_{\cO_{\bar X}}(E_0,E_0(*D)/E_0)\otimes\Omega_{\bar X}(\log D)
    \end{tikzcd}\]
    where $(E_0,\nabla_0)$ is the closed point underlying $(E',\nabla')$.  Locally near a smooth point of $D$ where $q$ is a local defining equation for $D$, a section of the kernel is represented as a principal part $\sum_{i\geq -N} q^{-i}f_i$ for $f_i$ a section of $\cEnd_{\cO_{\bar X}}(E_0)$ with $f_{-N}\neq 0$.  Further, we have
    \begin{align*}0=(\ad\nabla)\sum_{i\geq -N} q^{-i}f_i &=\sum_{i\geq-N} \left(-iq^{-i}f_i\otimes \frac{dq}{q}+q^{-i}(\ad\nabla) f_i\right)\\
    &=(-N+\Res\ad\nabla)f_i\mod (q^{1-N})
    \end{align*}
    so $\Res\ad\nabla$ must have $N$ as an eigenvalue, and therefore two of the eigenvalues of $\Res\nabla$ differ by $N$, which is a contradiction.  Thus, lifts are unique if they exist.  
    
    Now, for any $\bC$-local system $V_0$, any lift $(E_0,\nabla_0)$ to the good locus is obtained via the Deligne construction with respect to some choice of fundamental domain for the exponential.  Moreover, for any artinian $A$ and $A$-local system $V$, lifts of $V$ with the fixed closed point $(E_0,\nabla_0)$ exist locally by the Deligne construction (which works just as well over $A$), and are unique by the above, hence glue to a global lift.  Thus, the map is a local isomorphism.
\end{proof}
Let $\cM_{DR}^\nilp(\bar X,D)\subset\cM_{DR}(\bar X,D)$ be the closed substack of logarithmic connections for which the residue of the connection is nilpotent.  Clearly, $\cM_{DR}^\nilp(\bar X,D)^\an\subset\cM_{DR}(\bar X,D)^{\good}$.  On the nilpotent substack Gieseker and slope (semi)stability are equivalent (with respect to any polarization) and semistability is automatic since the rational Chern classes of any logarithmic connection with nilpotent residues vanish \cite[Appendix B]{Esnault-Viehweg}.  Moreover every point is GIT-semistable by \cite[Theorem 4.10]{SimpsonmoduliI}.  Thus we may form the good moduli space $M_{DR}^\nilp(\bar X,D)$.

  Denote by $\cM_B^\unip(X)\subset \cM_B(X)$ the closed substack of local systems with unipotent local monodromy, and $M_B^\unip(X)$ the good moduli space.  We also define the corresponding notions for the framed moduli space $R_{DR}^\nilp(\bar X, D,x)$ and $R_B^\unip(X,x)$.
\begin{cor}\label{BettiDeRhamcomparison}
    The restriction of the Riemann--Hilbert morphism gives isomorphisms $RH^\nilp_{(\bar X,D)}:\cM_{DR}^\nilp(\bar X,D)^\an\to \cM_B^\unip(X)^\an$, $RH^\nilp_{(\bar X,D)}:R_{DR}^\nilp(\bar X,D,x)^\an\to R^\unip(X,x)^\an$, and $RH^\nilp_{(\bar X,D)}:M_{DR}^\nilp(\bar X,D)^\an\to M_B^\unip(X)^\an$.
\end{cor}
\begin{proof}
    The claim for the stacks follows from \Cref{local iso on good} and the fact that the map is bijective on points.  The claim for the framed moduli spaces follows similarly.  By \cite[Proposition 5.5]{SimpsonmoduliI}, the analytifications of the good moduli spaces are universal categorical quotient in the category of complex analytic spaces of the respective framed moduli spaces, from which the last claims follows.
\end{proof}

\subsection{Background on tame harmonic bundles}\label{sect:harmonic bundles}
We recall the most important elements of the theory of tame harmonic bundles, due to Corlette \cite{Corlette_JDG},  Simpson in the proper case \cite{simpsonhiggs} and the case of non-proper curves \cite{Simpson_noncompact} and Mochizuki \cite{mochizukitame,mochizukitameii} in general.

Let $\cE$ be a $\cC^\infty$-complex vector bundle on a complex manifold $X$ equipped with a flat connection $\nabla$. The choice of a smooth hermitian metric $h$ on $\cE$
induces a canonical decomposition $\nabla = \nabla_u + \Psi$, where $\nabla_u$ is a unitary connection on $\cE$ with respect to $h$ and $\Psi$ is self-adjoint for $h$. Both decompose in turn in
their components of type $(1, 0)$ and $(0, 1)$: $\nabla_u = \partial_\cE + \bar \partial_\cE$, $\Psi = \theta +  \theta^\ast$.
By definition, the metric $h$ is pluriharmonic if the operator $ \bar \partial_\cE + \theta$ is integrable, i.e. if the differentiable form $(\bar \partial_\cE + \theta)^2 \in \cA^2(\End(\cE))$ is zero.

\begin{defn}
A harmonic bundle $(\cE, \nabla, h)$ (or equivalently $(\cE,\bar\partial_\cE,\theta,h)$) on a complex manifold $X$ is the data of a $\cC^\infty$-complex vector bundle $\cE$ equipped with a flat connection $\nabla$ and a pluriharmonic metric $h$.
\end{defn}
If $(\cE, \nabla, h)$ is a harmonic bundle, then the holomorphic bundle $E^{Dol} := (\cE, \bar \partial_\cE)$ equipped with the one-form $\theta \in \cA^1(\End(\cE))$ defines a Higgs bundle. By definition, this means that $\theta$ is a holomorphic one-form with values in $\End(E^{Dol})$ that satisfies $\theta \wedge \theta = 0 $.

Let $D$ be a normal crossing divisor in a smooth projective complex algebraic variety $\bar X$. A harmonic bundle $(\cE, \nabla, h)$ on $X := \bar X \setminus D$ is called tame if the associated Higgs bundle $(E^{Dol}, \theta)$ on $X$ is the restriction of a logarithmic Higgs bundle $(\bar E, \theta)$ on $(\bar X, D)$. (It is sufficient that the coefficients of the characteristic polynomial of the Higgs field $\theta$ extends as logarithmic holomorphic symmetric forms.) A tame harmonic bundle $(\cE, \nabla, h)$ on $X$ is purely imaginary (resp. nilpotent) if the eigenvalues of the residues of $\theta$ in a logarithmic extension $(\bar E, \theta)$ of $(E^{Dol}, \theta)$ are purely imaginary (resp. zero). One easily checks that these definitions do not depend on the choice of the log-compactification $(\bar X, D)$ and of the extension $(\bar E, \theta)$.

In general the correspondence between flat bundles and polystable Higgs bundles involves parabolic structures on both sides.  The details will not be important for us (see \cite{brunebarbesemipos} for a more in-depth discussion), as there are canonical choices when the residues of the operators are nilpotent, but it is useful to recall the general idea.  A parabolic sheaf on $(\bar X,D)$ extending an algebraic sheaf $E$ on $X$ is a $\bR^{\pi_0(D^\reg)}$-indexed decreasing filtration $\bar E^\bullet$ of the associated meromorphic bundle $j_*E$ on $\bar X$ by coherent subsheaves satisfying a semicontinuity condition and such that $\bar E^{\alpha+e_i}=\bar E^\alpha(-D_i)$.  A parabolic bundle is a parabolic sheaf which is Zariski-locally isomorphic to a direct sum of parabolic line bundles (i.e. parabolic sheaves which
are locally-free of rank 1).  A locally free coherent extension $\bar E$ determines a parabolic bundle extending $E$ by setting $\bar E^\alpha:=E(\sum -\lfloor\alpha_i\rfloor D_i)$; such a parabolic bundle is said to be trivial.  A tame pluriharmonic metric naturally induces a parabolic extension---the moderate growth extension---of the associated flat and Higgs bundles according to order of growth of the norm.  Finally, a local system admits a natural parabolic bundle extension---called the Deligne--Manin parabolic extension---via the Deligne construction.    
\begin{thm}[{Mochizuki \cite[Theorem 1.1]{mochizukitameii}}]\label{existence_of_tame_purely_imaginary_harmonic_metrics}Let $(\bar X,D)$ be a projective log smooth variety with ample bundle $L$ and set $X=\bar X\setminus D$. 
\begin{enumerate}
\item A flat filtered regular $\lambda$-connection bundle $(\bar E^\bullet,D)$ is $\mu_L$-polystable with vanishing rational parabolic chern classes if and only if its restriction to $X$ admits a tame pluriharmonic metric for which the moderate growth parabolic extension agrees with $\bar E^\bullet$.  The metric is unique up to flat automorphisms.
\item For a tame harmonic bundle $(E,\nabla,h)$ on $X$, the following are true:
\begin{enumerate}
    \item The parabolic structure on the associated filtered regular flat bundle is the Deligne--Manin extension if and only if the tame harmonic bundle is purely imaginary.
    \item The parabolic structure on the associated filtered regular Higgs bundle is trivial if and only if the eigenvalues of the residues of the connection in the De Rham realization have integral real part, or equivalently if the local monodromy in the Betti realization has purely positive real eigenvalues.
\end{enumerate}
\end{enumerate}
\end{thm}

\begin{cor}\label{DolDRcorrespondence}
    Fix $(\bar X,D)$ and $L$ as above.  There is an equivalence of categories via purely imaginary tame harmonic bundles with unipotent local monodromy between semisimple logarithmic flat vector bundles with nilpotent residues and $\mu_L$-polystable logarithmic Higgs bundles on $\bar X$ with vanishing rational chern classes and nilpotent residues.
\end{cor}

If $f \colon X \to Y$ is an algebraic morphism between two smooth quasiprojective complex varieties, then the pull-back of a tame (resp. tame purely imaginary) harmonic bundle on $Y$ is a tame (resp. tame purely imaginary) harmonic bundle on $X$, see \cite[Lemma 25.29]{mochizukitameii}. In particular, one gets the following result.

\begin{thm}[Corlette, Mochizuki {\cite[Theorem 25.30]{Mochizuki-AMS2}}]\label{pullback ss}
Let $f \colon X \to Y$ be an algebraic morphism between two smooth algebraic spaces. If $V$ is a semisimple complex local system on $Y$, then its pull-back $f^*V$ is a semisimple complex local system on $X$.
\end{thm}
\begin{proof}
    The usual statement is for quasiprojective varieties.  In general, there is an affine subspace of $X$ mapping to an affine subspace of $Y$ \cite[\href{https://stacks.math.columbia.edu/tag/06NH}{Tag 06NH}]{stacks-project}, and we apply the result to them.
\end{proof}

\begin{rem}
Let $X$ be a smooth quasiprojective complex variety.  If $(E, \nabla)$ is a semisimple flat bundle on $X$, then thanks to Theorem \ref{existence_of_tame_purely_imaginary_harmonic_metrics} every choice of a tame purely imaginary pluriharmonic metric $h$ on $(E, \nabla)$ will yield the same (algebraic) Higgs bundle on $X$. 
\end{rem}

\begin{prop}[See  (21.19) in \cite{mochizukitameii}]\label{energy_versus_Higgs_field}
Let $(E, \theta, h)$ be a harmonic bundle on a Kähler manifold $X$. Then the following equality between the energy density $e^u$ of the corresponding twisted pluriharmonic map $u$ and the norm of the Higgs field holds: 
\[8 \cdot |\theta|^2 = e^u. \]
\end{prop}

\begin{prop}\label{unipotent_residues_equivalent_finite_energy}
Let $(\bar X,D)$ be a projective log smooth variety and set $X=\bar X\setminus D$. Endow $X$ with a Poincaré-type Kähler metric. Let $(E, \theta, h)$ be a purely imaginary tame harmonic bundle on $X$. Then the following properties are equivalent:
\begin{enumerate}
    \item the eigenvalues of the local monodromy have modulus one,
    \item the residues of the Higgs field are nilpotent,
    \item $(E, \theta, h)$ has finite energy.
\end{enumerate}
\end{prop}
We say that a harmonic bundle on a Kähler manifold $X$ has finite energy if the associated twisted pluriharmonic map has finite energy (of course this depends on the choice of the metric on $X$).

\begin{proof}
The equivalence between $(1)$ and $(2)$ is a consequence of the table in \cite[p.720]{Simpson_noncompact}. Assume that the residues of the Higgs field are nilpotent, so that $(E, \theta, h)$ is a nilpotent tame harmonic bundle. By \cite[Theorem 1]{Simpson_noncompact} and \cite[Corollary 4.1]{Mochizuki-JDG}, the norm of $\theta$ is bounded with respect to the Poincaré-type metric on $X$, so that by \Cref{energy_versus_Higgs_field} $(E, \theta, h)$  has finite energy. 
The implication $(3) \implies (2)$ can be proved with arguments similar to those appearing in the proof of \Cref{extended_characteristic_polynomial} below. Since this implication is not used in this paper, we leave the details to the interested reader.
\end{proof}


\subsection{The nilpotent residue De Rham--Dolbeault comparison}\label{section on comparison}
Assume $(\bar X,D)$ is a projective log smooth variety.  Recall that $\cM^\nilp_{Dol}(\bar X,D)\subset \cM_{Dol}(\bar X,D)$ is the closed substack of semistable logarithmic Higgs bundles whose Higgs fields has nilpotent residues and $M_{Dol}^\nilp(\bar X,D)$ is its good moduli space, whose points correspond to polystable logarithmic Higgs bundles with nilpotent residues. According to \Cref{DolDRcorrespondence}, solving for the harmonic metric yields a bijective map of sets
\begin{equation}\label{SM corr}SM^{\nilp}_{(\bar X,D)}:M_{Dol}^\nilp(\bar X,D)(\bC)\to M_{DR}^\nilp(\bar X,D)(\bC)\end{equation}
which is functorial with respect to pull-back along algebraic morphisms.

The main result of this subsection is the following:
\begin{thm}\label{comparison}For a projective log smooth curve $(\bar X,D)$, the comparison $SM^{\nilp}_{(\bar X,D)}$ is a homeomorphism in the euclidean topology.
    
\end{thm}

\begin{cor}\label{contDolDR}  For any projective log smooth variety $(\bar X,D)$, $SM^\nilp_{(\bar X,D)}$ is a continuous bijection.  In particular, for any compact $K\subset M^\nilp_{Dol}(\bar X,D)(\bC)$ it induces a homeomorphism $K\to SM_{\bar X,D}^\nilp(K)$. 
\end{cor}
\begin{proof}[Proof of \Cref{contDolDR} assuming \Cref{comparison}]
   By an appropriate Lefschetz theorem (e.g. \cite[\S II.5.1]{stratmorse}), for a general sufficiently ample curve $C\subset X$, the inclusion $i:C\to X$ induces a surjection $i_*:\pi_1(C,x_0)\to \pi_1(X,x_0)$ for $x_0\in C$.  The pull-back $i^*_B:\cM_B(X)\to\cM_B(C)$ is therefore a closed immersion, as is the coarse map $i^*_B:M_B(X)\to M_B(C)$; likewise for the unipotent local monodromy loci, and the nilpotent residues locus in the De Rham realization, by \Cref{BettiDeRhamcomparison}.  We have the commutative diagram:
\[\begin{tikzcd}
    M_{Dol}^\nilp(\bar X,D_X)\ar[r,"i^*_{Dol}"]\ar[d,"SM^\nilp_{(\bar X,D_X)}",swap]&M_{Dol}^\nilp(\bar C,D_C)\ar[d,"SM^\nilp_{(\bar C,D_C)}"]\\
    M_{DR}^\nilp(\bar X, D_X)\ar[r,"i^*_{DR}"]&M_{DR}^\nilp(\bar C,D_C).
\end{tikzcd}\]
Since $i_{DR}^*$ is a closed immersion and both vertical maps are bijective on the level of sets, it follows from the theorem that the left vertical map is continuous.
\end{proof}

The main step of the proof of \Cref{comparison} is the following:   
\begin{prop}\label{harmonic sequence}
Let $(\bar X,D)$ be a projective log smooth curve with basepoint $x\in X=\bar X\setminus D$.  Let $\mathbf{E}_i=(\cE_i,\bar\partial_i,\theta_i,h_i,\phi_i)$ be a sequence of framed rank $r$ tame nilpotent harmonic bundles on the curve $X$ with unipotent local monodromy such that the coefficients of the characteristic polynomial $P(\theta_i)$ are uniformly bounded (in the space of global log symmetric forms) and such that the framing $\phi_i:(V_0,h_0)\xrightarrow{\cong}(\cE_{i,x},h_{i,x})$ respects the metric.  Then there is a framed tame nilpotent harmonic bundle $(\cE_\infty,\bar \partial_\infty,\theta_\infty,h_\infty,\phi_\infty)$ with unipotent local monodromy and whose framing respects the metric such that, up to passing to a subsequence, the associated sequence of framed logarithmic flat bundles $(\bar E^{DR}_i,\nabla_i,\phi_i)$ (resp. framed polystable logarithmic Higgs bundles $(\bar E_i^{Dol},\theta_i,\phi_i)$) converges in $R_{DR}^\nilp(\bar X,D,x)$ (resp. $R_{Dol}^\nilp(\bar X,D,x)$) to $(\bar E_\infty^{DR},\nabla_\infty,\phi_\infty)$ (resp. $(\bar E_\infty^{Dol},\theta_\infty,\phi_\infty)$).
\end{prop}

Throughout we use the following notation.  For a complex manifold $M$, we say a complex-valued function on $M$ is $L^\infty_{loc}$ bounded if its restriction to any compact $K\subset M$ is $L^\infty$-bounded.  Likewise for $W^{1,2}_{loc}$.  We also use the following elementary fact.
\begin{lem}\label{cauchylemma}
    A family of smooth complex-valued functions $f$ on the disk $\bD$ for which both $f$ and $\bar\partial f$ are uniformly bounded in $L^\infty_{loc}$ has relatively compact restriction to $L^\infty(K)$ for any compact $K\subset\bD$.
\end{lem}
\begin{proof}
    See for example \cite[Chapter 3, Exercise 3.6]{dbarexercise}.  This follows from the generalized Cauchy inequality, which shows that such a family is locally uniformly H\"older continuous with exponent $1-\epsilon$ for any $0<\epsilon<1$, as well as the Arzel\`a--Ascoli theorem.
\end{proof}
\begin{proof}[Proof of \Cref{harmonic sequence}]  We begin with some general remarks.  Let $\Delta$ be the symmetric space of positive definite hermitian forms on a fixed rank $r$ vector space $V_0$ and choose a basepoint $h_0$ of $\Delta$.  Given a harmonic bundle $\mathbf{V}:=(\cV,\bar\partial,\theta,h)$ on a complex manifold $M$, a basepoint $m\in M$, and a metric framing $\psi:(V_0,h_0)\xrightarrow{\cong}(\cV_{m},h_{m})$, using the flat connection on $\mathbf{V}$ we obtain a pluriharmonic map $f_{\mathbf{V},\psi}:\tilde M\to \Delta$ which is equivariant with respect to the monodromy representation $\rho_{\mathbf{V},\psi}:\pi_1(M,m)\to\GL(V_0)$.
\vskip1em
\noindent\emph{Step 1.}\label{step 1}  Let $M$ be a simply-connected complex manifold equipped with a Riemannian metric, $m_0\in M$ a basepoint, and $\mathbf{V}_i=(\cV_i,\bar \partial_i,\theta_i,h_i)$ a sequence of harmonic bundles on $M$ with framings $\psi_{i}:(V_0,h_0)\xrightarrow{\cong}(\cV_{i,m},h_{m})$.  Assume the associated pluriharmonic maps $f_{\mathbf{V}_i,\psi_i}:M\to \Delta$ have locally uniformly bounded energy.  Then after passing to a subsequence, the $f_{\mathbf{V}_i,\psi_i}$ converge strongly in $W^{1,2}_{loc}$ and $L_{loc}^\infty$ to a (smooth) pluriharmonic map $f_{\infty}: M\to \Delta$.  The resulting metrics $h_i$ on the trivial $C^\infty$ bundle $C^\infty_M\otimes V_0$ converge strongly in $W^{1,2}_{loc}$ and $L^\infty_{loc}$ to $h_\infty$, and the operators $\bar\partial_i,\theta_i$ converge strongly in $L^\infty_{loc}$ to $\bar\partial_\infty,\theta_\infty$.
\begin{proof}
The statement about the convergence of the maps follows from \Cref{limit_of_harmonic_is_sometime_harmonic}.  As the harmonic metric is pulled back from $\Delta$, the convergence statements follow for $h_i$.  The harmonic metric $h$ uniquely determines the operators $\bar\partial,\theta$, since the connection form of $\partial+\bar\partial$ in the flat basis is $\frac{1}{2}h^{-T}dh$, which is also the matrix of $-(\theta+\theta^*)$.  Moreover, $\theta_i$ (and therefore also the difference $A_i:=\bar\partial_i-\bar\partial\otimes \id$) is uniformly $L^\infty_{loc}$ bounded since the coefficients of the $P(\theta_i)$ are uniformly bounded \cite[Lemma 2.7]{simpsonhiggs}.  Since $\bar \partial_i\theta_i=0$ we have $(\bar\partial\otimes\id) \theta_i=-[A_i,\theta_i]$, which is uniformly $L^\infty_{loc}$ bounded.  Thus, by \Cref{cauchylemma}, after passing to a further subsequence, the operators $\bar\partial_i,\theta_i$ converge strongly in $L^\infty_{loc}$ to $\bar\partial_\infty,\theta_\infty$.    
\end{proof}
We now return to the setting of the proposition.
\vskip1em
 \noindent\emph{Step 2.}\label{step 2} After passing to a subsequence of the framed harmonic bundles in the proposition, the monodromy representations $\rho_{\mathbf{E}_i,\phi_i}$  converge to a limit representation $\rho_{\infty}:\pi_1(X,x)\to\GL(V_0)$ and there is a smooth pluriharmonic $\rho_\infty$-equivariant map $f_\infty:\tilde X\to \Delta$ to which the $f_{\mathbf{E}_i,\phi_i}$ converges strongly in $W^{1,2}_{loc}$ and $L^\infty_{loc}$.  Moreover, the resulting harmonic bundle $(\cE_\infty,\bar\partial_\infty,\theta_\infty,h_\infty)$ is tame and purely imaginary.  Finally, for any relatively compact $u\in U\subset X$, there are identifications $\alpha_i:\cE_i|_U\to \cE_\infty|_U$ of $C^\infty_U $-bundles preserving the metric at $u$ such that, up to passing to a subsequence, $\alpha_{i*}h_i$ (resp. $\alpha_{i*}\bar\partial_i$ resp. $\alpha_{i*}\theta_i$) converge strongly to $h_\infty$ (resp. $\bar\partial_\infty$ resp. $\theta_\infty$) in $W^{1,2}_{loc}$ and $L^\infty_{loc}$ (resp. $L^\infty_{loc}$ resp. $L^\infty_{loc}$).
\begin{proof}
    As in the previous step, by \Cref{energy_versus_Higgs_field} the energies of the maps $f_{\mathbf{E}_i,\phi_i}$ are locally uniformly bounded.  Choose loops $\gamma_1.\ldots,\gamma_n$ generating $\pi_1(X,x)$.  Applying Step 1 to a finite cover of each $\gamma_i$ by simply-connected relatively compact open subsets $x\in U\subset X$, it follows that, after passing to a subsequence, the monodromy operators $\rho_{\mathbf{E}_i,x}(\gamma_j)$ are constrained to lie in a compact subset of $\GL(V_0)$, and therefore after passing to a further subsequence it follows that the limit representation $\rho_\infty$ exists.  For any $\gamma\in\pi_1(X,x)$, taking a finite cover of $X$ by simply connected $x\in U\subset X$, we have that $f_i|_{\gamma\tilde U}=\rho_{\mathbf{E}_i,x}(\gamma)\circ f_i|_{\tilde U}$.  It follows that $\rho_\infty(\gamma)\circ f_\infty|_{\tilde U} $ agrees with $f_\infty|_{\gamma\tilde U}$ on overlaps, and since the pointwise limit of the $f_i$ is unique if it exists, the first claim follows from Step 1.  The boundedness of the characteristic polynomials implies the second claim.  For the last claim, we may glue together local identifications of the $C^\infty$ bundles via flat sections on finitely many relatively compact simply-connected open sets using a fixed partition of unity, and the metrics and operators will have the same convergence properties.
    \end{proof}

Let $\bD^*$ be a disk neighborhood of a puncture of $X$ with coordinate $q$ and basepoint $u\in\bD^*$.  Choose framings $\chi_{i}:(V_0,h_0)\xrightarrow{\cong}(\cE_{i,u},h_{i,u})$, $\chi_{\infty}:(V_0,h_0)\xrightarrow{\cong}(\cE_{\infty,u},h_{\infty,u})$ be the induced framings obtained via flat transport.  It follows from Step 2 that the associated harmonic maps $f_{\mathbf{E}_i|_{\bD^*},\chi_i}:\tilde\bD^*\to \Delta$ converge in $W^{1,2}_{loc}$ and $L^\infty_{loc}$ to $f_{\mathbf{E}_\infty|_{\bD^*},\chi_\infty}$. 

Let $V:=\cO_\bD\otimes V_0$ be the trivial holomorphic bundle.  We have a canonical identification between $V$ and the Deligne extension $\bar{E}^{DR}_i|_{\bD}$ (resp. $\bar{E}_\infty^{DR}|_{\bD}$) of the flat bundle underlying each $\mathbf{E}_i|_{\bD^*}$ (resp. $\mathbf{E}_\infty|_{\bD^*}$) by identifying, for each $v\in V_0$, the constant section $1\otimes v$ of $V$ with $e^{zN_i}\tilde v$ (resp. $e^{zN_\infty}\tilde v$), where $\tilde v$ is the flat continuation of $\chi_i(v)$ (resp. $\chi_\infty(v)$), $z=\log q$ and $N_i,N_\infty$ are the nilpotent local monodromy logarithms.  Note that the $N_i$ converge to $N_\infty$.  Let $h_i^{V},\bar\partial_i^{V},\theta_i^{V}$ be the transport of the metrics and operators to $C^\infty_{\bD^*}\otimes V_0$ via this identification.  Note that these operators converge in the same way to their limits on $C^\infty_{\bD^*}\otimes V_0$, since they differ from the flat trivialization by $e^{zN_i}$.

According to \cite[Prop. 7.4]{mochizukitame} (taking $U_0=\{0\}\subset(a-1,a)$ in the notation therein), there is a $\bar\partial^{V}_i$-holomorphic frame $\mathbf{e}_i$ of $C_{\bD^*}\otimes V_0$ with the property that:
\begin{enumerate}
\item Under the above identification, $\mathbf{e}_i$ is a holomorphic frame of $\bar E_i^{Dol}|_\bD$.
\item $h_i^V(\mathbf{e}_i(j))\leq C (\Im z)^k$ for some $k$ and some fixed constant $C>0$.
\item $C'^{-1}\leq h_i^V(\det \mathbf{e}_i)\leq C'$ for some fixed constant $C'>0$, where $\det\mathbf{e}_i=\mathbf{e}_i(1)\wedge\cdots\wedge \mathbf{e}_i(r)$.
\end{enumerate}
    
    \vskip1em
    \noindent\emph{Step 3.}\label{step 3}  There is a $\bar\partial^V_\infty$-holomorphic frame $\mathbf{e}_\infty$ of $C^\infty_{\bD^*}\otimes V_0$ to which the frames $\mathbf{e}_i$ converge strongly in $L^\infty_{loc}$ after passing to a subsequence.  The limit frame $\mathbf{e}_\infty$ satisfies properties (1)-(3) above (with $i=\infty$).  
\begin{proof}The frames $\mathbf{e}_i$ are uniformly bounded in $L^\infty_{loc}$, hence converge weakly to an $r$-tuple of sections $\mathbf{e}_\infty$ in $L^q_{loc}$ for all $q<\infty$.  Since $\bar\partial_i^V$ converges strongly in $L_{loc}^\infty$ to $\bar\partial_\infty^V$, it follows that $\bar\partial_\infty^V\mathbf{e}_\infty=0$ weakly, so by $\bar\partial$-regularity it follows that $\mathbf{e}_\infty$ is in fact comprised of $\bar\partial^V_
\infty$-holomorphic sections.  Let $g_i$ be the change of basis matrix from $\mathbf{e_i}$ to $\mathbf{e}_\infty$, which is $L^\infty_{loc}$ bounded.  
\begin{claim}
    The $g_i$ converge strongly in $L^\infty_{loc}$ to $\id$ after passing to a subsequence.
\end{claim}
\begin{proof}
    Let $\bar\partial^V_i=\bar\partial^V_\infty+B_i$, so $B_i\to 0$ strongly in $L^\infty_{loc}$; we have  $\bar\partial_\infty^V g_i=-g_iB_i$ which is uniformly $L^\infty_{loc}$ bounded. Thus, by \Cref{cauchylemma} the $g_i$ converge strongly in $L^\infty_{loc}$ after passing to a subsequence.
\end{proof}
Thus, the norms $h_\infty^V(\mathbf{e}_\infty(j))$ satisfy inequalities (2) and (3).  In particular, $h_\infty^V(\det\mathbf{e}_\infty)$ is nowhere vanishing on $\bD^*$, and neither vanishes nor goes to infinity faster than a power of $\Im z$ at the puncture.  The extension $\bar E_\infty^{Dol}$ is the moderate growth extension with respect to $h_\infty$, so property (1) follows.
\end{proof}
By mapping the $\mathbf{e}_i$ frame to $\mathbf{e}_\infty$, we obtain holomorphic identifications $\beta_i^{Dol}:\bar E_i^{Dol}|_{\bD}\to \bar E_\infty^{Dol}|_\bD$.  From Steps 2 and 3, it follows that $\beta^{Dol}_{i*}h_i$ (resp. $\beta^{Dol}_{i*}\theta_i$) converges to $h_\infty$ (resp. $\theta_\infty$) strongly in $L^\infty_{loc}$.
        \vskip1em
    \noindent\emph{Step 4.}\label{step 4}  With respect to any metric on $\bar E_\infty^{Dol}|_\bD$, $\beta_{i*}\theta_i$ converges to $\theta_\infty$ in $L^\infty(\End(\bar E_\infty^{Dol}|_\bD)\otimes\omega_\bD(0))$, possibly after shrinking $\bD$.
\vskip1em
\begin{proof}
    The $\beta_{i*}^{Dol}\theta_i$ are a sequence of sections of a holomorphic vector bundle which converge strongly in $L^\infty_{loc}$ to a holomorphic section.  The Cauchy integral formula implies both uniform boundedness and equicontinuity on $\bD$, hence uniform convergence.
\end{proof}
    \noindent\emph{Step 5.}\label{step 5}  End of proof.
\vskip1em
Using a fixed partition of unity, we may glue together the identifications $\alpha_i:\cE_i|_U\to  \cE_\infty|_U$ of the $C^\infty$ bundles of Step 2 on the complement $U$ of small disk neighborhoods of the punctures with the holomorphic identifications $\beta_i^{Dol}:\bar E_i^{Dol}|_\bD\to \bar E_\infty^{Dol}|_\bD$ from Step 4 to obtain identifications $\eta_i:C^\infty_{\bar X}\otimes \bar E_i^{Dol}\to C^\infty_{\bar X}\otimes \bar E_\infty^{Dol}$ which preserve the metric at a chosen basepoint $x$.  We will therefore consider the operators $\bar\partial_i,\theta_i$ on a fixed $C^\infty_{\bar X}$ bundle $\bar \cE$; from Steps 2 and 4 we deduce that they both converge to $\bar\partial_\infty,\theta_\infty$ in $L^\infty$ with respect to any choice of metric, which in particular implies they converge in operator norm as operators $W^{1,\infty}(\bar E^{Dol}_\infty)\to L^\infty(\bar E^{Dol}_\infty\otimes\omega_{\bar X}(D))$.  According to \cite[Lemma 5.12]{SimpsonmoduliI}, this implies the corresponding sequence of framed polystable logarithmic Higgs bundles $(\bar E_i^{Dol},\theta_i,\phi_i)$ converges to $(\bar E_\infty^{Dol},\theta_\infty,\phi_\infty)$.  The convergence in the De Rham realization follows from the convergence of the monodromy representations, by Riemann--Hilbert.    
\end{proof}
Let $(V_0,h_0)$ be a fixed vector space with a hermitian form of the relevant rank.  As in \cite[\S7]{SimpsonmoduliII}, we denote by $R_{DR}^{\nilp}(\bar X,D,x)(\bC)^{(V_0,h_0)}\subset R_{DR}^\nilp(\bar X,D,x)(\bC)$ (resp. $R_{Dol}^{\nilp}(\bar X,D,x)(\bC)^{(V_0,h_0)}\subset R_{Dol}^\nilp(\bar X,D,x)(\bC)$) the subspace of semisimple logarithmic connections with nilpotent residues (resp. polystable logarithmic Higgs bundles with nilpotent residues and trivial Chern class) which are framed at $x$ and which admit a tame purely imaginary harmonic metric which is identified with $h_0$ at $x$ via the framing.  

As in \cite[\S7]{SimpsonmoduliII} we deduce from \Cref{harmonic sequence}:
\begin{cor}\label{framed metric comparison}
The correspondence 
    \[RSM^{\nilp,(V_0,h_0)}_{(\bar X,D)}:R_{Dol}^{\nilp}(\bar X,D,x)(\bC)^{(V_0,h_0)}\to R_{DR}^{\nilp}(\bar X,D,x)(\bC)^{(V_0,h_0)}\]
    is a homeomorphism.
\end{cor}
\begin{proof}
To show a bijection $f:X\to Y$ of metric spaces is continuous, it suffices to show any convergent sequence in the source $x_i\to x_\infty$ has a subsequence such that $f(x_i)\to f(x_\infty)$.  For any sequence of points $(\bar E_i^{Dol},\theta_i,\phi_i)$ in $R_{Dol}^{\nilp}(\bar X,D,x)(\bC)^{(V_0,h_0)}$ converging to $ (\bar E_\infty^{Dol},\theta_\infty,\phi_\infty)$ in $R_{Dol}^\nilp(\bar X,D,x)(\bC)$, the image in $M_{Dol}^\nilp(\bar X,D)$ converges, hence the corresponding sequence of framed harmonic bundles satisfies the condition of the proposition, and the images $RSM^\nilp_{\bar X,D}(\bar E_i^{Dol},\theta_i,\phi_i)$ converge to $RSM^\nilp_{\bar X,D}(\bar E_\infty^{Dol},\theta_\infty,\phi_\infty)$ up to passing to a subsequence.  This shows continuity of $SM^\nilp_{\bar X,D}$.  

 To show continuity of the inverse, observe that one can construct harmonic maps by minimizing the energy (see Theorem \ref{existence_harmonic_finite_energy}).  Since for $K\subset R_B^{\unip}(X,x,r)$ compact there are $\rho$-equivariant maps $f_\rho:\tilde X\to \Delta$ with uniformly bounded energy for all $\rho\in K$, for example by the construction of an initial metric as in \cite[section 2.4]{Jost-Yang-Zuo}, it follows that a given convergent sequence of points in the target $(\bar E_i^{DR},\nabla_i,\phi_i)\to (\bar E_\infty^{DR},\nabla_\infty,\phi_\infty)$ can be lifted to a sequence of harmonic bundles satisfying the conditions of the proposition, hence converges in the Dolbeault realization. 
\end{proof}

\begin{cor}\label{framed stuff is proper}
$R_{DR}^{\nilp}(\bar X,D,x)(\bC)^{(V_0,h_0)}\to M_{DR}^\nilp(\bar X,D)(\bC)$ and $R_{Dol}^{\nilp}(\bar X,D,x)(\bC)^{(V_0,h_0)}\to M_{Dol}^\nilp(\bar X,D)(\bC)$ are proper.

\end{cor}
\begin{proof}
    The Dolbeault part is immediate from \Cref{harmonic sequence}. 
 For the De Rham part, according to the Kempf--Ness theorem (specifically in the form \cite[(4.7) Corollary]{schwarzgroup}), there is a subset $S\subset R_B(X,x,r)(\bC)$ such that $S\to M_B(X,r)(\bC)$ is proper with respect to the euclidean topology.  Thus, by lifting each simple factor and passing to a subsequence there is a convergent lift $(\bar E_i^{DR},\nabla,\phi'_i)$ of $(\bar E_i^{DR},\nabla_i)$ to $R_{DR}^{\nilp}(\bar X,D,x)(\bC)^\ss$.  By the argument of the previous corollary, the sequence of tame nilpotent harmonic bundles associated to $(\bar E_i^{DR},\nabla,\phi'_i)$ has uniformly bounded energy, and the same will be true after modifying the framings to get a lift $(\bar E_i^{DR},\nabla,\phi_i)$ to $R_{DR}^{\nilp}(\bar X,D,x)(\bC)^{(V_0,h_0)}$.  Passing to a subsequence, there is then a limit by the proposition.
\end{proof}

\begin{proof}[{Proof of \Cref{comparison}}]The previous two corollaries imply that any convergent sequence on either side of \eqref{SM corr} corresponds to a convergent sequence on the other side after passing to a subsequence.  
\end{proof}

\subsection{The $\bR_{>0}$-action on the Betti realization}\label{sect:R* on Betti}

Let $(\bar X,D)$ be a projective log smooth variety and set $X=\bar X\setminus D$. Let $L$ be an ample line bundle on $\bar X$. If $(E_\ast, \theta)$ is a $\mu_L$-polystable regular filtered Higgs bundle on $(\bar X,D)$ with vanishing first and second rational parabolic Chern classes, then $(E_\ast, t \cdot \theta)$ is a $\mu_L$-polystable regular filtered Higgs bundle on $(X,D)$ with vanishing first and second rational parabolic Chern classes for every $t \in \bC^\ast$. Moreover, if $t \in \bR_{>0}$, $(E_\ast, \theta)$ is purely imaginary if and only if $(E_\ast, t \cdot \theta)$ is purely imaginary. Therefore, using the correspondence between semisimple complex local systems on $X$ and purely imaginary $\mu_L$-polystable regular filtered Higgs bundle on $(\bar X,D)$ with vanishing first and second rational parabolic Chern classes (which follows from \Cref{existence_of_tame_purely_imaginary_harmonic_metrics}), we get a set-theoretic action of $\bR_{>0}$ on the points of the good moduli space $M_B(X)(\bC)$.  Moreover, this action extends to a $\bC^*$-action on $M_B^\unip(X)(\bC)$ by \Cref{DolDRcorrespondence}.  

\begin{lem}[Simpson, Mochizuki]\label{Cstar_fixed_points}
    A semisimple complex local system $V$ underlies a complex variation of pure Hodge structures if and only if the corresponding point of $M_B(X)(\bC)$ is fixed by $\bR_{>0}$ (or any infinite subgroup thereof).
\end{lem}
\begin{proof}
    The same argument as in \cite[Theorem 8]{Simpson_noncompact} works given the correspondence \Cref{existence_of_tame_purely_imaginary_harmonic_metrics}.
\end{proof}

\begin{rem}\label{C* families}
    For a $\mu_L$-polystable regular filtered Higgs bundle $(E_*,\theta)$ (with no condition on the residues) we can still associate a local system to $(E_*, t\cdot \theta)$ for $t\in\bC^*$.  Choosing a hermitian metric $h_0$ on a framing space $V_0$ and framing compatibly with the harmonic metric, Mochizuki proves \cite[Theorem 10.1]{mochizukitame} that the resulting map $\bC^*\to U(V_0,h_0)\backslash R_B(X,x)(\bC)^{(V_0,h_0)}$ (which is independent of choices) is continuous.  In particular, the map to $M_B(X)(\bC)$ is continuous.  Moreover, the image of $\bD^*$ is relatively compact \cite[Lemma 10.2]{mochizukitame}.  This does not yield an action on $M_B(X)(\bC)$ however, since the family depends on the choice of parabolic structure on the initial local system, and the resulting parabolic structure on the translate may not be the Deligne--Manin parabolic structure.
\end{rem}

It follows from the functoriality of the Simpson--Mochizuki correspondence that the $\bR_{>0}$ action (i) is independent of the choice of compactification, and therefore (ii) can be defined functorially for connected normal algebraic spaces:

\begin{lem}There is a unique extension of the $\bR_{>0}$-action to $M_B(X)(\bC)$ for any connected normal algebraic space $X$ which is functorial with respect to pullback and agrees with the above description for $X=\bar X\setminus D$ where $(\bar X,D)$ is projective log smooth.
\end{lem}
\begin{proof}
    Functoriality with respect to morphisms $f:X\to Y$ of smooth varieties is standard.  For $X$ a connected normal algebraic space, we must define the action on $M_B(X)$ via the embedding $M_B(X)\subset M_B(U)$ induced by restriction to any smooth affine $U\subset X$.  It remains to check compatibility with respect to pullback.  Let $f:X\to Y$ be a morphism of connected normal algebraic spaces, let $\pi:\tilde Y\to Y$ be a resolution.  There is a proper generically finite dominant morphism $g:\tilde X\to X$ with $\tilde X$ smooth such that $\tilde X\to Y$ lifts to $\tilde f:\tilde X\to \tilde Y$, and there is a dense open smooth subset $U\subset X$ for which the base-change $g_U:\tilde U\to U$ of $g$ is finite \'etale.  It suffices to check that in the pullback diagram on the right induced by the commutative diagram on the left, the bottom map must be $\bR_{>0}$-equivariant.
    \[\begin{tikzcd}
        \tilde U\ar[d,"g_U"]\ar[r,"\tilde f_U"]&\tilde Y\ar[d,"\pi"]&&M_B(\tilde U)&M_B(\tilde Y)\ar[l,"\tilde f_U^*",swap]\\
        U\ar[r,"f_U"]&Y&&M_B(U)\ar[u,"f_U^*"]&M_B(Y)\ar[u,"\pi^*",swap]\ar[l,"f_U^*",swap].
    \end{tikzcd}\]
    Since $M_B(U)\to M_B(\tilde U)$ is quasifinite and each $\bR_{>0}$-orbit is continuous, it follows that the inverse image of any $\bR_{>0}$-orbit in $M_B(\tilde U)$ is a disjoint union of $\bR_{>0}$-orbits in $M_B(U)$ each of which maps isomorphically and equivariantly onto its image.  For $V\in M_B(Y)(\bC)$, the image $f_U^*\bR_{>0}V$ is thus one of the components of the preimage of $\tilde f_U^*\pi^*\bR_{>0}V$, hence $f_U^*$ is equivariant.
\end{proof}

The similarity with the compact case is restored if we restrict our attention to local systems with unipotent local monodromy using \Cref{comparison}:  for a projective log smooth curve $(\bar X,D)$, we obtain a continuous action of $\bC^\ast$ on $M_B^{\unip}(X)$.  Our main use of \Cref{comparison} will be for the following:   
\begin{thm}\label{R* stable has fixed}
Let $X$ be a connected normal algebraic space.  Let $\Sigma$ be a $\bR_{>0}$-stable $\bC$-constructible subset of $M_B^\unip(X)(\bC)$. Then
\begin{enumerate}
\item $\overline{\Sigma}$ is $\bR_{>0}$-stable.
    \item Each irreducible component of $\overline{\Sigma}$ is $\bR_{>0}$-stable.
    \item For any $V\in \overline{\Sigma}$, the $\bR_{>0}$-orbit of $V$ completes to a continuous map $\bR_{\geq 0}\to \overline{\Sigma}$ and the image of 0 is a $\bC^*$-fixed point.
\end{enumerate}
  In particular, every irreducible component of $\Sigma$ contains a $\bC^\ast$-fixed point.
\end{thm}
Before the proof we establish some preliminary facts.
\begin{lem}
    Let $X$ be a reduced algebraic space and $X^{\mathrm{ball}}\subset X^\an$ the open subset of points of $X^\an$ with a neighborhood homeomorphic to a ball.  Then taking closure gives a bijection from connected components of $X^{\mathrm{ball}}$ to irreducible components of $X$.
\end{lem}
\begin{proof}We first claim that the inclusion $X^{\mathrm{reg}}\subset X^{\mathrm{ball}}$ of the regular locus induces a bijection on connected components.  Observe that the dimension of a ball neighborhood of a point is a locally constant function on $ X^{\mathrm{ball}}$, hence constant on connected components, so we may assume it is constant and equal to $n$.  Since $X^\mathrm{ball}$ can be given a countable locally finite finite-dimensional triangulation such that $X^{\mathrm{ball}}\setminus X^{\mathrm{reg}}$ is a subcomplex of dimension $\leq n-2$, from the long exact sequence in Borel--Moore homology we have
 \[0=H_{n}^{BM}(X^{\mathrm{ball}}\setminus X^{\mathrm{reg}})\to H_n^{BM}(X^{\mathrm{ball}})\to H_n^{BM}(X^{\mathrm{reg}})\to H_{n-1}^{BM}(X^{\mathrm{ball}}\setminus X^{\mathrm{reg}})=0\]
whence the claim.

To finish, it remains to observe that since $X^\mathrm{reg}$ is dense in $X^\mathrm{ball}$, each connected component of $X^\mathrm{reg}$ has the same closure as the corresponding connected component of $X^\mathrm{ball}$.  
\end{proof}
\begin{cor}\label{cor:homeomorphism_preserve_irreducible_components}
    Let $X$ be a reduced algebraic space. Let $f$ be a homeomorphism of the set $X(\bC)$ equipped with the euclidean topology. Then $f$ sends every irreducible component of $X(\bC)$ homeomorphically onto a (possibly different) irreducible component of $X(\bC)$.
\end{cor}

\begin{prop}\label{C* inv of comp}
Let $(\bar X,D)$ be a proper log smooth algebraic space and set $X=\bar X\setminus D$. Let $\Sigma$ be a $\bC^\ast$-stable (resp. $\bR_{>0}$-stable) Zariski closed subset of $M_B^\unip(X)(\bC)$. Then, every irreducible component of $\Sigma$ is $\bC^\ast$-stable  (resp. $\bR_{>0}$-stable).
\end{prop}
\begin{proof}By taking a projective modification, Lefschetz, and functoriality of the $\bC^*$-action, we may assume $X$ is a curve.  By \Cref{comparison}, the map $\bC^\ast \times M_B^\unip(X)(\bC) \to M_B^\unip(X)(\bC)$ defining the action of $\bC^\ast$ on $M_B^\unip(X)(\bC)$ is continuous. Therefore, the induced map $\bC^\ast \times \Sigma \to \Sigma$ is also continuous, so each $t\in\bC^*$ acts as a homeomorphism on $\Sigma$. By \Cref{cor:homeomorphism_preserve_irreducible_components}, $t$ sends each irreducible component of $\Sigma$ to an irreducible component of $\Sigma$. Therefore we get an action of $\bC^\ast$ on the set of irreducible components of $\Sigma$. Since the map $\bC^\ast \times \Sigma \to \Sigma$ is continuous, this action is locally constant. Since $\bC^\ast$ is connected and $1 \in \bC^\ast$ acts by the identity, we get that the action of $\bC^\ast$ stabilizes every irreducible component of $\Sigma$. The proof for the $\bR_{>0}$-stable case is the same.
\end{proof}

\begin{proof}[Proof of \Cref{R* stable has fixed}]
Again by taking a projective modification and Lefschetz we may assume $X$ is a curve.  Then (1) and (2) follow immediately from \Cref{comparison} and \Cref{C* inv of comp}.  For (3) we may then assume that $\Sigma$ is Zariski closed and irreducible.  Thanks to \Cref{BettiDeRhamcomparison} and \Cref{comparison}, the comparison map $M_{Dol}^{\nilp}(\bar X, D)(\bC)\to M_B^\unip(X)(\bC)$ is a homeomorphism. For any $(E, \theta) \in M_{Dol}^{\nilp}(\bar X, D)(\bC)$, the limit of $(E, t \cdot \theta)$ when $t$ goes to zero exists in $M_{Dol}^{\nilp}(\bar X, D)(\bC)$ by Proposition \ref{existence_of_Gm_limit} and it is a $\bC^\ast$ fixed point. If $(E, \theta)$ corresponds to a point in $\Sigma$, since $\Sigma$ is closed and $\bR_{>0}$-stable, the limit is a $\bC^\ast$-fixed point in $\Sigma$.    
\end{proof}

\begin{cor}\label{comp has R* fixed}
    Let $X$ be a connected normal algebraic space and $\Sigma\subset M_B(X)$ a nonempty $\bR_{>0}$-stable subset of semisimple local systems with quasiunipotent local monodromy.  Then $\overline{ \Sigma}$ contains a $\bR_{>0}$-fixed point.  
\end{cor}
\begin{proof}By the functoriality of the $\bR_{>0}$-action, we may assume $X$ is smooth and quasiprojective.  For any $V\in \Sigma$, there is a finite \'etale cover $\pi:X'\to X$ (see \Cref{from quasiunipotent to unipotent}) such that $\pi^*V\in \cM_B^{\unip}(X')(\bC)$.  The $\bR_{>0}$-orbit of $\pi^*V$ has a limit at $t\to 0$ in $\overline{\pi^*\Sigma}$, by \Cref{R* stable has fixed}.  By \cite[Lemma 10.2]{mochizukitame} (see \Cref{C* families}), the $\bR_{>0}$-orbit of $V$ has an accumulation point in $\overline{\Sigma}$ as $t\to 0$.  Since any lift of an $\bR_{>0}$-fixed point of $M_B^{\unip}(X')$ to $M_B^{\unip}(X)$ is $\bR_{>0}$-fixed (since variations of Hodge structures push forward under finite \'etale maps and factors of local systems underlying variations underlie variations \cite[Proposition 1.13]{delignefiniteness}), it follows that $\overline{\Sigma}$ contains a $\bR_{>0}$-fixed point.
\end{proof}


\section{The Deligne--Hitchin space}\label{DeligneHitchin sect}

In this section we use the Deligne--Hitchin space equipped with its $\bG_m$ to get around the lack of a $\bR_{>0}$-action on the Betti stack.  See the works of Simpson \cite{Simpson-Hodge-filtration,Simpsonrank2twistor,simpsonHitchinDeligne} for related discussions and results.  

\subsection{From quasiunipotent to unipotent}

\begin{prop}\label{injectivity_N_roots}
Let $R \subset \bC$ be a finitely generated $\bZ$-algebra. Let $N$ be a positive integer. Then there exists a maximal ideal $I \subset R$ such that the reduction morphism $R \to R \slash I$ is injective in restriction to the $N$-roots of unity.
\end{prop}
\begin{proof}
The integral closure $S$ of $\bZ$ in $R$ is a finitely generated $\bZ$-algebra which contains the roots of unity of $R$.  If such an ideal $J \subset S$ exists for $S$ then we may take $I$ to be any maximal ideal containing $J$.  Thus we may assume $R \subset \bar \bQ$. The fraction field $K \subset \bar \bQ$ of $R$ is a number field, and $R$ differs from $\cO_K$ only at finitely many places. Therefore one can assume from the beginning that $R = \cO_K$. If $a \in R \setminus \{1 \}$ satisfies $a-1 \in I$, then the norm of $I$ divides the norm of $1-a$. The result follows by letting $a$ being any $N$-root of unity.    
\end{proof}

\begin{prop}\label{from quasiunipotent to unipotent}
Let $X$ be a connected complex algebraic variety. Let $\rho \colon \pi_1(X) \to \GL(n, \bC)$ be a representation with quasi-unipotent local monodromy. Then there exists a finite étale cover $X^\prime \to X$ such that the corresponding finite index subgroup $\pi_1( X^\prime)$ of $\pi_1(X)$ contains $\ker \rho$ and $\rho_{|\pi_1( X^\prime)}$ has unipotent local monodromy.
\end{prop}
\begin{proof}
Since $\pi_1(X)$ is finitely generated, $\rho$ takes values in $\GL(n, R)$ for $R \subset \bC$ a finitely generated $\bZ$-algebra. Let $N$ be a positive integer such that the eigenvalues of the monodromies at infinity of $\rho$ are all $N$-roots of unity. Thanks to the Proposition \ref{injectivity_N_roots}, there exists a maximal ideal $I \subset R$ such that the reduction morphism $R \to R \slash I$ is injective in restriction to the $N$-roots of unity. Then one can take for $X^\prime \to X$ the finite étale cover corresponding to the kernel of the composed morphism $\pi_1(X) \to \GL(n, R) \to \GL(n, R / I) $.  
\end{proof}

\begin{prop}\label{reduction_to_unipotent_monodromy}
Let $X$ be a connected complex algebraic variety and $\Sigma \subset \cM_B(X)(\bC)$ a constructible subset containing only local systems with quasiunipotent local monodromy. Then there exists a finite étale cover $X^\prime \to X$ corresponding to a finite quotient of the image of the monodromy representation of $\bigoplus_{V \in \Sigma} V$
such that the image of $\Sigma$ in $\cM_B(X^\prime)(\bC)$ consists only of elements with unipotent (rather than quasiunipotent) local monodromy.
\end{prop}
\begin{proof}
It is sufficient to consider the case where $\Sigma$ is an irreducible locally closed subset of $\cM_B(X)(\bC)$. For every finite étale cover $X^\prime \to X$, consider the closed algebraic subset $\Sigma^\prime$ of $\Sigma$ consisting of the elements whose pull-back to $X^\prime$ have unipotent (rather than quasiunipotent) local monodromy. It follows from the assumption and \Cref{from quasiunipotent to unipotent} that the union of all those subsets is equal to $\Sigma$. It follows from Baire's category theorem that there exists a finite étale cover $X^\prime \to X$ such that $\Sigma^\prime$ has non-empty interior in $\Sigma$ (for the euclidean topology). Since $\Sigma$ is irreducible and $\Sigma^\prime$ is a closed algebraic subset of $\Sigma$, it follows that $\Sigma^\prime = \Sigma$. 
\end{proof}

\subsection{Construction of the Deligne--Hitchin space}
The construction here is largely based on Simpson's description \cite{Simpson-Hodge-filtration}; see \cite{Simpsonrank2twistor,simpsonHitchinDeligne} for some related discussion.

Let $(\bar X,D)$ be a connected log smooth algebraic space with $X=\bar X\setminus D$ and $\cM_{\Hod}(\bar X,D)$ the stack of logarithmic $\lambda$-connections.  Precisely, an $S$-point of $\cM_\Hod(\bar X,D)$ consists of a triple $(\lambda,\bar E,\lnabla)$ where $\lambda\in\cO_S(S)$ and $\bar E$ is a locally free $\cO_{\bar X\times S}$-module on $\bar X_S:=\bar X\times S$ equipped with a flat logarithmic $\lambda$-connection $\lnabla$---that is, an operator $\lnabla:E\to E\otimes\Omega_{\bar X_S/S}(\log D_S)$ satisfying  $\lnabla(fs)=\lambda s\otimes d_{\bar X}f+f\lnabla s$ and $\lnabla^2=0$.  As for the De Rham stack, $\cM_\Hod(\bar X,D)$ is a countably finite type algebraic stack (see \ref{subsect:DR stack} for the definition of this notion, and \cite[\S2]{Simpsonrank2twistor} for general references on representability).  There is a natural morphism $\lambda:\cM_\Hod(\bar X,D)\to\bA^1$, as well as a $\bG_m$-action on $\cM_{\Hod}(\bar X,D)$ by scaling the connection which covers the scaling action on $\bA^1$.  

It will often be simpler to consider the framed space $R_\Hod(\bar X,D,x)$ for $x\in X$, whose $S$-points are $S$-points $(\lambda,\bar E,\lnabla)$ of $\cM_\Hod(\bar X,D)$ together with an isomorphism $\phi:\bar E|_{x\times S}\xrightarrow{\cong} \cO_{S}^{\rk \bar E}$.  Then $R_\Hod(\bar X,D,x)$ is a countably finite type algebraic space which comes with a morphism $\lambda:R_\Hod(\bar X,D,x)\to\bA^1$ and a $\bG_m$-action.  Clearly, the rank $r$ substack $\cM_\Hod(\bar X,D,r)$ is identified with the quotient $[\bGL_r\backslash R_\Hod(\bar X,D,x,r)]$ of the framed rank $r$ space.

The symmetry offered by the $C^\infty$ perspective motivates the construction of the Deligne--Hitchin space, so we describe it informally.  An $S$-point of $\cM_\Hod(\bar X,D)^\an$ in particular gives a $\mathscr{A}_{\bar X_S}=C^\infty_{\bar X}\boxtimes \cO_{S}$-module $\bar\cE$ on $\bar X_S^\an$ together with a flat $\lambda\partial_{X^\an}+\bar\partial_{X^\an}$ connection $\mathscr{D}=\nabla+\bar\partial_{\bar E}$ on $\cE:= \bar\cE|_{X^\an}$, where $\bar \partial_{\bar E}$ is the holomorphic structure on $\bar E$.  Note that $\nabla=\mathscr{D}^{1,0}$ and $\bar\partial_{\bar E}=\mathscr{D}^{0,1}$.  These are the same types of operators that define a variation of mixed twistor structures, restricted to $\bP^1\setminus \infty$.  As realized by Simpson and Sabbah \cite{Simpson_noncompact,Sabbah_twistor_D_modules}, the eigenvalues of the residues of the $\lambda$-connection associated to a variation of mixed twistor structures behave in a regular way, although it depends on the parabolic structures as well.  For $\lambda\in\bC$, define the bijection (see \cite{Sabbah_twistor_D_modules}) $\frak{k}_\lambda:\bR\times\bC\to\bR\times\bC$ by
\begin{equation}\label{formula}\frak{k}_\lambda(a,\alpha)=(\frak{p}_\lambda(a,\alpha),\frak{e}_\lambda(a,\alpha))\hspace{.5in}\begin{cases}
   \frak{p}_\lambda(a,\alpha)&:=a+2\Re(\lambda\bar \alpha ) \\
   \frak{e}_\lambda(a,\alpha)&:=\alpha-a\lambda-\bar \alpha \lambda^2.
\end{cases}\end{equation}
For any $\lambda \in \bC$ the set of pairs $(a,\alpha)\in\bR\times\bC$ where $\alpha$ is an eigenvalue of the residue of the logarithmic extension occurring in the $a$th graded piece of the parabolic structure in the $\lambda$ specialization is called the KMS-spectrum at $\lambda$.  The KMS-spectrum at $\lambda$ is then the image under $\frak{k}_\lambda$ of the KMS-spectrum of the $\lambda=0$ specialization.

For $n\in\bN$, denote by $R_\Hod^{\qu|n}(\bar X,D,x)\subset R_\Hod(\bar X,D,x)$ the closed subspace of $\lambda$-connections with residues contained in $\lambda\cdot (\frac{1}{n}\bZ\cap (-1,0])\subset \bC$.  Let $R_\Hod^{\qun,\loc}(\bar X,D,x)$ be the germ of an open neighborhood of $R_\Hod^{\qun}(\bar X,D,x)^\an$ in $R_\Hod(\bar X,D,x)^\an$.  It is useful for example to keep in mind the open neighborhoods for which the $\lambda$-connection has residual eigenvalues in $\lambda\cdot B_\epsilon(\frac{1}{n}\bZ\cap (-1,0])$, where for $\Xi\subset \bC$ we let $B_\epsilon(\Xi)$ be the union of radius $\epsilon$ balls centered at points of $\Xi$ and $\epsilon<\frac{1}{2n}$ is sufficiently small. 

 Let $(\bar X^c,\bar D^c)$ be the complex conjugate variety.  There is a natural holomorphic isomorphism\footnote{The construction is usually described via an isomorphism $R_\Hod^{\qun,\loc}(\bar X,D,x)|_{\bG_m}\cong R_\Hod^{\qu|n,\loc}(\bar X,D,x)^c|_{\bG_m^c}$ given as $(\lambda,\cE,\scrD,\phi)\mapsto (-\bar{\lambda}^{-1},(\cE^\vee)^c,-\bar{\lambda}^{-1}(\scrD^{0,1\vee})^c+\bar\lambda^{-1} (\scrD^{1,0\vee})^c)$, but as $(\lambda,\cE,\scrD)\mapsto(-\bar\lambda,(\cE^\vee)^c,-(\scrD^{0,1\vee})^c+(\scrD^{1,0\vee})^c)$ gives an identification $R^{\qun,\loc}(\bar X^c,D^c,x)\xrightarrow{\cong}R^{\qun,\loc}(\bar X,D,x)^c$ these give the same space.  Note however with this description we would then take the residual eigenvalues to lie in a neighborhood of $\lambda\cdot \frac{1}{n}\bZ\cap[0,1)$ in the target.} $R_\Hod^{\qun,\loc}(\bar X,D,x)|_{\bG_m}\cong R_\Hod^{\qu|n,\loc}(\bar X^c,D^c,x)|_{\bG_m}$ covering the involution $\lambda\mapsto \lambda^{-1}$ and locally equivariant with respect to the $\bG_m$-action after twisting by $t\mapsto t^{-1}$ on the target which is described as follows.  For an $S$-point of $R_\Hod^{\qu|n,\loc}(\bar X,D,x)|_{\bG_m}$ given by $(\lambda,\bar E,\nabla,\phi)=(\lambda,\bar\cE,\scrD,\phi)$, the rescaled framed connection $(\bar \cE,\lambda^{-1}\mathscr{D}^{1,0}+\scrD^{0,1},\phi)$ is a family of flat connections whose residues have eigenvalues contained in $(-1+\epsilon,\epsilon]+ i\bR$.  It follows that the extension is the Deligne extension of the underlying family of framed local systems $S\to R_B(X,x)^\an$, which works in families since the eigenvalues of the local monodromy admit a continuous logarithm.  We can form the antiholomorphic Deligne extension $(\bar\cE',\scrD^{0,1}+\lambda^{-1}\scrD^{1,0},\phi)$ with eigenvalues in $(-1+\epsilon,\epsilon]+ i\bR$, and rescaling $(\lambda^{-1},\bar\cE',\lambda^{-1}\scrD,\phi)$ provides the required $S$-point  of $R_\Hod^{\qu|n,\loc}(\bar X^c, D^c,x)|_{\bG_m}$.  Gluing $R_\Hod^{\qu|n,\loc}(\bar X,D,x)$ to $R_\Hod^{\qu|n,\loc}(\bar X^c,D^c,x)$ via this identification we obtain a countably finite type complex analytic space $R_{DH}^{\qu|n,\loc}(\bar X,D,x)$ with an analytic map $\pi:R_{DH}^{\qu|n,\loc}(\bar X,D,x)\to\bP^1$, restricting to the corresponding structure on the two $R_\Hod$ spaces.  We similarly construct $\cM_{DH}^{\qu|n,\loc}(\bar X,D)$ by gluing, and identify $\cM_{DH}^{\qu|n,\loc}(\bar X,D)$ with the quotient of $R_{DH}^{\qu|n,\loc}(\bar X,D,x)$ by $\bGL_r(\bC)$.  The derivation $\Theta$ associated to the $\bG_m$-action glues to give a natural global derivation $\Theta$ compatibly on $\cM_{DH}^{\qu|n,\loc}(\bar X,D)$ and $R_{DH}^{\qu|n,\loc}(\bar X,D,x)$.  In fact, the $\bG_m$-action on $R_{DH}(\bar X,D,x)$ stabilizes $R^{\qu|n}_{DH}(\bar X,D,x)$, so there is a well-defined $\bG_m$-action on $R_{DH}^{\qu|n,\loc}(\bar X,D,x)$ as a germ of an analytic space containing $R_{DH}^{\qu|n}(\bar X,D,x)$ as a closed subspace.

 The moduli functor of $R_{DH}^{\qu|n,\loc}(\bar X,D,x,r)$ is described as follows.  For an analytic space $S$, an $S$-point of $R_{DH}^{\qu|n,\loc}(\bar X,D,x,r)$ yields a tuple $(\pi, \cE,\bar E,\bar E',\mathscr{D},\phi)$ where:
 \begin{enumerate}
 \item  $\pi:S\to \bP^1$ is a morphism;
 \item $ \cE$ is a $C^\infty_X\boxtimes \cO_{S}$-module on $ X_{S}$;
 \item $\mathscr{D}:\cE\to\cE\otimes \Omega_{ X_S/S}(1)$ is a flat $x\partial_{ X_S/S}+y\bar\partial_{ X_S/S}$ connection, where $x,y$ are fixed sections of $\cO_{\bP^1}(1)$ vanishing at $0,\infty$.  We require that:
 \begin{enumerate}
     \item Using the trivialization $y$ of $\pi^*\cO_{\bP^1}(1)$ on $S_0:=\pi^{-1}(\bP^1\setminus 0)$, $( \cE,\mathscr{D}^{0,1})$ is a rank $r$ locally free holomorphic vector bundle on $ X_{S_0}$ and $\bar E$ is a locally free $\cO_{\bar X_S}$-module extension (meaning it comes equipped with an isomorphism $j_{S_0}^*\bar E\to(\cE,\mathscr{D}^{0,1})$), where $j_{S_0}:X_{S_0}\to\bar X_{S_0}$ is the inclusion) to which $\mathscr{D}^{1,0}$ extends as a flat logarithmic connection.  
     \item Using the trivialization $x$ of $\pi^*\cO_{\bP^1}(1)$ on $S_\infty:=\pi^{-1}(\bP^1\setminus \infty)$, $( \cE,\mathscr{D}^{1,0})$ is a rank $r$ locally free holomorphic vector bundle on $ X^c_{S_0}$ and $\bar E'$ is a locally free $\cO_{\bar X^c_S}$-module extension to which $\mathscr{D}^{0,1}$ extends as a flat logarithmic connection.

 \end{enumerate}
      \item $\phi:\cE|_{x\times S}\xrightarrow{\cong}\cO_S^r$ is a framing.
 \end{enumerate}
Moreover, any such $(\pi, \cE,\bar E,\bar E',\mathscr{D},\phi)$ on $S$ arises from a morphism $S\to R_{DH}^{\qu|n,\loc}(\bar X,D,x,r)$ after shrinking to a sufficiently small neighborhood of the closed analytic space $S^{\qu|n}\subset S$ where the residual eigenvalues of $(\bar E,\mathscr{D}^{1,0})$ (resp. $(\bar E',\mathscr{D}^{0,1})$) are contained in $\lambda\cdot (\frac{1}{n}\bZ\cap (-1,0])$ (resp. $-\lambda^{-1}\cdot (\frac{1}{n}\bZ\cap (-1,0])$).

\begin{rem}
    Note that $S$-points of $\cM_{DH}^{\qu|n,\loc}(\bar X,D)$ are tuples $(\pi,\cE,\bar E,\bar E',\mathscr{D})$ that are trivializable over $x$, which is a nontrivial condition.  In fact, we could have twisted the above gluing to produce a version of $R_{DH}$ and $\cM_{DH}$ where $\cE|_{x\times S}\cong \pi^*F$ where $F$ is a fixed locally free sheaf on $\bP^1$. 
\end{rem}

\begin{rem}
    The global construction of the Deligne--Hitchin space is complicated by both the monodromy of the residual eigenvalues if they are allowed to roam freely, and ``resonant'' phenomena when they differ by integers (as in \Cref{local iso on good}).  This is more seriously contended with in recent work of Simpson \cite{Simpsonrank2twistor,simpsonHitchinDeligne}, although there the ``generic'' case is assumed.  By the results of \Cref{sect:bialg}, bialgebraic substacks of $\cM_B(X)$ are determined by their germ around the quasiunipotent local monodromy locus, and neither issue arises in the germ of the Deligne--Hitchin space around this locus.
\end{rem}

The following is straightforward; we leave the proof to the reader.
\begin{lem}
Let $f:(\bar X,D)\to(\bar Y,E)$ be a morphism of log smooth pairs compatible with a choice of basepoints.  Then there are pullback morphisms $f^*:\cM_{DH}^{\qu|n,\loc}(\bar Y,E)\to\cM_{DH}^{\qu|n,\loc}(\bar X,D)$ and $f^*: R_{DH}^{\qu|n,\loc}(\bar Y,E,y)\to R_{DH}^{\qu|n,\loc}(\bar X,D,x)$ of analytic stacks (up to shrinking).
\end{lem}

\subsection{Preferred sections and twistor germs}
Variations of twistor structures yield sections of the Deligne--Hitchin space:
\begin{lem}\label{tw deligne ext}
    To every $\bC$-AVMTS $(\cE,W_\bullet\cE,\mathscr{D})$ with quasiunipotent local monodromy on $X$, there is a functorially associated filtered logarithmic locally free extension $(\cE,W_\bullet\cE,\bar E,\bar E',\mathscr{D})$ on $\bar X_{\bP^1}$ in the above sense whose residues have eigenvalues in $\lambda\cdot (\frac{1}{n}\bZ\cap(-1,0])$. 
\end{lem}
\begin{proof} 
By the above discussion, the family of flat connections $\lambda^{-1}\mathscr{D}^{1,0}$ on $(\cE,\mathscr{D}^{0,1})$ has constant residual eigenvalues in $(-1,0]$ and the Deligne extension provides such extensions $\bar E,\bar E'$ outside of the codimension 2 subset $D\times 0\cup D\times\infty $ in $\bar X_{\bP^1}$.  The existence of the claimed extensions $\bar E,\bar E'$ therefore reduces to the local existence, which is by assumption the case for a mixed twistor $\mathscr{D}$-module (see \cite[\S9]{mochizukimixedtwistor}).
\end{proof}

\begin{cor}\label{cor preferred sections}
    Let $A$ an artinian $\tate$-MTS-algebra and $\scrE=(\cE,W_\bullet\cE,\mathscr{D})$ an $A$-AVMTS on $X$ with quasiunipotent local monodromy whose eigenvalues have order dividing $n$.  Then $\scrE$ is pulled back via a $\bP^1$-morphism $\underline{\mathrm{Spec}}\, A\to \cM_{DH}^{\qu|n,\loc}(\bar X,D)$. 
 If in addition there is a framing $\phi \colon \scrE_x\xrightarrow{\cong}A^r$, then $(\scrE,\phi)$ is pulled back via a $\bP^1$-morphism $\underline{\mathrm{Spec}}\, A\to R_{DH}^{\qu|n,\loc}(\bar X,D,x)$.
\end{cor}
\begin{proof}
    We can interpret the extension $(\cE,W_\bullet\cE,\bar E,\bar E',\mathscr{D})$ on $\bar X_{\bP^1}$ given by applying the lemma to the underlying $\bC$-AVMTS together with its action by the sheaf of rings $A$ as such a tuple on $\bar X_{\underline{\mathrm{Spec}}\, A}$.  
\end{proof}
\begin{defn}
    A \emph{quasiunipotent preferred section} is a section of $\cM_{DH}^{\qu|n,\loc}(\bar X,D)\to\bP^1$ or $R_{DH}^{\qu|n,\loc}(\bar X,D,x)\to\bP^1$ as in \Cref{cor preferred sections} resulting from a weight 0 $\bC$-VTS (that is, a tame purely imaginary harmonic bundle) with quasiunipotent local monodromy.

\end{defn}
Note the image of a quasiunipotent preferred section is contained in $\cM_{DH}^{\qu|n,\loc}(X)$ or $R_{DH}^{\qu|n}(\bar X,D,x)$.  We now specialize to the framed spaces, but there are versions of the following discussion for $\cM_{DH}$, provided we replace all formal isomorphism claims with a miniversality claim.

\begin{cor}For any $(V,\phi)\in R_B^{\qu|n}(X,x)(\bC)^\ss$ we have:
\begin{enumerate}
    \item The associated framed Deligne extension $(V_{DR},\phi)\in R_{DR}^{\qu|n}(\bar X,D,x)(\bC)$ (with residual eigenvalues in $(-1,0]$).
    \item A framed quasiunipotent preferred section $s_{(V,\phi)}$ of $\pi \colon R_{DH}^{\qu|n,\loc}(\bar X,D,x)\to\bP^1$ specializing to $(V_{DR},\phi)$ at $\lambda=x/y=1$.
    \item A uniquely determined pro-$\tate$-MTS-algebra $\hat\cO_{R_B(X,x),(V,\phi)}$ and a morphism 
    $$\hat s_{R_B(X,x),(V,\phi)} \colon \underline{\mathrm{Spec}}\,\hat\cO_{R_B(X,x),(V,\phi)}\to R_{DH}^{\qu|n,\loc}(\bar X,D,x)$$
    fitting into a commutative diagram
    \[\begin{tikzcd}
        \bP^1\ar[rrrd,equals, bend right =10]\ar[r,"0"]\ar[rrr,"s_{(V,\phi)}",bend left=20]&\underline{\mathrm{Spec}}\, \hat\cO_{R_B(X,x),(V,\phi)}\ar[rr,"\hat s_{R_B(X,x),(V,\phi)}"]&&R_{DH}^{\qu|n,\loc}(\bar X,D,x)\ar[d,"\pi"]\\
        &&&\bP^1.
    \end{tikzcd}\]
    Moreover, the specialization to $\lambda=1$ of $\hat s_{R_B(X,x),(V, \phi)}$ is a formal isomorphism to the formal completion of $R^{\qu|n,\loc}_{DR}(\bar X,D,x)$ at $(V_{DR},\phi)$.
\end{enumerate}
\end{cor}
\begin{proof}
    By \Cref{thm:versal frame} and \Cref{local iso on good}.
\end{proof}

\begin{lem}\label{formal iso R} In the notation of the previous corollary, for any $(V,\phi)\in R_B^{\qu|n}(X,x)(\bC)^\ss$, the morphism 
\begin{equation}\label{formal section}\hat s_{R_B(X,x),(V,\phi)} \colon \underline{\mathrm{Spec}}\, \hat\cO_{R_B(X,x),(V,\phi)}\to R_{DH}^{\qu|n,\loc}(\bar X,D,x) \end{equation}
    is a formal isomorphism to the formal completion of $R^{\qu|n,\loc}_B(\bar X,D,x)$ along $s_{(V,\phi)}$ over $\bG_m$.  If $X$ is a curve and $n=1$, the same is true over all of $\bP^1$.
\end{lem}
\begin{proof}
    We first reduce to the claim assuming $X$ is a curve.  If $i:C\to X$ is a morphism from a smooth affine curve such that $i_*:\pi_1(C,c)\to\pi_1(X,x)$ is surjective (which exists by \Cref{lefschetz curve}), we have a commutative diagram
    \[\begin{tikzcd}
        \uSpec\hat\cO_{R_B(X,x),(V,\phi)}\ar[rrr,"\hat s_{R_B(X,x),(V,\phi)}"]\ar[d,"i^*"]&&&R_{DH}^{\qu|n,\loc}(\bar X,D,x)\ar[d,"i^*"]\\
        \uSpec\hat\cO_{R_B(C,c),(i^*V,i^*\phi)}\ar[rrr,"\hat s_{R_B(C,c),(i^*V,i^*\phi)}"]&&&R_{DH}^{\qu|n,\loc}(\bar C,D_C,c)
    \end{tikzcd}\]
    The pullback $i^*$ is a closed immersion in the Betti realization, hence also in the De Rham realization (locally around the chosen component of the quasiunipotent locus).  By \Cref{tw loc triv}, the left morphism is a closed immersion, while by the $\bG_m$-action on $R_\Hod^{\qu|n,\loc}(\bar X,D,x)|_{\bG_m}$ the right morphism is a closed immersion over $\bG_m$.  Thus, if the bottom morphism is an isomorphism on completions over $\bG_m$, the top morphism is a closed immersion on completions over $\bG_m$.  For any $t\in \bC^*$, taking $V'=t V$ we may consider a framed quasiunipotent preferred section $s':=s_{( V',\phi')}$ such that $t^{-1}s(t)=s'(1)$, where $s=s_{(V,\phi)}$.  It then follows by acting by $t$ that there is a surjection
    of the $\lambda=1$ specialization of $\hat\cO_{R_B(X,x),(V',\phi')}$ onto the $\lambda=t$ specialization $\hat\cO_{R_B(X,x),(V,\phi)}$, so in particular we have an inequality of the dimensions of the graded pieces with respect to the maximal ideal.  By \Cref{tw loc triv} the same is true for the $\lambda=1$ specializations of both, and swapping $(V,\phi)$ and $(V',\phi')$ we find the top morphism is an isomorphism on completions for all $\lambda\neq 0, \infty$.

    It remains therefore to take $X$ to be a smooth affine curve.  In this case both $\underline{\mathrm{Spec}}\, \hat\cO_{R_B(X,x),(V,\phi)}$ and $R^{\qu|n,\loc}_{DH}(X,x)|_{\bG_m}$ are smooth over $\bG_m$, so for the claim over $\bG_m$ it is enough to show the map \eqref{formal section} is an isomorphism to first order in the fiber direction at every $\lambda\in\bG_m$.  This follows from the fact that the specialization to $\lambda\in\bG_m$ of the (derived) pushforward of a tame nilpotent mixed twistor module (specifically, $\End(V)$) computes the cohomology of the De Rham complex of the associated $\lambda$-connection \cite[Proposition 14.1.20]{mochizukimixedtwistor}, which in turn computes the deformation space of the logarithmic connection by \Cref{local iso on good}.  
    
    For $\lambda=0$ (and $\lambda=\infty$ is the same, after passing to the conjugate), we have:

        \begin{lem}\label{lemma_Sabbah}
        Let $(\bar X,D)$ be a log smooth projective curve with $X=\bar X\setminus D$, let $j:X\to \bar X$ be the inclusion, and let $E$ be a smooth tame nilpotent (algebraic) pure twistor $\scrD$-module on $X$ with trivial parabolic structure.  Then the cohomology of $\mathrm{sp}_0\pt_*j_*E$ is canonically isomorphic to the cohomology of the log de Rham complex of the $\lambda=0$ specialization of the logarithmic extension $(\bar E,\theta)$ from above.
    \end{lem}
    \begin{proof}
        We thank Claude Sabbah for communicating the following argument.  Let $x$ be a local defining equation for $D$.  Let $\cM[*D]$ (see \cite[\S3.1.2]{mochizukimixedtwistor}) be the $\mathscr{R}_X$-module underlying $j_*E$ away from $\lambda=\infty$.  Then locally $\cM[*D]$ has a $V$-filtration for which $\lambda\partial_x:\gr^V_b\cM[*D]\to\gr^V_{b+1}\cM[*D]$ is an isomorphism for $b>-1$ and $x: V_0\cM[*D]\to V_{-1}\cM[*D]$ is an isomorphism.  We therefore have quasi-isomorphisms of two-term complexes.
        \[\begin{tikzcd}
        V_{-1}\cM[*D]\ar[r," \lambda\partial+\theta"]&V_{-1}\cM[*D]\otimes\omega_{\bar X}(D)\\
V_{-1}\cM[*D]\ar[d]\ar[r,"\lambda\partial+\theta"]\ar[u,equals]&V_0\cM[*D]\otimes\omega_{\bar X}\ar[d]\ar[u,"x\otimes x^{-1}",swap]\\
            \cM[*D]\ar[r,"\lambda\partial+\theta"]&\cM[*D]\otimes\omega_{\bar X}
        \end{tikzcd}\]
        The bottom row is the De Rham complex of $\cM[*D]$ and the top row specializes to the log Higgs complex of $V_{-1}\cM[*D]$, which in this case is $\bar E$.  Since the terms of the complexes are flat over $\bP^1$ and the cohomology of the push-forward to the point is locally free, its formation commutes with specialization.
    \end{proof}
We now claim that \eqref{formal section} is an isomorphism on tangent spaces at $\lambda=0$.  Since $V$ has unipotent local monodromy, the same is true for $\End(V)$.  Thus taking endomorphism objects commutes with the Deligne extension for $V$, and the same is true for the extensions in the sense of \Cref{tw deligne ext}.  The tangent space of the $\lambda=0$ fiber of $\cM_{DH}(\bar X,D)$ is the first cohomology of the log De Rham complex of $(\overline{\End(V)}|_0,\ad\theta)$, which is by the lemma identified with the relative tangent space of $\uSpec\hat\cO_{\cM_B(X),V}$ at $\lambda=0$.  The same is then true for the framed space, so \eqref{formal section} is an isomorphism on relative tangent spaces at $\lambda=0$, and because of the presence of the section this implies the map on tangent spaces is an isomorphism.

Finally, any local morphism $A\to B$ of complete noetherian local rings which is an isomorphism on tangent spaces and for which $B$ is smooth is an isomorphism, and this completes the proof.
\end{proof}
\begin{rem}
    We could have directly used the formalism of \Cref{sect:miniversal} to show that for $\lambda\in\bG_m$ the deformation theory of a $\lambda$-connection with residual eigenvalues in $\lambda\cdot \frac{1}{n}(-1,0]$ admits an enhancement by twistor structures.  This follows from \cite[Proposition 14.1.20]{mochizukimixedtwistor}, which implies that the second part of \Cref{lem: six functors tw} holds for $\lambda\in\bG_m$.  For $\lambda=0,\infty$, our version of $R^{\qu|n,\loc}_{DH}(\bar X,D,x)$ is not the correct moduli space since it ignores parabolic structures. 
 If one keeps track of this, then the same argument can likely be made to work using that the $\lambda=0$ specialization of the push-forward of the endomorphism object computes the cohomology of the log De Rham complex.
\end{rem}

For any $(V,\phi)\in R_B(X,x)(\bC)^{\ss,\qu}$, let $i \colon C\to X$ be a Lefschetz curve (see \Cref{lefschetz curve})and $p \colon C'\to C$ a finite \'etale cover for which $V' := f^*V$ has unipotent local monodromy, where $f=i\circ p$.  Then $f^* \colon \uSpec\hat\cO_{R_B(X,x),(V,\phi)}\to\uSpec\hat\cO_{R_B(C',c'),(V',\phi')}$ is a closed immersion of formal spaces by \Cref{lem immersive} and \Cref{tw loc triv}.
\begin{cor}\label{theta is canonical}
    The canonical derivation $\Theta$ lifts functorially to each $\uSpec\hat\cO_{R_B(X,x),(V,\phi)}$ for $(V,\phi)\in R_B(X,x)(\bC)^{\qu,\ss}$. 
\end{cor}
\begin{proof}
The derivation $\Theta$ lifts canonically over $\bG_m$ so we just need to check it extends to $\bP^1$ and this holds on $\uSpec\hat\cO_{R_B(C',c'),(V',\phi')}$ by \Cref{formal iso R}.
\end{proof}

For a subspace $Z\subset R_B(X,x)$ which is formally twistor at $(V,\phi)\in Z^{\qu|n}(\bC)^{\ss}$ we define $\hat s_{Z,(V,\phi)}$ as the composition
\[\begin{tikzcd}
    \uSpec\hat\cO_{Z,(V,\phi)}\ar[rd]\ar[rr,"\hat s_{Z,(V,\phi)}"]&&R_{DH}^{\qu|n,\loc}(\bar X,D,x)\\
    &\uSpec\hat\cO_{R_B(X,x),(V,\phi)}\ar[ru,"\hat s_{R_B(X,x),(V,\phi)}",swap]&
\end{tikzcd}\]
Since $\hat \cO_{Z,(V,\phi)}$ is a quotient pro-$\tate$-MTS-algebra of $\hat\cO_{R_B(X,x),(V,\phi)}$, it follows from \Cref{formal iso R} that $\hat s_{Z,(V,\phi)}$ is a closed immersion, as this is true after post-composition with $R_{DH}^{\qu|n,\loc}(\bar X,D,x)\xrightarrow{i^*}R_{DH}^{\qu|n,\loc}(\bar C,E,c')\xrightarrow{p^*}R_{DH}^{\unip,\loc}(\bar C',E',c')$.

We denote by $Z_{\Hod}(V,\phi) $ (resp. $Z_{\overline{\Hod}}(V,\phi) $) the Zariski closure of the image of $\hat s_{Z,(V,\phi)}|_{\bA^1}$ (resp. $\hat s_{Z,(V,\phi)}|_{\bP^1\setminus 0}$) in $R_{\Hod}(\bar X,D,x)$ (resp. $R_{\overline{\Hod}}(\bar X^c,D^c,x)$).

\begin{defn}
    In the above situation, we say $\hat s_{Z,(V,\phi)}$ is \emph{algebraic} over $\bG_m$ if:
    \begin{enumerate}
    \item $\hat s_{Z,(V,\phi)}|_{\bG_m}$ induces an isomorphism onto the completion of $Z_{\Hod}(V,\phi) |_{\bG_m}$ along the image of $s_{(V,\phi)}$.
    \item The same is true for $\hat s_{Z,(V,\phi)}|_{\bG_m}$ and $Z_{\overline{\Hod}}(V,\phi) |_{\bG_m}$.
    \end{enumerate}
\end{defn}
Note that if $\hat s_{Z,(V,\phi)}$ is algebraic, then the algebraic germs $Z_{\Hod|\bG_m}^{\qu|n,\loc}(V,\phi) :=Z_{\Hod|\bG_m}(V,\phi)\cap R_{DH}^{\qu|n,\loc}(\bar X,D,x) $ and $Z_{\overline{\Hod}|\bG_m}^{\qu|n,\loc}(V,\phi)$ (defined similarly) are identified as analytic germs via the gluing.
\begin{cor}\label{extend germ}
    Let $Z\subset R_B(X,x)$ a subspace which is formally twistor at $(V,\phi)\in Z^{\qu|n}(\bC)^\ss$ such that $\hat s_{Z,(V,\phi)}$ is algebraic.  Then:
    \begin{enumerate}
        \item $\hat s_{Z,(V,\phi)}|_{\bA^1}$ induces an isomorphism onto the the completion of $Z_{\Hod}(V,\phi) $ along the image of $s_{(V,\phi)}|_{\bA^1}$.  The same is true for $Z_{\overline{\Hod}}(V,\phi) $ over $\bP^1\setminus 0$.

    \item There is a germ of a closed analytic subspace $Z_{DH}^{\qu|n,\loc}(V,\phi)\subset R_{DH}^{\qu|n,\loc}(\bar X,D,x)$ for which $\hat s_{Z,(V,\phi)}$ induces an isomorphism onto the completion of $Z_{DH}^{\qu|n,\loc}(V,\phi) $ along the image of $s_{(V,\phi)}$.
    \end{enumerate}
\end{cor}
\begin{proof}
The second assertion is a direct consequence of the first.
 For simplicity write $W=Z_{\Hod}(V,\phi)$ with completion $\hat W$ along the image of $s_{(V,\phi)}|_{\bA^1}$.  First observe that the morphism $\cO_{\hat W}\to \hat \cO_{Z,(V,\phi)}$ is injective.  Indeed, by \Cref{tw loc triv}, the localization $\hat \cO_{Z,(V,\phi)}\to\hat \cO_{Z,(V,\phi)}[\lambda^{-1}]$ is injective, but since $\cO_{\hat W}\to \hat \cO_{Z,(V,\phi)}$ is an isomorphism over $\bG_m$ by assumption, its ideal is exactly the kernel of $\cO_{\hat W}\to\cO_{\hat W}[\lambda^{-1}]$.  By noetherianity and flatness of completion, this ideal is the completion of the corresponding ideal in $\cO_W$, hence must vanish be zero since $W$ is the Zariski closure.
 
We therefore need only show the surjectivity.  As in \Cref{formal iso R}, by Lefschetz we may reduce to the case $X$ is an affine curve, since this induces a closed immersion of formal twistor spaces.  By \Cref{from quasiunipotent to unipotent} there is a finite \'etale cover $p:X'\to X$ such that $p^*V$ has unipotent local monodromy.  Since $p^*:R_B(X,x)\to R_B(X',x')$ is immersive by \Cref{lem immersive}, by the same argument we reduce to the case that $V$ has unipotent local monodromy, and further that $Z=R_B(X,x)$, which is \Cref{formal iso R}.
\end{proof}
\begin{cor}\label{defo is tw}
With the setup of the previous corollary, the formal deformation theory of $s_{(V,\phi)}$ as a section of $Z_{DH}^{\qu|n,\loc}(V,\phi) \to \bP^1$ is naturally equivalent to the formal deformation theory of the zero section of $\uSpec\hat\cO_{Z,(V,\phi)} \to \bP^1$.\end{cor}

\subsection{Moduli of sections}
Let $\Sect(R_{DH}^{\qu|n,\loc}(\bar X,D,x)/\bP^1)$ denote the space of sections of the morphism $R_{DH}^{\qu|n,\loc}(\bar X,D,x)\to\bP^1$, which is an open subspace of the Douady space and in particular has the structure of a second countable analytic space.  Since $R_{DH}^{\qu|n,\loc}(\bar X,D,x)$ is only defined as a germ of a neighborhood of $R_{DH}^{\qu|n}(\bar X,D,x)$, the space $\Sect(R_{DH}^{\qu|n,\loc}(\bar X,D,x)/\bP^1)$ is well-defined as a germ of a neighborhood of the locus of sections contained in $R_{DH}^{\qu|n,\loc}(\bar X,D,x)$.  For any $\lambda\in\bP^1$ there is an evaluation morphism $$\ev_\lambda:\Sect(R_{DH}^{\qu|n,\loc}(\bar X,D,x)/\bP^1)\to R_{DH}^{\qu|n,\loc}(\bar X,D,x)_\lambda$$
which is compatible with the $\bGL_r$-action (with $r$ interpreted as the locally constant rank function).  Note also that for any morphism $f:(\bar X,D)\to (\bar Y,E)$ of log smooth projective pairs respecting the choice of basepoints, there is an analytic map
\begin{equation}f^*:\Sect(R_{DH}^{\qu|n,\loc}(\bar Y,E,y)/\bP^1)\to \Sect(R_{DH}^{\qu|n,\loc}(\bar X,D,x)/\bP^1)\label{pullback DH}\end{equation}
which is also compatible with the $\bGL_r$-action. 

  We denote by $\PrefSect^{\qu|n}_{DH}(\bar X, D,x)\subset \Sect(R_{DH}^{\qu|n,\loc}(\bar X,D,x)/\bP^1)$ the set of quasiunipotent preferred sections with the subspace topology, which is well-defined as they are contained in $R_{DH}^{\qu|n,\loc}(\bar X,D,x)$.  Pullback as in \eqref{pullback DH} maps preferred sections to preferred sections.  The evaluation at $\lambda=1$ restricts to a bijection
\begin{equation}\ev_1:\PrefSect^{\qu|n}_{DH}(\bar X,D,x)\to R_{DR}^{\qu|n}( \bar X,D,x)(\bC)^{ss}.\label{eval at 1}\end{equation}
There is a natural action of $\bR_{>0}$ on $\PrefSect^{\qu|n}_{DH}(\bar X,D,x)$ coming from the $\bR_{>0}$-action on $R_{DR}^{\qu|n}(\bar X,D,x)(\bC)^\ss$.

As in \Cref{section on comparison}, for a fixed a vector space $(V_0,h_0)$ with hermitian metric of the relevant rank, we define $\PrefSect^{\qu|n}_{DH}(\bar X,D,x)^{(V_0,h_0)}\subset\PrefSect^{\qu|n}_{DH}(\bar X,D,x)$ to be the subspace of preferred sections passing through $R_{DR}^{\qu|n}(\bar X,D,x)(\bC)^{(V_0,h_0)}$ over $\lambda=1$, or equivalently which are framed compatibly with the metric for all $\lambda$. 
\begin{prop}\label{section convergence}
    Let $(\bar X,D)$ be a log smooth projective curve.  Then the evaluation 
    \[\ev_1: \PrefSect^{\unip}_{DH}(\bar X,D,x)^{(V_0,h_0)}\to R_{DR}^{\unip}( X,x)(\bC)^{(V_0,h_0)}\]
    is a homeomorphism.
\end{prop}
\begin{proof}
The evaluation map is continuous, so we need only show the inverse is continuous.  As in the proof of \Cref{framed metric comparison}, it suffices to show any convergent sequence in the target has a subsequence which converges on the source to the point corresponding to the limit.

The claim is essentially a restatement of \Cref{harmonic sequence}.  Indeed, \cite[Theorem 5.12]{SimpsonmoduliI} implies that a metric on $R^{\nilp}_{\Hod}(\bar X,D,x)$ is given by
    \begin{align*}d((\lambda_1,\bar E_1,\nabla_1,\phi_1),&(\lambda_2,\bar E_2,\nabla_2,\phi_2))\\
    &=\inf_{\psi}|\lambda_1-\lambda_2|^2+|\phi_1-\psi_*\phi_2|^2+\|\nabla_1-\psi_*\nabla_2\|^2+\|\bar\partial_{\bar E_1}-\psi_*\bar\partial_{\bar E_1}\|^2\end{align*}
    where the infimum is over all $C^\infty$ isomorphisms $\psi:C^\infty\otimes \bar E_1\to C^\infty\otimes \bar E_2$ and the norms are operator norms as operators $L_1^q(\bar E_1)\to L^q(\bar E_1\otimes\omega_{\bar X}(D))$.  The proof of \Cref{harmonic sequence} implies if $(V_i,\phi_i)$ converges to $(V_\infty,\phi_\infty)$ and has framings which are compatible with the metric, then after passing to a subsequence, the associated preferred sections converge uniformly as maps $\bP^1\to R_{DH}^{\unip}(\bar X,D,x)$ to the section corresponding to $(V_\infty,\phi_\infty)$.     
\end{proof}

\begin{cor}\label{orbit is C0}
Let $(\bar X,D)$ be a connected log smooth proper algebraic space.  For any $(V,\phi)\in R_B^{\unip}(X,x)(\bC)^\ss$, there is a continuous map $\bR_{>0}\to \PrefSect^{\unip}_{DH}(\bar X,D,x)$ such that the composition with 
\[\PrefSect^{\unip}_{DH}(\bar X,D,x)\xrightarrow{\ev_1} R_B^{\unip}(X,x)(\bC)^{\ss}\to M^{\unip}_B(X)(\bC)\]
is the continuous map $\bR_{>0}\to M_B^\unip(X)(\bC)$ given by the $\bR_{>0}$-orbit of $V$.
       
\end{cor}
\begin{proof}Pullback to a Lefschetz curve $f:C\to X$ induces a closed immersion on section spaces, so we may assume $X$ is a curve.  By \Cref{framed stuff is proper}, $R^\unip_B(X,x)^{(V_0,h_0)}\to M^\unip_B(X)(\bC)$ is proper, and in corestriction to the locally closed stratification of $M^\unip_B(X)(\bC)$ by Jordan--H\"older type it is a locally trivial fiber bundle (by the existence of a slice \cite[Theorem 2.1]{mostow}).  The $\bR_{>0}$-orbit of $V$ in $M_B^\unip(X)(\bC)$ is contained in one of these strata, hence a lift to $R^\unip_B(X,x)^{(V_0,h_0)}$ exists through any point above $V$, and by \Cref{section convergence} this yields a lift to $\PrefSect^{\unip}_{DH}(\bar X,D,x)^{(V_0,h_0)}$.  Acting by $\bGL_r(\bC)$, we obtain the claim for any framing.
\end{proof}

\subsection{Hodge substacks}

\begin{defn}\label{defn Hodge substack}
    Let $X$ be a connected smooth algebraic space with a log smooth compactification $(\bar X,D)$.  A locally closed substack $\cZ\subset\cM_B(X)$ is \emph{Hodge} at $V\in\cZ(\bC)^{\qu,\ss}$ if the following conditions are satisfied.  Let $R\cZ\subset R_B(X,x)$ be the base-change to the framed space and let $\phi$ be a framing of $V$.  Then we have:
    \begin{enumerate}
    \item $R\cZ$ is formally twistor at $(V,\phi)$, and formally Hodge if $V$ underlies a $\CVHS$.
        \item The formal twistor subspace $\uSpec\hat\cO_{R\cZ,(V,\phi)}\subset\uSpec\hat\cO_{R_B(X,x),(V,\phi)}$ is tangent to $\Theta$.
        \item The formal twistor germ $\hat s_{R\cZ,(V,\phi)}$ is algebraic. 
    \end{enumerate}
A locally closed substack $\cZ\subset\cM_B(X)$ is a \emph{Hodge substack} if it is closed under semisimplification and Hodge at every $V\in\cZ(\bC)^{\qu,\ss}$.

    If $X$ is a connected normal algebraic space we define the corresponding notion for $\cZ\subset \cM_B(X)$ if the restriction $\cZ\subset \cM_B(X)\subset \cM_B(U)$ is Hodge for any smooth affine $U\subset X$.

\end{defn}
The definition is easily seen to not depend on the log smooth compactification.  Note that all of the conditions only involve the Deligne--Hitchin space over $\bG_m$, where its functorial behavior is controlled by the Betti stack.  Nonetheless, each of the three structures extend to the full Deligne--Hitchin space, and in particular by \Cref{extend germ} the algebraic germ $Z_{\Hod}(V,\phi)$ (resp. $Z_{\overline{\Hod}}(V,\phi)$) will extend over $\bA^1$ (resp. $\bP^1\setminus 0$) and agree with the twistor germ $\hat s_{R\cZ,(V,\phi)}$.

  The following functorial properties follow from the definitions and what we've proven.
\begin{prop}\label{basic prop hodge substack}
        Let $X$ be a connected normal algebraic space.

    \begin{enumerate}
        \item For any $r$, $\cM_B(X,r)$ and $\{\triv_r\}$ are Hodge substacks of $\cM_B(X)$.
        \item Assume $X$ is a curve.  For any $V\in \cM_B(X)(\bC)^{\qu,\ss}$, the fixed local monodromy leaf $FM(V)\subset\cM_B(X)$ is Hodge at $V$.  For any $V\in \cM_B(X)(\bC)^{\qu,\ss}$, the fixed residual eigenvalues leaf $FE(V)\subset\cM_B(X)$ is a closed Hodge substack.

        \item Intersections and reductions of Hodge substacks are Hodge substacks.  Irreducible components of reduced Hodge substacks are Hodge substacks.  Finite unions of closed Hodge substacks are Hodge substacks.   
        \item Let $f:X\to Y$ be a morphism of connected normal algebraic spaces and $f^*:\cM_B(Y)\to \cM_B(X)$ the pullback morphism.  Then:
        \begin{enumerate}
           \item For any Hodge substack $\cZ\subset \cM_B(X)$, $(f^*)^{-1}(\cZ)\subset \cM_B(Y)$ is a Hodge substack.
        \item For any Hodge substack $\cZ\subset \cM_B(Y)$, $f^*\cZ\subset \cM_B(X)$ is a Hodge substack provided $f$ is dominant or a Lefschetz curve (see \Cref{defn lefschetz}).
        \end{enumerate}
        \item Inverse images of Hodge substacks under the direct sum and tensor product morphisms $\oplus,\otimes:\cM_B(X)^2\to\cM_B(X)$ are Hodge substacks\footnote{With the obvious definition of Hodge substack of $\cM_B(X)^2$.}. 
 
    \end{enumerate}
\end{prop}
\begin{proof}
The formally twistor condition follows from \Cref{lem: prop Hodge sub of Betti}. 
For $\cM_B(X,r)$, the algebraicity condition is 
\Cref{formal iso R}. For the rest, the $\Theta$-tangent condition and the algebraicity condition are easily seen to have the required properties.   
\end{proof}

The following will be essential, especially \Cref{abs Hodge contain R*} below.

\begin{thm}\label{abs Hodge contain R* germ}
    Let $\cZ\subset \cM_B(X)$ be a Hodge substack, and $Z\subset M_B(x)$ the image in the good moduli space.  Then for any $n$, any irreducible component $Z_0$ of $Z^{\qu|n}$ satisfies the following property.  For any $V\in Z_0(\bC)$, a germ of the $\bR_{>0}$-orbit of $V$ around $V$ is contained in $Z_0$.
\end{thm}
\begin{proof}

  By Lefschetz (see \Cref{lefschetz curve}) we may assume $X$ is an affine curve, with log smooth compactification $(\bar X,D)$.  We may also replace $\cZ$ with an irreducible component of the reduction of $\cZ^{\qu|n}$; note that it follows that $Z$ is irreducible.  We will need to use the following:

      \begin{lem}\label{sing of Hodge is Hodge}Let $\cZ\subset\cM_B(X)$ be a locally closed substack which is Hodge at each of its semisimple points.  Then the singular locus $\cZ^\mathrm{sing}\subset\cM_B(X)$ is also Hodge at each of its semisimple points.
    \end{lem}
\begin{proof}
The only condition to check is the formally twistor condition (the proof of the formally Hodge condition will be the same), as then the other conditions in \Cref{defn Hodge substack} are clear, since the algebraic germ will be the singular locus of the algebraic germ of $\cZ$, which is clearly $\Theta$-tangent. Let $R\cZ^\sing$ be the base-change to $R_B(X,x)$ (and likewise for $R\cZ$); we have $(R\cZ)^\sing=R\cZ^\sing$.  For a point $(\bsV,\bsphi)\in R\cW(\bC)^\ss$, let $\hat\cO_{R}$ be the deformation pro-$\tate$-MS-algebra of $R_B(X,x)$ at $(\bsV,\bsphi)$ and $I\subset \hat\cO_{R}$ be the ideal of $R\cZ$. 
By \Cref{tw equiv} (and the argument of \cite[Lemma 1.6]{ES}), $\Der_{\tate}(\hat\cO_{R},\hat\cO_{R})^\wedge$ has a natural pro-$\hat\cO_{R}$-MS-module structure:  there is a natural equivariant structure and an equivariant filtration by $\Der_{\tate}(\hat\cO_R,\fm_{\hat\cO_R}^\ell)^\wedge$ whose graded pieces $\Hom_{\tate}(\fm_{\hat\cO_R}/\fm_{\hat\cO_R}^2,\fm_{\hat\cO_R}^\ell/\fm_{\hat\cO_R}^{\ell+1})$ are twistor structures.  Since $R_B(X,x)$ is smooth, this puts a natural pro-$\hat\cO_R$-MS-module structure on $\hat\Omega_{\hat\cO_{R}/\tate}$, and the canonical morphism
 \[I/I^2\to \hat\cO_{R\cZ}\otimes_{\hat\cO_{R}}\hat \Omega_{\hat\cO_{R}/\tate}\]
 is a morphism of pro-$\hat\cO_R$-MS-modules.  The singular locus of $\cZ$ is cut out by the minors of this morphism, which generate a $\tate$-MS-ideal.  
\end{proof}

By \Cref{reduction_to_unipotent_monodromy}, there is a finite \'etale cover $p:X'\to X$ such that $p^*\cZ\subset \cM_B^\unip(X')$.  Recall that $p^*:\cM_B(X)\to\cM_B(X')$ is immersive by \Cref{lem immersive}.  Consider the composition 
\[\begin{tikzcd}
    \cM_B(X)\ar[r,"p^*"]&\cM_B(X')\ar[r,"\psi_D"]&\prod_{\bar x\in D}\cM_B(\hat \bD^*).
\end{tikzcd}\]

We claim there is a partition $\{\cZ_i\}$ of $\cZ$ by reduced locally closed smooth substacks such that:
\begin{enumerate}
    \item Each $\cZ_i$ is contained in some $(p^*)^{-1}FM(V'_i)$.
    \item Each $\cZ_i$ is Hodge at each of its semisimple points.
\end{enumerate}
Indeed, the reduced intersections $\cZ\cap (p^*)^{-1}FM(V')$ form a finite partition (since $\cZ=\cZ^{qu|n}$) by locally closed substacks which are Hodge at each of their semisimple points, by \Cref{basic prop hodge substack}, and by \Cref{sing of Hodge is Hodge} we may successively take the reduced singular locus.

It suffices to prove the statement of the theorem for each $\cZ_i$.  We are therefore reduced to the following:

\begin{claim}
    Suppose $\cZ\subset\cM_B(X)$ is a locally closed smooth substack which is Hodge at each of its semisimple points and which is contained in some $(p^*)^{-1}FM(V')$.  Let $Z$ be the image of $\cZ$ in $M_B(X)$.  Then, for any $V\in Z$, $Z$ contains a germ of the $\bR_{>0}$-orbit of $V$ around $V$. 
\end{claim}
\begin{proof}

Let $\bsV\in\cZ(\bC)^\ss$ and equip it with a framing $\bsphi$.  Set $(\bsVp,\bsphip):=p^*(\bsV,\bsphi)$.  Since $p^*: R_B(X,x)\to R_B(X',x')$ is immersive by \Cref{lem immersive}, it induces closed immersions on the level of germs $p^*:(R\cZ,(\bsV,\phi))\to(R_B^\unip(X,x),(\bsV',\bsphi'))$, as well as for the twistor germ 
$$p^*:R\cZ^{\qu|n,\loc}_{DH}(\bsV,\bsphi)\to R_{DH}^{\unip,\loc}(\bar X',D',x')(\bsV',\bsphi'),$$  
where we recall that the source (resp. target) is the germ of the Zariski closure of $\hat s_{R\cZ,(V,\phi)}$ (resp. $\hat s_{R_B(X',x'),(V',\phi')}$).
According to \Cref{orbit is C0}, there is a continuous family $(\boldsymbol{s}'_t)_{t\in\bR_{>0}}$ of preferred sections of $R_{DH}^{\unip,\loc}(\bar X',D',x')$ with $\bss_1'=s_{(\bsV',\bsphi')}$ and such that $\bss_t'(1)$ lifts the $\bR_{>0}$-orbit of $\bsV'$.  These sections satisfy the following properties:
\begin{enumerate}
    \item Since $\uSpec\hat\cO_{R\cZ,(\bsV,\bsphi)}$ is $\Theta$-tangent, $\boldsymbol{s}'_t(0),\boldsymbol{s}'_t(\infty)\in p^*R\cZ^{\qu|n,\loc}_{DH}(\bsV,\bsphi)$ for $t$ sufficiently close to $1$.
    \item By \cite[Theorem 7]{Simpson_noncompact} (see also the discussion in \cite{simpsonHitchinDeligne}), $\bss_t'$ is contained in the fixed local monodromy leaf $RFM_{DH}^{\unip,\loc}(\bsV',\bsphi')$.
\end{enumerate}

The space of sections satisfying these two conditions is an analytic subspace of $\Sect(R_{DH}^{\unip,\loc}(\bar X',D',x')/\bP^1)$, and $s_t'$ is a continuous family of sections, so it follows that there is an analytic family of sections $\sigma'_a$ over a connected analytic germ $(A,a_0)$ which satisfies these two conditions and such that there is a continuous map $\gamma:((1-\epsilon,1+\epsilon),1)\to (A,a_0)$ for $\epsilon>0$ with $s'_t=\sigma'_{\gamma(t)}$.

\begin{subclaim}
    The family $\sigma_a'$ is contained in $p^*R\cZ^{\qu|n,\loc}_{DH}(\bsV,\bsphi)$.
\end{subclaim}
\begin{proof}
    Denote by $\hat\cO_{\twR'},\hat\cO_{\twF}$ the pro-$\tate$-MTS-algebras of $R_B(X',x'), RFM(\bsV',\bsphi')$ at $(\bsV',\bsphi')$, and $\hat \cO_{\twZ}$ that of $R\cZ$ at $(\bsV,\bsphi)$.  Let $\twR',\twF,\twZ$ be the associated formal spaces obtained by applying $\uSpec$, with $\twZ$ embedded in $\twR'$ via $p^*$.  By \Cref{defo is tw}, the deformation theory of $\sigma_{a_0}$ in $R^{\unip,\loc}_{DH}(\bar X',D',x')$ is equivalent to the deformation theory of the zero section in $\twR'$.  The relative obstruction space of $\twZ$ in $\twF$ is given by $(I_{\twZ/\twF}/\fm_{\twF}I_{\twZ/\twF})^\vee$, where $I_{\twZ/\twF}$ is the ideal of $\cZ$ in $\cF$, and we have a commutative diagram with exact rows and columns
    \[\begin{tikzcd}
&I_{\twF/\twR'}/\fm_{\twR'}I_{\twF/\twR'}\ar[r,equals]\ar[d]&I_{\twF/\twR'}/\fm_{\twR'}I_{\twF/\twR'}\ar[d]&&\\
       0\ar[r] &I_{\twZ/\twR'}/\fm_{\twR'} I_{\twZ/\twR'}\ar[d]\ar[r]&\fm_{\twR'}/\fm_{\twR'}^2\ar[r]\ar[d]&\fm_\twZ/\fm_\twZ^2\ar[d,equals]\ar[r]&0\\
0\ar[r,dashed]&I_{\twZ/\twF}/\fm_\twF I_{\twZ/\twF}^2\ar[r]\ar[d]&\fm_\twF/\fm_\twF^2\ar[r]\ar[d]&\fm_\twZ/\fm_\twZ^2\ar[r]&0\\
 &       0&0&&
    \end{tikzcd}\]
    where the injection comes from the smoothness of $\twR'$ and $\twZ$.  By a diagram chase, we obtain exactness including the dashed arrow.  
    
    According to \Cref{FLM obs} and \Cref{IH and weights}, $\fm_\twF/\fm_\twF^2$ has weights 0 and -1, so $(I_{\twZ/\twF}/\fm_\twF I_{\twZ/\twF})^\vee$ has weights 0 and 1.  In particular, any section which vanishes at two distinct points vanishes.  Since $\sigma_a'$ is contained in $\twF$ and contained in $\twZ$ over $0$ and $\infty$, it follows that $\sigma_a'$ is contained in $\twZ$ everywhere. 
\end{proof}

Thus, there is a continuous lift of the family of preferred sections $s'_t$ of $R_{DH}^{\unip,\loc}(\bar X',D',x')$ to a family of sections $s_t$ of $R\cZ_{DH}^{\qu|n,\loc}(\bsV,\bsphi)$.  Note that any lift of a preferred section $s_{E'}$ is preferred, since such a section yields a flat subbundle of $p_*E'$, which is therefore a sub-harmonic bundle.  Thus, $\bss_t$ is a continuous family of preferred sections of $R\cZ_{DH}^{\qu|n,\loc}(\bsV,\bsphi)$ through $(\bsV,\bsphi)$ whose image in $M_B(X)(\bC)$ is a lift of an arc of the $\bR_{>0}$-orbit of $\bsV'$.  Since $p^*:M_B(X)\to M_B(X')$ is quasi-finite (in fact finite \cite[Lemma 1]{rapinchuk98}), any continuous lift of a germ of an $\bR_{>0}$-orbit is a germ of a $\bR_{>0}$-orbit, and this completes the proof of the claim, and therefore also of the theorem. 
\end{proof}
\end{proof}
\begin{cor}\label{abs Hodge contain R*}
    Let $\cZ\subset \cM_B(X)$ be a closed Hodge substack and $Z$ its image in $M_B(X)$.  Then for any $n$, each irreducible component of $Z^{\qu|n}$ is $\bR_{>0}$-stable.
\end{cor}

\Cref{abs Hodge contain R*} will allow us to produce points underlying $\CVHS$ in Hodge substacks provided $Z^\qu$ is nonempty, which will always be the case by the results of \Cref{sect:bialg} if we require bialgebraicity.

\begin{rem}In the proper case, Simpson shows there is an $\bR_{>0}$-action on $\cM_B(X)(\bC)$ lifting the $\bR_{>0}$ action on $M_B(X)$ \cite{simpsonhiggs}; by pre- and post-composing with complex conjugation, we obtain another action.  The resulting $\bR_{>0}*\bR_{>0}$-action is expected to have as fixed points the local systems underlying $\bC$-VMHSs \cite[Remark, p. 45]{simpsonhiggs}, and perhaps the induced action of $\bR_{>0}*\bR_{>0}$ on the formal germ at a fixed point would yield the $\bC$-MHS of \Cref{thm:versal}.  Thus, the formally twistor condition might be related to a $\bR_{>0}*\bR_{>0}$-stability condition, and it is conceivable the notion of a Hodge substack could be equivalent to the notion of a $\bR_{>0}*\bR_{>0}$-stable substack.
\end{rem}

\section{Properties of constructible subsets of local systems}\label{sect:constructible}

Let $X$ be a connected normal algebraic space. 
 We begin with some general remarks about various properties of constructible subsets of the Betti stack $\cM_B(X)$, and eventually formulate a notion of absolute Hodge subsets. It will be important to work with points of the stack as opposed to the coarse space since we will ultimately want to allow non-semisimple local systems.

\subsection{Nonextendability}\label{sect:nonextend}We briefly review some of the notions from \cite{brunebarbeshaf}.

\begin{defn} Let $X$ be a connected normal algebraic space and $\Sigma\subset \cM_B(X)(\bC)$ a subset.
\begin{enumerate}

\item For $X\subset \bar X$ a partial compactification by a connected normal algebraic space (meaning an inclusion as an open subset of a connected normal algebraic space), we say $\Sigma$ extends to $\bar X$ if $\Sigma$ is contained in the image of the restriction $\cM_B(\bar X)(\bC)\to\cM_B(X)(\bC)$.
\item If $X$ is smooth we say $\Sigma$ is \emph{nonextendable} if the following condition is met.  For any finite \'etale cover $f:X'\to X$ and any partial log smooth compactification $X'\subset \bar X'$, if $f^*\Sigma$ extends to $\bar X'$ then $X'=\bar X'$.  
\item In general, we say $\Sigma$ is \emph{nonextendable} if for some (hence any) resolution $f:Y\to X$, $f^*\Sigma$ is nonextendable.
\end{enumerate}
\end{defn}

Note that any $\Sigma$ is nonextendable if $X$ is proper.

\begin{prop}
    Let $X$ be a connected normal algebraic space and $\Sigma\subset \cM_B(X)(\bC)$.  The following are equivalent.
    \begin{enumerate}
        \item\label{nonext1} $\Sigma$ is nonextendable. 
        \item\label{nonext2} For every proper morphism $f:C\to X$ from a smooth quasiprojective curve $C$, $f^*\Sigma\subset\cM_B(C)(\bC)$ is nonextendable. 
        \item\label{nonext3} For every morphism $f:\bD^*\to X^\an$ from the punctured disk which completes to a morphism from $\bD$ to some partial compactification of $X^\an$, $f^*\Sigma\subset\cM_B(\bD^*)(\bC)$ contains a non-trivial local system.
    \end{enumerate}
\end{prop}
\begin{proof}
    Nonextendability is unaffected by passing to a finite \'etale cover, and maps from disks and curves will lift to a resolution up to an \'etale cover of the source, so we may assume $X$ is smooth.  In the log smooth situation, a local system extends if and only if its local monodromy around every divisor is trivial, which can clearly be checked on curves and disks.
\end{proof}
\begin{cor}\label{pullback-nonextendable}
    Let $f:X\to Y$ be a proper morphism of connected normal algebraic spaces and $\Sigma\subset\cM_B(Y)(\bC)$.  Then if $\Sigma$ is nonextendable, so is $f^*\Sigma$.  If $f$ is in addition dominant, then $\Sigma$ is nonextendable if and only if $f^*\Sigma$ is.
\end{cor}
\begin{prop}\label{nonextendable_Zariski-closure}
    Let $\Sigma\subset \cM_B(X)(\bC)$ a subset, and $\Sigma^\Zar$ the $\bQ$-Zariski closure.  Then $\Sigma$ is nonextendable if and only if $\Sigma^\Zar$ is.
\end{prop}
\begin{proof}
 The subset $\Sigma$ is extendable if and only if there is a resolution  $Y\to X$ and a finite étale cover $Y'\to Y$ with composition $f:Y'\to X$ and a partial log smooth compactification $i:Y'\to \bar Y'$ such that $f^*\Sigma\subset i^*\cM_B(\bar Y')(\bC)$.  Since $i^*\cM_B(\bar Y')(\bC)$ is $\bQ$-Zariski closed in $\cM_B(Y')(\bC)$ the claim follows.
\end{proof}
\begin{prop}\label{extend to nonextend}  Let $X$ be a smooth algebraic space and $\Sigma\subset \cM_B(X)(\bC)$ of bounded rank.  Then there is a finite \'etale cover $\pi:X'\to X$ arising from a finite quotient of the image of the monodromy of $\bigoplus_{V\in\Sigma}V$, a partial log smooth compactification $i:X'\to \bar X'$, and a nonextendable $\bar \Sigma'\subset\cM_B(\bar X')$ such that $i^*\bar\Sigma'=\pi^*\Sigma$. 
\end{prop}
\begin{proof}
Taking Zariski closure by \Cref{nonextendable_Zariski-closure} and the sum of local systems corresponding to generic points of each irreducible component we may reduce to a single local system and this is \cite[Proposition 3.5]{brunebarbeshaf}.
\end{proof}
Note that in the previous proposition we can do without the finite \'etale cover if we allow a partial compactification by a Deligne--Mumford stack.  In fact, as a consequence of \Cref{mainShaf} such a partial compactification exists even when $X$ is only assumed to be normal.

Let $X$ be the complement of a normal crossing divisor in a smooth algebraic space $\bar X$. For every $x \in \bar X$, the local fundamental group of $X$ at $x$ (inside $\bar X$) is by definition the group $\pi_1(X, \bar X, x) := \pi_1(B \cap X, x^\prime)$,
where $B$ is a small euclidean ball in $\bar X$ around $x$ and $x^\prime \in B \cap X$ (since $\pi_1(B \cap X, x^\prime)$ is abelian, it does not depend of the choice of $x^\prime$).

\begin{lem}\label{characterization_nonextendability_local_fundamental_group}
Let $X$ be a smooth algebraic space and $\Sigma\subset \cM_B(X)(\bC)$. Let $\bar X \supset X$ be a log smooth compactification of $X$. Then, $\Sigma$ is nonextendable if and only if the induced representations of the local fundamental group $\pi_1(X, \bar X, x) $ at each point $x$ of $\bar X \setminus X$ are collectively injective.
\end{lem}
\begin{proof}
Indeed, for every $x \in \bar X$, every element of the local fundamental group $\pi_1(X, \bar X, x)$ is obtained as the image of a generator of $\pi_1(\Delta^\ast)$ for a holomorphic map $f:\bD^\ast\to X^\an$ from the punctured disk which completes to a holomorphic map from $\bD$ to $\bar X$ with $x = f(0) \in \bar X$. 
\end{proof}

\begin{prop}\label{nonextendable_reduction_to_finite}
Let $X$ be a smooth algebraic space and $\Sigma\subset \cM_B(X)(\bC)$. Let $\bar X \supset X$ be a log smooth compactification of $X$. Assume that $\Sigma$ contains only local systems with unipotent local monodromy.  If $\Sigma$ is nonextendable, then there exists a finite subset $\Sigma_0 \subset \Sigma$ which is nonextendable.
\end{prop}
Observe that \Cref{nonextendable_reduction_to_finite} is not true without any assumption on $\Sigma$, as shown by the example where $X = \bC^\ast$ and $\Sigma\subset \cM_B(X)(\bC)$ consists of the torsion rank one local systems.

\begin{proof}
Thanks to Proposition \ref{characterization_nonextendability_local_fundamental_group}, it is sufficient to find a finite subset $\Sigma_0 \subset \Sigma$ such that the induced representations of the local fundamental group at each point of $\bar X \setminus X$ are collectively injective. Since there are only finitely many strata, it is sufficient to look at what happen in the neighborhood of one point $x \in \bar X \setminus X$. For every $V \in \Sigma$, let $K^x_V$ be the kernel of the induced representation of the free abelian group of finite type $\pi_1(X, \bar X, x)$. Since the image of this representation is made of commuting unipotent matrices, $K^x_V$ is a saturated subgroup of $\pi_1(X, \bar X, x)$, i.e. the quotient group is torsion-free. By assumption, the intersection $\cap_{V \in \Sigma} K^x_V = \{0\}$. Since the $K^x_V$'s are saturated, it follows that there is a finite subset $\Sigma_0 \subset \Sigma$ such that $\cap_{V \in \Sigma_0} K^x_V = \{0\}$, finishing the proof.
\end{proof}


\subsection{Absolute constructible subsets}\label{sect:abssets}

In this section we show that the notion of bialgebraicity with respect to the Riemann--Hilbert correspondence as in \Cref{sect:bialg} is essentially equivalent to absoluteness, which is intrinsic to the Betti side. 

Let $X$ be a smooth complex algebraic space. Let $LS(X)$ be the abelian category of complex local systems on $X^{\an}$. Let $\mathrm{Conn}(X)^{\reg}$ be the abelian category of algebraic flat vector bundles on $X$ with regular singularities. Deligne's version of the Riemann-Hilbert correspondence states that the functor $\mathrm{Conn}(X)^{\reg} \to LS(X)$ which associates to an algebraic flat vector bundle with regular singularities $(V, \nabla)$ the locally constant sheaf of $\nabla$-flat sections of $V^{\an}$ is an equivalence of abelian categories \cite{Deligne_book}. For any field automorphism $\sigma \in \Aut(\bC/\bQ)$ and any complex scheme $Y$, let $Y^\sigma \to Y$ denote the base change of $Y$ via $\sigma$ (it is a morphism of schemes but not of complex schemes). By pulling-back along the morphism of $X^\sigma \to X$, we get an equivalence of categories $\mathrm{Conn}(X)^{\reg} \simeq \mathrm{Conn}(X^\sigma)^{\reg}$, and therefore a bijection between the isomorphism classes of objects in both categories. Therefore, any automorphism $\sigma \in \Aut(\bC/\bQ)$ induces a canonical bijection $\cM_B(X)(\bC) \simeq \cM_B(X^\sigma)(\bC)$. Given a subset $\Sigma \subset \cM_B(X)(\bC)$, we denote by $\Sigma^\sigma$ its image in $\cM_B(X^\sigma)(\bC)$ through this canonical bijection.

\begin{defn}[Simpson {\cite{Simpsonrankone}}, Budur-Wang {\cite{Budur-Wang}}]
Let $X$ be a smooth complex algebraic space. Let $K \subset \bC$ be a subfield. A subset $\Sigma  \subset \cM_B(X)(\bC)$ is absolute $K$-constructible (resp. absolute $K$-closed or absolute $K$-locally closed) if for any automorphism $\sigma \in \Aut(\bC/\bQ)$ the subset $\Sigma^\sigma \subset \cM_B(X^\sigma)(\bC)$ defined above is $K$-constructible (resp. $K$-closed or $K$-locally closed).
\end{defn}

If $X$ is merely a normal complex algebraic space, the open immersion $X^{\reg} \subset X$ induces a surjective homomorphism between their fundamental groups. Therefore, the induced morphism of algebraic stacks $\cM_B(X) \to \cM_B(X^{\reg})$ is a closed immersion.

\begin{defn}
Let $X$ be a normal complex algebraic space. Let $K \subset \bC$ be a subfield. 
\begin{enumerate}
    \item A subset $\Sigma  \subset \cM_B(X)(\bC)$ is absolute $K$-constructible (resp. absolute $K$-closed or absolute $K$-locally closed) if its image by the injective map $\cM_B(X)(\bC) \to \cM_B(X^{\reg})(\bC)$ is absolute $K$-constructible (resp. absolute $K$-closed or absolute $K$-locally closed).
    \item A morphism between (products of) Betti stacks is an absolute $K$-morphism if its graph is absolute $K$-constructible.
\end{enumerate}

\end{defn}

\begin{prop}\label{abs basic prop}
Let $X$ be a normal complex algebraic space. Let $K \subset \bC$ be a subfield. Then 
\begin{enumerate}
    \item $\cM_B(X)(\bC)$ is an absolute $K$-closed subset of $\cM_B(X)(\bC)$.
    \item Absolute $K$-constructible subsets of $\cM_B(X)(\bC)$ form a Boolean subalgebra.
    \item Pullback under algebraic morphisms, tensor products, and direct sums are absolute $K$-morphisms.
    \item Images and preimages of absolute $K$-constructible subsets under absolute $K$-constructible morphisms are absolute $K$-constructible.
\end{enumerate}
\end{prop}
\begin{proof}
Immediate.
\end{proof}

\begin{lem}\hspace{1em}\label{ssabs}
\begin{enumerate} 
    \item The semisimple locus $\cM_B(X)(\bC)^{\ss}\subset \cM_B(X)(\bC)$ is absolute $\bQ$-constructible.
    \item The maps $c_X:\cM_B(X)(\bC)\to M_B(X)(\bC)$, $p_X:\cM_B(C)(\bC)^{\ss}\to M_B(X)(\bC)$ and $\ss_X:\cM_B(X)(\bC)\to \cM_B(X)(\bC)$ are absolute $\bQ$-morphisms and commute with pull-back under any algebraic map $f:X\to Y$.
\end{enumerate}
    
\end{lem}
\begin{proof}
For (1), a local system is semisimple if and only if the associated flat bundle is semisimple, and semisimplicity is preserved under the Galois action.    For (2), clearly the morphism $c_X:\cM_B(X)(\bC)\to M_B(X)(\bC)$ is an absolute $\bQ$-morphism, and the absoluteness of $p_X$ and $\ss_X$ follows.  The final claim is obvious for $c_X$, and by the harmonic theory (see \Cref{pullback ss}) the pull-back of a semisimple local systems is semisimple, whence the remainder of the claim.
\end{proof}

\begin{prop}[Compare with {\cite[Theorem 6.3]{Simpsonrankone}}; see also {\cite[Proposition 7.4.6]{Budur-Wang}}]\label{criterion_algebraicity}
Let $X$ be a complex algebraic space defined over a countable field $L \subset \bC$. Let $S \subset X(\bC)$ be a subset such that as $\sigma$ runs through $\Aut(\bC/L)$, $S^\sigma$ runs through countably many subsets of $X(\bC)$. Suppose that $S$ is an analytically constructible subset of $X^{\an}$. 
Then $S$ is algebraically constructible in $X(\bC)$.
\end{prop}

Recall that an analytically constructible subset of a complex analytic variety is an element of the Boolean algebra generated by closed analytic subsets.

\begin{proof}
We follow closely the proof of \cite[Theorem 6.3]{Simpsonrankone}. 
Note that it is sufficient by induction on dimension to prove that the Euclidean closure $\bar S$ is Zariski closed in $X(\bC)$, and that is what we will prove.

Let us show that, up to enlarging $L$, one can assume that $S$ is stable by $\Aut(\bC/L)$. Let $\{S_j\}$ denote the set of subsets of $X(\bC)$ which occur as $S^\sigma$. For each pair $(j,k)$ with $S_j \neq S_k$, choose a point $x_{jk}$ in one of $S_j$ or $S_k$ but not the other. Then, among all the subsets, a given one is determined by the information of whether $x_{jk}$ is in it or not for all $j,k$. On the other hand, there are countably many points $x_{jk}$, so we may choose a countable field extension $L' /L$ so that $x_{jk}^\sigma = x_{jk}$ for every $\sigma \in \Aut(\bC/L')$. Then $S^\sigma = S$ for every $\sigma \in \Aut(\bC/L')$.

For every $x \in X(\bC)$, let $\overline{\{x\}}^L$ be the smallest closed algebraic subset of $X$ defined over $L$ and containing $x$.  For every $y\in \overline{\{x\}}^L(\bC)$, we have $\overline{\{y\}}^L\subset \overline{\{x\}}^L$ with equality if and only if $\overline{\{y\}}^L(\bC)$ meets the orbit $\Aut(\bC/L)\cdot x$.  Since $L$ is countable, there are only countably many strict suvarieties defined over $L$, so that their union is a meager subset of $\overline{\{x\}}^L(\bC)$. By the Baire category theorem, its complement $ \Aut(\bC/L)\cdot x$ is therefore dense in $\overline{\{x\}}^L(\bC)$ for the euclidean topology. Taking $x \in S$, one has $ \Aut(\bC/L)\cdot x \subset S$, hence $\overline{\{x\}}^L$ is contained in the euclidean closure $\bar S$ of $S$. 

Let $T$ be an irreducible component of $\bar S$. Then $T \cap S$ is covered by the closed subsets $\overline{\{x\}}^L \cap T \cap S$ for $x \in T \cap S$. Since there are countably many of those, it follows from Baire category theorem again that there exists $x \in T \cap S$ such that $\overline{\{x\}}^L$ contains a non-empty open subset of $T$. By analytic  continuation, we get that $T$ is equal to a geometric component of $\overline{\{x\}}^L$. 
\end{proof}

\begin{cor}[cf. {\cite[Proposition 7.4.5]{Budur-Wang}}]\label{algebraicity_of_irreducible_components}
Let $(\bar X,D)$ be a proper log smooth algebraic space and set $X=\bar X\setminus D$. Let $K \subset \bC$ be a countable subfield. If $\Sigma  \subset \cM_B(X)(\bC)$ is an absolute $K$-constructible subset, then any analytic irreducible component of the Euclidean closure of $RH^{-1}_{(\bar X,D)}(\Sigma)$ is a Zariski closed subset of $\cM_{DR}(\bar X,D)(\bC)$.
\end{cor}
\begin{proof}
Let $L$ be a countable field of definition of $\bar X$ and $D$, so that $\cM_{DR}(\bar X,D)$ is also defined over $L$. Therefore there exists $U$ a $L$-scheme and $U \to \cM_{DR}(\bar X,D)$ a smooth surjective morphism defined over $L$. Let $S = RH^{-1}_{(\bar X,D)}(\Sigma) $ and $T$ be the preimage of $S$ in $U(\bC)$. Since $\Sigma$ is absolute $K$-constructible and $K$ is countable, $T^\sigma$ runs through countably many subsets of $U(\bC)$ as $\sigma$ runs through $\Aut(\bC/L)$. Moreover, $T$ is the preimage by an analytic map of an algebraic (hence analytic) constructible subset of $\cM_B(X)(\bC)$, hence it is an analytic constructive subset of $U^{\an}$. Therefore, it follows from \Cref{criterion_algebraicity} that any analytic irreducible component of the Euclidean closure of $T$ in $U^{\an}$ is Zariski closed in $U(\bC)$. Conclude using the analytic version of \cite[\href{https://stacks.math.columbia.edu/tag/0DR5}{Tag 0DR5}]{stacks-project}.  
\end{proof}

\begin{prop}\label{Closure_absolute}
Let $X$ be a normal complex algebraic space, $K \subset \bC$ be a countable subfield, and $\Sigma  \subset \cM_B(X)(\bC)$ an absolute $K$-constructible.  Then
\begin{enumerate}
    \item For every $\sigma \in \Aut(\bC/\bQ)$ we have $\overline{\Sigma^\sigma} = \overline{\Sigma}^\sigma$ (where $\overline{(\cdot)}$ denotes euclidean closure). 
 In particular, $\bar \Sigma  \subset \cM_B(X)(\bC)$ is absolute $K$-closed. 
    \item \label{geometric irreducible component of absolute} If $\Sigma$ is absolute $K$-locally closed, then for every $\sigma \in \Aut(\bC/\bQ)$ the geometric irreducible components of $\Sigma^\sigma$ are the $\sigma$-translates of the geometric irreducible components of $\Sigma$.  In particular, each geometric irreducible component of $\Sigma$ is absolute $\bar K$-locally closed.
\end{enumerate}
\end{prop}
\begin{proof}
Replacing $X$ with $X^{\reg}$, one can assume that $X$ is smooth.  Let $(\bar X,D)$ be a proper log smooth algebraic space with $X=\bar X\setminus D$ (which we may even assume is projective, but this is not necessary).  Let 
\[RH_{(\bar X,D)}^\good:\cM_{DR}(\bar X,D)^\good\to\cM_B(X)^\an\]
be the restriction of the Riemann--Hilbert map, which is surjective and locally (on the source) an isomorphism by \Cref{local iso on good}.  Note also that:
\begin{enumerate}[label=(\roman*)]
    \item Pull-back under $RH_{(\bar X,D)}^\good$ commutes with euclidean closure.
    \item For any analytically constructible subset of $\cM_{DR}(\bar X,D)^\good$ whose euclidean closure in $\cM_{DR}(\bar X,D)$ is Zariski closed, the Galois action commutes with euclidean closure (in $\cM_{DR}(\bar X,D)^\good$).  This is because $\cM_{DR}(\bar X,D)^\good$ is Galois-stable and open, so inherits this property from the corresponding one on $\cM_{DR}(\bar X,D)$, which is clear.
    \item For any analytically closed subset of $\cM_{DR}(\bar X,D)^\good$ whose euclidean closure in $\cM_{DR}(\bar X,D)$ is Zariski closed, the $\sigma$-translates of components are precisely the components of the $\sigma$-translate.
\end{enumerate}
Thanks to \Cref{algebraicity_of_irreducible_components}, (i) and (ii) imply part (1).

For part (2), the geometric components of $\Sigma$ are the geometric components of $\bar \Sigma$ minus $\bar\Sigma\setminus \Sigma$, so by (1) we may assume $\Sigma$ is absolute $K$-closed.  For $\Sigma_0$ a geometric component of $\Sigma$, $(RH_{(\bar X,D)}^\good)^{-1}(\Sigma_0)$ is (globally) pure-dimensional and a union of components of $(RH_{(\bar X,D)}^\good)^{-1}(\Sigma)$.  Thus, by (iii) every component of $(RH_{(\bar X,D)}^\good)^{-1}(\Sigma_0)^\sigma=(RH_{(\bar X^\sigma,D^\sigma)}^\good)^{-1}(\Sigma_0^\sigma)$ is (globally) pure-dimensional and a union of components of $(RH_{(\bar X,D)}^\good)^{-1}(\Sigma^\sigma)$.  In particular, each component of $(RH_{(\bar X,D)}^\good)^{-1}(\Sigma_0)^\sigma$ dominates a component of $\Sigma^\sigma$.  In fact, every component of $(RH_{(\bar X,D)}^\good)^{-1}(\Sigma_0)^\sigma$ must dominate a single component $T_0$ of $\Sigma^\sigma$, since for any component $T_1\subset\Sigma^\sigma$, the locus of points $x\in\Sigma_0$ admitting a neighborhood $x\in U\subset \Sigma_0$ for which $U^\sigma\subset T_1$ is open, and such sets give an open partition of $T_0$.  Thus, $\Sigma_0^\sigma\subset T_0$, and applying the same argument to $\sigma^{-1}$ yields $T_0^{\sigma^{-1}}\subset\Sigma_0$, so $\Sigma_0^\sigma$ is closed.
\end{proof}

\begin{cor}\label{abs boolean}
Let $X$ be a normal complex algebraic space. Let $K \subset \bC$ be a countable subfield.
Every absolute $K$-constructible subset in $\cM_B(X)(\bC)$ is a Boolean combination of absolute $K$-closed subsets in $\cM_B(X)(\bC)$.
\end{cor}
\begin{proof}
If $S$ is an absolute $K$-constructible subset, then its closure $\bar S$ is an absolute
$K$-closed subset and the complement $\bar S \setminus S$ is an absolute $K$-constructible subset of $\cM_B(X)(\bC)$ of smaller dimension, hence the result follows by induction.
\end{proof}

\begin{cor}\label{bialg=abs}
    Let $X$ be a normal complex algebraic space and $K\subset\bC$ a countable field.  A $K$-Zariski closed subset $\Sigma_{B,1}\subset \cM_B(X)(\bC)$ is absolute $K$-closed if and only if it can be completed to a $K$-bialgebraic pair $(\Sigma_{DR},\{\Sigma_{B,\sigma}\})$ of the Riemann--Hilbert $\bQ$-correspondence stack $\cM_{RH,\bQ}(\bar X',D',E)$ for some (hence any) choice of log smooth resolution/compactification.
\end{cor}
\begin{proof}
    By replacing $X$ with $X^\reg$ we may assume $X$ is smooth. 
 On both sides of the implication, by definition $\Sigma_{B,1}$ is $K$-constructible.  We may assume $\Sigma_{B,1}$ is closed by \Cref{abs boolean}.  If $\Sigma_{B,1}$ is absolute $K$-closed, then taking $\Sigma_{DR}$ to be any component of $RH_{(\bar X,D)}^{-1}(\Sigma)$ meeting $\cM_{DR}(\bar X,D)^{\good}$, $(\Sigma_{DR},\{(\Sigma_{B,1})^\sigma\})$ is $K$-bialgebraic.  Conversely, if $(\Sigma_{DR},\{\Sigma_{B,\sigma}\})$ is $K$-bialgebraic, we may assume it is $K$-bialgebraic irreducible, and by taking Zariski closure we may assume all of $\Sigma_{DR},\Sigma_{B,\sigma}$ are closed and irreducible.  Then for each $\sigma$, $RH_{(\bar X^\sigma,D^\sigma)}(\Sigma_{DR}^\sigma)$ contains a euclidean open set $U_\sigma\subset\Sigma_{B,\sigma}$ which has nonempty interior.  Since the saturation $RH_{(\bar X^\sigma,D^\sigma)}^{-1}RH_{(\bar X^\sigma,D^\sigma)}(\Sigma^\sigma_{DR})$ is a countable union of algebraically constructible sets (as it is the set of all log connections whose meromorphic connection is isomorphic to the meromorphic connection of a point of $\Sigma_{DR}^\sigma$), it follows that for every $\sigma$ the euclidean closure in $\cM_{DR}(\bar X^\sigma,D^\sigma)$ of every component of $(RH_{(\bar X^\sigma,D^\sigma)}^\good)^{-1}(\Sigma_{B,\sigma})$ is algebraic.  Since $\Sigma_{DR}$ is a component of $RH_{(\bar X,D)}^{-1}(\Sigma_{B,1})$ and $\Sigma_{DR}^\sigma$ is a component of $RH_{(\bar X^\sigma,D^\sigma)}^{-1}(\Sigma_{B,\sigma})$ it follows that thet saturations are equal $RH_{(\bar X,D)}^{-1}(\Sigma_{B,1})=RH_{(\bar X^\sigma,D^\sigma)}^{-1}(\Sigma_{B,\sigma})$.  Thus, $\Sigma_{B,\sigma}=(\Sigma_{B,1})^\sigma$ and $\Sigma_{B,1}$ is absolute $K$-closed.
\end{proof}

\begin{cor}\label{qu dense in abs}
Let $X$ be a connected normal complex algebraic space. Let $\Sigma  \subset \cM_B(X)(\bC)$ be an absolute $\bar \bQ$-closed subset. Then 
\begin{enumerate}
    \item The locus of points $\Sigma^{\qu}\subset\Sigma$ with quasiunipotent local monodromy is Zariski dense in $\Sigma$.
    \item For each $n$, the locus $\Sigma^{\qu|n}\subset\Sigma$ with eigenvalues of local monodromy of order dividing $n$ is absolute $\bar\bQ$-closed.
\end{enumerate}

\end{cor}
\begin{proof}
    The first part is by \Cref{qu dense in bialg}, and the second part follows from \Cref{qu is bialg} or can be seen directly.
\end{proof}

\subsection{Graded nearby cycle functors}

\subsubsection{The Betti side}
Let $(\bar X,D)$ be a log smooth complex manifold and set $X := \bar X \setminus D$. Let $D_k$ be a (smooth) irreducible component of $D$. Let $D^k$ denotes $\cup_{j \neq k} D_j$. Then $D_k \cap D^k$ is a simple normal crossing divisor in $D_k$. Let $U_k$ be a tubular neighborhood of $D_k$ in $\bar X$. A fortiori, $U_k \setminus D^k$ is a tubular neighborhood of $D_k \setminus D^k$, and we have an exact sequence of groups:
\[ 1 \to K \to \pi_1(U_k \setminus  D) \to \pi_1(U_k \setminus  D^k) \to 1. \]
The group $K$ is canonically isomorphic to $\bZ$ and we denote by $\gamma_k \in \pi_1(U_k \setminus  D)$ its canonical generator (which corresponds geometrically to the class of a simple loop going around $D_k$ counterclockwise). Moreover, $K$ is central in $\pi_1(U_k \setminus  D)$, hence it acts on any local system defined on $U_k \setminus D$. Thanks to Deligne's version of the Riemann–Hilbert correspondence, the group $K$ acts also on any regular meromorphic connection on the pair $(U_k, U_k \cap D)$.

Recall the following well-known result.

\begin{lem}[see {\cite[Proposition 1.6.1]{Deligne-WeilII}}]\label{Deligne_weight_filtration}
Let $\cA$ be an abelian category. For each object
$Z$ of $\cA$ equipped with a nilpotent endomorphism $N$, there exists a unique finite increasing filtration $W_\bullet = W(N)_\bullet$ of $Z$ satisfying the following conditions:
\begin{enumerate}
    \item $N(W_k) \subset W_{k-2}$ for every $k$,
    \item $N^k$ induces an isomorphism $\gr_k^W Z\simeq \gr_{-k}^WZ$ for every $k \geq 0$.
\end{enumerate}
Moreover, the association $(Z,N) \mapsto (Z, W(N)_\bullet)$ defines a functor between the category of objects $Z$ of $\cA$ equipped with a nilpotent endomorphism $N$, a morphism between $(Z_1, N_1)$
and $(Z_2, N_2)$ being a morphism $f \colon Z_1 \to Z_2$ in $\cA$ such that $f \circ N_1 = N_2 \circ f$, and the category of objects $Z$ of $\cA$ equipped with a finite increasing filtration $W_\bullet$, the morphisms being the morphisms preserving the filtrations.
\end{lem}

Let $\cM_B^\unip(X)$ denote the closed substack of $\cM_B(X)$ consisting in local systems with unipotent local monodromy. For every $k$, we have the graded nearby cycle functor $$\gr\psi_{D_k} \colon \cM_B^\unip(X) \to \cM_B^\unip(D_k \setminus D^k)$$ defined as follows. If $A$ is a ring and $V$ is a $A$-local system on $X$ with unipotent local monodromy, then $\gamma_k$ induces a unipotent automorphism $T$ of its restriction $V_{|U_k \setminus D}$ to $U_k \setminus D$. Let $N := T - \mathrm{Id}$ be the induced nilpotent endomorphism of $V_{|U_k \setminus D}$ and $W_\bullet$ be the associated filtration in the abelian category $\cM_B^\unip(U_k \setminus D)(A)$. Then $\gr^W V_{|U_k \setminus D}$ is a $A$-local system on $U_k \setminus D$ with no monodromy around $D_k$. Therefore it extends uniquely as a $A$-local system on $U_k \setminus D^k$, whose restriction to $D_k \setminus D^k$ is by definition the graded nearby cycles $\gr \psi_{D_k}(V)$ along $D_k$. Note that $\gr \psi_{D_k}(V)$ has indeed unipotent local monodromies around the irreducible components of $D_k \cap D^k$.

\begin{thm}[Mochizuki]\label{graded_nearby_cycle_Mochizuki}
Assume that $(\bar X,D)$ is a projective log smooth variety and set $X=\bar X\setminus D$. Let $K$ be a field of characteristic zero. Let $V$ be a $K$-local system on $X$ with unipotent local monodromy. If $V$ is semisimple, then the $K$-local system $\gr\psi_{D_k}(V)$ on $D_k \setminus D^k$ is semisimple.    
\end{thm}
\begin{proof}
Since the fundamental group of $X$ is finitely generated, there exists a finitely generated field $F$ and a $F$-local system $V_F$ such that $V_F \otimes_F K = V$. Since on the one hand the property of being semisimple is invariant by changing coefficients, and on the other hand every finitely generated field can be embedded in $\bC$, one can assume that $K = \bC$. When $K = \bC$, this is \cite[Theorem 19.49]{Mochizuki-AMS2}. See also \cite{Sabbah_twistor_D_modules}.
\end{proof}

\subsubsection{The De Rham side}
We briefly recall the construction of a graded nearby cycle functor for logarithmic connections and refer to \cite{Brunebarbe_Sym, Langer-JEMS} for the details.

Let $\bar X$ be a smooth projective variety defined over an algebraically closed
field $k$. Let $D$ be a simple normal crossing divisor on $\bar X$ and let $Y$ be an irreducible component of $D$. Let $\iota\colon Y \to \bar X$ be the canonical embedding.

Let $\bL_Y= \left(L, [,], \alpha\right)$ be the Lie algebroid on $Y$, where $L = \iota^\ast T_X(-\log D)$ is equipped with the Lie algebra structure induced from the standard Lie algebra structure on $T_{\bar X}$ and the anchor map $\alpha \colon L \to \Der_k(\cO_Y) = T_Y$ is the canonical map induced by $\iota$. The restriction to $Y$ defines an obvious functor $(E, \nabla) \mapsto \iota^\ast (E, \nabla)$ from the abelian category $\cM_{DR}(\bar X, D)(k)$ to the category $\bL_Y$-Mod of coherent $\cO_Y$-modules with an $\bL_Y$-module structure. Moreover, the residue of $(E, \nabla)$ along $Y$ yields an endomorphism of $\iota^\ast (E, \nabla)$ seen as an element of $\bL_Y$-Mod. When the residue is nilpotent, we get from Lemma \ref{Deligne_weight_filtration} an increasing filtration $W_\bullet$ of $\iota^\ast (E, \nabla)$. By definition, the graded nearby functor of $(E, \nabla) \in \cM_{DR}^{\nilp}(\bar X, D)(k)$ is $\gr \psi_Y (E, \nabla) := \gr^W \iota^\ast (E, \nabla)$. It is an element of $\cM_{DR}^{\nilp}(Y, D_Y)$.

\subsubsection{Compatibilities}
The following proposition is a straightforward consequence of the constructions.
\begin{prop}\label{graded_nearby_cycle_commutation}
Let $D$ be a simple normal crossing divisor in a smooth projective complex algebraic variety $\bar X$. Let $D_k$ be a (smooth) irreducible component of $D$.
\begin{enumerate}
    \item The diagram
     \[
 \begin{tikzcd}
     \cM_{DR}^{\nilp}(\bar X, D)(\bC)  \ar[r," \gr \psi_{D_k}"]  \ar[d,"RH"]& \cM_{DR}^{\nilp}(D_k, D^k)(\bC)  \ar[d, "RH"]\\
     \cM_B^\unip(X)(\bC) \ar[r, " \gr \psi_{D_k}"]& \cM_B^\unip(D_k \setminus D^k)(\bC)
 \end{tikzcd}
 \]  
    is commutative.
    \item For every field automorphism $\sigma \in \Aut(\bC/\bQ)$ and every $(E, \nabla) \in \cM_{DR}^{\nilp}(\bar X, D)(\bC)$, one has
    \[\gr\psi_{D_k^\sigma}((E, \nabla)^\sigma) = \left(\gr\psi_{D_k}((E, \nabla))\right)^\sigma.\]
\end{enumerate}
\end{prop}

\begin{cor}\label{nearby cycles is abs}
    In the above setup, $\gr\psi_{D_k} \colon \cM_B^\unip(X) \to \cM_B^\unip(D_k \setminus D^k)$ is an absolute $\bQ$-morphism.
\end{cor}

\subsection{Absolute Hodge subsets}

\begin{defn}\label{abs Hodge better}
    Let $X$ be a connected normal algebraic space.  An \emph{absolute Hodge substack} $\cZ\subset\cM_B(X)$ is a Hodge substack for which $\cZ(\bC)$ is $\bQ$-locally closed and absolute $\bar\bQ$-locally closed.  An \emph{absolute Hodge subset} $\Sigma\subset\cM_B(X)(\bC)$ is a set of points $\Sigma=\cZ(\bC)$ underlying an absolute Hodge substack $\cZ\subset\cM_B(X)$.  

\end{defn}

By \Cref{abs basic prop}, absolute Hodge substacks of $\cM_B(X)$ satisfy the same functoriality properties as Hodge substacks, see \Cref{basic prop hodge substack}:
\begin{prop}\label{basic prop abs hodge substack}
        Let $X$ be a connected normal algebraic space.

    \begin{enumerate}
        \item For any $r$, $\cM_B(X,r)$ and $\{\triv_r\}$ are absolute Hodge substacks of $\cM_B(X)$.
        \item 
        Assume $X$ is a curve.  For any $V\in \cM_B(X)(\bC)^{\qu,\ss}$, the fixed residual eigenvalues leaf $FE(V)\subset\cM_B(X)$ is a closed absolute Hodge substack.
        \item Intersections and reductions of absolute Hodge substacks are absolute Hodge substacks.  $\bQ$-irreducible components of reduced absolute Hodge substacks are absolute Hodge substacks.  Finite unions of closed absolute Hodge substacks are absolute Hodge substacks.   
        \item Let $f:X\to Y$ be a morphism of connected normal algebraic spaces and $f^*:\cM_B(Y)\to \cM_B(X)$ the pullback morphism.  Then:
        \begin{enumerate}
           \item For any absolute Hodge substack $\cZ\subset \cM_B(X)$, $(f^*)^{-1}(\cZ)\subset \cM_B(Y)$ is an absolute Hodge substack.
        \item For any absolute Hodge substack $\cZ\subset \cM_B(Y)$, $f^*\cZ\subset \cM_B(X)$ is an absolute Hodge substack provided $f$ is dominant or a Lefschetz curve (see \Cref{defn lefschetz})
        \end{enumerate}
        \item Inverse images of absolute Hodge substacks under the direct sum and tensor product morphisms $\oplus,\otimes:\cM_B(X)^2\to\cM_B(X)$ are absolute Hodge substacks. 
 
\end{enumerate}
\end{prop}

\begin{cor}\label{every component contains Hodge}
    Every irreducible component of a closed absolute Hodge substack contains a point underlying a $\CVHS$.
\end{cor}
\begin{proof}
By \Cref{qu dense in abs}, \Cref{abs Hodge contain R*}, and \Cref{comp has R* fixed}.    
\end{proof}

For completeness we give a notion of absolute Hodge subset suited to the semisimple quasiunipotent local monodromy case, which is substantially easier since we do not have to deal with the miniversal families.
\begin{defn}
   Let $X$ be a connected normal algebraic space.  A \emph{coarse absolute Hodge subset} $\Sigma\subset\cM_B(X)(\bC)^{\ss,\qu}$ is a $\bQ$-constructible absolute $\bar\bQ$-constructible $\bR_{>0}$-stable subset. 
\end{defn}
We sometimes refer to subsets of $M_B(X)(\bC)$ being coarse absolute Hodge subsets, by which we mean the corresponding set of semisimple local systems is coarse absolute Hodge.
\begin{prop}\label{basic prop coarse abs hodge}
        Let $X$ be connected normal algebraic space.

    \begin{enumerate}
        \item Let $\cZ\subset\cM_B(X)$ be a closed absolute Hodge substack.  Then for each $n$, $\cZ^{\qu|n}(\bC)^\ss\subset M_B(X)$ is a coarse absolute Hodge subset.
        \item Intersections and finite unions of coarse absolute Hodge subsets are coarse absolute Hodge subsets.
        \item Let $f:X\to Y$ be a morphism of connected normal algebraic spaces and $f^*:\cM_B(Y)\to\cM_B(X)$ the pullback morphism.  Then:
        \begin{enumerate}
       
        \item For any coarse absolute Hodge subset $\Sigma\subset \cM_B(X)(\bC)$, $(f^*)^{-1}(\Sigma)^{\ss,\qu|n}\subset \cM_B(Y)(\bC)$ is a coarse absolute Hodge subset.
        \item For any coarse absolute Hodge subset $\Sigma\subset \cM_B(Y)(\bC)^{\ss,\qu}$, $f^*\Sigma\subset \cM_B(X)(\bC)$ is a coarse absolute Hodge subset.
      
        \end{enumerate}
    \end{enumerate}
\end{prop}
\begin{proof}
    Immediate from \Cref{abs Hodge contain R*}, \Cref{abs basic prop}, \Cref{pullback ss}, and the functoriality of the $\bR_{>0}$-action.
\end{proof}

Finally, we give an ad hoc definition that is not particularly well-behaved, but collects the properties we will need for the proofs in \Cref{sect:qu stein} and \Cref{sect:proofs}.  
\begin{defn}\label{defn:abs hodge}
Let $X$ be a connected normal algebraic space.  A \emph{weak absolute Hodge subset} $\Sigma\subset\cM_B(X)(\bC)$ is a $\bQ$-constructible absolute $\bar\bQ$-constructible subset such that:

\begin{enumerate}
       \item\label{R* condition} 
Each geometric irreducible component of $\Sigma^\Zar$ contains a point underlying a $\bC$-VHS with quasiunipotent local monodromy.
       \item\label{mini condition} For each geometric irreducible component $\Sigma_0$ of $\Sigma^\Zar$ which is not generically semisimple with quasiunipotent local monodromy, let $\cM_{\Sigma_0}\subset\cM_B(X)$ be the reduced substack with underlying set of points $ \Sigma_0$.  Then $\cM_{\Sigma_0}$ has a point underlying a $\bC$-VHS with quasiunipotent local monodromy at which it is formally Hodge as in \Cref{defn formally Hodge}.
    \item There exists a $\bQ$-constructible absolute $\bar\bQ$-constructible subset $\Sigma_0\subset \Sigma^{\qu,\ss}$ satisfying (1) and containing the $\bC$-VHS points of $\Sigma$ from (2).
\end{enumerate}
\end{defn}

Note that by definition a weak absolute Hodge subset has bounded rank.  Note also that a weak absolute Hodge subset $\Sigma\subset\cM_B(X)(\bC)^{\qu,\ss}$ is just a $\bQ$-constructible absolute $\bar\bQ$-constructible subset satisfying (1).

\begin{prop}\label{abs Hodge implies weak abs Hodge}\hspace{1in}
\begin{enumerate}
    \item For any coarse absolute Hodge subset $\Sigma\subset \cM_B(X)(\bC)$, there is a finite \'etale cover $p:X'\to X$ corresponding to a finite quotient of the image of the monodromy representation of $\bigoplus_{V\in\Sigma} V$ such that $p^*\Sigma$ is a weak absolute Hodge subset.
    \item For any closed absolute Hodge substack $\cZ\subset \cM_B(X)$, $\cZ(\bC)$ is a weak absolute Hodge subset.
\end{enumerate}
\end{prop}
\begin{proof}
    The first claim is immediate from by the density of the quasiunipotent locus (\Cref{qu dense in abs}) and \Cref{R* stable has fixed} combined with \Cref{reduction_to_unipotent_monodromy}.  The second part follows from \Cref{qu dense in abs}, \Cref{basic prop abs hodge substack}, and \Cref{every component contains Hodge}.
\end{proof}

\begin{cor}\label{MB is weak abs hodge} For each $r$, $\cM_B(X,r)(\bC)$ is a weak absolute Hodge subset of $\cM_B(X)(\bC)$.
\end{cor}

\begin{prop}\label{abs Hodge prop}
    Let $X$ be a connected normal algebraic space.

    \begin{enumerate}
    
                \item Finite unions and closures of weak absolute Hodge subsets are weak absolute Hodge.
                \item $\bQ$-irreducible components of weak absolute Hodge subsets are weak absolute Hodge.
               
        \item Let $f:X\to Y$ be a morphism of connected normal algebraic spaces and $f^*:\cM_B(Y)\to\cM_B(X)$ the pullback morphism.  Then for any weak absolute Hodge subset $\Sigma\subset \cM_B(Y)(\bC)$, $f^*\Sigma\subset \cM_B(X)(\bC)$ is a weak absolute Hodge subset.
 
    \end{enumerate}
\end{prop}

\begin{proof}
Immediate given \ref{abs basic prop}, \Cref{lem: prop Hodge sub of Betti}, the functoriality of the $\bR_{>0}$-action, and the fact that all operations are defined over $\bQ$.  
\end{proof}

\subsection{Generalities on Shafarevich morphisms}
\begin{defn}
        Let $X$ be a connected normal algebraic space (or more generally a connected generically inertia-free normal Deligne--Mumford stack) and $\Sigma\subset\cM_B(X)(\bC)$ a set of complex local systems.  We say $\Sigma$ is large if for every non-constant morphism $g:Z\to X$ from a connected normal algebraic space $Z$, $g^*\Sigma$ contains a nontrivial local system.
\end{defn}

\begin{lem}\label{algebraic space on cover}
 Let $X$ be a connected generically inertia-free normal Deligne--Mumford stack and $\Sigma\subset\cM_B(X)(\bC)$ a set of complex local systems. Assume that for every point $x\in X(\bC)$ the inertia of $x$ acts faithfully on $\bigoplus_{V\in \Sigma} i^*_x V$. Then there exists a connected finite Galois \'etale cover $p:X'\to X$ with Galois group $G$ arising from a finite quotient $\pi_1(X,x)\to \img\rho_\Sigma\to G$ such that $X'$ is an algebraic space.
\end{lem}
\begin{proof}
For any point $x\in X$, there is a finite set of local systems $\Sigma_x\subset\Sigma$ such that the inertia of $x$ acts faithfully on $\bigoplus_{V\in\Sigma_x} V$.  By Selberg's lemma\footnote{Selberg's lemma states that any finitely generated subgroup of $\bGL_r(\bC)$ has a torsion-free finite-index subgroup.}, we obtain a cover of the required form trivializing the inertia at points above $x$.  By noetherianity (since the inertia sheaf is constructible), such a cover exists trivializing all of the inertia.
\end{proof}

We say that a set of complex local systems $\Sigma$ on a connected generically inertia-free normal Deligne--Mumford stack $X$ is nonextendable if there exists a connected finite Galois \'etale cover by an algebraic space $p:X'\to X$ such that $p^* \Sigma$ is not extendable. This generalizes the definition for algebraic spaces thanks to \Cref{pullback-nonextendable}.

\begin{defn}\label{defn:shaf}
    Let $X$ be a connected normal algebraic space and $\Sigma\subset\cM_B(X)(\bC)$ a set of complex local systems.  A representable morphism $f:X\to Y$ to a connected generically inertia-free normal Deligne--Mumford stack $Y$ is an \emph{(algebraic) $\Sigma$-Shafarevich morphism} if: \begin{enumerate}
\item
$s:X\to Y$ is dominant and $K(Y)$ is algebraically closed in $K(X)$.
\item $\Sigma$ is the pull back of a large nonextendable $\Sigma_Y\subset\cM_B(Y)(\bC)$ and for every point $y\in Y(\bC)$ the inertia of $y$ acts faithfully on $\bigoplus_{V\in \Sigma_Y} i^*_yV$. 
\item If a morphism $g:Z\to X$ from a connected $Z$ has the property that $g^*V$ is trivial for every $V\in\Sigma$, then the composition $Z\to X\to Y$ factors through $Z\to\Spec\bC$.
\end{enumerate}
\end{defn}
In particular, a $\Sigma$-Shafarevich morphism contracts exactly those subvarieties on which the local systems in $\Sigma$ have uniformly finite monodromy. Also, there is always a natural \'etale cover $\tilde Y^{\Sigma_Y}\to Y$ by an analytic space and a diagram
\[\begin{tikzcd}
\tilde X^\Sigma\ar[r]\ar[d]&\tilde Y^{\Sigma_Y}\ar[d]\\
X^\an\ar[r]&Y^\an
\end{tikzcd}\]
where the top horizontal map is equivariant with respect to the action of the natural deck transformation groups.

Provided they exist, $\Sigma$-Shafarevich morphisms are functorial in the following sense: 

\begin{prop}\label{shaf functo}
Let $f:X\to X'$ be a morphism between two normal connected algebraic spaces, $\Sigma\subset \cM_B(X)(\bC)$ and $\Sigma'\subset \cM_B(X')(\bC)$ two subsets such that $f^*\Sigma'\subset\Sigma$. Suppose $s:X\to Y$ (resp. $s':X'\to Y'$) is a $\Sigma$-Shafarevich morphism (resp. $\Sigma'$-Shafarevich morphism).  Then, there exists a unique morphism $g:Y\to Y'$ such that we have a commutative diagram
\[
\begin{tikzcd}
    X\ar[r,"f"]\ar[d,"s",swap]&X'\ar[d,"s'"]\\
    Y\ar[r,"g"]&Y'.
\end{tikzcd}
\]  
\end{prop}
In particular, it follows that a Shafarevich morphism is unique up to unique isomorphism provided it exists. Before proving this, we state the following lemma, which says that the existence of a Shafarevich morphism can be checked on certain finite \'etale covers, and in fact there is always such a cover where the target is an algebraic space.  For any $\Sigma\subset \cM_B(X)(\bC)$, define $\rho_\Sigma$ to be the monodromy representation of $\bigoplus_{V\in\Sigma} V$.
\begin{lem}\label{Shaf pass to finite}Let $X$ be a connected normal algebraic space and $\Sigma\subset\cM_B(X)(\bC)$.
   \begin{enumerate}

   \item Suppose an algebraic $\Sigma$-Shafarevich morphism $s:X\to Y$ exists.  Given a connected finite Galois \'etale cover $q:Y'\to Y$ coming from a finite quotient $\pi_1(Y,y)\to \img\rho_{\Sigma_Y}\to G$, the base-change $s':X'\to Y'$ of $s$ is an algebraic $p^*\Sigma$-Shafarevich morphism, where $p:X'\to X$ is the base-change of $q$. 

   \item Given a connected finite Galois \'etale cover $p:X'\to X$ with Galois group $G$ arising from a finite quotient $\pi_1(X,x)\to \img\rho_\Sigma\to G$, and a $p^*\Sigma$-Shafarevich morphism $s':X'\to Y'$, then $s'$ is equivariant with respect to a $G$-action on $Y'$, and the quotient $s:X\to [G\backslash Y']$ is a $\Sigma$-Shafarevich morphism.

\end{enumerate} 
\end{lem}

\begin{proof}[Proof of \Cref{shaf functo} and \Cref{Shaf pass to finite}] Part (1) of the lemma is straightforward.  

We prove part (2) of the lemma and the proposition together in a sequence of more general cases.  We first prove \Cref{shaf functo} in the case $f$ is an open embedding, $\Sigma=f^*\Sigma'$, and $Y,Y'$ are algebraic spaces; we will prove that $g$ exists and is an isomorphism.  Let $X\to Z\to Y$ be a factorization of $s$ for which $X\to Z$ is an open embedding and $Z\to Y$ is proper; likewise $X\to Z'\to Y'$.  If $Z''$ is the normalization of the closure of $X$ in $Z\times Z'$ via the diagonal map, then since the composition $Z''\to Z\times Z'\to Y\times Y'$ is proper, the pullback of $\Sigma_Y \boxtimes \Sigma_{Y'}$ is nonextendable.  But this pullback is the pullback of $\Sigma_Y$ (resp. $\Sigma_{Y'}$) along $Z''\to Y$ (resp. $Z''\to Y'$), so it follows that both of these maps are proper.  On the other hand, $Z''\to Y$ and $Z''\to Y'$ have the same fibers, hence by the rigidity lemma \cite[Lemma 1.15]{debarre} there is a unique isomorphism $g:Y\to Y'$ making the obvious diagram commute.

 We now prove \Cref{shaf functo} when $Y,Y'$ are algebraic spaces.  Let $X\to  X_1 \to Y'$ be a relative compactification of the composition $s'\circ f$ and $ X_1\to Z\to Y'$ the Stein factorization.  Then $X\to Z$ is a $f^*\Sigma$-Shafarevich morphism.  Thus we may assume $X=X'$ and $f$ is the identity.  Let $X\to  X_2\to Y\times Y'$ be a relative compactification of $s\times s':X\to Y\times Y'$ with $X_2$ normal.  As in the previous paragraph, $X_2\to Y$ is proper, and $X_2\to Y'$ contracts the fibers of $X_2\to Y$, so by the rigidity lemma there is a factorization $X_2\to Y\to Y'$. 

We now prove \Cref{Shaf pass to finite}(2) provided $Y'$ is an algebraic space.  By the algebraic space case of \Cref{shaf functo} proven above, $G$ acts on the map $s'$, and the stack quotient is easily checked to be the $\Sigma$-Shafarevich morphism.  

We now prove \Cref{shaf functo} in general.  By \Cref{Shaf pass to finite}(1) base-changing along a finite \'etale Galois cover of $Y'$ yields the same setup, and by \Cref{algebraic space on cover} we may take such a cover for which $Y'$ is an algebraic space.  By the algebraic space case of \Cref{Shaf pass to finite}(2), this will be sufficient to prove the original claim by taking stack quotients.  Thus, we may assume $Y'$ is an algebraic space.  Again by \Cref{algebraic space on cover} and \Cref{Shaf pass to finite}(1), there is a finite \'etale cover $Y''\to Y$ with $Y''$ an algebraic space and Galois group $G$ such that, if $p:X''\to X$ is the base-change, the base-change $s'':X''\to Y''$ of $s$ is a $p^*\Sigma$-Shafarevich morphism.  By the algebraic space case of \Cref{shaf functo}, the morphism $X''\to Y'$ factors through a morphism $Y''\to Y'$ which is moreover $G$-invariant, and it follows that the required morphism $Y\to Y'$ exists.

Finally, the general case of \Cref{Shaf pass to finite}(2) follows by the same argument as above now using \Cref{shaf functo} for stacks.
\end{proof}

\begin{defn}
Let $X$ be a connected normal algebraic space and $\Sigma\subset \cM_B(X)(\bC)$. The rank $r$ Shafarevich saturation (or just $r$-saturation for short) of $\Sigma$ is the subset $\Sigma^{r\text{-}\sat}\subset \cM_B(X,\rk\leq r)(\bC)$ of those local systems of rank $\leq r$ which pullback trivially along any morphism $f:Z\to X$ for which $f^\ast V$ is trivial for all $V \in \Sigma$.  That is,
    \[\Sigma^{r\text{-}\sat}:=\bigcap_{\substack{f:Z\to X\\f^*\Sigma\subset\triv_Z}}(f^*)^{-1}(\triv_Z)\hspace{.25in}\mbox{in}\hspace{.25in}\cM_B(X,\rk\leq r)(\bC).\]

A subset $\Sigma \subset \cM_B(\bC)$ which is equal to its rank $r$ saturation  is called \textit{$r$-saturated}.
\end{defn}

The main point of this definition is the following:
\begin{cor}\label{Shaf reduce to wabs}
    Let $X$ be a connected normal algebraic space, $\Sigma\subset \cM_B(X)(\bC)$ a subset of local systems of bounded rank, and $s:X\to Y$ a morphism to a generically inertia-free Deligne--Mumford stack.  Then 
    \begin{enumerate}
        \item For any $r$, $\Sigma^\mathrm{r\text{-}sat}$ underlies a closed absolute Hodge subset.  In particular, it is a weak absolute Hodge subset.
        \item For $r\gg0$, $s$ is a $\Sigma$-Shafarevich morphism if and only if it is a $\Sigma^\mathrm{r\text{-}sat}$-Shafarevich morphism.
    \end{enumerate}

\end{cor}
\begin{proof}
    Part (1) is by \Cref{basic prop abs hodge substack} and \Cref{abs Hodge implies weak abs Hodge}.  Part (2) is clear.
\end{proof}

\section{Pluriharmonic local systems over non-archimedean local fields}\label{sect:norms}
For the proof of the algebraic integrability of the Katzarkov--Zuo foliation in the next section, we shall need to understand the graded nearby cycles functor in the category of non-archimedean pluriharmonic bundles.  More precisely, let $K$ be a non-archimedean field and $V$ a semisimple $K$-local system with unipotent local monodromy on a log smooth algebraic space $(\bar X,D)$.  Associated to a pluriharmonic non-archimedean norm on $V$ is a closed positive $(1, 1)$-current $\omega_{\bar X}$ with continuous potential on $\bar X$ which turns out to be independent of the choosen pluriharmonic norm on $V$.  On the other hand, for each boundary component $D_0$ there is a graded nearby cycles local system $\gr\psi_{D_0}V$, which likewise has its own closed positive $(1,1)$-current $\omega_{D_0}$; we show $\omega_{D_0}$ is the restriction of $\omega_{\bar X}$.  This is a corollary of a similar statement for the characteristic polynomials, see \Cref{characteristic_polynomial_nearby_cycle}.

\subsection{Generalities on non-archimedean norms}
References for this section include \cite{Goldman-Iwahori, Gerardin, Bruhat-Tits-SMF, Boucksom-Eriksson}.

In all this section, $K$ will be a non-archimedean local field with absolute value $|\cdot|\colon K \to \bR_{\geq 0}$, i.e. either a finite extension of $\bQ_p$ or $\mathbb{F}_p((T))$ for a prime $p$.  Write $|K^\ast | = q^{\bZ}$ for a positive integer $q$.

\begin{defn}
Let $V$ be a finite dimensional $K$-vector space. A norm on $V$ is a function $\| \cdot \|\colon V \to \bR_{\geq 0}$ such that
\begin{enumerate}
    \item $\|v \| = 0  \Longleftrightarrow v= 0$;
    \item $\| av \| = |a| \|v\|$ for all $a \in K, v \in V$;
    \item $\|v + w \| \leq \max\{\|v\|, \|w\|\}$ for all $v, w \in V$.
\end{enumerate}
We write $\cN(V)$ for the set of norms on $V$.
\end{defn}

Note that $\|v + w \| = \max\{\|v\|, \|w\|\}$ for all $v, w \in V$ such that $\|v\| \neq \|w\|$. Indeed
$\| w \| = \| (v + w ) - v \| \leq  \max\{\|v + w \|, \|v\|\} \leq  \max\{\|v\|, \|w\|\}$. If $\|v\| < \|w\|$, it follows that $\| w \| =  \max\{\|v + w \|, \|v\|\} $, and so $\| w \| =  \|v + w \|$.\\

Let $V$ be a finite dimensional $K$-vector space. Let $\| \cdot \|$ be a norm on $V$. A basis $\{v_1,\ldots,v_n\}$ of $V$ such that $\| \sum_i a_i v_i \| = \max_i (  \|a_i v_i\|)$ for every $a_1, \ldots, a_n \in K$ is said to be orthogonal for $\| \cdot \|$. Every norm admits an orthogonal basis \cite[Proposition 1.1]{Goldman-Iwahori}, and every two norms always have a common orthogonal basis \cite[Proposition 1.3]{Goldman-Iwahori}. \\

If $W \subset V$ is a sub-$K$-vector space, then the restriction of $\| \cdot \|$ to $W$ is a norm on $W$. There is also an induced norm on $V / W$ by letting $\| x + W \| := \inf_{y \in x + W} \| y \|$ for every $x \in V$.\\

The dual $V^\vee$ of $V$ is endowed with the dual norm defined as:
\[ \|\lambda \| := \sup_{v \in V \setminus \{0 \} }\frac{|\lambda(v)|}{\|v\|} .\]

If $\{v_1,\ldots,v_n\}$ is an orthogonal basis for $\| \cdot \|$, then its dual basis $\{v_1^\vee,\ldots,v_n^\vee\}$ is an orthogonal basis for the dual norm and $\|v_i^\vee \| = \| v_i \|^{-1}$ for every $i$.\\

Finally, for every positive integer $r$, there is a unique induced norm $\wedge^r \| \cdot \|$ on $\wedge^r V$ such that, if $\{v_1,\ldots,v_n\}$ is an orthogonal basis for $\| \cdot \|$ on $V$, then $\{v_{i_1} \wedge \ldots \wedge v_{i_r}\}_{i_1 < \ldots < i_r}$ is an orthogonal basis for $ \wedge^r\| \cdot \|$ on $\wedge^r V$ and $\wedge^r\|v_{i_1} \wedge \ldots \wedge v_{i_r} \| = \|v_{i_1} \| \cdots \| v_{i_r} \|$ for every $i_1 < \ldots < i_r$, see \cite[Proposition 3.9]{Goldman-Iwahori}.

\begin{lem}
Let $V$ be a $K$-vector space equipped with a norm $\| \cdot\|$. Let $v \in V\setminus \{0 \}$ and $\lambda \in V^\vee\setminus \{0 \}$. Then $\|v\| \cdot \|\lambda\| \geq |\lambda(v)|$, with equality if and only if $Kv$ is orthogonal to the hyperplane $\ker(\lambda)$.
\end{lem}
\begin{proof}
    The inequality is clear, as is the ``if'' part.  For the ``only if'' part, note that $V=Kv\oplus \ker(\lambda)$ is orthogonal if and only if for any $w\in \ker(\lambda)$ we have $\|v+w\|=\max(\|v\|,\|w\|)$, which is automatic if $\|v\|\neq \|w\|$.  If $\|v\|\cdot\|\lambda\|=|\lambda(v)|$, then for such a $w$ we have $\lambda(v+w)=\lambda(v)$, so $\frac{|\lambda(v+w)|}{\|v+w\|}\leq \frac{|\lambda(v)|}{\|v\|}$ implies $\|v\|\leq \|v+w\|\leq \max(\|v\|,|w\|)$, and if $\|v\|=\|w\|$ we conclude $\|v+w\|=\|v\|$.   
\end{proof}

\begin{cor}\label{switch}
    Let $V$ be a $K$-vector space equipped with a norm $\| \cdot\|$ and let $\{v_1,\ldots,v_n\}$ be an orthogonal basis.  If $v'=\sum_i  a_iv_i$ and $\|v'\|=|a_1|\|v_1\|$, then $\{v',v_2,\ldots,v_n\}$ is an orthogonal basis.
\end{cor}
\begin{proof}
Let $\{v_1^\vee,\ldots,v_n^\vee\}$ be the dual basis of $\{v_1,\ldots,v_n\}$. Taking $\lambda=v_1^\vee$, we find $\|\lambda\|=\|v_1\|^{-1}$ so $\|v'\|\cdot\|\lambda\|=|a_1|=|\lambda(v')|$.
\end{proof}

\begin{cor}\label{two_orthogonal_for_one}
Let $V$ be a $K$-vector space equipped with two norms $\| \cdot\|$ and $\| \cdot\|^\prime$. Let $v \in V\setminus \{0 \}$ that realizes $\sup_{u \in V \setminus \{0 \} }\frac{\N u \N^\prime}{\N u \N}$. Then every hyperplane $\ker({\lambda})$ which is $\|\cdot\|'$-orthogonal to $v$ is $\|\cdot\|$-orthogonal to $v$.
\end{cor}
\begin{proof}
For every $\lambda \in V^\vee\setminus \{0 \}$, we have
\[\|\lambda\|'\cdot \frac{\|v\|'}{\|v\|}=\sup_{u \in V \setminus \{0 \} }\frac{|\lambda(u)|}{\|u\|'}\cdot \sup_{u \in V \setminus \{0 \} }\frac{\|u\|'}{\|u\|}\geq \sup_{u \in V \setminus \{0 \} }\frac{|\lambda(u)|}{\|u\|}=\|\lambda\|\]
so $|\lambda(v)| \leq \|v\|\cdot \|\lambda\|\leq \|v\|'\cdot\|\lambda\|'$.  The conclusion follows from the lemma. 
\end{proof}

\begin{lem}\label{sup_quotient_norms}
Let $V$ be a $K$-vector space equipped with two norms $|| \cdot||$ and $|| \cdot||^\prime$. Let $\{v_1, \cdots, v_n\}$ be an orthogonal basis for $|| \cdot ||$. Then
\[ \sup_{v \in V \setminus \{0 \} }\frac{\N v \N^\prime}{\N v \N} = \max_{i}\frac{\N v_i \N^\prime}{\N v_i \N}. \]
\end{lem}
\begin{proof}
On the one hand, one trivially has $ \sup_{v \in V \setminus \{0 \} }\frac{\N v \N^\prime}{\N v \N} \geq \max_{i}\frac{\N v_i \N^\prime}{\N v_i \N}$. On the other hand, if $v = \sum_i a_i v_i$, then
\[  \| v \|^\prime \leq \max_i |a_i |  \|v_i \|^\prime \leq  \max_{i}\frac{\N v_i \N^\prime}{\N v_i \N} \cdot \max_i |a_i | \|v_i \|  =  \max_{i}\frac{\N v_i \N^\prime}{\N v_i \N} \cdot \| v \|. \]
\end{proof}

\subsection{Relative spectrum of two norms}

(Compare with \cite[section 2.5]{Boucksom-Eriksson}.) 
Let $V$ be a $K$-vector space of dimension $N$. Let $\| \cdot \|, \| \cdot \|^\prime \in \cN(V)$ be two norms on $V$. The relative spectrum of $\| \cdot \|$ with respect to $\| \cdot \|$ is the finite decreasing sequence
\[ \lambda_1(\| \cdot \|, \| \cdot \|^\prime) \geq \ldots \geq \lambda_N(\| \cdot \|, \| \cdot \|^\prime)\]
defined by the formula

\[ \lambda_i(\| \cdot \|, \| \cdot \|^\prime) := \sup_{W \subset V, \dim W \geq i} \left( \inf_{w \in W \setminus \{0 \} } \log \frac{\|w\|^\prime}{\|w\|} \right).\]

In particular,
\[ \lambda_1(\| \cdot \|, \| \cdot \|^\prime) := \sup_{v \in V \setminus \{0 \} } \log \frac{\|v\|^\prime}{\|v\|} \]
and 
\[ \lambda_N(\| \cdot \|, \| \cdot \|^\prime) := \inf_{v \in V \setminus \{0 \} } \log \frac{\|v\|^\prime}{\|v\|} = - \lambda_1(\| \cdot \|^\prime, \| \cdot \|) . \]

\begin{prop}{\cite[Proposition 2.24]{Boucksom-Eriksson}}\label{relative_spectrum_basis}
Let $V$ be a $K$-vector space of dimension $N$. Let $\| \cdot \|, \| \cdot \|^\prime \in \cN(V)$. Let $(e_i)$ be a basis of $V$ which is orthogonal for both norms, and order it so
that
\[  \frac{\|e_1\|^\prime}{\|e_1\|} \geq \ldots  \geq \frac{\|e_N\|^\prime}{\|e_N\|}.\]
Then, for every $i  \in \{1, \ldots, N\}$, 
\[ \lambda_i(\| \cdot \|, \| \cdot \|^\prime) =  \log \frac{\|e_i\|^\prime}{\|e_i\|}.\]
\end{prop}

\begin{cor}\label{sum_relative_spectrum}
For every $k \in \{1, \ldots, N\}$, one has
\[\sum_{i = 1}^k \lambda_i(\| \cdot \|, \| \cdot \|^\prime) = \lambda_1( \wedge^k\| \cdot \|, \wedge^k \| \cdot \|^\prime).\]
\end{cor}

\begin{cor}\label{relative_spectrum_decomposition}

Let $V = K v \oplus W$ be a decomposition which is orthogonal for both norms. If moreover $\lambda_1(\| \cdot \|, \| \cdot \|^\prime) =  \log \frac{\|v\|^\prime}{\|v\|}$, then, for every integer $i \in \{1, \ldots, \rk V - 1\}$, one has
\[ \lambda_i(\| \cdot \|_{|W}, \| \cdot \|_{|W}^\prime) = \lambda_{i +1}(\| \cdot \|, \| \cdot \|^\prime). \]

\end{cor}

\subsection{Space of norms}

Let $V$ be a $K$-vector space of dimension $N$.  The group $GL(V)(K)$ acts on $\cN(V)$ by composition. 

To each basis $e = (e_i)$ of $V$ is associated an injective map
\[ \iota_e \colon \bR^N \hookrightarrow \cN(V), \]
which takes $a \in \bR^N$ to the unique norm $\| \cdot \|_{e,a}$ that is diagonalized in $(e_i)$ and such that $\log \| e_i \|_{e, a} = - a_i$. The image $ \bA_e := \iota_e( \bR^n) \subset \cN(V)$ is thus the set of norms that are diagonalized in the given basis $e$, and is called an apartment (or flat) of $\cN(V)$.
The Bruhat-Tits metric on $\cN(V)$ is defined as:
\[ d_2(\| \cdot \|, \| \cdot \|^\prime) := \left(\sum_{i= 1}^N \left(\lambda_i(\| \cdot \|, \| \cdot \|^\prime)\right)^2\right)^{1/2}. \]

\begin{thm}[See {\cite[2.4.7, Corollaire 2]{Gerardin} and \cite[Theorem 3.1 and Corollary 3.3]{Boucksom-Eriksson}}]
The function $d_2\colon \cN(V) \times \cN(V) \to \bR_ {\geq 0}$ defines a metric on $\cN(V)$, and $(\cN(V), d_2)$ is a NPC complete metric space. The Bruhat-Tits metric $d_2$ is the unique metric on $\cN(V)$ for which $\iota_e \colon (\bR^N, l_2) \hookrightarrow (\cN(V), d_2)$ is an isometric embedding for each basis $e$ of $V$.
 \end{thm}

The extended Bruhat-Tits building $\Delta(\GL(V), K)$ can be canonically identified with $(\cN(V), d_2)$ as metric spaces equipped with an isometrical action of $\GL(V)(K)$, see \cite{Bruhat-Tits-SMF}.

\begin{prop}\label{induced_norm_Lipschitz}
Let $V$ be a finite-dimensional $K$-vector space.  Let $W \subset V$ be a sub-$K$-vector space. Then the canonical map $\cN(V) \to \cN(W) \times \cN(V / W)$ is Lipschitz continuous.
\end{prop}
\begin{proof}

First, $d_2$ is comparable to the Goldman-Iwahori distance $d_\infty(\| \cdot \|, \| \cdot \|^\prime):=\max(\lambda_1(\| \cdot \|, \| \cdot \|^\prime),\lambda_1(\| \cdot \|^\prime, \| \cdot \|))$, since any two norms on $\bR^N$ are equivalent, and $d_2,d_\infty$ are both pulled back from a norm on $\bR^N$ via the map given by the $\lambda_i$'s.  Thus, it is enough to prove $\cN(V) \to \cN(W) \times \cN(V / W)$ is distance-decreasing with respect to $d_\infty$.  The map $\cN(V)\to\cN(W)$ obtained by restricting the norm is clearly $d_\infty$-decreasing.  Moreover, the map $\cN(V)\to\cN(V^\vee)$ obtained by taking the dual norm is an involutive isometry with respect to $d_\infty$ \cite[Theorem 1.21]{Boucksom-Eriksson}.  Finally, the dual of the quotient norm on $V/W$ is the restriction of the dual norm of $V$ \cite[Lemma 1.23]{Boucksom-Eriksson}.  Thus, the quotient norm map $\cN(V)\to\cN(V/W)$ is also distance decreasing with respect to $d_\infty$, and the claim follows.

\end{proof}
\subsection{Pluriharmonic norms on local systems}

\begin{defn}
 Let $X$ be a complex analytic space equipped with a $K$-local system $L$. Let $\cN(V) \to X$ be the total space of the locally constant sheaf of spaces of norms on the stalks of $V$, so that the fiber $\cN(V)_x$ at $x \in X$ is canonically identified with $\cN(V_x)$. A norm on $V$ is then by definition a locally Lipschitz continuous section of $\cN(V) \to X$.
\end{defn}

Assume that $X$ is connected and fix $x \in X$. Let $\rho \colon \pi_1(X) \to \GL(V_x)$ be the monodromy representation of $V$. A norm $\| \cdot \|$ on $V$ is thus equivalent to a $\rho$-equivariant locally Lipschitz continuous map $\tilde{X} \to \cN(V_x)$.\\

Let $(V, \| \cdot \|)$ be a normed $K$-local system. A set of sections $v_1,\ldots, v_k$ is orthogonal if it is pointwise an orthogonal basis for its span.

\begin{defn}
Let $(V, \| \cdot \|)$ be a normed $K$-local system on a complex analytic space $X$. Its regular locus $\Reg(\| \cdot \|)$ is the open subset of points $x \in X$ for which an orthogonal trivializing set of sections for $V$ exists in some neighborhood of $x$. The complementary of the regular subset is called the singular set of $\| \cdot \|$ and denoted $\Sing(\| \cdot \|)$.
\end{defn}

\begin{defn}
A harmonic (resp. pluriharmonic) $K$-local system on a complex analytic space $X$ is a normed $K$-local system $(V, \| \cdot \|)$ on $X$ such that the associated section of $\cN(V) \to X$ is harmonic (resp. pluriharmonic): for every $x \in X$ and every simply-connected open neighborhood $U$ of $x$ in $X$, the induced map $U \to \cN(V_x)$ is harmonic (resp. pluriharmonic).
\end{defn}

We deduce from \Cref{Singular_locus_pluripolar} the following:
\begin{cor}\label{local system: pluripolar sing}
Let $(V,\|\cdot \|)$ be a pluriharmonic $K$-local system on a complex manifold $X$. Then its singular locus $\Sing(\|\cdot \|)$ is pluripolar.
\end{cor}

If $(V, \| \cdot \|)$ is pluriharmonic and $e_1,\ldots, e_k$ is an orthogonal trivializing set of sections, then the functions $\log \| e_i\|$ are pluriharmonic.

\begin{defn}\label{def-subharmonic-norm}
Let $V$ be a $K$-local system on a complex analytic space. A norm $\| \cdot \|$ on $V$ is called plurisubharmonic if for every local section $v$ of $V$ the function $\log \| v\|$ is plurisubharmonic. 
\end{defn}

\begin{prop}
A pluriharmonic norm is plurisubharmonic.    
\end{prop}
\begin{proof}
Let $V$ be a $K$-local system on a complex analytic space $X$. Let $\| \cdot \|$ be a pluriharmonic norm on $V$. Let $v$ be a local section of $V$. Locally on the regular locus of $\| \cdot \|$, there exists a flat orthogonal basis $(e_i)$ of $V$. Writing $v = \sum_i a_i e_i$, one has $\log \|v\| = \max_i \log \|a_i e_i\|$. This shows that $\log \|v\|$ is plurisubharmonic on $\Reg(\| \cdot \|)$. 
Since $\Sing(\| \cdot \|) \subset X$ is a pluripolar subset and $\log \|v\|$ is locally Lipschitz continuous, it follows that $\log \|v\|$ is plurisubharmonic on $X$.
\end{proof}

Theorem \ref{extension_pluriharmonic_map} immediately implies the following extension result.
\begin{thm}\label{extension_pluriharmonic_norm}
Let $(V, \| \cdot \|)$ be a normed $K$-local system on a  complex analytic space $X$. If $(V, \| \cdot \|)$ is pluriharmonic in restriction to the complementary of a closed pluripolar subset of $X$, then $(V, \| \cdot \|)$ is pluriharmonic.    
\end{thm}

We will prove later a stronger extension result, cf. Theorem \ref{extension_finite_energy_pluriharmonic}.
Since the singular locus of a pluriharmonic norm is a closed pluripolar subset of $X$ thanks to Proposition \ref{Singular_locus_pluripolar}, we get the following alternative characterization of pluriharmonicity. 
\begin{cor}
A normed $K$-local system $(V, \| \cdot \|)$ is pluriharmonic if and only if there exists a closed pluripolar subset $S \subset X$, such that for every point $x \in X \setminus S$ there exists a local trivializing orthogonal set of sections $\{e_i\}$ of $V$ such that the functions $z \mapsto \log \|e_i(z)\|$ are pluriharmonic.
\end{cor}

\begin{prop}\label{pluriharmonic_norm_sum}
Let $(V_1, \| \cdot \|_1)$ and $(V_2, \| \cdot \|_2)$ be two normed $K$-local systems on a complex analytic space $X$. Let $(V, \| \cdot \|) = (V_1, \| \cdot \|_1) \oplus (V_2, \| \cdot \|_2) $. Then, $(V, \| \cdot \|)$ is pluriharmonic if and only if $(V_1, \| \cdot \|_1)$ and $(V_2, \| \cdot \|_2)$ are pluriharmonic.
\end{prop}
\begin{proof}
Let $W_1$, $W_2$ and $W = W_1 \oplus W_2$ be $K$-vector spaces of finite dimension. Then there is  a canonical inclusion $ \cN(W_1) \times \cN(W_2) \to \cN(W)$ as the subset of norms on $W$ for which the decomposition $W = W_1 \oplus W_2$ is orthogonal.This realizes $ \cN(W_1) \times \cN(W_2)$ as a closed convex subset of $\cN(W)$. Therefore, the proposition is a consequence of the following general fact: if $M$ is a Riemannian manifold and $C \subset \cN$ is a closed convex subset of the complete metric space $\cN$ equipped with the induced metric, then a map $M \to C$ is pluriharmonic if and only if the induced map $M \to \cN$ is pluriharmonic.
\end{proof}

\subsection{Asymptotic behaviour of pluriharmonic norms}

\begin{prop}\label{flat_norm_quasiunipotent}
Let $V$ be a $K$-local system on $(\bD^\ast)^n$ with quasiunipotent local monodromies around zero. Then $V$ admits a flat norm. When the local monodromies are unipotent, $V$ admits a flat lattice, i.e. there exists a sub-$\cO_K$-local system $L \subset V$ such that $L \otimes_{\cO_K} K = V$. 
\end{prop}
\begin{proof}
Equivalently, one needs to prove that every finite set of commuting unipotent elements of $\bGL_r(K)$ stabilize a lattice of $K^r$, and that every finite set of commuting quasiunipotent elements of $\bGL_r(K)$ have a common fixed point in $\cN(K^r)$.

Let $U_1, \ldots, U_n$ be pairwise commuting unipotent elements of $\bGL_r(K)$. One easily proves that they stabilize a common complete flag of $K^r$. By multiplying every element of a basis adapted to the flag by an adequate power of the uniformizer, one obtains a lattice in $K^r$ stabilized by the $U_i$'s. 

Let $T_1, \ldots, T_n$ be pairwise commuting quasiunipotent elements of $\bGL_r(K)$. For every $i$, let $r_i$ be a positive integer such that $U_i := T_i^{r_i}$ is unipotent. Let $F \subset \cN(K^r)$ be the subset of norms that are fixed by the $U_i$'s. By the preceding paragraph, $F$ is nonempty. It is a closed convex subset of $\cN(K^r)$, hence it is a NPC space. Since the $T_i$'s commute pairwise, $F$ is stabilized by the action of $\bZ^n$ via the $T_i$'s. It follows from the definition of $F$ that this action factorizes through $(\bZ / \bZ^{r_i})^n$. Since $F$ is a non empty NPC space, this action has a fixed point by the Bruhat-Tits fixed point theorem.
\end{proof}

\begin{defn}
A normed $K$-local system $(V,\|\cdot \|)$ on $(\bD^\ast)^n$ with quasiunipotent local monodromies is said to be bounded in the neighborhood of $0$ if there exists a flat norm $\| \cdot \|^\prime$ on $V$ such that the function $z \mapsto d_{\cN(V_z)}(\|\cdot \|_z, \|\cdot \|_z^\prime)$ is bounded in the neighborhood of $0$.
\end{defn}
Note that by Proposition \ref{flat_norm_quasiunipotent}, every $K$-local system on $(\bD^\ast)^n$ with quasiunipotent monodromies admits a flat norm. Moreover, the distance between two flat norms is constant, hence the definition of boundedness does not depend on the choice of the flat norm $\| \cdot \|^\prime$.

\begin{defn}
Let $(\bar X,D)$ be a log smooth complex manifold of dimension $n$ and set $X=\bar X\setminus D$. A normed $K$-local system $(V,\|\cdot \|)$ on $X$ with quasiunipotent local monodromies around $D$ is said to be locally bounded in the neighborhood of $D$ (or just locally bounded at infinity) if for every admissible polydisk $\bD^n \subset \bar X$ the restriction of $(V,\|\cdot \|)$ to $\bD^n \setminus D$ is bounded in the neighborhood of $0$.
\end{defn}

\begin{prop}\label{finite_energy_implies_Lipschitz}
Let $(V,\|\cdot \|)$ be a pluriharmonic $K$-local system on $(\bD^\ast)^n$ with quasiunipotent local monodromies. Assume that there exists $U \subset \bD^n$ a neighborhood of $0$ such that the restriction of $(V,\|\cdot \|)$ to $U \setminus \{0\}$ has finite energy. Then $(V,\|\cdot \|)$ is Lipschitz continuous in restriction to a neighborhood of $0$ (and not only locally Lipschitz continuous).
\end{prop}
\begin{proof}
Fixing $r>0$ such that  $\{ (z_1, \ldots, z_n) \in \bD^n \, | \, \sum_i |z_i|^2 < r \} \subset U$, the proposition follows by applying \Cref{energy_control_Lipschitz} with $R = 1$ and $\Omega_\epsilon = \{ (z_1, \ldots, z_n) \in (\bD^\ast)^n \, | \, \epsilon < \sum_i |z_i|^2 < r \}$ (with the metric induced by the Poincaré metric on $(\bD^\ast)^n$) for every $0 < \epsilon <r$.
\end{proof}

\begin{prop}\label{finite_energy_implies_bounded}
Let $(V,\|\cdot \|)$ be a pluriharmonic $K$-local system on $(\bD^\ast)^n$ with quasiunipotent local monodromies around zero. If the restriction of $(V,\|\cdot \|)$ to a neighborhood of $0$ has finite energy, then $(V,\|\cdot \|)$ is bounded in the neighborhood of $0$.
\end{prop}
\begin{proof}
Thanks to Proposition \ref{flat_norm_quasiunipotent}, there exists a flat norm $\| \cdot \|^\prime$ on $V$. By Proposition \ref{distance_subharmonic} the function $\delta\colon z \mapsto d_{\cN(V_z)}(\| \cdot \|_z , \| \cdot \|^\prime_z)$ is plurisubharmonic on $(\bD^\ast)^n$. Since both $\| \cdot \|$ and $ \| \cdot \|^\prime$ have finite energy in a neighborhood of $0$, it follows from Proposition \ref{finite_energy_implies_Lipschitz} that the function $\delta$ is bounded in the neighborhood of $0$ by the function $z \mapsto C \log \left(-\log |z| \right)$ for some $C > 0$. In particular, $\delta$ extends as a plurisubharmonic function on a neighborhood of $0$ in $\bD^n$ and therefore it is bounded in a neighborhood of $0$.    
\end{proof}

\subsection{Extension of pluriharmonic norms}

\begin{thm}\label{extension_finite_energy_pluriharmonic}
Let $D$ be a normal crossing divisor in a complex manifold $X$ of dimension $n$. Let $V$ be a $K$-local system on $X$. Let $\|\cdot \|$ be a pluriharmonic norm on the restriction of $V$ to $X \setminus D$.
Assume that $\|\cdot \|$ has finite energy with respect to a metric of Poincaré type on $X \setminus D$. Then $\|\cdot \|$ is the restriction of a unique pluriharmonic norm on $V$.
\end{thm}

\begin{proof}
Since a pluriharmonic norm is continuous, the unicity is clear. In particular, to prove the existence of the extension, one can work locally on $X$. Therefore, one can replace $X$ by an admissible polydisk $\bD^n \subset X$ and prove the existence in a neighborhood of $0 \in \bD^n$. By assumption, the norm $\| \cdot \|$ belongs to $W^{1,2}((\bD_r^\ast)^n, \cN(V))$ for every $0< r < 1$. In particular, thanks to Proposition \ref{finite_energy_implies_bounded}, there exists $0< r < 1$ such that in addition the norm $\| \cdot \|$ is bounded on $(\bD_r^\ast)^n$. The trace of $\| \cdot \|$ is a well-defined element of $L^2(\partial (\bD_r^\ast)^n, \cN(V))$. Let $\| \cdot \|^\prime$ be the unique harmonic norm on $V_ {|\bD_r^n}$ with the same trace as $\| \cdot \|$ on $\partial \bD_r^n$, see Theorem \ref{Dirichlet-NPC}. The function $(\bD_r^\ast)^n \to \bR_{\geq 0}, z \mapsto d_{\cN(V_z)}(\| \cdot \|_z,  \| \cdot \|^\prime_z)$ is bounded. Moreover, it is subharmonic by Proposition \ref{distance_subharmonic}, hence by Brelot extension theorem, it extends uniquely as a subharmonic function on $\bD_r^n$. Since its trace in $L^2(\partial  (\bD_r^\ast)^n, \bR)$ is zero by Theorem \ref{KoSc-trace}, it is zero everywhere by the maximum principle. This shows that $\| \cdot \|$ admits an extension to $\bD_r^n$ as an harmonic norm, hence in particular as a locally Lipschitz continuous norm. Thanks to Theorem \ref{extension_pluriharmonic_map}, this extension is pluriharmonic.
\end{proof}


\subsection{Existence}
\begin{thm}\label{existence_pluriharmonic_norm}
Let $(\bar X,D)$ be a projective log smooth variety and set $X=\bar X\setminus D$. Let $V$ be a semisimple $K$-local system on $X$ with quasiunipotent local monodromy. Then $V$ admits a pluriharmonic norm of finite energy with respect to any Poincaré-type complete Kähler metric on $X$.
\end{thm}

\begin{proof}
It is sufficient to consider the case where $V$ is an irreducible $K$-local system.  Fix a Poincaré-type complete Kähler metric on $X$. Fix $x \in X$. Let $\rho \colon \pi_1(X,x) \to \GL(V_x)$ be the monodromy representation of $V$. 
 Let $X^\prime \to X$ be a finite étale cover such that the Zariski closure $\G$ of the image of $\rho^\prime := \rho_{|\pi_1(X^\prime)}$ is connected (and reductive). Thanks to Theorem \ref{existence_harmonic_nonarchimedean_quasiunipotentmonodromies}, there exists a $\rho^\prime$-equivariant harmonic map $\tilde{X} \to \Delta(\G,K)$ of finite energy. The induced $\rho^\prime$-equivariant map $u \colon \tilde{X} \to \Delta(\GL(V_x))$ is thus locally Lipschitz continuous with finite energy. Applying Proposition \ref{centerofmass_finite_energy}, we get that there exists a $\rho$-equivariant locally Lipschitz continuous map $\tilde{X} \to \Delta$ of finite energy. Since $\rho$ is irreducible, its image is not contained in any parabolic subgroup of $\GL(V_x)$, therefore the action of $\rho$ does not have any fixed point on $\partial \Delta(\GL(V_x))$. We conclude from \Cref{existence_harmonic_finite_energy} the existence of a $\rho$-equivariant harmonic map $u \colon \tilde{X} \to \Delta(\GL(V_x))$ of finite energy, which is necessarily pluriharmonic thanks to \Cref{pluriharmonic}.
\end{proof}

\subsection{Induced norm on sublocal systems}

\subsubsection{The local setting}
\begin{lem}\label{lem:basic}
Suppose $\{f_i\}_{i\in I}$ is a finite set of subharmonic functions on a connected Kähler manifold such that $f=\max_{i\in I}f_i$ is harmonic.  Then if $f=f_i$ holds at some point it holds everywhere.
\end{lem}
\begin{proof}
Suppose $f=f_i$ at some point.  Then $f_i-f$ is subharmonic and achieves its maximum value, namely zero, hence it is zero everywhere.  
\end{proof}

\begin{lem}\label{switch_harmonic}
Let $(V,\|\cdot \|)$ be a harmonic $K$-local system on a connected Kähler manifold $X$. Let $\{v_1, \ldots, v_n\}$ be a flat orthogonal basis. Let $v$ be a nonzero section of $V$ such that $\log\|v\|$ is harmonic. Then, up to renumbering the $v_i$'s, $\{v, v_2, \ldots, v_n\}$ is a flat orthogonal basis.
\end{lem}
\begin{proof}
Write $v = \sum_i a_i v_i$, so that $\log \|v\| = \max_{1 \leq i \leq n} \log \| a_i v_i \|$. Since $\{v_1, \ldots, v_n\}$ is a flat orthogonal basis, the functions $\log \| a_i v_i \| = \log|a_i| + \log \| v_i \|$ are harmonic. Since $\log\|v\|$ is harmonic by assumption, it follows from Lemma \ref{lem:basic} that there exists $i$ such that $\|v \| = \| a_i v_i \|$ on $X$. Finally, thanks to Corollary \ref{switch}, up to renumbering the $v_i$'s, $\{v, v_2, \ldots, v_n\}$ is a flat orthogonal basis.
\end{proof}

\begin{lem}\label{induced_harmonic_norm_local_with_basis}
Let $(V,\|\cdot \|)$ be a harmonic $K$-local system on a Kähler manifold $X$.  Let $W\subset V$ be a sub-$K$-local system equipped with a harmonic norm $\|\cdot\|_W$. Let $\| \cdot \|_V$ denote the norm on $W$ induced by $\|\cdot \|$. Assume that for every $1 \leq i \leq  \rk W$ the function $\lambda_i(\| \cdot \|_W, \| \cdot \|_V)$ is harmonic.
Assume that both $(V,\|\cdot \|)$ and $(W,\|\cdot \|_W)$ admit a flat orthogonal basis. Then there exists a flat basis $(e_i)$ of $W$ which is orthogonal for both $\| \cdot \|_V$ and $\| \cdot \|_W$, that can be completed into a flat orthogonal basis of $(V,\|\cdot \|)$ and such that $\lambda_i(\| \cdot \|_W, \| \cdot \|_V) =  \log \frac{\|e_i\|_V}{\|e_i\|_W}$.
\end{lem} 
\begin{proof}
By induction on the rank of $V$. There is nothing to prove if $V$ has rank zero. Therefore, let us assume that the rank of $V$ is positive. Let $\{v_1, \ldots, v_n\}$ be a flat orthogonal basis of $(V,\|\cdot \|)$ and $\{w_1, \ldots, w_r\}$ be a flat orthogonal basis of $(W,\|\cdot \|_W)$. Thanks to Proposition \ref{relative_spectrum_basis}, one has
\[ \lambda_1(\| \cdot \|_W, \| \cdot \|_V) = \max_{1 \leq i \leq r}  \log \frac{\|w_i\|_V}{\|w_i\|_W}.\]
By assumption, the functions $\log \frac{\|w_i\|_V}{\|w_i\|_W}$ are subharmonic. It follows from Lemma \ref{lem:basic} that $\lambda_1(\| \cdot \|_W, \| \cdot \|_V) = \log \|w_i\|_V - \log \|w_i\|_W$ for some $i$. This shows that the function $ \log \|w_i\| = \lambda_1(\| \cdot \|_V, \| \cdot \|_W) + \log \|w_i\|_W$ is harmonic. Therefore, thanks to Lemma \ref{switch_harmonic}, up to renumbering the $v_i$'s, one can assume that $\{w_i, v_2, \ldots, v_n\}$ is a flat orthogonal basis of $(V,\|\cdot \|)$.

Let $V^\prime$ be the sublocal system of $V$ generated by the $v_i$'s with $i \geq 2$, so that the decomposition $V = K w_i \oplus V^\prime$ is $\| \cdot \|$-orthogonal. Therefore the norm $\|\cdot \|_{V^\prime}$ induced by $\|\cdot \|$ on $V^\prime$ is harmonic by Proposition \ref{pluriharmonic_norm_sum}.

Let $W^\prime := W \cap V^\prime$. Then the
decomposition $W = K w_i \oplus W^\prime$ is $\| \cdot \|_V$-orthogonal. Thanks to Corollary \ref{two_orthogonal_for_one}, the
decomposition $W = K w_i \oplus W^\prime$ is also $\| \cdot \|_W$-orthogonal.
In particular, the norm $\|\cdot \|_{W^\prime}$ induced by $\|\cdot \|_W$ on $W^\prime$ is harmonic by Proposition \ref{pluriharmonic_norm_sum}. Moreover, thanks to Corollary \ref{relative_spectrum_decomposition}, for every integer $i \in\{1, \ldots, \rk V - 1\}$, one has
\[ \lambda_i(\| \cdot \|_{W^\prime}, (\| \cdot \|_{V^\prime})_{|W^\prime}) = \lambda_{i +1}(\| \cdot \|_W, \| \cdot \|_V). \]
Therefore one can apply the induction to $(V^\prime,  \| \cdot \|_{V^\prime})$ and $(W^\prime,  \| \cdot \|_{W^\prime})$ and conclude the proof.
\end{proof}

\begin{lem}\label{induced_harmonic_norm_local}
Let $(V,\|\cdot \|)$ be a pluriharmonic $K$-local system on a Kähler manifold $X$.  Let $W\subset V$ be a sub-$K$-local system equipped with a pluriharmonic norm $\|\cdot\|_W$. Let $\| \cdot \|_V$ denote the norm on $W$ induced by $\|\cdot \|$. Assume that for every $1 \leq i \leq  \rk W$ the function $\lambda_i(\| \cdot \|_W, \| \cdot \|_V)$ is pluriharmonic. Then the norm on $W \oplus (V / W)$ induced by $\| \cdot \|$ is pluriharmonic. 
\end{lem}
\begin{proof}
By Proposition \ref{induced_norm_Lipschitz}, the norm on $W \oplus (V / W)$ induced $\| \cdot \|$ is locally Lipschitz continuous. Since $\Sing(\|\cdot \|)$ and $\Sing(\|\cdot \|_W)$ are closed pluripolar subsets of $X$ by \Cref{local system: pluripolar sing}, thanks to Theorem \ref{extension_pluriharmonic_norm} it is sufficient to prove that the norm on $W \oplus (V / W)$ induced $\| \cdot \|$ is pluriharmonic on $\Reg(\|\cdot \|) \cap \Reg(\|\cdot \|_W)$. Since pluriharmonicity is a local property, one can assume that both $(V,\|\cdot \|)$ and $(W,\|\cdot \|_W)$ admit a flat orthogonal basis, so that one can apply Lemma \ref{induced_harmonic_norm_local_with_basis} to conclude.
\end{proof}

\begin{cor} 
Let $(V,\|\cdot \|)$ be a pluriharmonic $K$-local system on a connected complex analytic space $X$. Let $W \subset V$ be a sub-$K$-local system. If the norm on $W$ induced by $\| \cdot \|$ is pluriharmonic, then the norm on $V / W$ induced $\| \cdot \|$ is pluriharmonic. 
\end{cor}
\begin{proof}
This a direct consequence of Lemma \ref{induced_harmonic_norm_local}, in which we take for $\|\cdot\|_W$ the norm on $W$ induced by $\| \cdot \|$.
\end{proof}


\subsubsection{The global setting}

\begin{prop}\label{global_relative_spectrum}
Let $(\bar X,D)$ be a compact log smooth K\"ahler manifold and set $X=\bar X\setminus D$. 
Let $V$ be a $K$-local system on $X$, with quasiunipotent local monodromy.
Let $\|\cdot \|$ be a plurisubharmonic norm on $V$ (see Definition \ref{def-subharmonic-norm}), which is locally bounded at infinity. Let $W\subset V$ be a sub-$K$-local system equipped with a pluriharmonic norm $\| \cdot \|_W$ which is locally bounded at infinity. Then the functions $\lambda_i(\| \cdot \|_W,  \| \cdot \|_{|W})$ are constant for every $1 \leq i \leq \rk W$.
\end{prop}
\begin{proof}
Thanks to Corollary \ref{sum_relative_spectrum}, the function $\sum_{i=1}^k \lambda_i(\| \cdot \|_W,  \| \cdot \|_{|W})$ is equal to the function $ \lambda_1( \wedge^k\| \cdot \|_W, \wedge^k\| \cdot \|_{|W})$ for every $k$. But the latter is a plurisubharmonic on $X$ and locally bounded at infinity, hence it is constant on $X$. It follows that the function $\lambda_i(\| \cdot \|_W,  \| \cdot \|_{|W})$ are constant.
\end{proof}

\begin{thm}\label{induced_harmonic_norm_global}
Let $(\bar X,D)$ be a compact log smooth K\"ahler manifold and set $X=\bar X\setminus D$. 
Let $V$ be a $K$-local system on $X$, with quasiunipotent local monodromy.
Let $\|\cdot \|$ be a pluriharmonic norm on $V$, which is locally bounded at infinity. Let $W \subset V$ be a sub-$K$-local system. Assume that $W$ admits a pluriharmonic norm which is locally bounded at infinity. Then the norm induced by $\|\cdot \|$ on $W \oplus V/W$ is pluriharmonic and locally bounded at infinity.  
\end{thm}

\begin{proof}
Let $\| \cdot \|_W$ be a pluriharmonic norm on $W$ which is locally bounded at infinity. Thanks to Proposition \ref{global_relative_spectrum}, the functions $\lambda_i(\| \cdot \|_W,  \| \cdot \|_{|W})$ are constant. It follows from Lemma \ref{induced_harmonic_norm_local} that the norm induced by $\|\cdot \|$ on $W \oplus V/W$ is pluriharmonic. Moreover, it is locally bounded at infinity thanks to Proposition \ref{induced_norm_Lipschitz}.  
\end{proof}

\begin{cor}\label{induced_pluriharmonic_norm_graded}
Let $(\bar X,D)$ be a compact log smooth K\"ahler manifold and set $X=\bar X\setminus D$. 
Let $V$ be a $K$-local system on $X $, with quasiunipotent local monodromy.
Let $\|\cdot \|$ be a pluriharmonic norm on $V$, which is locally bounded at infinity. Let $W_\bullet \subset V$ be a filtration by sub-$K$-local systems. Assume that $\gr^W V$ admits a pluriharmonic norm which is locally bounded at infinity. Then the norm induced by $\|\cdot \|$ on $\gr^W V$ is pluriharmonic and locally bounded at infinity.   
\end{cor}
\begin{proof}
By induction on $\rk V$.
Let $W \subset V$ be the smallest non-zero piece of the filtration $W_\bullet$. By assumption, $W$ admits a pluriharmonic norm which is locally bounded at infinity. Thanks to Theorem \ref{induced_harmonic_norm_global}, the norm induced by $\|\cdot \|$ on $W \oplus V / W$ is pluriharmonic and locally bounded at infinity.    
Let $W^\prime_\bullet \subset V/W$ be the filtration induced by $W_\bullet$ on $V / W$. Since $\gr^{W^\prime} V / W$ is a factor of $\gr^W V$, it admits a pluriharmonic norm which is locally bounded at infinity, see Proposition \ref{pluriharmonic_norm_sum}. Therefore one can conclude by induction.  
\end{proof}

\begin{cor}\label{induced_pluriharmonic_norm_semisimple_graded}
Let $(\bar X,D)$ be a compact log smooth K\"ahler manifold and set $X=\bar X\setminus D$.  
Let $V$ be a $K$-local system on $X $, with quasiunipotent local monodromy.
Let $\|\cdot \|$ be a pluriharmonic norm on $V$, which is locally bounded at infinity. Let $W^\bullet \subset V$ be a filtration by sub-$K$-local systems. Assume that $\gr^W V$ is a semisimple $K$-local system. Then the induced norm on $\gr^W V$ is pluriharmonic and locally bounded at infinity.  
\end{cor}
\begin{proof}
Follows from Theorem \ref{existence_pluriharmonic_norm} and Corollary \ref{induced_pluriharmonic_norm_graded}.
\end{proof}
\subsection{The characteristic polynomial of a pluriharmonic local system}\label{sect:KZ foliation}
The material presented in this section is close to some material presented in \cite{Katzarkov97, Jost-Zuo00, Eyssidieux, Klingler03, Corlette-Simpson, BDDM} in a slightly different setting. 

Let $(V,\|\cdot \|)$ be a pluriharmonic $K$-local system of rank $r$ on a complex manifold $X$.  Let $x \in X \setminus \Sing(V,\|\cdot \|)$ and $\{v_1, \ldots , v_r\}$ be a flat orthogonal basis of $V$ in a neighborhood $U$ of $x$. Let $\omega_i := \partial \log \| v_i \|$ for every $i$. Since the functions $\log \| v_i \| $ are pluriharmonic, the $\omega_i$'s are holomorphic one-forms. Another choice of a flat orthogonal basis yields the same collection, up to permutation, thanks to the following lemma.

\begin{lem}
Let $(V,\|\cdot \|)$ be a harmonic $K$-local system on a connected complex manifold $U$.
Let $\{v_1, \ldots , v_r\}$ and $\{v_1^\prime, \ldots , v_r^\prime\}$ be two flat orthogonal basis of $V$. Then, up to reordering the first basis, one has $\log \| v_i \| = \log \| v_i^\prime \| + k_i q$ for some integers $k_i$.
\end{lem}
\begin{proof}
For every $x \in U$, one has $\log \| V_x \setminus \{0\} \| = \cup_i \log \| v_i \| + q \bZ$. Therefore, for a fixed $i$, $U$ is the union the closed subsets $\{ \log \| v_i^\prime \| = \log \|v_j \| + k \}$ with $j \in \{1, \ldots, r\}$ and $k \in \bZ$. By Baire category theorem, one of these subsets has non empty interior, hence is equal to $U$ by analytic continuation.    
\end{proof}

In particular, if $T$ is a formal variable, the coefficients of the polynomial $P(T):= \prod_{i} (T -\omega_i) = T^r + \sigma_1 T^{r-1} + \cdots + \sigma_r$ are well-defined holomorphic symmetric forms on $U\setminus \Sing(V,\|\cdot \|)$. Since these forms do not depend on any choices, varying the point $x$, we get a well-defined polynomial $P(T)= T^r + \sigma_1 T^{r-1} + \cdots + \sigma_r$ whose coefficients are holomorphic symmetric forms on $X\setminus \Sing(V,\|\cdot \|)$. The pluriharmonic map $\tilde{X} \to \cN(V_{x_0})$ associated to $(V,\|\cdot \|)$ is locally Lipschitz continuous, therefore the forms $\sigma_i$ are locally bounded in the neighborhood of every point of $\Sing(V,\|\cdot \|)$. Since by Theorem \ref{Gromov-Schoen-regular} the singular locus of $(V,\|\cdot \|)$ has Hausdorff codimension at least $2$, they extend to $X$ by \cite[Lemma 3.(ii)]{Schiffman68}. We call the polynomial $P(T)$ the characteristic polynomial of $(V,\|\cdot \|)$.

\begin{prop}\label{characteristic_polynomial_pullback}
Let $(V,\|\cdot \|)$ be a pluriharmonic $K$-local system on a complex manifold $X$. Let $Y$ be a complex manifold and $f \colon Y \to X$ be a holomorphic map.
Then the characteristic polynomial of $f^*(V,\|\cdot \|)$ is equal to the pull-back of the  characteristic polynomial of $(V,\|\cdot \|)$.
\end{prop}
\begin{proof}
As in \cite[\S1.4]{Eyssidieux}.  The pluriharmonic norm corresponds to a pluriharmonic map $u:\tilde X\to \Delta$ which is locally Lipschitz continuous (\Cref{Dirichlet-NPC}).  The building $\Delta$ has a length structure:  any Lipschitz continuous arc $\gamma:[0,1]\to\Delta$ has a well defined length by using the local metric embedding in euclidean space.  Concretely, any such arc can be decomposed into arcs contained in an apartment, and the length of such an arc is the standard length with respect to the euclidean metric given by the canonical coordinates.  By pushing forward arcs, $\tilde X$ acquires a length structure.  By deforming any arc to the regular locus using \Cref{Gromov-Schoen-regular}, it is equivalently described by taking the integral of the square root of the sum of the squares of the real parts of the roots of $P_V$.  The length structure on $\tilde Y$ is clearly the pullback of that of $\tilde X$, as both are pulled back from $\Delta$.  It therefore remains to show the length structure determines the characteristic polynomial.  In the regular locus we may find arcs for which all but one of the real parts of the roots vanish, and then the length function determines the restriction of the form.  Taking such arcs through a point we thereby determine the real parts of the roots, hence the roots, and therefore the characteristic polynomial at that point.  As the regular locus is dense, the length structure determines $P_V$.
\end{proof}

\begin{prop}\label{characteristic_polynomial_exact_sequence}
If $0 \to (V_1,\|\cdot \|_1) \to (V,\|\cdot \|) \to (V_2,\|\cdot \|_2) \to 0$ is an exact sequence of pluriharmonic $K$-local systems on a complex manifold $X$, then the characteristic polynomial of $(V,\|\cdot \|)$ is equal to the product of the  characteristic polynomials of $(V_1,\|\cdot \|_1)$ and $(V_2,\|\cdot \|_2)$.
\end{prop}
\begin{proof}
One can assume that $X$ is connected. Let $x \in \Reg (V,\|\cdot \|) \cap  \Reg(V_1,\|\cdot \|_1)$. Thanks to Lemma \ref{induced_harmonic_norm_local_with_basis} there exists in the neighborhood of $x$ a flat orthogonal basis $(v_1, \ldots , v_k)$ of $ (V_1,\|\cdot \|_1)$ that completes into a flat orthogonal basis $(v_1, \ldots, v_r)$ of $(V,\|\cdot \|)$. A fortiori, the images of $(v_{k +1}, \ldots, v_r)$ in $V_2$ give a flat orthogonal basis of $(V_2,\|\cdot \|_2)$.
Therefore, the equality from the statement holds in a neighborhood of $x$, hence on the whole of $X$ by analytic continuation.
\end{proof}

\begin{defn}
Let $(V,\|\cdot \|)$ be a pluriharmonic $K$-local system of rank $r$ on a complex manifold $X$.
We say that its characteristic polynomial $P(T)$ is split if there exists $r$ holomorphic one-forms $\omega_1, \cdots, \omega_r \in H^0(X, \Omega_{X}^1)$ such that $P(T):= \prod_{i} (T -\omega_i)$.    
\end{defn}

By construction, the characteristic polynomial of a pluriharmonic $K$-local system $(V,\|\cdot \|)$ is split in the neighborhood of every point of $X \setminus \Sing(V,\|\cdot \|)$.

\begin{prop}\label{extended_characteristic_polynomial}
Let $(\bar X,D)$ be a compact log smooth K\"ahler manifold and set $X=\bar X\setminus D$. Let $(V,\|\cdot \|)$ be a pluriharmonic $K$-local system of rank $r$ on $X$, with quasiunipotent local monodromy and finite energy. Then the coefficients of its characteristic polynomial extend (uniquely) as holomorphic logarithmic symmetric forms on $(\bar X, D)$.   
Moreover, if the characteristic polynomial $P(T)$ of $(V,\|\cdot \|_V)$ is split, then there exist $n$ holomorphic one-forms $\omega_1, \ldots, \omega_r$ on $\bar X$ such that $P(T)= \prod_{i} (T - {\omega_i}_{|X})$.
\end{prop}
In particular, if $X$ is algebraic, then by GAGA the coefficients of the characteristic polynomial are algebraic symmetric forms on $X$. 

\begin{proof}
Let $P(T)= T^r + \sigma_1 T^{r-1} + \cdots + \sigma_r$ be the characteristic polynomial of $(V,\|\cdot \|)$. Its coefficients are holomorphic symmetric forms on $X$. Let $x \in \bar X$ be a smooth point of $D$, so that there exists an admissible polydisk $\bD^d \subset \bar X$ centered at $x$ such that $D \cap \bD^d = \{ z_1 = 0\}$. Thanks to Proposition \ref{finite_energy_implies_Lipschitz}, the restriction of the forms $\sigma_k$ to $\bD^d \setminus D$ are bounded in the neighborhood of $0$ with respect to the Poincaré metric on $\bD^d \setminus D$. Write $\sigma_k = \sum_{|\alpha| = k} \tau_\alpha(z) dz^\alpha$ for some holomorphic functions $\tau_\alpha$ on $\bD^d \setminus D$. Here $\alpha = (\alpha_1, \ldots, \alpha_d) \in \bN^d$, with $|\alpha| := \sum_{i= 1}^d \alpha_i$ and $dz^\alpha := dz_1^{\alpha_1} \ldots dz_d^{\alpha_d}$. Then 
\[  |\sigma_k|^2_{\omega_P} = \sum_{|\alpha| = k} |\tau_\alpha(z)|^2 \left( |z_1| \log |z_1|^2 \right)^{2 \alpha_1} \prod_{i = 2}^d  (1- |z_i|^2)^{2 \alpha_i},  \]
so that the function $|\tau_\alpha(z)|^2 \left( |z_1| \log |z_1|^2 \right)^{2 \alpha_1}$ is bounded in the neighborhood of $0$ for every $\alpha$. It follows that $z \mapsto \tau_\alpha(z) \cdot z_1^k$ extends as a holomorphic function on $\bD^d$. This shows that $\sigma_k$ extends as a holomorphic logarithmic symmetric form on the complementary of the singular locus of $D$ in $\bar X$, hence on the whole of $\bar X$ since the later has codimension at least $2$ in $\bar X$.

Assume now that the characteristic polynomial $P(T)$ of $(V,\|\cdot \|_V)$ is split, so that there exist $r$ holomorphic one-forms $\omega_1, \ldots, \omega_r$ on $X$ such that $P(T)= \prod_{i} (T - {\omega_i})$.  The holomorphic one-forms $\omega_1, \ldots, \omega_r$ are orthogonal on the regular locus of $(V,\|\cdot \|)$, hence everywhere on $X$. Since $(V,\|\cdot \|)$ has bounded energy, it follows that the $\omega_i$'s are square integrable in a Poincaré type metric on $X$. Working near a smooth point of $D$ as above and writing $\omega_i = \sum_{l= 1}^d  \phi_{il} dz_l$ for some holomorphic functions $\phi_{il}$ on $\bD^d \setminus D$, we get that
\[\int_{\bD^\ast \times \bD^{d-1}} |\phi_{i1}(z)|^2 i^d dz_1 \wedge d \bar z_1 \wedge \ldots \wedge dz_d \wedge d \bar z_d < \infty \]
and that
\[\int_{\bD^\ast \times \bD^{d-1}} |\phi_{il}(z)|^2 \frac{1}{|z_1|^2 (\log |z_1|^2 +1)^2} i^d dz_1 \wedge d \bar z_1 \wedge \ldots \wedge dz_d \wedge d \bar z_d < \infty \]
for every $l > 1$. This implies that the $\phi_{il}$'s extend as holomorphic functions on $\bD^d$, hence the $\omega_i$'s extend as holomorphic one-forms on $\bar X$.
\end{proof}

\begin{lem}\label{spectral cover ben}
    Let $(\bar X,D)$ be a log smooth K\"ahler manifold.  Let $P(T)$ be a monic polynomial with log symmetric form coefficients on $(\bar X,D)$ which is locally split on a dense Zariski open.  Then there is a log smooth K\"ahler manifold $(\bar X',D')$ and a proper generically finite $f:(\bar X',D')\to(X,D)$ such that $f^*P$ is globally split on $\bar X'$ with log form roots.  
\end{lem}
\begin{proof}
    Let $q:E\to \bar X$ be the total space of the vector bundle $\Omega_{\bar X}(\log D)$, and the tautological section $s$ of $q^*\Omega_{\bar X}(\log D)$ yields a tautological section $\alpha$ of $\Omega_E(\log q^{-1}(D))$.  There is then a closed analytic subvariety $Z(P)\subset E$ where $s$ satisfies the polynomial $q^*P$ as a section of $q^*\Omega_{\bar X}(\log D)$, or equivalently where $\alpha$ satisfies $q^*P$ as a form.  Since $P$ is monic, $Z(P)$ is finite over $\bar X$.  Since $P$ is locally split on a dense Zariski open, it follows that $Z(P)$ dominates $\bar X$, so there is an irreducible component $Z$ of the reduction of $Z(P)$ which is not the zero section and which dominates $\bar X$.  Let $\pi:(\bar Y,D_Y)\to (Z,Z\cap \pi^{-1}(D))$ be a log resolution and $g:(\bar Y,D_Y)\to(\bar X,D)$.  Then the pullback $\pi^*\alpha$ is a log form and a root of $g^*P$, so $g^*P$ factors.  Since the factors also have log symmetric form coefficients, by induction on the degree of $P$ the proof is completed. 
\end{proof}

\begin{prop}\label{characteristic_polynomial_canonical}
Let $V$ be a $K$-local system on a smooth complex algebraic variety $X$ with quasiunipotent local monodromy.
Let $\|\cdot \|$ and $\| \cdot \|^\prime$ be two pluriharmonic norms on $V$, that are locally bounded at infinity. Then their characteristic polynomials are equal.
\end{prop}
\begin{proof}
Thanks to Proposition \ref{global_relative_spectrum}, the relative spectrum of  $\|\cdot \|$ and $\| \cdot \|^\prime$ is constant. Let $x \in X \setminus \left(\Sing(\| \cdot \|) \cup \Sing(\| \cdot \|^\prime)\right)$. Then by Lemma \ref{induced_harmonic_norm_local_with_basis}  there exist a flat basis $(e_i)$ of $V$ in a neighborhood of $x$ which is orthogonal for both norms and such that $\lambda_i(\| \cdot \|, \| \cdot \|^\prime) =  \log \frac{\|e_i\|^\prime}{\|e_i\|}$. By differentiating, we get that the characteristic polynomials of $\|\cdot \|$ and $\| \cdot \|^\prime$  are equal in a neighborhood of $x$, hence everywhere by analytic continuation. 
\end{proof}

\begin{cor}\label{pluriharmonic_norm_trivial_local_system}
Let $V$ be a $K$-local system on a smooth complex algebraic variety $X$, equipped with a pluriharmonic norm $\|\cdot \|$ which is locally bounded at infinity. If $V$ is trivializable, then $\|\cdot \|$ is necessarily flat.
\end{cor}
\begin{proof}
Thanks to Proposition \ref{characteristic_polynomial_canonical}, the coefficients of the characteristic polynomial of $\|\cdot \|$ are zero (except the top coefficient). It follows that the pluriharmonic map corresponding to $\|\cdot \|$ is constant on the 
regular locus of $\|\cdot \|$, hence everywhere by density of the regular locus.  
\end{proof}
\subsection{The canonical current}\label{nonarchimedean_canonical_current}
Let $(\bar X,D)$ be a compact log smooth K\"ahler manifold and set $X=\bar X\setminus D$.  Let $(V,\|\cdot \|)$ be a pluriharmonic $K$-local system of rank $r$ on $X$, with quasiunipotent local monodromy and finite energy. In the neighborhood of any regular point of $(V,\|\cdot \|)$, there is a collection of $r$ holomorphic one-forms $\omega_i$, well-defined up to permutation. Therefore the smooth semipositive $(1,1)$-form $\sqrt{-1}\sum_i \omega_i \wedge \bar \omega_i$ is well-defined on the regular locus of $(V,\|\cdot \|)$. Moreover, it extends uniquely as closed positive $(1, 1)$-current $\omega$ with continuous potential on $\bar X$. This is an obvious consequence of Proposition \ref{extended_characteristic_polynomial} if the characteristic polynomial of $(V,\|\cdot \|)$ is split. The general case follows from the existence of the spectral cover (see \Cref{spectral cover ben}) and a classical average procedure.

The current $\omega$ is called the canonical current associated to $(V,\|\cdot \|)$, and it follows from Proposition \ref{characteristic_polynomial_canonical} that it depends only on the underlying local system $V$.  

\begin{prop}[See {\cite[Proposition 2.2]{Gromov-Schoen}} and {\cite[Proposition 3.3.6]{Eyssidieux}}] \label{current_controls_exhaustion}
Let $(\bar X,D)$ be a compact log smooth K\"ahler manifold and set $X=\bar X\setminus D$. Let $(V,\|\cdot \|)$ be a pluriharmonic $K$-local system on $X$, with quasiunipotent local monodromy and finite energy. Let $\omega_V$ be the associated current. Fix $x_0 \in X$. Let $\rho \colon \pi_1(X,x_0) \to GL(V_{x_0})$ be the monodromy representation of $V$ and $u \colon \tilde{X} \to \cN(V_{x_0})$ be the $\rho$-equivariant pluriharmonic map corresponding to $(V,\|\cdot \|)$.
Choose $Q \in \cN(V_{x_0})$ and let $\phi \colon \tilde{X} \to \bR_{\geq 0}$ be the function defined as $\phi(x) := 2 \cdot d_2(u(x), Q)^2$ for every $x \in \tilde{X}$. Then $d d^c \phi \geq \omega_V$.
\end{prop}

\subsection{The graded nearby cycle functor for pluriharmonic local systems}

\begin{prop}\label{nearby_cycles_pluriharmonic}
Let $(V, \|\cdot \|)$ be a pluriharmonic $K$-local system on $(\bD^\ast)^k \times \bD^{n-k}$, with quasiunipotent local monodromy and finite energy. Assume that $(V, \|\cdot \|)$ extends as a normed $K$-local system to $(\bD_r^\ast)^k \times \bD_r^{n-k}$ for some $r>1$. Let $W_\bullet$ be a filtration of $V$ by sub-$K$-local systems such that the $K$-local system $\gr^W V$ extends to $\bD^n$. Then, up to going to a finite étale cover $(\bD^\ast)^k \times \bD^{n-k} \to(\bD^\ast)^k \times \bD^{n-k}$, there exists a pluriharmonic $K$-local system $(V^\prime, \|\cdot \|^\prime)$ on $\bD^n$ equipped with a filtration $W_\bullet^\prime$ by sub-$K$-local systems and an isomorphism $\gr^{W^\prime} V^\prime \simeq \gr^W V$ as graded normed $K$-local systems on $(\bD^\ast)^k \times \bD^{n-k} $.
\end{prop}

\begin{proof}
Since $(V, \|\cdot \|)$ has finite energy, thanks to \Cref{finite_energy_implies_bounded} it is locally bounded in the neighborhood of $0$. Fix a point $P \in (\bD^\ast)^k \times \bD^{n-k}$ and let $u \colon \bH^k  \times \bD^{n-k} \to \cN(V_P)$ be the $\bZ^k$-equivariant pluriharmonic map associated to $\| \cdot \|$. Since the image of $u$ is bounded, it meets only finitely many simplices of $\cN(V_P)$. Since the monodromy action preserves the simplicial structure of $\cN(V_P)$, a finite index of the monodromy group $\bZ^k$ acts trivially on the image of $u$. Therefore, $u$ factorizes through a finite étale cover of $(\bD^\ast)^k \times \bD^{n-k}$. Note that the induced map 
$v \colon (\bD^\ast)^k \times \bD^{n-k} \to \cN(V_P)$ has finite energy with respect to the Poincaré metric on $(\bD^\ast)^k \times \bD^{n-k}$. Let $V^\prime$ (resp. $W^\prime_i$) be the constant local system on $\bD^n$ with fiber $V_P$ (resp. $(W_i)_P$) at $P$. The map $v$ endows the restriction of $V^\prime$ to $(\bD^\ast)^k \times \bD^{n-k}$ with a pluriharmonic norm, which thanks to Theorem \ref{extension_finite_energy_pluriharmonic} extends (uniquely) as a pluriharmonic norm on $V^\prime$ on $\bD^n$. By construction, there is an isomorphism $\gr^{W^\prime} V^\prime \simeq \gr^W V$ as graded normed $K$-local systems on $(\bD^\ast)^k \times \bD^{n-k}$, so that the statement follows.
\end{proof}

\begin{prop}\label{criterion_pluriharmonicity_filtration}
Let $D$ be a compact Kähler manifold. Let $E \subset D$ be a simple normal crossing divisor. Let $N \geq 0$. For every $0 \leq k \leq N$, let $G_k$ be a $K$-local system on $D \setminus E$ with quasiunipotent local monodromy, and $\|\cdot \|_k$ be a norm on $G_k$ which is locally bounded at infinity. Assume the following conditions hold:
\begin{itemize}
    \item Every $G_k$ admits a pluriharmonic norm which is locally bounded at infinity. 
    \item In the neighborhood of every $x \in D \setminus E$, there exists a pluriharmonic $K$-local system $(V^\prime, \|\cdot \|^\prime)$ equipped with a filtration $W_\bullet^\prime$ by sub-$K$-local systems such that, for every $k$, the $K$-local system $\gr^{W^\prime}_k V^\prime$ equipped with the norm induced by $\|\cdot \|^\prime$ is isomorphic to $(G_k, \|\cdot \|_k)$ as normed $K$-local systems. 
\end{itemize}
Then the norms $\|\cdot\|_k$ are pluriharmonic for every $k$.
\end{prop}
\begin{proof}
By induction on $N$. There is nothing to prove if $N = 0$. Assume that $N \geq 1$. Without lost of generality, one can assume that $G_0$ is non-zero. By assumption, $G_0$ admits a pluriharmonic norm $\| \cdot \|_H$ which is locally bounded at infinity. 

Let $x \in D \setminus E$, and let $(V^\prime, \|\cdot \|^\prime)$ and $W_\bullet^\prime$ as in the statement of the proposition. Let $\| \cdot \|_0^\prime$ be the plurisubharmonic norm on $W^\prime_0$ induced by the pluriharmonic norm $\|\cdot \|^\prime$. By assumption, $(W^\prime_0,  \| \cdot \|_0^\prime) \simeq (G_0, \| \cdot \|_0)$ as normed $K$-local systems in the neighborhood of $x$. In particular, by letting $x$ cover $D \setminus E$, we get that the norm $\| \cdot \|_0$ on $G_0$ is plurisubharmonic. It follows from Proposition \ref{global_relative_spectrum} (applied to $W = V = G_0$) that the functions $\lambda_i(\| \cdot \|_H, \| \cdot \|_0)$ are constant for every $i$.
On the other hand, working again in the neighborhood of $x \in D \setminus E$, the functions $\lambda_i(\| \cdot \|_H, \| \cdot \|^\prime_0)$ are equal to the functions $\lambda_i(\| \cdot \|_H, \| \cdot \|_0)$, hence constant. Therefore, by Proposition \ref{induced_harmonic_norm_local}, the norm induced by $\| \cdot \|^\prime$ on $W^\prime_0 \oplus (V^\prime / W^\prime_0)$ is pluriharmonic. In particular, the norm $\| \cdot \|_0$ on $G_0$ is pluriharmonic.
But replacing $(V^\prime, \|\cdot \|^\prime)$ and $W_\bullet^\prime$ by $(V^\prime/ W^\prime_0, \|\cdot \|^\prime)$ and $W_\bullet^\prime/ W^\prime_0$ in the neighborhood of every $x \in D \setminus E$, we can apply the induction and conclude the proof.
\end{proof}

We consider the following setting:
\begin{itemize}
    \item Let $D \subset U$ be an inclusion of a compact Kähler manifold in a Kähler manifold. Let $E \subset U$ be a divisor such that $D \cup E$ is a simple normal crossing divisor of $U$.
    \item Let $V$ be a $K$-local system on $U \setminus (D \cup E)$ with quasiunipotent local monodromy.
    \item Let $\|\cdot \|$ be a pluriharmonic on $V$, with finite energy with respect to a Poincaré-like Kähler metric on  $U \setminus (D \cup E)$.
    \item Let $W_\bullet$ be a filtration of $V$ by sub-$K$-local systems on $U \setminus (D \cup E)$, such that the $K$-local system $\gr^W V$ extends to $U \setminus E$.
\end{itemize}

\begin{thm}\label{Construction_graded_nearby_cycles_pluriharmonic}
Setting as above.
\begin{enumerate}
    \item The norm $\gr^W \| \cdot \|$ induced by $\| \cdot \|$ on $\left( \gr^W V \right)_{U \setminus (D \cup E)}$ extends as a continuous norm on the extension of $\gr^W V$ to $U \setminus E$, and this norm is locally bounded in the neighborhood of $E$. 
    \item Assume in addition that the $K$-local system $\left(\gr^W V \right)_{|D \setminus E}$ admits a pluriharmonic norm which is locally bounded at infinity (this is the case for example if it is semisimple by \Cref{existence_pluriharmonic_norm} and \Cref{finite_energy_implies_bounded}). Then the restriction of the norm $\gr^W \| \cdot \|$ to $D \setminus E$ is pluriharmonic and locally bounded in the neighborhood of $E$.
\end{enumerate}
\end{thm}
\begin{proof}
For part $(1)$, the fact that the norm $\gr^W \| \cdot \|$ extends as a continuous norm on $U \setminus E$ can be checked locally on $D \setminus E$ and therefore follows from \Cref{nearby_cycles_pluriharmonic}. Since $(V,\|\cdot \|)$ has finite energy, it follows from \Cref{finite_energy_implies_bounded} that $\|\cdot \|$ is locally bounded in the neighborhood of $E$. Using \Cref{induced_norm_Lipschitz} it follows that $\gr^W \| \cdot \|$ is locally bounded in the neighborhood of $E$, therefore so is its continuous extension to $U \setminus E$.

For part $(2)$, it follows immediately from part $(1)$ that the restriction of the norm $\gr^W \| \cdot \|$ to $D \setminus E$ is locally bounded in the neighborhood of $E$. Moreover, thanks to Proposition \ref{nearby_cycles_pluriharmonic}, every $x \in D \setminus E$ admits an open neighborhood in $D \setminus E$ over which there exists a pluriharmonic $K$-local system $(V^\prime, \|\cdot \|^\prime)$ equipped with a filtration $W_\bullet^\prime$ by sub-$K$-local systems and an isomorphism $\gr_{W^\prime} V^\prime \simeq \gr_W V$ as graded normed $K$-local system. Therefore, the pluriharmonicity of the restriction to $D \setminus E$ of $\gr^W \|\cdot \| $ is a consequence of Proposition \ref{criterion_pluriharmonicity_filtration}.
\end{proof}

\begin{thm}\label{characteristic_polynomial_nearby_cycle}
Let $(\bar X,D)$ be a compact log smooth K\"ahler manifold and set $X=\bar X\setminus D$.  Let $(V,\|\cdot \|_V)$ be a pluriharmonic $K$-local system on $X$ with unipotent local monodromy and finite energy. Let $D_k$ be an irreducible component of $D$. Assume that the $K$-local system $V$ is semisimple. Then the graded nearby-cycle $K$-local system $\gr \psi_{D_k} V$ inherits a canonical pluriharmonic norm which is locally bounded at infinity. Moreover, if the characteristic polynomial of $(V,\|\cdot \|_V)$ is split, then its restriction to $D_k$ is equal to the characteristic polynomial of $\gr\psi_{D_k} (V,\|\cdot \|_V)$ (which in particular is split).
\end{thm}

\begin{proof}
The construction of $\gr\psi_{D_k} (V,\|\cdot \|_V)$ is a particular case of Theorem \ref{Construction_graded_nearby_cycles_pluriharmonic}. The assertion on the characteristic polynomial follows from the construction in view of \Cref{characteristic_polynomial_canonical}, \Cref{characteristic_polynomial_pullback} and \Cref{characteristic_polynomial_exact_sequence}.
\end{proof}
\begin{cor}\label{nearby cycles canonical current}
    With the setup of \Cref{characteristic_polynomial_nearby_cycle}, we have $\omega_{\gr\psi_{D_k}V}=\omega_V|_{D_k}$.
\end{cor}
Note that $\omega$ has continuous potentials locally on $\bar X$, so the restriction makes sense.


\section{Algebraic integrability of the Katzarkov--Zuo foliation}\label{sect:KZ}
Throughout we fix a prime $p$ and consider $\bar\bQ$ with a choice of $p$-adic valuation $v$.  The valuation $v$ on $\bar\bQ$ yields a completion $\bar\bQ_v$ of $\bar \bQ$ and we let $\bar\bQ_p$ be the algebraic closure of $\bQ_p$ in $\bar\bQ_v$, equipped with the natural embedding $\bar\bQ\subset\bar\bQ_p$.  We reserve the notation $K$ for a general non-archimedean local field and $\bar K$ an algebraic closure of $K$.  

The main goal of this section is to prove the following generalization of a theorem of Eyssidieux \cite{Eyssidieux} to the non-proper case: 
\begin{thm}\label{thm KZ integrable}
Let $X$ be a connected normal algebraic space, $\Sigma\subset \cM_B(X)(\bC)^{\ss}$ an absolute $\bar\bQ$-constructible set of semisimple local systems with quasiunipotent local monodromy.  Then the $v$-adic Katzarkov--Zuo foliation of $X$ associated to the $\bar\bQ$-points $\Sigma(\bar\bQ)$ of $\Sigma$ (equipped with the valuation $v$) is algebraically integrable. 
\end{thm} 
We define these notions more precisely in the following section. 
 The proof follows the strategy of Eyssidieux, the main new input being \Cref{characteristic_polynomial_nearby_cycle}.
\subsection{Preliminaries}

Note that any $\bar K$-local system on a complex algebraic variety $Y$ is defined over some finite extension of $K$, since the fundamental group of $Y$ is finitely generated.  Let $(\bar X,D)$ be a log smooth proper algebraic space, with $X=\bar X\setminus D$.  Let $V$ be a semisimple $\bar K$-local system on $X$ with quasiunipotent local monodromy, which is defined over a finite extension $L$ of $K$.  According to \Cref{existence_pluriharmonic_norm} and \Cref{finite_energy_implies_bounded}, $V$ can be equipped with a pluriharmonic norm which is locally bounded at infinity.  As in \Cref{sect:KZ foliation}, the resulting characteristic polynomial $P_V$ whose coefficients are log symmetric forms is independent of the choice of such a norm on $V$ by \Cref{characteristic_polynomial_canonical}.  In particular, since the base extension of a pluriharmonic norm which is locally bounded at infinity is also pluriharmonic and bounded at infinity, $P_V$ is canonically associated to $V$ (i.e., independent of $L$).  By \Cref{spectral cover ben}, there is a generically finite proper morphism $f:(\bar X',D')\to (\bar X,D)$ such that $f^*P_V=P_{f^*V}$ is split and whose roots are closed one-forms.  We refer to the Stein factorization $f_V:X_V\to X$ as a spectral cover.  The key property of these symmetric forms is the following:

\begin{lem}Let $f:\tilde X\to \Delta$ be the pluriharmonic map given by the norm on the universal cover $\pi:\tilde X\to X^\an$.  Let $Z$ be a connected normal analytic space and $g: Z\to \tilde X$ an analytic morphism.  Then $f\circ g$ is constant if and only if the coefficients of $(\pi\circ g)^*P_V$ vanish as sections of $\Sym^*\Omega_{Z^\reg}$ on $Z^\reg$.  Moreover, if this is the case, then the monodromy of $g^*V$ is bounded.
\end{lem}
\begin{proof}
The map $f\circ g$ is pluriharmonic and by \Cref{Gromov-Schoen-regular}, there is locally an orthogonal basis on $Z^\reg\setminus W$ where $W$ has Hausdorff codimension 2 in the nonsingular locus.  By \Cref{characteristic_polynomial_pullback} the coefficients of $(\pi\circ g)^*P_V$ are then constant if and only if the norms of this basis (and therefore the norm itself) are locally constant on this set.  This gives the forward direction, and the reverse direction follows since $Z^\reg\setminus W$ is dense in $Z$.  For the final claim, we need only observe that the image of $(f\circ g)$ is stabilized by the monodromy of $g^*V$ and the stabilizer of a point in $\Delta$ is bounded. 
\end{proof}

We now extend this picture to normal algebraic spaces.  The following terminology is useful.
\begin{defn}
    As in \cite{Eyssidieux}, for a connected normal algebraic space $X$, by an extendable one-form (resp. log extendable one-form) $\alpha$ on $X$ we mean a one-form on $X^{\reg}$ such that for some (hence any) pair $(Y,\bar Y)$ consisting of a resolution $\pi:Y\to X$ and a log smooth compactification $\bar Y$ of $Y$, $\alpha$ extends to a one-form on $\bar Y$ (resp. a log one-form on $\bar Y$ with poles along $\bar Y\setminus Y$) that vanishes on fibers $F$ of $\pi$ as a section of $\Omega_{F^\reg}$.
\end{defn}

\begin{lem}
    Let $X$ be a connected normal algebraic space.  
    \begin{enumerate}
        \item The natural Hodge structure on $H^1(X,\bC)$ has weights 1 and 2.
        \item Taking de Rham cohomology class yields an isomorphism between the space of extendable one-forms (resp. log extendable one-forms) on $X$ and $W_1F^1H^1(X,\bC)$ (resp. $F^1H^1(X,\bC)$).
        \item Integration yields natural Albanese maps (unique up to translation on the target)
        \[\begin{tikzcd}
            X\ar[r]\ar[rd]& \Alb(X)\ar[d]\ar[r,equals]&H_1(X,\bC)/F^0+H_1(X,\bZ) \\
            &\gr^W_{-1}\Alb(X)\ar[r,equals]&\gr^W_{-1}H_1(X,\bC)/F^0+\gr^W_{-1}H_1(X,\bZ)
        \end{tikzcd}\]
    and pull-back along $X\to\gr^W_{-1}\Alb(X)$ (resp. $X\to \Alb(X)$) is an isomorphism on spaces of extendable one-forms (resp. log extendable one-forms).  The Albanese maps are naturally algebraic.
    \end{enumerate}

\end{lem}
\begin{proof}
    Standard, see for example \cite{Delignehodgeii,DeligneHodgeiii}.
\end{proof}
Let $X$ be a connected normal algebraic space.  For a semisimple $\bar K$-local system $V$ on $X$ with quasiunipotent local monodromy, we may take a resolution $\pi:Y\to X$ and by the above there is a spectral cover $f_{\pi^*V}:Y_{\pi^*V}\to Y$.  We refer to the Stein factorization $f_V:X_V\to X$ as a spectral cover of $V$.
\begin{cor}
    Let $X$ be a connected normal algebraic space, $V$ a semisimple $\bar K$-local system with quasiunipotent local monodromy, and $f_V:X_V\to X$ the spectral cover.  Then the roots of $P_{f_V^*V}=f_V^*P_V$ are extendable one-forms.
\end{cor}
\begin{proof}
    Immediate by \Cref{extended_characteristic_polynomial}.
\end{proof}

For any finite set $\Sigma\subset \cM_B(X)(\bar K)$ of semisimple local systems with quasiunipotent local monodromy, we thus have the following:
\begin{itemize}
    \item A finite dominant cover $f_\Sigma:X_\Sigma\to X$ from a connected normal space $X_\Sigma$.  After taking Galois closure we may assume we also have a group action of a finite group $G_\Sigma$ on $X_\Sigma$ realizing $f_\Sigma$ as the quotient map.
    \item A $G_\Sigma$-invariant space $\form_\Sigma\subset W_1F^1H^1(X_\Sigma,\bC)$ of extendable 1-forms.
    \item A minimal quotient $\Alb(X_\Sigma)\to A_\Sigma$ to an abelian variety $A_\Sigma$ from which the forms in $\form_\Sigma$ are pulled back, which is therefore equivariant with respect to the natural $G_\Sigma$ action.
    \item A $G_\Sigma$-equivariant map $a_\Sigma:X_\Sigma\to A_\Sigma$, well-defined up to translation.

\end{itemize}

The space $\form_\Sigma$ naturally defines a foliation on $A_\Sigma$ whose leaves are (locally) affine subspaces.  For any $y\in X_\Sigma$, we obtain a natural germ of an analytic subspace of $X_\Sigma$ by taking the reduced pullback by $a_\Sigma$ of a leaf through $a_\Sigma(y)$.  The group $G_\Sigma$ naturally acts on these germs, so to any point $x\in X$ we obtain a natural germ of a reduced analytic subspace $\cL_{\Sigma}(x)$, which we call a $\Sigma$-leaf.  This germ has the property that for any irreducible algebraic subvariety $i:Z\to X$ through $x$, the local systems $i^*\Sigma$ have bounded monodromy if and only if the germ of $Z$ is contained in $\cL_\Sigma(x)$.

\begin{defn}
Let $\Sigma\subset \cM_B(X)(\bar K)$ be a set of semisimple local systems with quasiunipotent local monodromy.  The $\Sigma$-leaf through $x$ (also denoted $\cL_{\Sigma}(x)$) is the $\Sigma_0$-leaf through $x$ for sufficiently large finite $\Sigma_0\subset\Sigma$.  We say $\Sigma$ is algebraically integrable if each $\Sigma$-leaf is algebraic---that is, has nonempty interior in its Zariski closure.

For a subset $\Sigma\subset\cM_B(X)(\bar\bQ)$, we may speak of the $\Sigma$-leaves via the embedding $\bar\bQ\subset\bar\bQ_p$.

\end{defn}

\begin{rem}\label{KZ reduction}The map $X\to S_{\Sigma,\bar K}:=G_\Sigma\backslash A_\Sigma$ is called a $\bar K$-Katzarkov--Zuo reduction (with respect to $\Sigma$).  It has the property that the connected component of the fiber through $x$ is the maximal connected algebraic subvariety containing $x$ which is tangent to the leaf $\cL_\Sigma(x)$.  In the proper case one often takes the Stein factorization to make it the unique fibration with this property. 
\end{rem}

\begin{lem}\label{lem KZ1}Let $X$ be a connected normal algebraic space.  Let $\Sigma\subset \cM_B(X)(\bar K)$ be a set of semisimple local systems with quasiunipotent local monodromy.  Then there is a finite set $\Sigma_0\subset \Sigma$ such that for every point $x\in X$, $\cL_\Sigma(x)=\cL_{\Sigma_0}(x)$.
\end{lem}
\begin{proof}
 Clearly $\cL_\Sigma(x)\subset\cL_{\Sigma_0}(x)$ for any point $x$ and any subset $\Sigma_0\subset \Sigma$.  Since $A_\Sigma$ is a proper abelian variety, the foliation extends to a compactification $\bar X$ of $X$.  Observe that for any finite set $\Sigma_0\subset\Sigma$, the sets $\bar X(\Sigma_0,n):=\{x\in \bar X\mid \dim \cL_{\Sigma_0}(x)\geq n\}\subset \bar X$ are closed algebraic subsets.  Indeed, the corresponding sets in $\bar X_{\Sigma_0}$ (meaning the normalization of $\bar X$ in $X_{\Sigma_0}$) are closed analytic subsets, hence algebraic by Chow's theorem.  For each $n$, an increasing sequence of finite subsets $\Sigma_0$ yields a decreasing sequence of closed algebraic subsets $\bar X(\Sigma_0,n)$ which must therefore stabilize.  Thus there exists a finite subset $\Sigma_0\subset\Sigma$ for which the leaves of $\Sigma_0$ and $\Sigma$ have the same dimension at every point $x$.  Applying the same argument to the subsets $\{x\in \bar X\mid \mbox{\# branches of }\cL_{\Sigma_0}(x)\geq n\}$, the claim follows.
\end{proof}

\begin{lem}\label{lem KZ2}  Let $\Sigma\subset\cM_B(X)(\bar K)$ be a set of semisimple local systems with quasiunipotent local monodromy.  

\begin{enumerate}
    \item Let $g:Z\to X$ be an algebraic map from a connected normal algebraic space $Z$.  Then the leaves of $g^*\Sigma$ are the pullbacks of the leaves of $\Sigma$.
    \item Let $\Sigma=\{V\}$ with $V=\bigoplus_{i=1}^n V_i$ and set $\Sigma'=\{V_i\}_{i=1}^n$.  Then the leaves of $\Sigma$ are equal to the leaves of $\Sigma'$.
    \item Let $\Sigma=\{V\}$ with $V=\bigotimes_{i=1}^n V_i$ and set $\Sigma'=\{V_i\}_{i=1}^n$.  Assume $\det V_i$ is trivial for $i>1$. 
 Then the leaves of $\Sigma$ are equal to the leaves of $\Sigma'$.
\end{enumerate}
\end{lem}
\begin{proof}
    \begin{enumerate}
        \item Clear.
        \item It suffices to observe that for a single $\bar K$-local system $V=V'\oplus V''$, the leaves of $V$ are the intersections of the leaves of $V'$ and $V''$, since the characteristic polynomial of $V$ is the product of those of $V'$ and $V''$ by \Cref{characteristic_polynomial_exact_sequence}.
        \item Likewise, it suffices to assume $V=V'\otimes V''$ where $V''$ has trivial determinant.  The roots of $p_V$ are $\alpha_i'+\alpha_j''$, where $\alpha_i'$ (resp. $\alpha_j''$) are the roots of $p_{V'}$ (resp. $p_{V''}$).  Certainly then the intersections of the leaves of $V'$ and $V''$ are contained in the leaves of $V$.  The characteristic polynomial of the determinant of $V''$ is $T-\sum_j\alpha''_j$, so $0=\sum_i\alpha''_j$ vanishes.  Thus, since $\sum_j\alpha_i'+\alpha_j''=\rk(V'')\alpha_i'$, we have the reverse containment as well.\qedhere
    \end{enumerate}
\end{proof}
\begin{cor}\label{cor KZ}
    Let $\Sigma\subset\cM_B(X)(\bar K)$ be a set of semisimple local systems with quasiunipotent local monodromy.  After passing to a finite \'etale cover $f:X'\to X$ and replacing $\Sigma$ with $f^*\Sigma$, there is a finite subset of irreducible local systems $\Sigma_0\subset \cM_B(X)(\bar K)$ with the same leaves as $f^*\Sigma$ such that any $V\in\Sigma_0$ is either abelian or has monodromy which is Zariski dense in a simple $\bar K$-group.  
\end{cor}

\begin{proof}   Let $\Sigma_0\subset\Sigma$ be a finite set with the same leaves as $\Sigma$, using \Cref{lem KZ1}.  By (1), passing to a finite \'etale cover, the leaves of $f^*\Sigma$ and $f^*\Sigma_0$ will be the same.  Thus, we may assume that the algebraic monodromy of each local system in $\Sigma_0$ is connected.  By (2) we may then replace $\Sigma_0$ with a finite set of irreducible local systems by taking irreducible factors.  For each element $V$ of $\Sigma_0$, we may replace $V$ with $\{\End(V),\det(V)\}$:  if $p_V$ has roots $\alpha_i$, then $\End(V)$ (resp. $\det(V)$) has roots $\alpha_i-\alpha_j$ (resp $\sum\alpha_i$), and $\{\alpha_i\}_{i\leq r}$ and $\{\alpha_i-\alpha_j\}_{i,j\leq r}\cup\{\sum\alpha_i\}$ have the same span.  Finally, the algebraic monodromy of $\End(V)$ is the adjoint form of the derived group of the algebraic monodromy of $V$, so $\End(V)$ splits as a product of local systems each of which has simple algebraic monodromy.  By (3) we may replace $\End(V)$ with these factors, and this completes the proof.
\end{proof}

\begin{cor}\label{extract tensor}
    Let $\Sigma\subset \cM_B(X)(\bC)$ be an absolute $\bar\bQ$-constructible set of semisimple local systems with quasiunipotent local monodromy.  There is a finite \'etale cover $f:X'\to X$, an absolute $\bar\bQ$-constructible $\Sigma'\subset \cM_B(X')(\bC)$ set of semisimple local systems with quasiunipotent local monodromy, and a finite subset $\Sigma_0\subset \Sigma'(\bar\bQ)$ such that $f^*\Sigma(\bar\bQ),\Sigma'(\bar\bQ),\Sigma_0$ have the same leaves and  any $V\in\Sigma_0$ is either abelian or has monodromy which is Zariski dense in a simple $\bar\bQ$-group.  
\end{cor}
\begin{proof}
    By \Cref{abs basic prop}, the direct factors and tensor components of the previous corollary can be extracted with $\bar\bQ$-absolute operations.
\end{proof}
\subsection{Main step}
The main step of the proof of \Cref{thm KZ integrable} is the following, whose proof closely follows the strategy of Eyssidieux using Simpson's Lefschetz theorem as in \cite[\S5]{Eyssidieux}.
\begin{prop}\label{KZ main step}
    Let $X$ be a connected smooth algebraic space, $(\bar X,D)$ a log smooth compactification, and $\Sigma\subset \cM_B(X)(\bar K)$ a set of semisimple local systems with unipotent local monodromy.  Assume that:
    \begin{enumerate}
    \item For each normalization of a strict subvariety $i:Z\to X$, $i^*\Sigma$ is algebraically integrable. 
    \item For each component $D_0$ of $D$, the image under the graded nearby cycles functor $\gr\psi_{D_0}\Sigma$ is algebraically integrable. 
    \end{enumerate}
 Then either $\Sigma$ is algebraically integrable or the commutator subgroup $[\pi_1(X,x),\pi_1(X,x)]\subset\pi_1(X,x)$ has bounded image under the monodromy representation $\rho_{V,x}$ for any $V\in \Sigma$.
\end{prop}

\begin{proof}
By Lemma \ref{lem KZ1} may replace $\Sigma$ with a finite subset with the same leaves.  Replacing $X$ with a resolution of the spectral cover, we have a map $a:X\to A$ to an abelian variety (spanned by $X$) and a space $\form\subset H^0(A,\Omega_A)$ of one-forms defining a foliation of $A$ by affine subspaces.  By the minimality condition on $A$, there are no algebraic subvarieties of $A$ which are contained in an affine subspace cut out by $\form$.  Throughout, by a leaf of a subvariety $Y\subset A$ we mean the germ of an intersection of an affine subspace cut out by $\form$ with $Y$.  

From now on, we assume that $\Sigma$ is not algebraically integrable, in which case we will show that the second alternative holds.  Thus, $a(X)$ contains a positive-dimensional leaf.  The assumption (1) implies:  for any strict subspace $Z\subset X$, the image of $Z$ under $X\to A$ has no positive-dimensional leaves.  Thus, any positive-dimensional leaf of $a(X)$ is Zariski-dense in $a(X)$.  Moreover, every positive-dimensional leaf of $a(X)$ must be 1-dimensional, or else it would intersect a strict algebraic subvariety in positive dimension.  Thus, $\dim \form\geq\dim a(X)-1$.  Finally, $a$ is generically finite: 
 otherwise, by taking a generic multisection of $a$, (1) would imply that $\Sigma$ is algebraically integrable.

Let $\bar X$ be a log smooth compactification of $X$ with boundary $D$. Note that the map $a:X\to A$ automatically extends to a map $\bar a:\bar X\to A$.  According to Theorem \ref{characteristic_polynomial_nearby_cycle}, the restrictions of the one-forms $\form$ to any component $D_0$ of $D$ define the foliation corresponding to the image of the nearby cycles functor, and therefore by (2) it follows that $\bar a(D)$ contains no positive-dimensional leaves.

The proof proceeds as in Eyssidieux's proof in the projective case, although we simplify the argument by replacing the use of Eyssidieux's higher-dimensional version of the Castelnuovo--de Franchis theorem by the so-called weak Ax--Schanuel theorem:

\begin{thm}[{Ax \cite{axsemiab}}]\label{axschanuel}Let $A$ be an abelian variety, $\cL\subset A$ the germ of an affine subspace, and $S\subset A$ an algebraic subvariety.  If $U$ is an irreducible germ of the intersection $\cL\cap S$ with
\[\codim_S U<\codim_A \cL\]
then $U$ is contained in a translate of a proper abelian subvariety $B\subset A$.  In particular, if $S$ spans $A$, then $U$ is not Zariski dense in $S$. 
\end{thm}

\begin{rem}
    \Cref{axschanuel} immediately implies the classical Castelnuovo--de Franchis theorem.  Indeed, let $X$ be a smooth proper algebraic variety with two linearly independent forms $\alpha,\beta\in H^0(\Omega_X)$ for which $\alpha\wedge \beta=0$.  Let $\Alb(X)\to A$ be the smallest quotient from which $\alpha,\beta$ are pulled back and $S\subset A$ the image of $X$ in $A$.  Note that in particular there is no nontrivial subvariety of $A$ on which $\alpha,\beta$ vanish.  Since $\alpha$ and $\beta$ are pointwise dependent (or equivalently satisfy a meromorphic linear relation), the leaves of $S$ cut out by $\alpha$ are equal to the leaves cut by the span of $\alpha$ and $\beta$.  Any leaf therefore has codimension 1 in $S$ but is the intersection with a codimension 2 affine subspace in $A$, and Theorem \ref{axschanuel} implies it is algebraic, which is a contradiction unless $S$ is a curve.  
\end{rem}

\begin{claim}$\dim \form=\dim a( X)-1$.
\end{claim}
\begin{proof}If $\dim \form\geq \dim a(X)$, then any positive-dimensional leaf $U$ is an atypical intersection between $\bar a(\bar X)$ and an affine subspace of $A$.  By Ax--Schanuel, this implies $U$ (and therefore $\bar a(\bar X)$) is contained in a coset of a strict abelian subvariety of $A$, which is a contradiction.  
\end{proof}
We finish with a very slight modification of Eyssidieux's version \cite{Eyssidieux} of the Lefschetz-type theorem of Simpson \cite{Simpsonlefschetz}, but for convenience we sketch the remaining steps.

Let $\bar Y\to \bar X$ be the minimal cover for which there is a lift $\bar Y\to\tilde A$ to the universal cover, namely, the cover corresponding to the quotient $\pi_1(\bar X,x)\to H_1(\bar X,\bZ)\to H_1(A,\bZ)$.  By integrating the forms in $\form$ we then have a map $g:\bar Y\to \form^\vee$ which is $H_1(A,\bZ)$-equivariant. 
 Observe that:
\begin{enumerate}
\item The fibers of $g$ are precisely lifts of (analytically continued) leaves of $\bar X$, all of which are at least 1-dimensional.  In particular, there is a leaf with positive-dimensional image in $A$ through every point of $X$.    
\item Let $Z\subset\bar Y$ be the vanishing locus of the canonical $\det \form$-valued $\dim (\form)$-form $\omega_\form$ on $\bar X$.  This is precisely the non-smooth locus of the map $g$, the inverse image of the corresponding vanishing locus $W$ in $\bar X$.  We claim that $\bar a(W)$ has codimension $\geq2$ in $\bar a(\bar X)$.  Indeed, for any component $T$, since $\omega_\form$ vanishes on $T$, the restrictions of the one forms in $\form$ to the regular locus $T^\reg$ satisfy a linear relation with meromorphic coefficients.  It follows that the generic intersection of an affine subspace of $A$ cut out by $\form$ with $T^\reg$ is equal to the intersection with an affine subspace of codimension $\dim (\form)-1$.  If $T$ were divisorial, this would mean it contains a positive-dimensional leaf, which is a contradiction.    

Note that since the leaves of $W$ are discrete, the connected components of the fibers of $Z\to \form^\vee$ are compact, so locally on $\bar Y$ there is a strict analytic subset $\Xi\subset \form^\vee$ outside of which $g$ is smooth.  
\item Let $D_Y$ be the inverse image of $D$ in $\bar Y$ and $Y$ the complement, that is, the preimage of $X$ in $\bar Y$.  Likewise, the leaves of $D$ are discrete, so locally on $\bar Y$ there is a strict analytic subset $\Xi\subset \form^\vee$ outside of which the map $g|_{\bar Y\setminus Y}:\bar Y\setminus Y\to \form^\vee$ is a local isomorphism.  Combining this with the previous point, there is a countable union of (global) strict analytic subvarieties $\Xi\subset \form^\vee$ outside of which $g$ is smooth and $g|_{\bar Y\setminus Y}$ is a local isomorphism. 
\item $g|_Y:Y\to \form^\vee$ is surjective.  Indeed, $g$ is dominant, and the images of $g$, $g|_Y$, and $g|_{Y\setminus Z}$ all coincide, hence the image is open.  Let $\Gamma$ be the image of $H_1(A,\bZ)$ in $\form^\vee$.  By general principles, the identity component $\bar\Gamma^0$ of the closure $\bar \Gamma$ is a real vector subspace and $\form^\vee/\bar\Gamma$ is a product of circles.  The image of the map $\bar X\to \form^\vee/\bar\Gamma$ is both compact and open, hence it is surjective.  As the complement of the image of $g$ is stable under the action of $\bar\Gamma$, it follows that $g$ is surjective. 
\item Putting all of the above together, for each point $y\in \bar Y$ there is a cylindrical neighborhood $\bar \Omega$ with the following properties:
\begin{enumerate}
\item $\bar\Omega$ is relatively compact in $\bar Y$.
\item $g(\bar \Omega)$ is a ball $B$ of fixed radius centered at $g(y)$.
\item\label{serrefib} Set $\Omega=\bar \Omega\cap Y$.  There is a strict analytic subset $\Xi_B\subset B$ such that $g|_{\Omega\setminus g^{-1}(\Xi_B)}:\Omega\setminus g^{-1}(\Xi_B)\to B\setminus \Xi_B$ is a Serre fibration.
\item Set $\Omega^\circ=\Omega\setminus Z$. 
 Then $g|_{\Omega^\circ}:\Omega^\circ\to B$ admits a section.
\end{enumerate}
\end{enumerate}

\begin{claim}Let $F$ be a fiber of $g|_Y$ above a point in $\form^\vee\setminus \Xi$.  Then $(Y,F)$ is 1-connected.
\end{claim}
\begin{proof} As in \cite{Eyssidieux}.  Recall that 1-connectedness means that any path $\gamma$ in $Y$ beginning and ending in $F$ can be homotoped to a path in $F$ via a homotopy whose restriction to the endpoints is contained in $F$.  It suffices to show this for a path beginning at any chosen component of $F$.  Briefly, one first shows that $(\Omega,\Omega\cap F)$ is 1-connected by taking a path $\gamma$ starting at the component meeting the section, homotoping into $\Omega\setminus g^{-1}(\Xi)$, concatenating with the inverse of the lift via the section of its image $g(\gamma) $ in $B\setminus \Xi_B$ (which is nulhomotopic), and using (\ref{serrefib}) to lift a nulhomotopy of the new image in $B\setminus\Xi_B$, which is by construction $g(\gamma)g(\gamma)^{-1}$.  One then concludes by gluing the 1-connectedness of the local neighborhoods using the surjectivity of $g$ and the contractibility of $\form^\vee$. 
\end{proof}

By construction, for any $y\in F$ lifting the basepoint $x\in X$, $\pi_1(F,y)$ has bounded image under $\rho_{V,x}$ for any $V\in\Sigma$, and thus the same is true for $\pi_1(Y,y)$, which is the kernel of the composition $\pi_1(X,x)\to H_1(X,\bZ)\to H_1(A,\bZ)$ and therefore contains $[\pi_1(X,x),\pi_1(X,x)]$.
\end{proof}

\subsection{Proof of \Cref{thm KZ integrable}}

We prove the claim by induction on $\dim X$, the case $\dim X=1$ being trivial.  By \Cref{extract tensor} we may assume by passing to a finite \'etale cover and changing $\Sigma$ that there is a finite set $\Sigma_0\subset\Sigma(\bar \bQ)$ with the same leaves such that every local system in $\Sigma_0$ is irreducible and either abelian or has monodromy which is dense in a simple $\bar \bQ_p$-group.  By passing to a further finite \'etale cover, we may assume the local systems in $\Sigma$ have unipotent local monodromy, cf. \Cref{from quasiunipotent to unipotent}.  By replacing $X$ with a resolution we may assume it is smooth.  By the inductive hypothesis, \Cref{abs basic prop}, and \Cref{nearby cycles is abs}, the two conditions of \Cref{KZ main step} are satisfied.  We are done if the first conclusion of \Cref{KZ main step} holds, so we may assume the monodromy of every $V\in\Sigma_0$ has bounded image on the commutator subgroup of $\pi_1(X,x)$.  By the following two lemmas, the monodromy of every $V\in\Sigma_0$ is either virtually abelian or bounded.
    \begin{lem}[{\cite[Lemma 5.3]{brunebarbehyp}}]
    Let $G$ be a geometrically simple
$K$-algebraic group. Let $\Gamma\subset G(K)$ be a Zariski-dense and unbounded subgroup. Let $\Theta\subset\Gamma$ be a normal subgroup. If $\Theta$ is bounded in $G(K)$, then $\Theta$ is finite.
\end{lem}
\begin{lem}
    A finitely generated group $\Gamma$ with finite commutator subgroup $\Gamma'$ is virtually abelian.
\end{lem}
\begin{proof}
    For any element $\gamma\in\Gamma$ of infinite order, some power commutes with $\Gamma'$, and some further power is central since $[g,\gamma^n]=[g,\gamma^{n-1}]\gamma^{n-1}[g,\gamma]\gamma^{1-n}$.  Thus, by lifting a set of generators of $\Gamma/\Gamma'$ and taking sufficiently divisible powers, we obtain a finite-index abelian subgroup.  
\end{proof}
After replacing $X$ with a finite \'etale cover, we may assume that the monodromy of every $V\in\Sigma_0$ is abelian or bounded, and the abelian ones have no local monodromy.  Note that bounded image local systems have trivial leaves.  If $\Sigma^{\mathrm{ab}}\subset\Sigma$ is the intersection with $\cM_B(X,1)$, then since $\Sigma_0$ and $\Sigma(\bar\bQ)$ have the same leaves, it follows that $\Sigma^{\mathrm{ab}}(\bar\bQ)$ and $\Sigma(\bar\bQ)$ have the same leaves.  Thus, we may assume $\Sigma=\Sigma^\mathrm{ab}$, and therefore we may replace $X$ with an abelian variety $A$.  By \Cref{simpsonbialg} and \Cref{bialg=abs}, $\Sigma$ must be the union of torsion translates of pull-backs of Zariski-dense subsets of $M_B(A',1)(\bC)$ for (potentially different) quotient abelian varieties $A\to A'$.  If $\Sigma=\bigcup_i\Sigma_i$ for absolute $\bar\bQ$-constructible $\Sigma_i$, the foliation for $\Sigma(\bar\bQ)$ is the intersections of the foliations for $\Sigma_i(\bar\bQ)$, so it suffices to assume $\Sigma$ is Zariski dense in $M_B(A,1)(\bC)$.  The proof is then concluded from the following:
\begin{lem}
\Cref{thm KZ integrable} holds for $X=A$ an abelian variety and $\Sigma$ Zariski dense in $M_B(A,1)(\bC)$.    
\end{lem}
\begin{proof}
    For any character $\chi$ in $H^1(A,\bC^*)$, the pluriharmonic norm is $e^{\int\alpha}$, where $\alpha$ is the harmonic form such that the image of $\alpha$ under the exponential $H^1(A,\bR)\to H^1(A,\bR^*)$ agrees with the norm $|\chi|^2$.  The associated holomorphic form defining the foliation is $\alpha^{1,0}$.  For any Zariski dense $\bar\bQ$-constructible subset $\Sigma$ of $H^1(A,\bC^*)$, the image under the log norm map $H^1(A,\bC^*)\to H^1(A,\bR)$ has nonempty interior, and since $\bar\bQ$ points are topologically dense, it follows there is a set of $\bar\bQ$ points of $\Sigma$ whose images under $H^1(A,\bC^*)\to H^1(A,\bR)$ span.  Thus, the space $E$ of 1-forms is all of $F^1H^1(A,\bC)$, and the foliation has trivial leaves.
\end{proof}
\qed

\section{A criterion of Steinness}\label{sect:stein}
A complex analytic space $X$ is said to be:
\begin{itemize}
     \item holomorphically convex if for every compact $K \subset X$, the holomorphically convex hull of $K$:
\[\hat{K}_X = \{x \in X : |f (x)| \leq||f||_K \text{ for all }f \in \cO(X) \},\]
where $||f||_K = \sup_K |f |$, is compact;
    \item holomorphically separated if for all points $x, y \in X$, there is a global holomorphic function $f \colon X \to \bC$ such that $f(x) \neq f (y)$;
    \item Stein if it is holomorphically separated and holomorphically convex.
\end{itemize}
Note that $X$ is holomorphically convex if and only if it admits a proper map to a Stein space \cite[Example 1]{Cartan}.

The goal of this section is to prove the following result, which is essentially proved in \cite{mok_stein}.
\begin{thm}\label{criterion of Steiness}
Let $X$ be a complex analytic space. Let $\phi$ be a continuous plurisubharmonic exhaustion function on $X$. Let $Z \subset X$ be a closed analytic subspace. Assume that  
\begin{enumerate}
\item $Z$ is Stein,
\item $\phi | _{X \smallsetminus Z}$ is strictly plurisubharmonic.
\end{enumerate}
Then $X$ is Stein.
\end{thm}
Recall that a continuous function $\phi \colon M \to \bR$ is said to be exhaustive if for every $c \in \bR$ the sublevel set $M_c = \{x \in M ; \phi(x) < c \}$ is relatively compact in $M$. Before explaining the proof of the theorem, we start by recalling some definitions and well-known results.

\begin{defn}
Let $Y$ be a Stein open subset of a Stein space $X$. The pair $(Y , X)$ is called Runge if $\cO(X)$ is dense in $\cO(Y)$ in the topology of compact convergence.
\end{defn}

\begin{prop}[Stein {\cite{Stein}}]\label{Stein_exhaustion}
Let $X$ be a complex space and $X_1 \Subset X_2 \Subset \ldots $ be an exhaustion of $X$ by
Stein relatively compact open subsets. If every pair $(X_\nu, X_{\nu + 1})$ is Runge, then $X$ is Stein.
\end{prop}

\begin{thm}[Narasimhan, {\cite[Corollary 1, p.211]{Narasimhan_Levi_problem}}] \label{Narasimhan_Runge}
Let $X$ be a Stein space and $\psi$ be a plurisubharmonic function on $X$. Then, for any real number $c$ the open subset $X_c = \{ x \in X | \psi(x) < c \} \subset X$ is Runge in $X$.
\end{thm}

\begin{thm}[Narasimhan, {\cite[Theorem II]{Narasimhan_Levi_problem}}] \label{Narasimhan_Stein}
A complex space is Stein if and only if it admits a continuous strictly plurisubharmonic exhaustion function.
\end{thm}

\begin{proof}[Proof of Theorem \ref{criterion of Steiness}]
For every $c \in \bR$, let $X_c := \{ x \in X | \phi(x) < c\}$.
It is sufficient to prove that for every $c \in \bR$ the space $X_c$ is Stein. Indeed, it follows then from Theorem \ref{Narasimhan_Runge} that the pairs $(X_c, X_{c^\prime})$ are Runge for any $c < c^\prime$, and by Proposition \ref{Stein_exhaustion} that $X$ is Stein. \\

Since $Z$ is Stein by assumption, it follows from a theorem of Siu \cite{Siu-Stein} that there is a Stein neighborhood $U$ of $Z$ in $X$. Therefore, $U$ admits a continuous strictly plurisubharmonic function $\psi \colon U \to [- \infty, + \infty)$. Let $\theta \colon X \to \bR_+$ be a smooth function that equals $1$ in a neighborhood of $Z$ in $X$ and such that $\mathrm{Supp}(\theta) \subset U$. Let $\tilde{\psi}\colon X \to [- \infty, + \infty)$ be the function that equals $\theta \psi$ on $U$ and zero outside. Then $\tilde{\psi}$ is a continuous function on $X$, which is strictly plurisubharmonic in a neighborhood of $Z$ in $X$ and satisfies $\mathrm{Supp}(\tilde{\psi}) \subset U$.

Fix any $c \in \bR$. Since $X_c $ is relatively compact in $X$, there exists a positive constant $K_c$ such that $K_c \phi + \tilde{\psi}$ is strictly plurisubharmonic on $X_c$. The continuous function $\frac{1}{c - \phi}  + K_c \phi + \tilde{\psi} $ is then strictly plurisubharmonic on $X_c$ and is an exhaustion function since $K_c \phi + \tilde{\psi}$ is bounded. We conclude by Theorem \ref{Narasimhan_Stein} that $X_c$ is Stein.
\end{proof}

\section{Proof of Theorem \ref{mainShaf} in the semisimple case and Theorem \ref{main result} in the quasiunipotent semisimple case}\label{sect:qu stein}
In this section, we first briefly recall the construction of \cite{brunebarbeshaf} of the analytic Shafarevich morphism associated to a bounded rank set of semisimple local systems and prove quasiprojectivity of its target using \Cref{thm:alg} and \Cref{prop:qp}.  We then use the pluriharmonic maps at the archimedean and non-archimedean places associated to a semisimple $\bar \bQ$-local system $V$ on $X$ with quasiunipotent local monodromy to define a pluriharmonic exhaustion on $\tilde{X}^V$.  The strict plurisubharmonicity of the exhaustion will follow from \Cref{sect:KZ}, and this will prove \Cref{main result} in the special case of a nonextendable absolute Hodge subset consisting only of semisimple local systems with quasiunipotent local monodromy.

\subsection{Quasiprojectivity of the target of reductive Shafarevich morphisms}
We first briefly recall the construction of \cite{brunebarbeshaf}.  Let $\Sigma$ be a bounded rank subset of $\cM_B(X)(\bC)^{\ss}$.  In this case, by replacing $\Sigma$ with the semisimplification of its saturation, we may assume $\Sigma$ is absolute $\bar\bQ$-closed, defined over $\bQ$, and $\bR_{>0}$-stable.  Let $f_{KZ}:X\to S$ be a Katzarkov--Zuo reduction of $X$ with respect to $\Sigma$ (at all non-archimedean places).  For any finite collection of $\bC$-local systems in $\Sigma$ underlying $\bC$-VHSs, by taking their direct sum we obtain a $\bC$-VHS with monodromy $\rho:\pi_1(X^\an,x)\to \bGL_r(\bC)$.  The period map then gives a $\rho$-equivariant map
\begin{equation}\label{eq:period}\tilde X^{\rho}\to (\check{D}\times S)^\an.\end{equation}
We may assume $\Sigma$ is nonextendable.  Then by \cite{brunebarbeshaf} we can take such a collection such that \eqref{eq:period} has compact fibers.  There is then a Stein factorization
\[\begin{tikzcd}
\tilde X^{\rho}\ar[rd," \psi",swap]\ar[rr,"\phi"]&&(\check{D}\times S)^\an\\
&\tilde \cY\ar[ur,"\chi",swap]&
\end{tikzcd}\]
and $\psi:\tilde X^{\rho}\to\tilde\cY$ descends to the Shafarevich morphism $\sigma:X^\an\to \cY$ to a generically inertia-free analytic Deligne--Mumford stack $\cY$.
  
\begin{thm}\label{thm:ss qp shaf}There is an algebraic morphism $g: X\to Y$ to a generically inertia-free Deligne--Mumford stack $Y$, unique up to isomorphism (as a map with fixed source), which analytifies to $\sigma:X^\an\to \cY$.  Moreover, $\cY$ has quasiprojective coarse space.
\end{thm}

\begin{proof}We may pass to a finite \'etale cover of $X$ so that $\cY$ is inertia-free (as in \cite{brunebarbeshaf}) and thus an analytic space.  \Cref{setup:twisted map} applies by as in \Cref{baby qp image}, so the algebraicity follows from \Cref{thm:alg}.  It remains to show the quasiprojectivity.  By replacing $X$ with $Y$ we may assume $\phi$ has discrete fibers.  

We now show that we are in \Cref{setup:qp}.  Take $A$ ample on $S$, $\scrM$ to be $\check D\times S$, $\mathscr{N}$ to be the Griffiths bundle on $\check D$, and $\scrL=\mathscr{N}\boxtimes A$.  Then \Cref{setup:alg} is satisfied, and \Cref{alg bundle} applies to $\scrL$.  

We claim that for any generically finite $g:Z\to X$ from a smooth variety $Z$ with a log smooth compactification $\bar Z$, there is a functorial big and nef extension $L_{\bar Z}$ by \Cref{lem:big and nef}:  the nefness is clear. For the bigness, note that on the generic fiber $F$ of $h:Z\ra S$, the corresponding period map on $F$ has generically discrete fibers, and thus by \Cref{lem:big and nef} it follows that $N_{\bar F}$ is big. Since $A$ is ample, it follows that the top intersection product $$\mathrm{c}_1(L_{\bar Z})^{\dim X} = \mathrm{c}_1(A)^{\textrm{codim} F} \cdot \mathrm{c}_1(N_{\bar F})^{\dim F}\cdot \deg(h) > 0,$$ and hence $L$ is big.   Therefore the quasiprojectivity follows from \Cref{prop:qp}.
\end{proof}

\subsection{Currents associated to semisimple $\bar \bQ$-local systems with quasiunipotent monodromy}

Let $X$ be a connected smooth quasiprojective complex algebraic variety. Let $V$ be a semisimple $\bQ$-local system with quasiunipotent local monodromy. One can associate to $V$ a canonical positive $(1,1)$-current on $X$ as follows. Fix $\bar X \supset X$ a smooth projective compactification such that $\bar X \setminus X$ is a normal crossing divisor (one easily check that the currents constructed below do not depend on the compactification $\bar X$). 

For every prime $p$, we let $\omega_V^p $ be the positive closed current on $X$ associated to $V_p := V \otimes_{\bQ} \bQ_p$, see \Cref{nonarchimedean_canonical_current}.  We define $\omega^f_V=\sum_p\omega_V^p$, which is sensible since almost all terms vanish.  On the other hand, considering the associated semisimple complex local system $V_\infty := V \otimes_{\bQ} \bC$, we get from the correspondence recalled in \Cref{sect:harmonic bundles} a well-defined algebraic Higgs bundle $(E, \theta)$ on $X$. We let $\omega_V^\infty := i \mathrm{tr}(\theta \wedge \bar \theta)$.  It is a smooth semipositive closed $(1,1)$-form on $X$, see \cite[Lemma 3.3.3]{Eyssidieux}. Finally, we define $\omega_V$ by summing over all places: $$\omega_V = \omega_V^\infty+\omega_V^f.$$ It is a semipositive $(1,1)$-current with continuous potentials. Observe that by construction, the rank of the kernel of $\omega_V$ is constant on a dense Zariski-open subset of $X$.

If one starts with a semisimple $K$-local system $V$ with quasiunipotent local monodromy instead, for some number field $K$, its restriction of scalars $\mathrm{Res}_{K / \bQ} V$ is a semisimple $\bQ$-local system with quasiunipotent local monodromy (of rank $[K:\bQ] \cdot \rk{V}$). We set $\omega^f_V := \frac{1}{[K:\bQ]} \cdot \omega^f_{\mathrm{Res}_{K / \bQ} V}$ and $\omega_V := \frac{1}{[K:\bQ]} \cdot \omega_{\mathrm{Res}_{K / \bQ} V}$. It follows directly from the definition that for any number field $L \supset K$, 
$\omega^f_{V \otimes_K L} = \omega^f_V$ and $\omega_{V \otimes_K L} = \omega_V$. Since every $\bar \bQ$-local system is defined over some number field, this permits to define the canonical current associated to every semisimple $\bar \bQ$-local system with quasiunipotent local monodromy

Finally, for any finite set $\Sigma \subset \cM_B(X)(\bar \bQ)$ of semisimple local systems with quasiunipotent local monodromy, we define 
$$ \omega^f_\Sigma := \sum_{V \in \Sigma} \omega^f_V \;\;\text{ and }\;\; \omega_\Sigma := \sum_{V \in \Sigma} \omega_V .$$

\subsection{Plurisubharmonic functions associated to semisimple $\bQ$-local systems with quasiunipotent monodromy}
The goal of this section is to prove the following result.

\begin{prop}\label{exhaustion_function}
Let $X$ be a connected smooth quasiprojective complex algebraic variety.  
Let $V$ be a semisimple $\bQ$-local system with quasiunipotent local monodromy. 
Let $\pi \colon \tilde X^{V} \to X$ be the corresponding cover and $\omega_V$ its associated semipositive closed current.
Then there exists a continuous plurisubharmonic function $\phi_V \colon \tilde{X}^V \to \bR_{\geq 0}$ that satisfies $d d ^c \phi_V \geq  \pi^\ast \omega_V $. Moreover, if $V$ is nonextendable, then $\phi_V$ is proper.    
\end{prop}
\begin{proof}

Let $\bar X \supset X$ be a smooth projective compactification such that $\bar X \setminus X$ is a normal crossing divisor. Fix $\tilde{x}_0 \in \tilde{X}^{V}$ with image $x_0 \in X$. Fix an isomorphism $V_{x_0} \simeq \bQ^r$. Let $\rho \colon \pi_1(X, x_0) \to \bGL_r(\bQ)$ the monodromy representation of $V$. 

For every prime number $p$, let $\cN((\bQ_p)^r)$ denote the space of norms on $(\bQ_p)^r$. Thanks to Theorem \ref{existence_pluriharmonic_norm}, there exists a $\rho$-equivariant pluriharmonic map $u_{V}^p \colon \tilde{X}^{V} \to \cN((\bQ_p)^r)$, with finite energy with respect to any Poincaré-type complete Kähler metric on $X$. The function $\phi_{V}^p \colon \tilde{X}^{V} \to \bR_{\geq 0}$ defined by
\[ \phi_{V}^p(x) = 2 \cdot d^2_{\cN((\bQ_p)^r)}(u(x), u(\tilde{x}_0)) \]
is continuous psh and satisfies $d d ^c \phi_V^p \geq  \pi^\ast \omega_V^p $, where 
$\omega_V^p $ is the semipositive closed current on $X$ associated to $V_p$, see \Cref{current_controls_exhaustion}. On the other hand, the complex local system $V \otimes_\bQ \bC$ admits a purely imaginary tame harmonic metric, see \cite[Part 5]{Mochizuki-AMS2}. Letting $\cN(\bC^r)$ denote the space of positive definite hermitian 
forms on $\bC^r$, thanks to \Cref{unipotent_residues_equivalent_finite_energy}, the corresponding $\rho$-equivariant pluriharmonic map $u_{V}^\infty \colon \tilde{X}^{V} \to \cN(\bC^r)$ has finite energy with respect to any Poincaré-type complete Kähler metric on $X$. The function $\phi^\infty_V \colon \tilde{X}^V \to \bR_{\geq 0}$ defined by
\[ \phi_V^\infty(x) = 2 \cdot d^2(u(x), u(\tilde{x}_0)) \]
is continuous psh and satisfies $d d ^c \phi_V^\infty \geq  \pi^\ast \omega_V^\infty $, where 
$\omega_V^\infty $ is the semipositive closed current on $X$ associated to $V$, see \cite[Proposition 3.3.2, Lemme 3.3.4]{Eyssidieux}. We define $\phi_V$ by summing over all places (this makes sense since almost all of the terms are zero).
Then the function $\phi_V \colon \tilde{X}^V \to \bR_{\geq 0}$ is continuous psh and satisfies $d d ^c \phi_V \geq  \pi^\ast \omega_V $.

Assuming that $V$ is nonextendable, let us prove that $\phi_V$ is proper. Since $\pi_1(X)$ is finitely generated, the image $\Gamma$ of $\rho$ takes values in $\bGL_r(\bZ[\frac{1}{N}])$ for a positive integer $N$. Up to replacing $\Gamma$ with a finite index subgroup, therefore replacing $X$ with a finite étale cover, one can assume by Selberg's lemma that $\Gamma$ is torsion-free (this does not change $\tilde X^V$ and $\phi_V$). Moreover, the image of the diagonal embedding $\bGL_r(\bZ[\frac{1}{N}]) \to \bGL_r(\bR) \times \prod_{p | N}  \bGL_r(\bQ_p)$ is discrete, and $\bGL_r(\bR)$ (resp. each $\bGL_r(\bQ_p)$) acts properly on $\cN(\bC^r)$ (resp. each $\cN(\bQ_p^r)$). Therefore, we get a pluriharmonic map
\[ X \to   \Gamma \backslash \left( \cN(\bC^r) \times \prod_{p | N} \cN(\bQ_p^r) \right).  \]
It is sufficient to prove that this map is proper. Consider by contradiction a sequence of points in $X$ that goes to infinity and whose images do not. Up to taking a subsequence, we get a sequence of points of $X$ converging to a point in $\bar X \setminus X$ and whose images converge to a point in $\Gamma \backslash \left( \cN(\bC^r) \times \prod_{p | N} \cN(\bQ_p^r) \right)$. This is in contradiction with Proposition \ref{harmonic_non_converge} below.
\end{proof}

\begin{prop} \label{harmonic_non_converge}
Let $\Delta$ be a NPC space and $\Gamma$ a discrete group acting properly on $\Delta$.
Equip $(\bD^\ast)^r \times \bD^s$ with its complete Poincaré metric, and consider a locally liftable harmonic map of finite energy
\[ f \colon (\bD^\ast)^r \times \bD^s \to   \Gamma \backslash \Delta.  \]
Assume that the local monodromies are infinite. Let $\{z_n\}= \{z_n^1, \cdots, z_n^{r+s} \}$ be a sequence of points of $(\bD^\ast)^r \times \bD^s$ with $\inf_{1 \leq i \leq r} |z_n^i| \to 0$ as $n \to \infty$. Then the sequence $\{ f(z_n)\}$ of points of $\Gamma \backslash \Delta$ does not converge.    
\end{prop}
For variations of Hodge structures, this is \cite[Proposition 9.11]{Giii}.
\begin{proof}
By applying \Cref{energy_control_Lipschitz} to geodesic balls of radius $1$ in $(\bD^\ast)^r \times \bD^s$, it follows that the map $f$ is Lipschitz continuous in a neighborhood of the origin (with respect to the Poincaré hyperbolic metric on $(\bD^\ast)^r \times \bD^s$). 

Up to extracting a subsequence, we can assume that one of the coordinates $z_n^i$---say $z_n^1$---tends to $0$. We assume by contradiction that $f(z_n)$ has a limit $w$ in $\Gamma \backslash \Delta$. Let $T_1$ be the image in $\Gamma$ of the loop corresponding to $(1, 0, \ldots, 0)$ through the canonical isomorphism $\pi_1 \left((\bD^\ast)^r \times \bD^s \right) = \bZ^r$. Let $\bar w$ be a preimage of $w$ in $\Delta$. Let $U$ be a neighborhood of $w$ and let $\bar U$ be a neighborhood of $\bar w$, such that the projection of $\bar U$ in $\Gamma \backslash \Delta$ is contained in $U$. We can moreover assume that $T_1(\bar U)$ is at a stricly positive distance $\delta$ of $\bar U$. Let $\bD^\ast = \bD^\ast \times \{0\} \times \ldots \times \{0\} \subset (\bD^\ast)^r \times \bD^s$ and let $\gamma_n$ be the circle of center $0$ 
and radius $|z_n|$ in $\bD^\ast$. We choose a local lift $\bar{f}$ of $f$ in $\Delta$ along the loop $\gamma_n \times \{z_n^2\} \times \ldots \times \{z_n^{r+ s} \}$. We denote by $\bar w_n$ the image of $z_n$ and we assume that it is contained in $\bar U$, for sufficiently large $n$. After a loop, the multivalued function $\bar f$ takes the value $T_1(\bar w_n)$. Hence, the distance $d_\Delta(\bar w_n, T_1(\bar w_n))$ is at least $\delta$. On the other hand, this distance $d_\Delta(\bar w_n, T_1(\bar w_n))$ is bounded by a constant times the length of $\gamma_n$. Since this length tends to zero, we get a contradiction.
\end{proof}

\subsection{Proof}

We now complete the proof of \Cref{main result} in the quasiunipotent semisimple case.  Note that we will not need to use the algebraicity result of \Cref{thm:ss qp shaf}, although its use would provide a slight simplification.

\begin{defn}
Let $X$ be a connected normal algebraic space and $\Sigma \subset \cM_B(X)(\bC)$ a set
of complex local systems. We say $\Sigma$ is generically large if there exists a nowhere dense closed analytic subset $W \subset X^{\an}$ such that for every non-constant morphism $g \colon Z \to X$ from a connected normal algebraic space $Z$ and whose image is not contained in $W$, $g^\ast \Sigma$ contains a nontrivial local system.
Assuming that the analytic $\Sigma$-Shafarevich morphism exists, $\Sigma$ is generically large exactly when the analytic $\Sigma$-Shafarevich morphism is a biholomorphism an a dense analytic Zariski open subset, in which case $W$ can be taken to be the complement.   
\end{defn}

\begin{prop}\label{reduction_to_finitely_many_rational_representations}
Let $X$ be a connected normal complex algebraic space and $\Sigma \subset \cM_B(X)(\bC)$ an absolute  $\bar\bQ$-constructible subset containing only semisimple local systems with quasiunipotent local monodromy. If $\Sigma$ is nonextendable (respectively nonextendable and generically large), then there exists a finite subset $\Sigma_0 \subset \Sigma(\bar \bQ)$ which is nonextendable (respectively nonextendable and generically large).
\end{prop}
\begin{proof}
Assume that $\Sigma$ is nonextendable and let us prove that there exists a finite subset $\Sigma_0 \subset \Sigma(\bar \bQ)$ which is nonextendable. One can assume that $X$ is smooth. Thanks to \Cref{reduction_to_unipotent_monodromy}, there exists a finite étale cover $X^\prime \to X$ such that the pull-back $\Sigma^\prime$ of $\Sigma$ to $X^\prime$ consists only of local systems with unipotent local monodromy. Since by \Cref{pullback-nonextendable} $\Sigma^\prime$ is nonextendable if and only if $\Sigma$ is nonextendable, one can assume from the beginning that all local monodromies are unipotent. It follows then from \Cref{nonextendable_Zariski-closure} that $\Sigma(\bar \bQ)$ is nonextendable, and from \Cref{nonextendable_reduction_to_finite} that there exists a finite subset $\Sigma_0 \subset \Sigma(\bar \bQ)$ which is nonextendable. 

Assume that $\Sigma$ is nonextendable and generically large. Since $\Sigma(\bar \bQ)$ is Zariski-dense in $\Sigma$, and both $\Sigma(\bar \bQ)$ and $\Sigma$ are nonextendable, they define the same (analytic) Shafarevich morphism $\Sh_\Sigma \colon X^{\an} \to S$, which by assumption is a holomorphic proper modification. Therefore, there exists $x \in X$ such that $\{x\} = \Sh_{\Sigma(\bar \bQ)}^{-1} (\Sh_{\Sigma(\bar \bQ)}(x))$. By the preceding discussion, there exists a finite subset $\Sigma_0 \subset \Sigma(\bar \bQ)$ which is nonextendable. Consider for every finite subset $\Sigma_1 \subset \Sigma(\bar \bQ)$ containing $\Sigma_0$ the closed algebraic subset $\Sh_{\Sigma_1}^{-1} (\Sh_{\Sigma_1}(x))$ of $X$. By assumption, the intersection of those subsets is equal to $\{x \}$. By noetherianity, there exists a finite subset $\Sigma_1 \subset \Sigma(\bar \bQ)$ containing $\Sigma_0$ such that $\{x\} = \Sh_{\Sigma_1}^{-1} (\Sh_{\Sigma_1}(x))$. It follows that $\Sh_{\Sigma_1}$ is a modification, in other words $\Sigma_1$ is generically large. Moreover, $\Sigma_1$ is nonextendable since it contains the nonextendable $\Sigma_0$.
\end{proof}

\begin{rem}
    For the next proposition, we will need to use the results of \Cref{sect:KZ} for all valuations $v$ on $\bar\bQ$ simultaneously.  We summarize the relevant statements.  For any $\Sigma\subset \cM_B(X)(\bar \bQ)$ consisting of semisimple local systems with quasiunipotent local monodromy, we define the leaves of the Katzarkov--Zuo foliation of $\Sigma$ to be the intersection of the $\Sigma$-leaves for all valuations $v$ on $\bar \bQ$.  As in \Cref{lem KZ1}, there is always a finite subset $\Sigma_0\subset\Sigma$ with the same leaves.  If $\Sigma$ is the set of $\bar\bQ$-points of an absolute $\bar\bQ$-constructible subset of $\cM_B(X)(\bC)$, then by \Cref{thm KZ integrable} the leaves are algebraic, and, as in \Cref{KZ reduction}, there is an algebraic map $f:X\to Y$ integrating the absolute foliation.  Such a map is called a Katzarkov--Zuo reduction. 
\end{rem}

\begin{prop}\label{canonical_current_generically_kahler}
Let $X$ be a connected smooth quasiprojective complex algebraic variety. Let $\Sigma \subset M_B(X)(\bC)$ a weak absolute Hodge subset containing only semisimple local systems with quasiunipotent local monodromy. If $\Sigma$ is generically large, then there exists a finite subset $\Sigma_0 \subset \Sigma(\bar \bQ)$ such that $\omega_{\Sigma_0}$ is a Kähler form on a Zariski-dense open subset of $X$.    
\end{prop}
\begin{proof}
Thanks to the remark above, the Kazarkov-Zuo foliation associated to $\Sigma(\bar \bQ)$ is algebraically integrable and there exists a finite subset $\Sigma_0 \subset \Sigma(\bar \bQ)$ such that the Katzarkov-Zuo foliations of $\Sigma$ and $\Sigma_0$ coincide. 

If the generic leaves of the Katzarkov-Zuo foliation of $\Sigma_0$ are zero-dimensional, then the sum over the finite places of the canonical currents associated to $\Sigma_0$ is already a Kähler form on a Zariski-dense open subset of $X$, hence the result follows. We can therefore assume that the Katzarkov-Zuo foliation of $\Sigma_0$ has only positive dimensional leaves. Let $F$ be a desingularization of a fiber of a  Katzarkov-Zuo reduction of $\Sigma_0$ (see \Cref{KZ reduction}). Thanks to \cite[Theorem 6.8]{brunebarbeshaf}, the restriction of $\Sigma$ to $\cM_B(F)$
consists of finitely many points. In particular these points are defined over $\bar \bQ$. For each of them, take a preimage in $\Sigma$ defined over $\bar \bQ$, and let $\Sigma_1 \subset \Sigma(\bar \bQ)$ be the finite subset obtained by adding all these new elements to $\Sigma_0$. By construction, the restriction of $\Sigma_1$ to $\cM_B(F)$ consists of finitely many points that are fixed by the $\bR_{>0}$-action, hence underlie $\bC$-VHSs. Moreover, the monodromy representation of the sum of these $\bC$-VHSs has an integral structure (see \cite[Proof of Theorem 9.1]{brunebarbeshaf}), hence the corresponding period map descends to $F$.

By construction, the kernel of $\omega^f_{\Sigma_1}$ coincides on a Zariski-dense open subset of $X$ with the kernel of the differential of the Katzarkov-Zuo reduction of $\Sigma_1$, and the kernel of $\omega_{\Sigma_1}^ \infty$ coincides with the intersection of the kernels of the Higgs field of the Higgs bundles associated to the elements of $\Sigma_1$.
On the other hand, the kernel of the restriction of $\omega_{\Sigma_1}^ \infty$ to $F$ coincides with the kernel of the differential of the period map. Since $\Sigma$ is generically large by assumption, it follows that when $F$ is a desingularization of a generic fiber of a Katzarkov-Zuo reduction of $\Sigma$, the kernel of $\omega_{\Sigma_1}$ is zero at a generic point of $F$. Therefore, the kernel of $\omega_{\Sigma_1}$ is in fact zero on a dense Zariski-open subset of $X$.
\end{proof}

\begin{thm}
Let $X$ be a connected normal algebraic space and $\Sigma \subset M_B(X)(\bC)$ be a nonextendable weak absolute Hodge subset containing only semisimple local systems with quasiunipotent local monodromy. Then there exists a finite subset $\Sigma_0 \subset \Sigma(\bar \bQ)$ such that $\tilde{X}^{\Sigma_1}$ is holomorphically convex for every subset $\Sigma_1 \subset \Sigma$ containing $\Sigma_0$. In particular $\tilde{X}^\Sigma$ is holomorphically convex.
\end{thm}
\begin{proof}
By induction on the dimension on $X$. We start with an easy observation. Let $f \colon T \to S$ be a holomorphic fibration between connected normal analytic spaces, so that the induced map $f^\ast \colon \cM_B(S) \to \cM_B(T)$ is a closed immersion. Let $\Sigma_1 \subset \cM_B(S)$ be any subset. Then the induced holomorphic map $\tilde S^{f^\ast \Sigma_1} \to \tilde T^{\Sigma_1}$ is a fibration, therefore $\tilde T^{f^\ast \Sigma_1}$ is holomorphically convex if and only if $\tilde S^{\Sigma_1}$ is holomorphically convex. In particular, thanks to \Cref{pullback-nonextendable}, one can assume in addition that $X$ is smooth quasiprojective. Moreover, thanks to the existence of analytic Shafarevich morphism \cite{brunebarbeshaf} and Proposition \ref{algebraizing_fibration}, one can assume in addition that $\Sigma$ is generically large.

Let $\sh_\Sigma \colon X^{\an} \to S = \Sh_\Sigma(X)$ be the (analytic) Shafarevich morphism associated to $\Sigma$. Up to replacing $X$ with a finite étale cover, one can assume that $\Sigma$ belongs to the image of the algebraic map $\cM_B(S) \to \cM_B(X)$, see \cite[Theorem B]{brunebarbeshaf}.

Thanks to Proposition \ref{reduction_to_finitely_many_rational_representations}, there exists a finite subset $\Sigma_0 \subset \Sigma(\bar \bQ)$ which is generically large and nonextendable. Moreover, thanks to Proposition \ref{canonical_current_generically_kahler}, one can assume that $\omega_{\Sigma_0}$ is a Kähler form outside a strict closed algebraic subvariety $Z$ of $X$.

Let $Z^\prime \to Z$ be the normalization of $Z$ and $\{Z^\prime_i\}_{i \in I}$ the connected components of $Z^\prime$. The restriction of $\Sigma$ to any connected component of $Z^\prime$ is again a nonextendable weak absolute Hodge subset containing only semisimple local systems with quasiunipotent local monodromy. Therefore, by induction, up to enlarging $\Sigma_0$, one can assume that for every $i \in I$ the corresponding étale cover $\widetilde{Z^\prime_i}^{\Sigma_0}$ is holomorphically convex (or equivalently that $\widetilde{\Sh_{\Sigma}(Z^\prime_i)}^{\Sigma_0}$ is Stein). Since by assumption $\Sigma$ is $\bQ$-constructible, one can also assume that $\Sigma_0$ is stable by Galois conjugation.

We will prove that $\tilde X^{\Sigma_1}$ is holomorphically convex for every set $\Sigma_0 \subset \Sigma_1 \subset \Sigma$. Since $\sh_\Sigma$ induces a proper holomorphic map $\tilde X^{\Sigma_1} \to \tilde S^{\Sigma_1}$, it will be sufficient to prove that $\tilde S^{\Sigma_1}$ is Stein. Since $\tilde S^{\Sigma_1}$ is a cover of $\tilde S^{\Sigma_0}$, it is sufficient to prove that $\tilde S^{\Sigma_0}$ is Stein. Let $V$ be the $\bQ$-local system obtained by summing all the elements in $\Sigma_0$, and $\rho\colon \pi_1(X) \to \bGL_r(\bQ)$ its monodromy representation . Let $\pi \colon \tilde X^{V} = \tilde X^{\Sigma_0}\to X$ be the corresponding cover and $\omega_V$ its associated semipositive closed current. Thanks to Proposition \ref{exhaustion_function}, there exists a proper continuous plurisubharmonic function $\phi_V \colon \tilde X^{\Sigma_0} \to \bR_{\geq 0}$ that satisfies $d d ^c \phi_V \geq  \pi^\ast \omega_V $.

By the maximum principle, the psh function $\phi_V$ is constant on the (compact connected) fibers of the holomorphic fibration $\tilde X^{\Sigma_0} \to \tilde S^{\Sigma_0}$. Therefore,
there exists a (necessarily proper continuous psh) map $\psi_V \colon \tilde S^{\Sigma_0} \to \bR_{\geq 0}$ whose composition with $\tilde X^{\Sigma_0} \to \tilde S^{\Sigma_0}$ gives $\phi_V$. Let $T \subset S$ be the image of $Z$ by the proper holomorphic map $\sh_\Sigma$. By construction, $\psi_V$ is strictly psh on the complementary of the preimage $\tilde{T}$ of $T$ in $\tilde S^{\Sigma_0}$.  Let us prove that $\tilde{T}$ is Stein. Thanks to Theorem \ref{criterion of Steiness}, this will imply that $\tilde S^{\Sigma_0}$ is Stein and finish the proof.

Let $T^\prime \to T$ be the normalization of $T$ and $\{T^\prime_j\}_{j \in J}$ the connected components of $T^\prime$. Thanks to \cite{Narasimhan_normalization_Stein}, $\tilde{T}$ is Stein if and only if its normalization is Stein, and the latter is nothing but the base-change of $\tilde S^{\Sigma_0} \to S$ by the composition $T^\prime \to T \to S$. It follows that the normalization of $\tilde{T}$ is a disjoint union of analytic spaces biholomorphic to $\widetilde{T^\prime_j}^{\Sigma_0}$ for some $j \in J$. The surjective proper holomorphic map $Z \to T$ induces a surjective proper holomorphic map $Z^\prime \to T^\prime$. Therefore, for every connected component $T^\prime_j$ of $T^\prime$, there exists $i \in I$ such that $\sh_{\Sigma}$ induces a surjective proper holomorphic map $Z^\prime_i \to T^\prime_j$ and therefore finite surjective holomorphic maps $\Sh_{\Sigma}(Z^\prime_i) \to T^\prime_j$ and $\widetilde{\Sh_{\Sigma}(Z^\prime_i)}^{\Sigma_0} \to \widetilde{T^\prime_j}^{\Sigma_0}$. Since $\widetilde{\Sh_{\Sigma}(Z^\prime_i)}^{\Sigma_0}$ is Stein for every $i \in I$, it follows from \cite[Theorem 2]{Narasimhan_normalization_Stein} that $\widetilde{T^\prime_j}^{\Sigma_0}$ is Stein for every $j \in J$, so that $\widetilde{T^\prime}^{\Sigma_0}$ is Stein.
\end{proof}

\begin{cor}\label{ss qu large case of shaf}
Let $X$ be a connected normal algebraic space and $\Sigma \subset M_B(X)(\bC)$ be a large nonextendable weak absolute Hodge subset containing only semisimple local systems with quasiunipotent local monodromy. Then there exists a finite subset $\Sigma_0 \subset \Sigma(\bar \bQ)$ such that $\tilde{X}^{\Sigma_1}$ is Stein for every set $\Sigma_0 \subset \Sigma_1 \subset \Sigma$. In particular $\tilde{X}^\Sigma$ is Stein. 
\end{cor}

\section{Proof of Theorem \ref{mainShaf}, Theorem \ref{main result}, and Theorem \ref{main baby}}\label{sect:proofs}
Using the results of the previous section together with the period maps of the mixed variations from \Cref{thm:versal} and \Cref{abs Hodge contain R*}, we construct the Shafarevich morphism in general and prove its algebraicity using \Cref{thm:alg}.  We use the same period maps and the affineness of mixed period domains over the associated graded period domains to deduce Steinness of the trivializing cover of a large nonextendable weak absolute Hodge subset. 

\subsection{Proof of \Cref{mainShaf}}\label{proof:shafalg}Let $X$ be a connected normal algebraic space and $\Sigma\subset \cM_B(X)(\bC)$ a set of local systems of bounded rank.  By \Cref{Shaf pass to finite} and \Cref{extend to nonextend}, we may assume $\Sigma$ is nonextendable.  By \Cref{Shaf reduce to wabs} we may assume $\Sigma$ is a closed weak absolute Hodge subset.  If $\Sigma_j$ are the $\bQ$-irreducible components of $\Sigma$ and $s_j:X\to Y_j$ is a $\Sigma_j$-Shafarevich morphism, then the Stein factorization of $\prod s_j:X\to \prod Y_j$ is a $\Sigma$-Shafarevich morphism.  Thus we may assume $\Sigma$ is $\bQ$-irreducible.  Let $\Sigma_j$ now be the geometric irreducible components of $\Sigma$.  If some $\Sigma_j$ is generically semisimple with quasiunipotent local monodromy, the same is true for every $\Sigma_j$, and we are done by \Cref{thm:ss qp shaf}.  Thus, by \Cref{from quasiunipotent to unipotent} and \Cref{Shaf pass to finite}, we may assume each $\Sigma_j$ contains a point $V_0^j\in\Sigma_j$ underlying a $\bC$-VHS with unipotent local monodromy at which $\Sigma_j$ is formally Hodge.  Let $\cZ\subset\cM_B(X)$ (resp. $\cZ_j\subset\cM_B(X)$) be the reduced closed substack with underlying set of points $\Sigma$ (resp. $\Sigma_j$).  Let $\hat V^j$ be the universal $\hat\cO_{\hat V^j}:=\hat\cO_{\cZ_j,V_0^j}$-AVMHS and set $V_k^j:=\hat V^j/\fm_{\hat\cO_{V_0}}^{k+1}\hat V^j$.  Set $U_k=\bigoplus_j V_k^j$, $\hat U= \bigoplus_j \hat V^j$, and $\hat\cO_{\hat U}=\prod_j\hat\cO_{\hat V^j}$.  By \Cref{nonextendable_Zariski-closure} and \Cref{nonextendable_reduction_to_finite} we may assume $U_k$ is nonextendable on $X$ for $k\gg 0$.

Let $s:X\to Y_0$ be the $U_0$-Shafarevich morphism of \Cref{thm:ss qp shaf}.  By Selberg's lemma and \Cref{Shaf pass to finite} again, we may assume $Y_0$ is inertia-free.  Let $D_k$ be the period domain\footnote{The mixed period domain for the $\bC$-VMHS $U_k$ keeps track of both $F^\bullet$ and $F'^\bullet$, the latter using conjugate coordinates so the map is holomorphic on both factors.  It is perhaps easier to think of $D_k$ as recording the $F^\bullet$ of the associated $\bR$-VMHS $U_k\oplus \overline U_k$.} corresponding to $U_k$.  The period maps of the $U_k$ yield a commutative square
\[\begin{tikzcd}
    \tilde X^{U_k}\ar[d]\ar[r,"\phi_k"]&D_k\ar[d]\\
    \tilde Y_0^{U_0}\ar[r]&D_0
\end{tikzcd}\]
and morphism
\begin{equation}\label{eq:period2}\psi_k:=(s\circ p_k)\times \phi_k:\tilde X^{ U_k}\to Y_0^\an \times  D_{k}\end{equation}
where $p_k:\tilde X^{U_k}\to X^\an$ is the covering map.
\begin{prop}\label{cpt fibers}For $k\gg 0$, the connected components of the fibers of $\psi_k$ are compact.  
\end{prop}
\begin{proof}
Let $y\in Y$ and let $y\in B\subset Y_0^\an$ be a contractible neighborhood equipped with a lift to $\tilde Y_0^{U_0}$.  Let $X_B:=s^{-1}(B)$ with embedding $i_B:X_B\to X$ and $f_B:X_B\to B$ the restriction of $f$.

\vskip1em
    \noindent\emph{Step 1.}  Fix a basepoint $x\in X_B$.  For infinitely many $k\in\bN$, $i_B^*U_k$ has a subquotient $E_k$ which has an integral structure.  Moreover, for $k\gg 0$, we have
\[\ker(\rho_{E_k,x})=\ker(\rho_{i_B^*U_k,x})=\ker(\rho_{i_B^*\hat U,x})\]
where for a local system $L$ on $X_B$, we denote by $\rho_{L,x}$ the monodromy representation of $L$ at $x$.

\begin{proof}
    The stack $\cM_B(X_B)$ is an algebraic stack defined over $\bQ$, and there is a miniversal family $(\hat\cO_{\hat W},\hat W)$ for the Zariski closure of $i_B^*\cZ$ at the trivial local system $W_0=i_B^*V_0^j$ which is defined over $\bQ$.  Moreover, each finite order quotient $W_\ell$ has a $\bZ$-structure since $W_0$ does.  By miniversality, there is an injective morphism $\hat\cO_{\hat W}\to \hat\cO_{\hat U}$ such that $i_B^*\hat U \cong \hat\cO_{\hat U}\otimes_{\hat\cO_{\hat W}}\hat W$.  For every $\ell$ there is a $k$ such that $\hat\cO_{\hat W}\cap \fm_{\hat\cO_{\hat U}}^{k+1}\subset\fm_{\hat\cO_{\hat W}}^{\ell+1}$ (by Krull's theorem and the finite-dimensionality of $\hat\cO_{\hat W}/\fm_{\hat\cO_{\hat W}}^{\ell+1}$), so we have a diagram
    \[\begin{tikzcd}
                W_\ell&\ar[l,two heads]\hat W/\Big(\hat\cO_{\hat W}\cap\fm_{\hat\cO_{\hat U}}^{k+1}\Big)\hat W\ar[r,hook]&i_B^*U_k.
    \end{tikzcd}\]
    We take $E_k$ to be $W_\ell$, for $\ell$ ranging over $\bN$.  The remaining claim then follows from:
    \begin{lem}\label{ker stab}  For $k,\ell\gg 0$, the subgroups $\ker\rho_{i^*U_k,x},\ker\rho_{W_\ell,x}\subset\pi_1(X_B,x)$ stabilize to the same subgroup.
\end{lem}
\begin{proof}
    We have
    \[\bigcap_{k\geq 0}\ker\rho_{i^*U_k,x}=\ker\rho_{i^*\hat U,x}=\ker\rho_{\hat W,x}=\bigcap_{\ell\geq 0}\ker\rho_{W_\ell,x}\]
    where the middle inequality comes from the fact that $\hat W$ and $i^* \hat U$ have the same Zariski closure in $\cM_B(X_B)$, by construction.  If $W_\mathrm{gen}\in i^*\Sigma$ is very general, then we also have $\ker\rho_{W_\mathrm{gen},x}=\ker\rho_{\hat W,x}$.  

    Since each $W_{\ell}$ has trivial semisimplification, $\img \rho_{W_\ell,x}$ is torsion-free for all $\ell$.  It follows that $\ker\rho_{W_\ell,x}$ is saturated in $\pi_1(X_B,x)$, meaning that the only subgroup of $\pi_1(X_B,x)$ containing $\ker\rho_{W_\ell,x}$ as a finite-index subgroup is $\ker\rho_{W_\ell,x}$ itself.  It further follows that $\rho_{W_{\mathrm{gen}},x}(\ker\rho_{W_\ell,x})$ is saturated in $\img\rho_{W_{\mathrm{gen}},x}$, and the following claim then implies that $\ker\rho_{W_\ell,x}$ stabilizes.
\begin{claim}Let $\mathbf{P}\subset\bGL_r$ be a unipotent complex algebraic subgroup, $\Gamma\subset \mathbf{P}(\bC)$ a finitely generated subgroup.  Then any decreasing sequence of saturated subgroups of $\Gamma$ stabilizes.
\end{claim}
The same argument applied to $\ker\rho_{i^*\hat U_k,x}$ completes the proof.
\end{proof}
 
\end{proof}

\vskip1em
    \noindent\emph{Step 2.}  The monodromy $\Gamma_{B,k}:=\img(\rho_{i_B^*U_k})$ of $i_B^*U_k$ acts properly on $D_k$ for $k\gg0$.  

        \begin{proof}
    There is a natural action of $\exp(W_{-1}\End(U_{k,x}))$ on $D_{k}$, and this action is proper\footnote{Again, $D_k$ records both $F^\bullet$ and $F'^\bullet$; the stabilizers of the $\exp(W_{-1}\End(V_{U,x}))$ action are in fact trivial since $W_{-1}\cap F^0\cap F'^0=0$.} \cite[Proposition 3.7]{bbkt}.  Because $\gr^Wi_B^*U_{k}$ is pulled back from $B$ by construction, $\Gamma_{B,k}$ is contained in $\exp(W_{-1}\End(U_{k,x}))$.  Let $E_k$ be as in the previous step, and $L_k\subset i_B^*U_{k}$ the sub-local system of which $E_k$ is a quotient.  Consider the subspace $H_k\subset W_{-1}\End(U_{k,x})$ of endomorphisms which preserve the inclusion $L_{k,x}\to U_{k,x}$, and for which the restriction to $L_{k,x}$ preserves the quotient $L_{k,x}\to E_{k,x}$.  There is then a natural map $H_k\to\End(E_{k,x})$.  For $k$ as in the previous step, the monodromy $\Gamma_{B,k}$ is contained in $\exp(H_k)$, is isomorphic to its image in $\GL(E_{k,x})$, and has discrete image there because it preserves an integral structure.  It follows that $\Gamma_{B,k}\subset \exp(W_{-1}\End(U_{k,x}))$ is discrete, and therefore by the above the action is proper.
    \end{proof}

    \vskip1em
    \noindent\emph{Step 3.}  The map $X_B\to \Gamma_{B,k}\backslash(B\times D_k)$ is proper for $k\gg0$.
    \begin{proof}
First, observe that the proof of \cite[Lemma 2.2]{bbt2} is easily adapted to show the following valuative criterion:
\begin{lem}
    Let $\cX,\cY\to\cS$ be morphisms of definable analytic spaces, and assume $\cX\to\cS$ factorizes as $\cX\to\bar\cX\to\cS$ for a definable analytic space $\bar X$ where $\cX\to \bar\cX$  is the embedding of a dense definable Zariski open subset and $\bar\cX\to\cS$ is proper.  Then an $\cS$-morphism $f:\cX\to\cY$ is proper if and only if any definable analytic map $v:\bD^*\to \cX$ extends as soon as $f\circ v:\bD^*\to \cY$ does. 
\end{lem}
Note that $X_B\to B$ satisfies the relative compactifiability condition of the lemma since $X\to Y_0$ is compactifiable.  By \cite{BMull} and the admissibility condition, the period map $\phi_k:\tilde X^{U_k}\to D_k$
is $p_k$-definable.  The properness of $f_B:X_B\to \Gamma_{B,k}\backslash(B\times D_k)$ can be checked analytically locally on the target.  Letting $\mathbb{B}\subset \Gamma_{B,k}\backslash(B\times D_k)$ be a small ball with its natural structure as a $\bR_\an$-definable analytic variety, we may apply the lemma to the map $f_B^{-1}(\mathbb{B})\to \mathbb{B}$.  For any definable analytic $v:\bD^*\to f_B^{-1}(\mathbb{B})$ whose composition $\bD^*\to \mathbb{B}$ extends, $v^*i_B^*U_k$ has no monodromy, so by the nonextendability of $U_k$ such a disk must extend, and the claim follows.        
    \end{proof}

\vskip1em
    \noindent\emph{Step 4.}  End of proof.
    
The preimage of $B\times D_k$ in $\tilde X^{U_k}$ is a disjoint union of translates of $\widetilde{X_B}^{i^*_BU_k}$, and by the previous step the map $\widetilde{X_B}^{i^*_BU_k}\to B\times D_k$ is proper for $k\gg0$.  Thus, every connected component of $\psi_k$ is compact for $k\gg0$.
\end{proof}
It follows that for $k\gg 0$ there is a Stein factorization
\[\begin{tikzcd}
\tilde X^{U_k}\ar[rd," \tilde \sigma_k",swap]\ar[rr,"\psi_k"]&&Y_0^\an\times D_{k}\\
&\tilde \cY_k\ar[ur,"\chi_k",swap]&
\end{tikzcd}\]
which descends to a proper morphism $\sigma_k:X^\an\to\cY_k$ to an analytic Deligne--Mumford stack.  
\begin{prop}In the above setup, for $k\gg 0$ there is an algebraic morphism $s_k: X\to Y_k$, unique up to isomorphism (as a map with fixed source), which analytifies to $\sigma_k:X^\an\to \cY_k$.  For $k\gg0$, this map stabilizes to $s_\infty:X\to Y_\infty$, and this is a $\Sigma$-Shafarevich morphism.  
\end{prop}

\begin{proof}
    As observed above, the period map of a $\bC$-AVMHS is $p_k$-definable, so by \Cref{thm:alg} $\sigma_k$ is algebraic for $k\gg0$.  By noetherianity, the map must stabilize.  By construction, $\hat U$ is pulled back from $Y_\infty$, as is $\Sigma$, since the image of $s_\infty^*:\cM_B(Y_\infty)\to\cM_B(X)$ contains the union of the Zariski closures of each $\hat V^j$, which is $\Sigma$.  Moreover, $\hat U$ is nonextendable on $X$ by construction, hence also on $Y$ by \Cref{pullback-nonextendable}, and the inertia of $Y$ is trivial by construction.  Finally, for any algebraic morphism $g:Z\to X$ with connected source for which $g^*\Sigma$ is trivial, $g^*\hat U$ is trivial.  This means $g$ must factor through a fiber of $s_0$ (since $g^*V_0$ is trivial), but also that $g$ lifts to $\tilde X^{U_k}$ and factors through a fiber of the period map $\phi_k$ for each $k$, since a $\bC$-AVMHS with no monodromy on an algebraic variety must be trivial.  But then $g$ factors through a fiber of $\psi_k$ for each $k$.  Thus, $g$ factors through a fiber of $s_\infty$.  
\end{proof}

This completes the proof of \Cref{mainShaf}.\qed
\subsection{Proof of \Cref{main result} and \Cref{main baby}}  Given \Cref{MB is weak abs hodge} and \Cref{abs Hodge implies weak abs Hodge}, it suffices to prove the following:
\begin{prop}Let $X$ be a connected normal algebraic space and $\Sigma\subset\cM_B(X)(\bC)$ be a nonextendable weak absolute Hodge subset.  Then $\tilde X^\Sigma$ is holomorphically convex. 
\end{prop}
\begin{proof}
    We may assume $\Sigma$ is closed since this does not effect the cover.  Let $\Sigma_j$ be the $\bQ$-irreducible components of $\Sigma$; each $\Sigma_j$ is a weak absolute Hodge subset.  If $s_j:X\to Y_j$ is the $\Sigma_j$-Shafarevich morphism, and $\Sigma_{Y_j}\subset\cM_B(Y_j)(\bC)$ is such that $s_j^*\Sigma_{Y_j}=\Sigma_j$, then $\prod s_j:X\to \prod Y_j$ is proper and $\tilde X^\Sigma$ is a component of the base-change of $\prod \tilde Y_j^{\Sigma_{Y_j}}\to \prod Y_j$.  Note that each $\Sigma_{Y_j}$ is a weak absolute Hodge subset.  Since a closed subspace of a Stein space is Stein and a product of Stein spaces is Stein, by replacing $X$ with $Y_j$ and $\Sigma$ with $\Sigma_{Y_j}$, it suffices to assume $\Sigma$ is a $\bQ$-irreducible nonextendable large weak absolute Hodge subset, and we must show $\tilde X^\Sigma$ is Stein.  Note that since we can freely pass to any finite \'etale cover of $X$ arising from a quotient the image of the monodromy representation of $\bigoplus_{V\in\Sigma}V$, $X$ can still be taken to be an algebraic space.

   Let $\Sigma_j$ now be the geometric irreducible components of $\Sigma$.  As in the previous section we may assume, after passing to a finite \'etale cover, that for each $j$ there is a point $V_0^j$ of $\Sigma_j$ underlying a $\bC$-VHS with unipotent local monodromy at which $\Sigma_j$ is formally Hodge.  By definition, there is a weak absolute Hodge subset $T\subset\Sigma^{\qu,\ss}$ containing each $V_0^j$.  Let $s:X\to Y$ be the $T$-Shafarevich morphism.  According to \Cref{ss qu large case of shaf}, $\tilde Y^{T}$ is Stein.  Let $\cZ$ (resp. $\cZ_j$) be the closed substack with underlying set of points $\Sigma$ (resp. $\Sigma_j$), let $(\hat\cO_{\hat V^j},\hat V^j)$ be the miniversal family for $\cZ_j$ at $V_0^j$, $V_k^j$ the finite quotients, and $\hat\cO_{\hat U}=\prod_j\hat\cO_{\hat V^j}$, $\hat U=\bigoplus \hat V^j$, $U_k=\bigoplus_j V_k^j$.  As before, $U_k$ is nonextendable for $k\gg0$, and by taking the period map of $U_k$ we have a commutative square\[\begin{tikzcd}
    \tilde X^{\Sigma}\ar[d,"\tilde s"]\ar[r,"\phi_k"]&D_k\ar[d]\\
    \tilde Y^{T}\ar[r]&D_0
\end{tikzcd}\]
Form the relative period map 
\begin{equation}\label{eq:period3}\psi_k:=\tilde s\times \phi_k:\tilde X^{ \Sigma}\to \tilde Y^{T} \times_{D_0} D_{k}.\notag\end{equation}
Let $y\in Y$, let $y\in B\subset Y^\an$ be a contractible Stein neighborhood equipped with a lift to $\tilde Y^{T}$, and again let $X_B=s^{-1}(B)$.  The proof of \Cref{cpt fibers} shows that $\widetilde{X_B}^{i^*_BU_k}\to B\times_{D_0}D_k$ is proper (in fact finite, since the fibers are discrete) for $k\gg 0$.  Since $\tilde Y^{T}\times_{D_0}D_k$ is affine over $\tilde Y^{T}$, it follows that $B\times_{D_0}D_k$ is Stein.  Finally, the preimage of $\psi_k$ over $B\times_{D_0}D_k$ is a disjoint union of lifts of $(X_B)^{i_B^*U_k}$ for $k\gg0$ so by \cite[Th\'eor\`eme 1]{lebarz} we conclude that $\tilde X^\Sigma$ is Stein.
\end{proof}

 \bibliography{biblio.shafarevich}
 \bibliographystyle{plain}
\end{document}